%% file: PS-23-LM.tex
\documentclass[openany]{now}

\usepackage{amsmath}
\usepackage{amssymb,amsbsy}
\usepackage{graphicx}
\usepackage{booktabs}
\usepackage{dsfont}
\usepackage{natbib}
\usepackage{pdfsync}
\usepackage{hyperref}
\usepackage{tikz}
\usepackage{pgfplots}
\usepackage{caption}
\usepackage{cleveref}
\usepackage{upgreek}
\usepackage{enumerate}
\usepackage[utf8]{inputenc}
\usepackage{devanagari} 
\usepackage{censor}
\usetikzlibrary{arrows,decorations.pathmorphing,backgrounds,positioning,fit,petri,mindmap,shapes.geometric, shapes.misc}
\tikzset{cross/.style={cross out, draw=black, minimum size=2*(#1-\pgflinewidth), inner sep=0pt, outer sep=0pt},
cross/.default={1pt}}

\input{Commands}

\title{Introduction to probability and statistics: a computational framework of randomness}

\author{Lakshman Mahto \\
Assistant Professor (Mathematics), IIIT Dharwad, India\\
lakshman@iiitdwd.ac.in}

\begin{document}

\frontmatter

\maketitle

\tableofcontents

\mainmatter

\begin{abstract}
This text presents an unified approach of probability and statistics in the pursuit of understanding and computation of randomness in engineering or physical or social system with prediction with generalizability. Starting from elementary probability and theory of distributions, the material progresses towards conceptual and advances in prediction and generalization in statistical models and large sample theory. We also pay special attention to unified derivation approach and one-shot proof of each and every probabilistic concept. Our presentation of intuitive and computation framework of conditional distribution and  probability are strongly influenced by unified patterns of linear models for regression and for classification. The text ends with a future note on the unified approximation of the linear models, the generalized linear models and the discovery models to neural networks and a summarized ML system.  
\end{abstract}

\chapter*{preface}
\input{pref}

\chapter{Introduction}
\label{intro}
\input{intro}

\chapter{Elementary probability}
\label{elmprob}
\input{elmprob}

\chapter{Probability distribution}
\label{dist}
\input{dist}

\chapter{Large sample theory}
\label{lsample}
\input{lsample}

\chapter{Statistical methods}
\label{statm}
\input{statm}

\chapter{Statistical modelling}
\label{statmod}
\input{statmod}
%
\begin{acknowledgements}
I dedicate this text to my wife, my son, my parents, my grand-parents, my in-law-parents, my teachers and all those who help and inspire me, because through you I become whole.

\noindent This text developed out from lectures given at Indian Institute of Information Technology Dharwad in 2022 and 2023 for engineering students specializing in computation mathematics, probabilistic modelling, statistical estimation \& inference.

\noindent I thank to my colleagues at IIIT Dharwad: Somen, Arun and Jagadish for their help, support and valuable discussions.

\noindent Last but not the least; I would like to convey my special thanks to my dear wife, my dear son, grandfather, parents, teachers, brothers and sisters who tolerated the onslaughts of time to bring me up to this position and for their blessings, continuous encouragement and moral support.
\end{acknowledgements}

\bibliographystyle{plainnat}
\bibliography{newbib}
\end{document}

%% file: Commands.tex
\renewcommand{\phi}{\varphi}

\def\ds1{\mathds{1}}
\renewcommand{\epsilon}{\varepsilon}

\newlength{\minipagewidth}
\newcommand{\beq}{\begin{equation}}
\newcommand{\eeq}{\end{equation}}

\newcommand{\beqa}{\begin{eqnarray}}
\newcommand{\eeqa}{\end{eqnarray}}

\newcommand{\beqan}{\begin{eqnarray*}}
\newcommand{\eeqan}{\end{eqnarray*}}

\def\ba#1\ea{\begin{align*}#1\end{align*}} 
\def\banum#1\eanum{\begin{align}#1\end{align}} 

%% file: pref.tex
\parbox{\textwidth}{
\begin{align*}
\mbox{pothee padhi padhi jag mua, pandit bhaya na koy |} & \\
\mbox{dhaee aakhar prem ka, padhe so pandit hoy  |} & \\
                                                    & \mbox{Kabir}.
\end{align*} 
Here, prem may refer, devotion to observation, understanding, \& applicability.} 

\vspace*{0.5cm}

\noindent The central objects of our studies are probability as a measure of randomness and statistics as a computation of a sample statistic (with goodness of fit) from given data-sets.

\section*{Objective and approach}
This text assumes that you know approximately nothing about probability and statistics. Its goal is to provide you with observation, understanding and applications of each every concepts that ultimately leads to solving real world problems.

\noindent We cover a large number of techniques, from the simplest and most commonly used (such as conditioning) to some of the advance techniques (e.g. estimation and generalization of a statistical model).

\noindent Computational framework (in the intersection of probability \& statistics, optimization \& numerical methods, algorithm \& programming, and computer) means a practical understanding of a concept and its applicability in solving a related problem in computers effectively with the help of some numerical method. The text favours a computational approach of each and every concept with a better observation, understanding and applicability through concrete working examples and their simulation in computer. Furthermore, the approach in order to solve a real world problem characterizes with an accuracy, (i.e. approximately very close to the correct/optimal result), an efficiency (i.e. a polynomial time computational and iteration complexities),a robustness (i.e. work for every possible realization of a random sample of any size) and an stability (i.e. not sensitive to small perturbations in the input).
\begin{center}
\begin{figure}[h]
    \centering
    \includegraphics[width=1.0\textwidth]{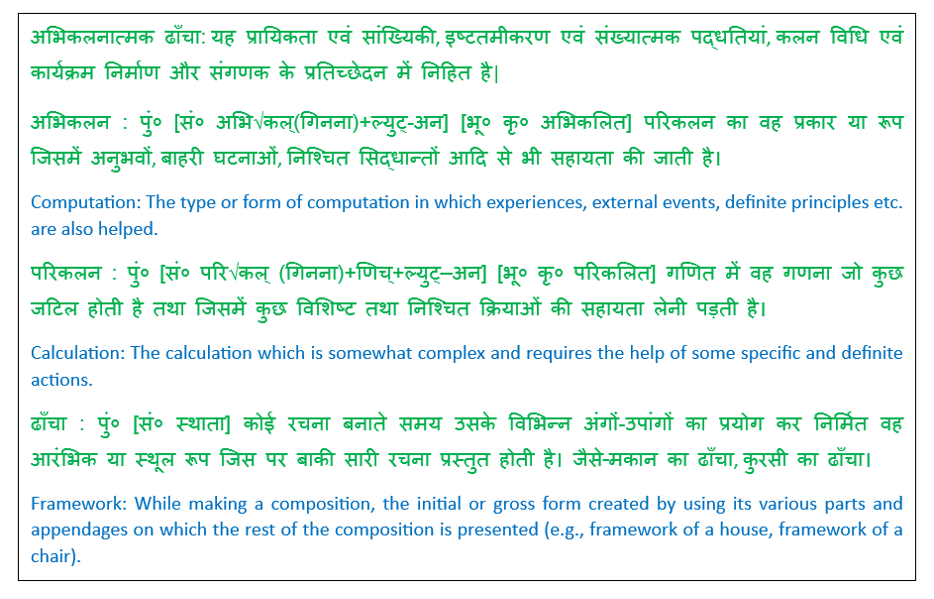}
    \caption{A self-explanatory meaning of computational framework in Hindi with corresponding translation in English.}
    \label{fig:comp-fram-01}
\end{figure}
\end{center}

\section*{Prerequisite}
This text assumes that you have some basic understanding of linear algebra, calculus, programming and algorithm.

\section*{Roadmap} \label{sec:rmap}
The central objects of our studies are probability as a measure of randomness and statistics as a computation of a sample statistic (with goodness of fit) from given data-sets. So, the text is decomposed into two parts. 

\noindent Part I, probability, covers the following topics:
\begin{enumerate}
\item[i.] What is probability? What is probabilistic modelling of a random experiment?  What are the types of probabilistic modelling? How a prior information or conditioning affects a probabilistic modelling?
\item[ii.] The central concept of probability, probability distribution, as a computational measure of randomness in a random experiment.
\item[iii.] The most common probability distributions and their derived distributions.
\item[iv.] Joint probability distributions and their derived distributions, convolutions.
\item[v.] Central representation of a probability distribution: expectation, variance, co-variance.
\item[vi.] Moment generating function and its applications
\end{enumerate}

\noindent Part II, statistics, covers the following topics:
\begin{enumerate}
\item[i.] What is sample? What is data from a sample? What is sample statistic? What is stochastic convergence?
\item[ii.] Probabilistic inequalities and non-asymptotic large sample theory.
\item[iii.] Limit theorems and asymptotic large sample theory.
\item[iv.] Methods of point estimation and goodness of fit (e.g. bias, variance, efficiency, sufficiency).
\item[v.] Interval estimation and limit theorems.
\item[vi.] Hypothesis testing and statistical inferences
\item[vii.] Linear models and prediction (e.g. simple univariate linear models, simple multivariate linear models, embedded linear models, multiple multivariate linear models).
\item[viii.] Simulation with Python.
\end{enumerate}

\noindent Note: don't jump into solving a noisy observation based problem in a computer too hastily: while solving the problem in computer is no doubt one of the most exciting areas in computational mathematics, you should master the fundamentals first.

\section*{Silent features of the book}
\begin{enumerate}
\item[a.] Derivation of various concepts such as:
\begin{enumerate}
\item[i.] computation of probability of an subinterval $[a,b]$ in the sample space $\Omega$ of a continuous probability model 
\begin{align*}
P([a,b])=\frac{\mbox{length of} \ [a,b]}{\mbox{length of} \ \Omega}=\frac{b-a}{\mbox{length of} \ \Omega}.
\end{align*} 
\item[ii.] conditional probability as a technique to understand, compute and apply to central concepts such as concepts joint probability, total probability, Bayes rule, conditional distribution, joint distribution and conditional expectation.
\end{enumerate}
\item[b.] presentation of each and every example in an ordered 3-tuple of backdrop, question and answer.
\item[c.] One-shot derivation of various proofs and solutions such as derived distribution, limit theorem, point estimation, interval estimation and computation $\&$ estimation of parameters for a linear model.
\item[d.] A unified approach to compute and estimate parameters of simple univariate, simple multivariate, embedded and multiple multivariate linear models.
\item[e.] Goodness of fit for every estimate of a parameter or a parametric model.
\item[f.] Python demonstration of almost every probability distribution, large sample theory and statistical methods.
\item[g.] An effort to utilise self-explanatory capability of a word in Indian languages to understand, observe and apply at the best (i.e. upto optimality) in-line with new education policy 2023.
\begin{center}
\begin{figure}[h]
    \centering
    \includegraphics[width=1.0\textwidth]{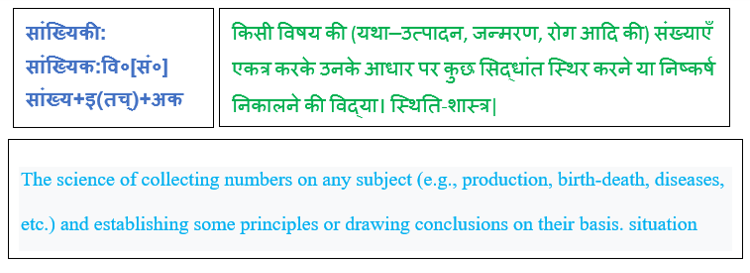}
    \caption{A self-explanatory meaning of statistics in Hindi with corresponding translation in English.}
    \label{fig:artime-01}
\end{figure}
\end{center}
\item[h.] This is a text, not a reference book, for probability and statistics of undergraduate studies. I have tried to present the fundamental ideas of the subject in simple and intuitive contexts in computational frameworks.
\end{enumerate}

%% file: intro.tex
The purpose of the text is to present the computational framework of probability and statistics as a mathematical discipline that observe, understand and applicability of a problem based on probabilistic phenomenon or a mathematical model based on noisy measurements.

\section{Perceiving randomness and probabilistic problem}
Natural sciences (e.g. theory of errors in measurements, problems of the theory of ballistics, problems of statistics) made it necessary to develop the theoretical probability with advanced analytical tools.

\noindent The modern development of probability is characterized by the notions of measure theory and, subsequently, functional analysis that lead to an extended content of theoretical probability (e.g. large sample theory, stochastic processes, stochastic modelling).

\noindent In the middle of 20th century, applicability of probability developed in modern science immeasurably (e.g. molecular orbital theory, statistical physics, microscopic understanding of diffusion and transport of particles through interaction, intra-action and characterization of motion).

\noindent At microscopic level, particle theory establish a fact that every substance is composed of very many particles i.e. in a homogeneous body their properties are very similar. Under these circumstances standard mathematical research methods are useless in the investigation of physical theories. For example, the techniques of differential equations are incapable of yielding any serious results under these conditions.

\begin{example}
Consider a diffusion model 
\begin{align}\label{eq:int_de-01}
\frac{\partial u}{\partial t} = D \frac{\partial^2}{\partial t^2}\left(u \right).
\end{align}
Then, 
\begin{enumerate}
\item[a.] At macroscopic level, the equation \eqref{eq:int_de-01} refers a deterministic system.
\item[b.]While, at microscopic level, the motion of each individual molecule/particle is having a Brownian motion, which is random in nature.
\end{enumerate}
\end{example}

\noindent Schematic a phenomenon or selection of mathematical tools to investigate the phenomenon. Goodness of fit or quality can be judged by how well the theory agrees with experiment and with practice. 

\section{Statistics and randomness}
Consider a statement, "$80\%$ chance of rain on this Friday". So, we perceive the following:
\begin{enumerate}
\item[a.] the chance $80\%$ is estimated from data statistically (i.e. a statistical estimate).
\item[b.] an estimation is performed in a deterministic way, because it involves a computation technique.
\item[c.] so, a statistical estimation is carried through measurement (or observation) using a computation technique.
\item[d.] furthermore, a measurement involve measurement randomness due to computation approximation and measuring unit. 
\end{enumerate} 

\section{Strategies for solving probabilistic problems} \label{sec:rmap}
The central objects of our studies are probability as a measure of randomness and statistics as a computation of a sample statistic (with goodness of fit) from given data-sets. So, the text is decomposed into two parts. 

\noindent Part I, probability, covers the following topics:
\begin{enumerate}
\item[i.] What is probability? What is probabilistic modelling of a random experiment?  What are the types of probabilistic modelling? How a prior information or conditioning affects a probabilistic modelling?
\item[ii.] The central concept of probability, probability distribution, as a computational measure of randomness in a random experiment.
\item[iii.] The most common probability distributions and their derived distributions.
\item[iv.] Joint probability distributions and their derived distributions, convolutions.
\item[v.] Central representation of a probability distribution: expectation, variance, co-variance.
\item[vi.] Moment generating function and its applications
\end{enumerate}

\noindent Part II, statistics, covers the following topics:
\begin{enumerate}
\item[i.] What is sample? What is data from a sample? What is sample statistic? What is stochastic convergence?
\item[ii.] Probabilistic inequalities and non-asymptotic large sample theory.
\item[iii.] Limit theorems and asymptotic large sample theory.
\item[iv.] Methods of point estimation and goodness of fit (e.g. bias, variance, efficiency, sufficiency).
\item[v.] Interval estimation and limit theorems.
\item[vi.] Hypothesis testing and statistical inferences
\item[vii.] Linear models and prediction (e.g. simple univariate linear models, simple multivariate linear models, embedded linear models, multiple multivariate linear models).
\item[viii.] Simulation with Python.
\end{enumerate}

\noindent Note: don't jump into solving a noisy observation based problem in a computer too hastily: while solving the problem in computer is no doubt one of the most exciting areas in computational mathematics, you should master the fundamentals first.

\section{Overview}
\begin{enumerate}
\item[a.] Derivation of various concepts such as:
\begin{enumerate}
\item[i.] computation of probability of an subinterval $[a,b]$ in the sample space $\Omega$ of a continuous probability model 
\begin{align*}
P([a,b])=\frac{\mbox{length of} \ [a,b]}{\mbox{length of} \ \Omega}=\frac{b-a}{\mbox{length of} \ \Omega}.
\end{align*} 
\item[ii.] conditional probability as a technique to understand, compute and apply to central concepts such as concepts joint probability, total probability, Bayes rule, conditional distribution, joint distribution and conditional expectation.
\end{enumerate}
\item[b.] presentation of each and every example in an ordered 3-tuple of backdrop, question and answer.
\item[c.] One-shot derivation of various proofs and solutions such as derived distribution, limit theorem, point estimation, interval estimation and computation $\&$ estimation of parameters for a linear model.
\item[d.] A unified approach to compute and estimate parameters of simple univariate, simple multivariate, embedded and multiple multivariate linear models.
\item[e.] Goodness of fit for every estimate of a parameter or a parametric model.
\item[f.] Python demonstration of almost every probability distribution, large sample theory and statistical methods.
\item[g.] An effort to utilise self-explanatory capability of a word in Indian languages to understand, observe and apply at the best (i.e. upto optimality) in-line with new education policy 2023.
\begin{center}
\begin{figure}[h]
    \centering
    \includegraphics[width=1.0\textwidth]{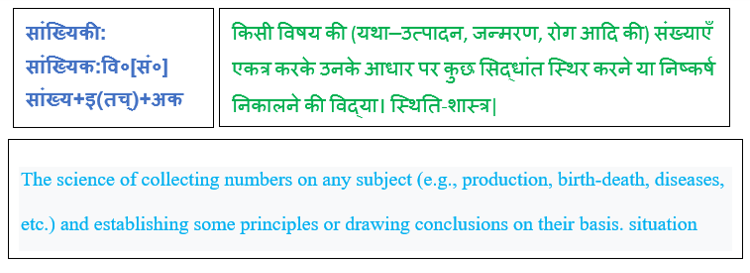}
    \caption{A self-explanatory meaning of statistics in Hindi with corresponding translation in English.}
    \label{fig:artime-01}
\end{figure}
\end{center}
\item[h.] This is a text, not a reference book, for probability and statistics of undergraduate studies. I have tried to present the fundamental ideas of the subject in simple and intuitive contexts in computational frameworks.
\end{enumerate}

%% file: elmprob.tex
\section{What is probability?}\label{sec:what_prob}
Probability is a mathematical discipline that concerning with numerical descriptions of how likely an event is to occur or a degree of belief of the occurrence of an event, or how likely it is that a proposition is true. The probability of an event is a number between 0 and 1. In simple, probability is the measure or quantification of randomness that falls in the intersection of mathematical thinking and intuitive observation/experience of patterns. It finds applications in . It diffuses through almost every discipline of sciences \& mathematics, engineering \& technology and social \& life sciences. And it finds applications in inference, decision making, experimental design, reliability theory, statistical theory of information, communication, control, and many more except motion of planets. A few example inspired from \citep{PS73} and \citep{RV06}:
\begin{enumerate}
\item[a.] Physics: a few quantities (e.g. temperature and pressure) arise as a direct consequence of the random motion of atoms and molecules. Heisenberg uncertainty principle in quantum mechanics.
\item[b.] Biology and medicine: in human evaluation, random mutations lead to the amazing diversity of life. Stochastic models are essential in understanding the spread of disease, both in a population (e. g. epidemics) or in the human body (e.g. cancer).
\item[c.] Chemistry: a chemical reaction takes place when molecules randomly meet. Stochastic models of chemical kinetics are particularly important in systems with very low concentrations (e.g. biochemical reactions in a single cell).
\item[d.] Electronics and communication engineering: the effect of noise to design a reliable communication protocol that we use on a daily basis in our cell phones. Observation or measurement of a signal using a suitable measuring unit happens to be noisy in nature (i.e. in practice, measurement of a signal contains a signal along-with an additive noise).
\item[e.] Computer science: randomness in an algorithm (for some best known method) brings robustness and stability in order to compute a solution of a complex problem.
\item[f.] Finance and economics: stock and bond prices are inherently unpredictable, so these lead to some kind of stochastic models. The modelling of randomly occurring rare events forms the basis for all insurance policies, and for risk management in banks.
\item[g.] Sociology: stochastic models provide basic understanding of the formation of social networks and of the nature of voting \& polling schemes.
\item[h.] Statistics and machine learning: statistical methods lead to stochastic models that form the foundation for almost all of data science (e.g. tabular data science, image data science, text data science or NLP). 
\item[i.] Optimization: in order to come out of stuck region of local in minimization a function, there is a need noise in an optimization algorithm through noisy or perturbed gradient (e.g. noisy gradient descent, stochastic gradient descent).
\end{enumerate}
\begin{center}
\begin{figure}[htb]
    \centering
    \includegraphics[width=1.0\textwidth]{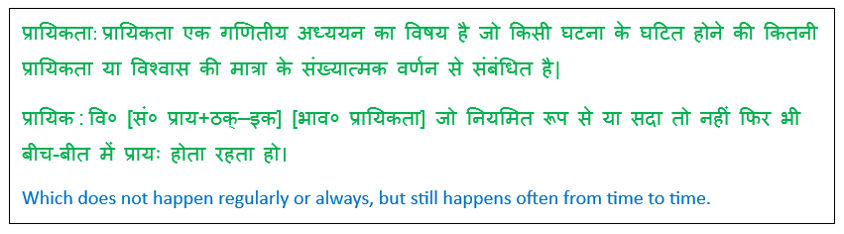}
    \caption{A self-explanatory meaning of probability in Hindi with corresponding translation in English.}
    \label{fig:sl-01}
\end{figure}
\end{center}

\begin{example}[A measurement error due to limitation of a scale]
Consider a measurement scale \eqref{fig:sl-01} that can measure 0 to 100 cm with 1cm accuracy. Now if we have to measure $99.73$cm with respect to the given scale, we find that $99.73$ cannot be measured deterministically on it. With respect to the scale,  $99.73$ has three significant figures, so it is $99.73 \approx 99.7$, where first two numbers i.e. $99$ are deterministic significant figures and the third one $.7$ is a uncertain or estimated or random significant figure. So, we can a measurement includes uncertainty due to limitation of the corresponding measuring instrument in use.
\begin{center}
\begin{figure}[h]
    \centering
    \includegraphics[width=1.0\textwidth]{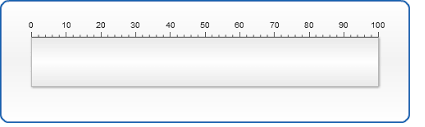}
    \caption{a measurement scale that can measure 0 to 100 cm with 1cm accuracy.}
    \label{fig:sl-01}
\end{figure}
\end{center}
\end{example}

\begin{example}[the toss of a fair coin] 
You know intuitively that there is a 50 percent chance of getting heads, and 50 percent chance of getting tails. If you want to do the math to calculate the probability of a head, here's the basic formula that count the number of times that the event will happen (e.g. in this case, there's just one chance of a head appearing, so it's 1). Divide this by the total number of possible outcomes (e.g. with a coin, it's either heads or tails, which is 2 possible outcomes). So the probability of getting a head is $0.5$ or 50 percent.
\end{example}

\begin{example}[Estimating a number or a function]
Consider an estimation of a number from its noisy measurement. For example, $\bar{x}=\mu +\epsilon$, where $\mu$ is the true mean, and$\bar{x}$ is the estimated mean with a noise $\epsilon$ as the error containing some randomness. Furthermore, consider an estimation of a deterministic function of $f(x)=x^2$ with its noisy measurement $x^2 +\epsilon$ in figure \eqref{fig:dnx2}.
\begin{center}
\begin{figure}[h]
    \centering
    \includegraphics[width=1.0\textwidth]{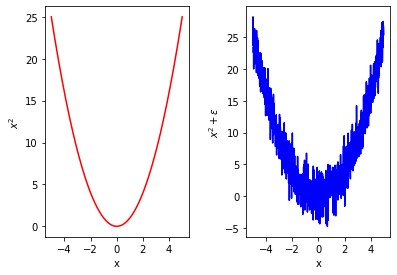}
    \caption{In left there a deterministic function $x^2$ to be estimated from its noisy version in right.}
    \label{fig:dnx2}
\end{figure}
\end{center}
\end{example}

\begin{example}
Backdrop: three men are on a trip. One of the three will volunteer to cook breakfast. A is twice as likely to volunteer as B, and B is three times as likely to volunteer as C. 

\noindent Question: compute the probability of volunteering for each man.

\noindent Answer: suppose that one and only one will volunteer among A, B, and C. So, we can assume that $\{A,B,C\}$ are mutually exclusive and totally exhaustive, i.e. $P(A) + P(B) + P(C)=1$. 

\noindent Now, from the question, we have the following assumptions:
\begin{align*}
& (a.) P(A)=2P(B), \ (b.) P(B)=3P(C), \ (c.) P(A) + P(B) + P(C)=1 \\
& \implies 2 \cdot 3P(C) + 3P(C) + P(C)=1 \implies P(C)=\frac{1}{10}.
\end{align*}
And hence, we get
\begin{align*}
P(A)=\frac{6}{10}, \ P(B)=\frac{3}{10}, \ \mbox{and} \ P(C)=\frac{1}{10}.
\end{align*}
\end{example}

\begin{remark}
In summary, we may perceive following at first instant observation from probability:
\begin{enumerate}
\item[a.] chance/unpredictable/random/stochastic,
\item[b.] possible outcomes in a random experiment (an experiment with unpredictable possible outcome in each of its trials). 
\end{enumerate}
So, generally we come across a random experiment with a list of possible outcomes, but a possible outcome in each trial of the experiment is unpredictable (i.e. random)
\end{remark}

\begin{definition}[Random experiment]
A random experiment is an experiment whose outcome is not predictable before the experiment is performed, however in advance we know all possible outcomes in the experiment. That is, an outcome of a random experiment is uncertain and a probabilistic modelling is providing an approach to the quantify the uncertainty.

\noindent Furthermore, there is no restriction on what constitutes an experiment (e.g. it could be a single toss of a coin, or three tosses, or an infinite sequence of tosses). So, it is important to note that in formulation of a probabilistic model, there is only one experiment (e.g. three tosses of a coin constitute a single experiment, rather than three experiments).
\end{definition}

\begin{example}[Tossing a coin] 
If we toss a coin, then either a head or tail comes up with some degree of belief, but we don't know in advance which one will turn up.
\end{example}

\begin{example}[Tossing a coin] 
In weather prediction, a meteorologist just know either it will rain or not rain with degree of belief, but we doesn't know in advance which one will turn up.
\end{example}

\section{Basic concepts of probabilistic modelling} \label{sec:bcpm}
Consider a few situation--(a.) the probability that it will rain the next day or (b.) the probability that you will win the lottery \citep{K06}. Here, we cannot be certain because, there are many factors that affect the weather and those leads to uncertainty that whether it will or not rain the next day. So, as a predictive tool we usually assign a number between 0 and 1 as our degree of certainty or belief that the event, rain, will occur. If we say that there is a 30\% chance of rain, means we believe that if identical conditions prevail, then 3 times out of 10, rain will occur the next day. To predict at the nature's standard, it is essential to come up with an accurate model by knowing everything of science of meteorology that gives the probability is either 0 or 1. Unfortunately, it is not possible at practice, i.e. uncertainty of an outcome is not because of the inaccuracy of our model, but because the experiment has been designed to produce uncertain results.

\noindent In the last section \eqref{sec:what_prob}, we see a list of common items e.g. a random experiment, a set of possible outcomes, and the probabilities assigned to these outcomes. These attributes are common to all probabilistic descriptions. And hence, it is very much essential to come up with a systematic approach as designed by Kolmogrow to compute randomness or uncertainty through three basic concepts of probabilistic modelling of random experiment:
\begin{enumerate}
\item[(a.)] a sample space, 
\item[(b.)] an event of interest, and 
\item[(c.)] probability measure of an event as an assignment or degree of belief between 0 and 1. 
\end{enumerate}

\subsection{A sample space i.e. specification of all possible outcomes}
Once a random experiment is being performed, we need to specify all possible outcomes of the random experiment as:
\begin{enumerate}
\item[(a.)] First we identify an outcome of the random experiment,
\item[(b.)] Secondly, we list out all possible outcomes of the experiment as a sample space.
\end{enumerate}

\begin{example}[Throwing two dice together]
Question: Throw two dice together and find its sample space.

\noindent Answer: Consider a throw of two dice together, then
\begin{enumerate}
\item[a.] we identify an outcome as an ordered pair $(i,j)$, where the first die comes up i and the second die comes up j. 
\item[b.] we list out all possible outcomes as $\{(i,j)| 1\leq i\leq 6, 1\leq j \leq 6\}$.
\end{enumerate}
So, we have the sample space,
\begin{align*}
\Omega = \{(i,j)| 1\leq i\leq 6, 1\leq j \leq 6\} \ \mbox{with cardinality} \ |\Omega|=36.
\end{align*}
\end{example}

\begin{example}[Waiting time for a bus at a bus stop]
Question: suppose you are waiting for a bus at a bus stop, then compute its sample space.

\noindent Answer: If you are waiting for a bus at a bus stop, then 
\begin{enumerate}
\item[a.] a real number $t\geq0$ is identified as a possible outcome of the experiment. 
\item[b.] $t=0h$ refers an outcome that on your arrival at the bus stop, immediately you get a bus. $t = 1.5h$ refers an outcome that on your arrival at the bus, you waited for $1.5h$ to get a bus. In this process, listing out all possible outcomes as $\{t\in \mathbb{R}|t\geq 0\}$.
\end{enumerate}
So, the sample space is given by
\begin{align*}
\Omega = \{t\in \mathbb{R}|t\geq 0\}=[0,\infty).
\end{align*}
\end{example}

\begin{example}[Path of a bee for 5 seconds]
Question: consider a random experiment of observing path of a bee for 5 seconds. Compute its sample space.

\noindent Answer: If we observe path of a bee for 5 seconds, then
\begin{enumerate}
\item[a.] a possible outcome is identified as a continuous map $\omega: [0,5] \to \mathbb{R}^3$ in \eqref{fig:pbee-04}:
\begin{center}
\begin{figure}[h]
    \centering
    \includegraphics[width=0.80\textwidth]{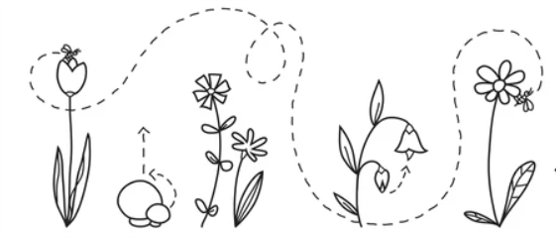}
    \caption{Path of a bee as continuous map $\omega: [0,5] \to \mathbb{R}^3$.}
    \label{fig:pbee-04}
\end{figure}
\end{center}
\item[b.] listing all possible outcomes as 
\begin{align*}
\Omega = \{\omega: [0,5] \to \mathbb{R}^3 | \omega \ \mbox{is a continuous map}\}.
\end{align*}
\end{enumerate}    
So, the sample space is given by
\begin{align*}
\Omega = \{\omega: [0,5] \to \mathbb{R}^3 | \omega \ \mbox{is a continuous map}\}.
\end{align*}
\end{example}

\subsection{The event (of interest)}
\begin{definition}[Informal] 
Once we have identified a sample space for a given random experiment, then we can come up with a statement in order to construct (or define) an event as a collection of some of outcomes (or favourable outcomes) satisfying the statement.
\end{definition}

\begin{definition}[Formal]
A subset $A$ of the sample space $\Omega$ which satisfies a given statement i.e. an element of a $\sigma-$ algebra of the sample space $\Omega$ (i.e. $A\in \sigma(\Omega)$). An $\sigma-$algebra of $\Omega$ is collection of subsets in $\Omega$ and closed with respect to complementation and countable union of its members. 

\noindent An event of interest can be computed as follows:
\begin{enumerate}
\item[a.] First, list out all possible outcomes as a sample space of the given random experiment.
\item[b.] Secondly, collect some possible outcomes satisfying a statement under interest/consideration.
\end{enumerate}
\end{definition}

\begin{example}
Question: consider an experiment of throwing a die and getting a number exactly divisible by 3. 

\noindent Answer: 
\begin{enumerate}
\item[a.] Listing out all possible outcomes as
\begin{align*}
\Omega =\{1,2,3,4,5,6\}.
\end{align*}
\item[b.] $A \to$ collection of some possible outcomes satisfying a statement that a number is exactly divisible by 3. So, we have
\begin{align*}
A=\{3,6\}.
\end{align*}
\end{enumerate}
\end{example}

\begin{example}[A throw of two dice]
Question: consider a throw of two dice together. Getting an outcome with a sum of is $7$.

\noindent Answer: 
\begin{enumerate}
\item[a.] When we toss two coins together, the sample space is given by
\begin{align*}
\Omega = \{(i,j)| i, j =1,2,3,4,5,6\}.
\end{align*}
\item[b.] $A \to$ collection of some possible outcomes satisfying a statement that sum is 7. So, we have
\begin{align*}
A=\{(1,6),(2,5),(3,4),(4,3),(5,2),(6,1)\}.
\end{align*} 
\end{enumerate}
\end{example}

\begin{example}[Waiting time for a bus at a bus stop]
Question: Suppose you are waiting for bus at a bus stop. Getting a bus in the first hour.

\noindent Answer: 
\begin{enumerate}
\item[a.] As you are waiting for a bus at a bus stop, so, the sample space is given by
\begin{align*}
\Omega=\{t \in \mathbb{R} | t \geq 0\}=[0,\infty).
\end{align*} 
\item[b.] $A \to$ collection of some possible outcomes satisfying a statement that a bus is in first. So, we have
\begin{align*}
A=\{t \in \mathbb{R} | 0 \leq t \leq 1 \} =[0,1] \subset \Omega.
\end{align*}
\end{enumerate}
\end{example}

\begin{example}[Path of a bee for 5 seconds]
Question: In a random experiment of observing path of a bee for 5 seconds. Observing path of bee stuck in a region C for first 1 second.

\noindent Answer: 
\begin{enumerate}
\item[a.] If we observe a bee for 5 second, so, the sample space is given by
\begin{align*}
\Omega =\{ \omega | \omega:[0,5] \to \mathbb{R}^3 \  \mbox{is a continuous map} \}.
\end{align*}
\item[b.] $A \to$ collection of some outcomes satisfying a statement that a bee got stuck in a region C in first second. So, we have 
\begin{align*}
A=\{\omega:[0,5] \to \mathbb{R}^3| \omega (t) \in C \ \mbox{for} \ t \in [0,1]\} \subset \Omega.
\end{align*}
\end{enumerate} 
\end{example}

\begin{remark}
\begin{enumerate}
\item[a.] If $\omega \in \Omega$ is an outcome of a random experiment, then an event $A$ occurs if $\omega \in A$.
\item[b.] $A^c$ occurs if $A$ does not occur.
\item[c.] $A \cup B$ occurs if A occurs or B occurs (or both).
\item[d.] $A\cap B$ occurs if A occurs and B occurs.
\item[e.] $A \textbackslash B =A -B=A\cap B^c$ occurs if $A$ occurs and B does not occur.
\item[f.] $A$ and $B$ are mutually disjoint if $A\cap B =\phi$ i.e. $A$ occurs  and B does not occurs
\end{enumerate}
\end{remark}

\subsection{Probability Measure}

\noindent Till now, we have seen the following:
\begin{enumerate}
\item[(a.)] Sample space ($\Omega \to$) as a collection of all possible outcomes from the random experiment.
\item[(b.)] An event ($A$) defined by a statement as a collection of some possible outcomes from the experiment. 
\end{enumerate}

\begin{definition}[Probability measure of an event]\label{def:pm-05}
Now, we define probability measure of an event A as a degree of confidence or a map $P:\sigma(\Omega)\to [0,1]$ satisfies the following properties or axioms:
\begin{enumerate}
\item[(i.)] $P(\Omega)=1, P(\phi)=0,$
\item[(ii.)] $P(A^c)=1-P(A) \in [0,1]$ for any event $A\in \sigma(\Omega)$,
\item[(iii.)] $P(A\cup B)=P(A)+P(B)$ for any two disjoint events A and B i.e. $A\cap B=\phi$,
\end{enumerate}
The third axiom or property can be generalized to ($\sigma-$ additive of) any countable collection of mutually disjoint events ${A_i}$ as follows:
\begin{align*}
P(\cup_{i=1}^{\infty}A_i) & =\sum_{i=1}^{\infty}P(A_i)  \ \mbox{for any mutually disjoint events i.e.} \\ \ A_i \cap A_j & =\phi \ \mbox{for} \ i\neq j. 
\end{align*}
\end{definition}

\noindent A probability measure of an event A can be computed using a suitable probability law (e.g. a uniform law or symmetrical outcomes) in the following two steps:
\begin{enumerate}
\item[(01)] first come up with an assumption or information from the random experiment, where first instant assumption happens to be uniform assumption (or equally likely outcomes or symmetrical outcomes). For example, in a single toss of a coin, the chances of getting both head and tail are equal due to symmetrical coin (with respect to head and tail).
\item[(02)] then computing the probability measure of an event using the assumption observed from the random experiment in step-01 and the axioms \eqref{def:pm-05} of the probability measure of the event.
\end{enumerate}

\begin{example}[A roll of a die]
Question: In a roll of a fair die, compute probability of occurrence of each and every outcomes.

\noindent Answer: In a roll of a fair die, the sample space is given by
\begin{align*}
\Omega=\{1,2,3,4,5,6\}.
\end{align*}

\noindent  As the die happens to be fair, so each outcome is equally likely in the roll of the fair die i.e. there is a uniform law of distribution of outcomes. So, we have
\begin{align}\label{as:eq-01}
P(1) = P(2) = P(3) = P(4) = P(5) = P(6)
\end{align}

\noindent Now, using axioms of the probability measure for an event, we have
\begin{align}\label{as:eq-02}
&P(\Omega)=1, \Omega = \{1\}\cup \{2\} \cup \{3\} \cup \{4\} \cup \{5\} \cup \{6\} \implies \nonumber \\
&P(1) + P(2) + P(3) + P(4) + P(5) + P(6)=1
\end{align}

Solving for equations \eqref{as:eq-01} and \eqref{as:eq-02}, we get
\begin{align*}
P(1) = P(2) = P(3) = P(4) = P(5) = P(6)=1/6.
\end{align*}
\end{example}

\begin{example}[Swimming in leisure time per week a day]
Ram likes swimming in leisure time per week on a day. An outcome is a day that is 
\begin{align*}
\Omega=\{Monday, Tuesday, Wednesday, Thursday, Friday, Saturday, Sunday\}.
\end{align*} 
The first instant assumption is the uniform as each day is equally likely, so we have
\begin{align*}
P(Monday)&=P(Tuesday)=P(Wednesday)=P(Thursday) \\
&=P(Friday)=P(Saturday)=P(Sunday).
\end{align*}
Now, using axioms of probability measure of an event, we have
\begin{align*}
P(Monday)&+P(Tuesday)+P(Wednesday)+P(Thursday) \\
&+P(Friday)+P(Saturday)+P(Sunday)=1.
\end{align*}
By solving above two equations, we get
\begin{align*}
P(Monday)&=P(Tuesday)=P(Wednesday)=P(Thursday)\\
&=P(Friday)=P(Saturday)=P(Sunday)=1/6.
\end{align*}
\end{example}

\begin{example}[swimming on weekday as often as weekend]
Ram goes for swimming on weekday as often as weekend. The first instant assumption says that weekdays and weekends are equally likely. So
\begin{align*}
P(Weekdays)=P(Weekends)=1/2.
\end{align*} 
Further, there are 5 days in weekdays, and each of these 5 days are equally likely, so, we have
\begin{align*}
P(Monday)&=P(Tuesday)=P(Wednesday)=P(Thursday) \\
&=P(Friday)=1/2 \cdot 1/5 =1/10.
\end{align*}
And there are two days in weekends, and each of these 2 days are equally likely, so, we have
\begin{align*}
P(Saturday)=P(Sunday)=1/2 \cdot 1/2 =1/4
\end{align*}
\end{example}

\begin{example}[Continuation of information from the previous two examples]
He goes on Friday equals to Sum of Wednesday and Thursday. Let $P(Wednesday) = P(Thursday) = p$. Then $P(Friday)=2p$ and $P(Monday) =  P(Tuesday) =\frac{1}{2} \cdot \left(\frac{1}{2}-4p\right)= \frac{1}{4} - 2p >0 \implies 0<p<1/8$.
\end{example}

\noindent Note: probability measure is not assigned to an individual outcome, it is assigned to an event containing the outcome.

\begin{example}
Five balls numbered $1,2,3,4,5$ are drawn from an urn without replacement. Then compute the probability that they will be drawn in the same order as their number.

\noindent Each outcome is represented by the 5-tuple $(z_1,z_2,z_3,z_4,z_5)$, so total possible outcomes are $|\Omega|=5!=120$. The only outcome in the event $A$ is $(1,2,3,4,5)$. Hence, the desired probability is
\begin{align*}
P(A)=\frac{1}{120}.
\end{align*}
\end{example}

\begin{example}[Monte Hall problem]
Background: suppose you're on a game show, and you're given the choice of three doors: behind one door there is a car, behind the others, there are goats. You pick a door (say No. 1), and the host, who knows what's behind the doors, opens another door (say No. 3), which has a goat. He then says to you, 'do you want to pick door No. 2?'

\noindent Question: Compute the probability with which the player wins by switching.

\noindent Answer: first we need to all possible outcomes followed by an event of interest (i.e. the player wins by switching) and probability measure of the event.
\begin{enumerate}
\item[a.] In order to identify an outcome, we recall a randomly determined quantities that may help to determine the outcome as:
\begin{enumerate}
\item[i.] the door concealing the car (as a first layer of information),
\item[ii.] the door initially chosen by the player (as a second layer of information),
\item[iii.] the door that the host opens to reveal a goat (as a third layer of information).
\begin{center}
\begin{figure}[h]
    \centering
    \includegraphics[width=1.0\textwidth]{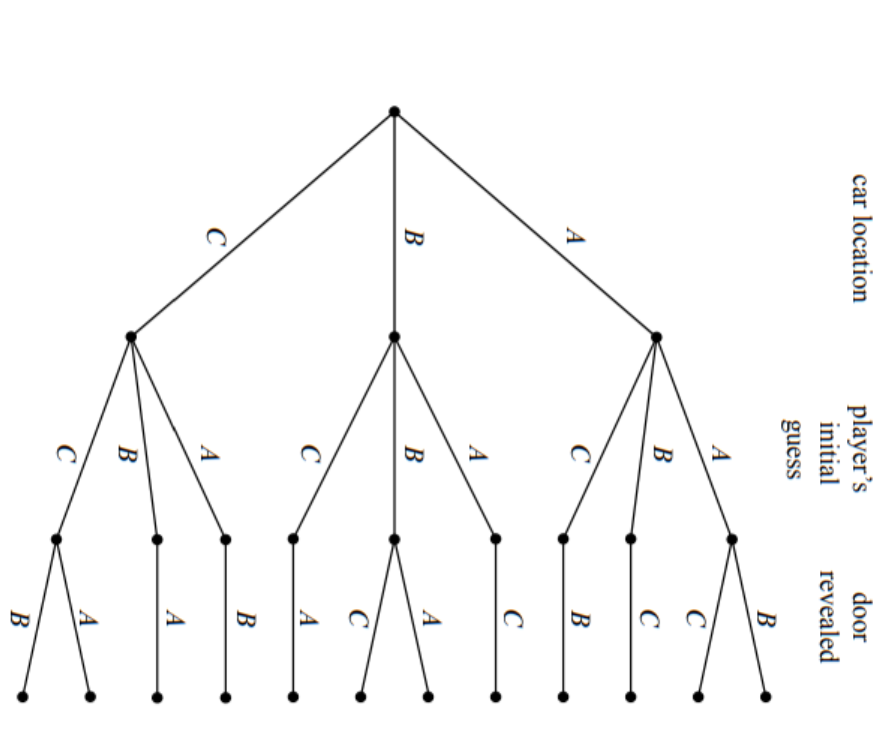}
    \caption{Three randomly determined quantities are defining three layers of information upside down in the tree plot.}
    \label{fig:3Layinf-01}
\end{figure}
\end{center}
\end{enumerate}
Consider the doors are A, B and C instead of 1, 2, and 3. Now as the first layer of the tree (see in figure \eqref{fig:3Layinf-01}) is the door concealing the prize. So, the prize may be behind any of the three doors A,B and C in equally likely manner.

\noindent The player could initially choose any of the three doors is the second layer of the tree (see in figure \eqref{fig:3Layinf-01}).

\noindent Then a third layer of the tree (see in figure \eqref{fig:3Layinf-01}) is that the host opens a door to reveal a goat.

\noindent If the prize is behind door A and the player picks door A, then the host could open either door B or door C. So, the possible outcomes are AAB or AAC.

\noindent If the prize is behind door A and the player picks door B, then the host must open door C. So, a possible outcomes is ABC. 

\noindent Repeating all the possibility of above three randomly-determined quantities, then the leafs of the tree (see in figure \eqref{fig:3Layinf-01}) determine all the 12 possible outcomes of the random experiment. So, sample space is given by
\begin{align*}
\Omega &=\left\lbrace AAB,AAC,ABC,ACB,BAC,BBA,BBC,BCA,CAB, \right. \\
       & \left. CBA,CCA,CCB\right\rbrace.
\end{align*}
\item[b.] The event of interest is determined by the statement that the player wins by switching. So, the desired event is given by
\begin{align*}
E=\{ABC,ACB,BAC,BCA,CAB,CBA\}.
\end{align*} 
\item[c.] Computing probability measure of the desired event E. Consider the topmost event of the tree (see in figure \eqref{fig:3Layinf-01}) i.e. the outcome $AAB$, then, the probability of $AAB$ computed as:
\begin{align*}
P(AAB)=P(A)P(A|A)P(B|A,A)=\frac{1}{3}\cdot \frac{1}{3}\cdot \frac{1}{2}=\frac{1}{18}.
\end{align*}
Similarly, we can compute probability of every possible outcome as a leaf of the tree (see in figure \eqref{fig:3Layinf-01}) in the following probability tree diagram:
 \begin{center}
\begin{figure}[h]
    \centering
    \includegraphics[width=1.0\textwidth]{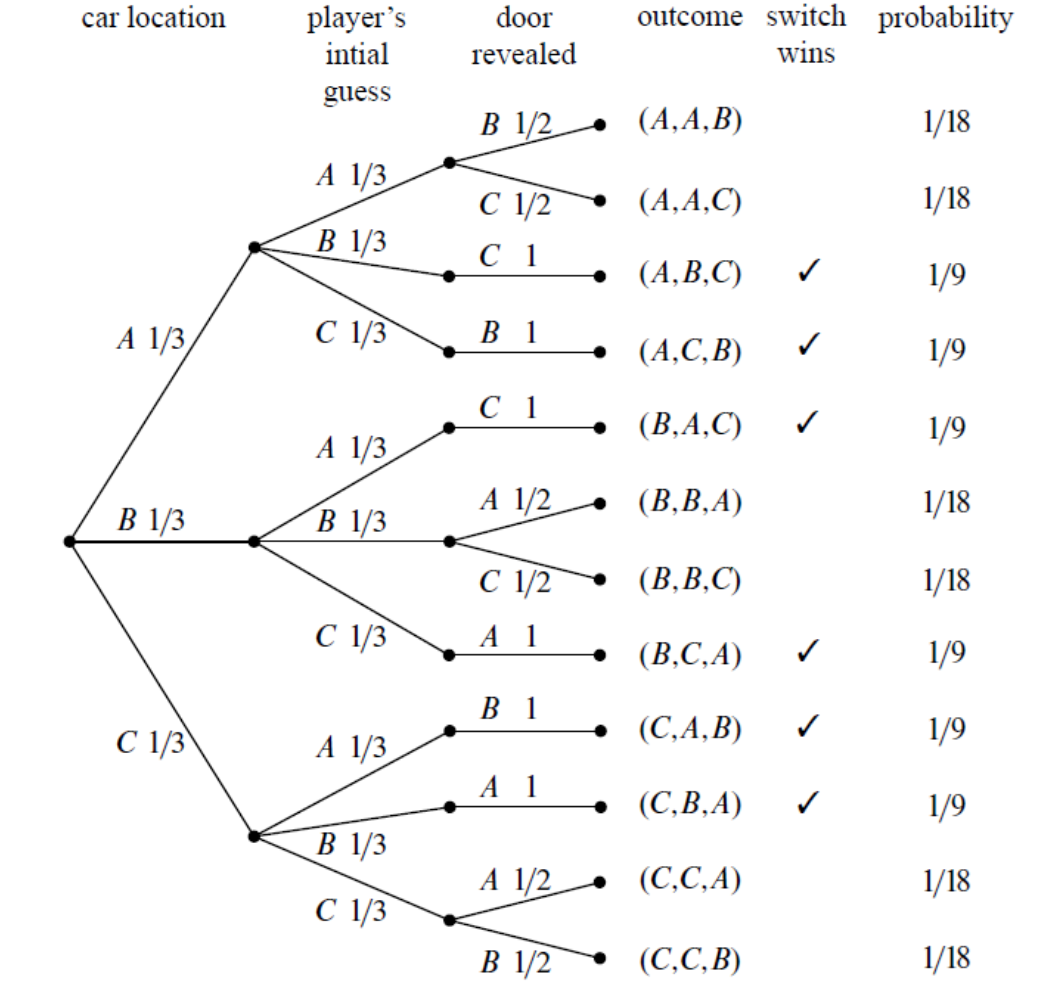}
    \caption{A plot of probability three is containing probability of every possible outcome of the Monte-Hall problem.}
    \label{fig:prob-tree-mhp-01}
\end{figure}
\end{center}
So, the desired probability is given by 
\begin{align*}
P(E)=P\left(\{ABC,ACB,BAC,BCA,CAB,CBA\}\right)=6\cdot \frac{1}{9}=\frac{2}{3}.
\end{align*}
\end{enumerate}
\end{example}

\begin{example}(sampling with order and replacement]
Compute number of distinct non-negative integer-valued solutions of the equation
\begin{align}
x_1+x_2+\ldots+x_k=n.
\label{eq:comb-01}
\end{align}
\noindent Consider a sequence of n $*$'s and m-1 $\vert$'s. There is a bijection between such sequences and non-negative integer-valued solutions to the equation. For example, if $m= 4$ and $n=3$, i.e.a sequence of $n=3 \ *$'s and $m-1=4-1=3 \ \vert$'s
\begin{align*}
\underbrace{\star \ \star}_{x_1=2} \ | \ \underbrace{ \ \ }_{x_2=0} | \underbrace{ \ \star \ }_{x_3=1} | \underbrace{ \ \ }_{x_4=0}.
\end{align*}
\noindent There are $\left(\begin{array}{cc} n+m-1 \\ n \end{array} \right)$ sequences of n $\*$'s and m-1 $\vert$'s and, hence, the same number of solutions to the equation \eqref{eq:comb-01}.
\end{example}

\begin{example}
Question: consider a urn with 15 identical (in shapes), but different (in colour) balls (i.e. an instant of seven white, two green, and six red). If you draw a ball from the urn, then compute the probability of drawing the ball from the urn.

\noindent Solution: Drawing a from the urn means either drawing a white as a type A, or drawing a red ball as a type B or drawing a green ball as a type C. So, out of the fifteen identical balls, A contains seven equally likely favourable outcomes, B contains six equally likely favourable outcomes and C contains two equally likely favourable outcomes. And hence, we have the following probabilities:
\begin{align*}
P(A)=\frac{7}{15}, \ P(B)=\frac{6}{15}, \ \mbox{and} \ P(C)=\frac{2}{15}.
\end{align*}
\end{example}
%

\begin{remark}
\begin{enumerate}
\item[a.] Two events $A$ and $B$ are mutually exclusive if $P(A\cap B)=0$.
\item[b.]
\begin{theorem}
Any two mutually disjoint events are mutually exclusive, but converse is not true.
\end{theorem} 
\begin{proof}
Consider $A$ and $B$ are two mutually disjoint events. Then, we have
\begin{align*}
A \cap B =\phi \implies P(A\cap B)=P(\phi)=0.
\end{align*}
So, $A$ and $B$ are two mutually exclusive.
\end{proof}
\end{enumerate}
\end{remark}

\section{Types of probabilistic models}
\noindent Consider the number of persons at a business location that are talking on their respective phones any-time between $9:00$ AM and $9:10$ AM. Clearly, the possible outcomes are $0,1,2,\ldots,n$, where $n$ is the number of persons in the office. On the other hand, if we are interested in the length of time a particular caller is on the phone during that time period, then the outcomes may be anywhere from $0$ to $T$ minutes, where $T= 10$. Now the outcomes are infinite in number since they lie within the interval $[0,T]$. In the first case, since the outcomes are discrete (and finite), we can assign probabilities to the outcomes $\{0,1,2,\ldots,n\}$. An equi-probable assignment would be to assign each outcome a probability of $1/(n+1)$. In the second case, the outcomes are continuous (and therefore infinite) and so it is not possible to assign a non-zero probability to each outcome. In the continuous case, it is not appropriate to ask for the probability that $T$ will be exactly, e.g. 5 minutes, because this probability will be zero. Instead, we inquire as to the probability that $T$ will be between 5 and 6 minutes.

\noindent So, based on discrete and continuum patterns of outcomes in a sample space of a random experiment, there are two types of probabilistic models:
\begin{enumerate}
\item[a.] Discrete probabilistic model
\item[b.] Continuous probabilistic model
\end{enumerate}
\subsection{Discrete probabilistic model}
A probabilistic model of a random experiment is discrete if the sample space of the random experiment is finite or at most countable (i.e. there are finitely many outcomes or infinitely many outcomes which can be arranged into a simple sequence) i.e. $\Omega=\{\omega\}_{k\in \mathbb{N}}$. And a probability measure of an event A can be computed in the following two steps:
\begin{enumerate}
\item[a.] Come up with an assumption or information from the random experiment, where first instant assumption happens to be uniform assumption (or equally likely outcomes).
\item[b.] Computing the probability measure of an event using the assumption observed from the random experiment in step-01 and the axioms of the probability measure of the event.
\end{enumerate}

\begin{theorem}[Discrete probabilistic modelling]
A probabilistic modelling of a random experiment is discrete if 
\begin{enumerate}
\item[a.] the sample space, $\Omega$, of the experiment is discrete i.e. $\Omega$ is at most countable i.e. $\Omega$ contains at most a countable collection of possible outcomes of the experiment i.e. $\Omega=\{\omega_n | \omega_n \ \mbox{is a possible outcome} \}$.
\item[b.] an even, $A$,  is at most countable as a sub-sequence of some possible outcomes i.e. $\{\omega_{n_k} | \omega_{n_k} \ \mbox{satisfies a given statement} \}$. 
\item[c.] Probability measure of an event (from a random experiment with a finite sample space) is defined as the ratio of number outcomes occur in the event $A$ to number outcomes occur in the sample space $\Omega$.
\end{enumerate}
\end{theorem}

\noindent Note: If the sample space of a random experiment is a finite i.e. $\Omega=\{x_1, x_2, \ldots, x_n\}$, then computation of probability measure of an event A is given by 
\begin{align*}
P(A)=\frac{\mbox{number of outcomes occur in A}}{\mbox{number of outcomes occur in} \ \Omega}.
\end{align*}

\begin{example}[Throwing a die]
Consider a throw of a die, then $\Omega = \{1,2,3,4,5,6\}$. Define an event $A\to$ even face or outcomes in the experiment, then $A=\{2,4,6\}$. So, $P(A)=3/6=1/2$. 
\end{example}
\begin{example}[Tossing a coin twice]
Consider two tosses of a coin, then $\Omega = \{HH, HT, TH, TT\}$. Define an event $A\to$ at least one head, then $A=\{HH, HT, TH\}$. So, $P(A)=3/4$. 
\end{example}

\begin{example}[Drawing ball from an urn \citep{F68}]
An urn has k red balls and $n-k$ black balls. If two balls are chosen in succession and at random with replacement, then compute the probability of a red ball followed by a black ball. 

\noindent To solve the problem, we first label the k red balls with $1,2,\ldots,k$ and the black balls with $k+1,k+2,\ldots,n$. In doing so the possible outcomes of the experiment can be represented by a 2-tuple $(z_1,z_2) : z_1=1,2,\ldots,n, z_2=1,2,\ldots,n$. A successful outcome is a red ball followed by a black one so that the successful event is $A=\{(z_1,z_2) : z_1=1,2,\ldots,k, z_2=k+1,k+2,\ldots,n\}$. Then the desired probability is given by
\begin{align*}
&a. \ P(A)=\frac{k(n-k)}{n^2}=\frac{k}{n}\left(1-\frac{1}{n} \right) \ \mbox{with replacement},\\
&b. \ P(A)=\frac{k(n-k)}{n(n-1)}=\frac{k}{n}\left(1-\frac{1}{n} \right)\frac{n}{n-1} \ \mbox{without replacement}.
\end{align*}
\end{example}

\begin{example}
Question: if two fair dice are tossed, find the probability that the same number will be observed on each one. Next, find the probability that different numbers will be observed.

\noindent Solution: if we toss two dice together than the sample is given by
\begin{align*}
\Omega = \left\lbrace(i,j)| 1\leq i \leq 6, 1\leq j \leq 6 \right\rbrace \implies |\Omega|=36.
\end{align*}
Now, probability that the same number will be observed on each one is given by
\begin{align*}
P\left(\{(1,1),(2,2),(3,3),(4,4),(5,5),(6,6)\}\right)=6\cdot \frac{1}{36}=\frac{1}{6}.
\end{align*}
And, the probability that different numbers will be observed is given by 
\begin{align*}
& P(\mbox{different numbers will be observed}) \\
&=1-P\left(\{(1,1),(2,2),(3,3),(4,4),(5,5),(6,6)\}\right)\\
&=1- 6\cdot \frac{1}{36}=1- \frac{1}{6}=\frac{5}{6}.
\end{align*}
\end{example}

\begin{example}[Birthday problem \citep{K06}]
A probability class has $n$ students enrolled. Then compute the probability that at least two of the students will have the same birthday. 

\noindent We first assume that each student in the class is equally likely to be born on any day of the year. To solve this problem consider a birthday urn that contains 365 balls. Each ball is labelled with a different day of the year. Now allow each student to select a ball at random, note its date, and return it to the urn. The day of the year on the ball becomes his/her birthday. The probability desired is of the event that two or more students choose the same ball. It is more convenient to determine the probability of the complement event or that no two students have the same birthday. Then, the desired probability is given by
\begin{align*}
& P(\mbox{at least 2 students have same birthday}) \\
& = 1-P(\mbox{no students have same birthday})\\
&=1-\frac{\left(\begin{array}{cc}
365 \\
n
\end{array} \right)}{365^n}.
\end{align*}
\begin{center}
\begin{figure}[h]
    \centering
    \includegraphics[width=1.0\textwidth]{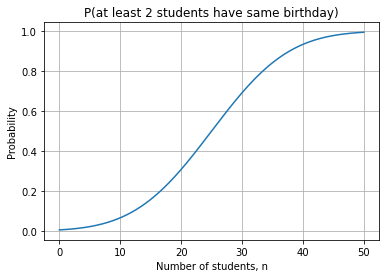}
    \caption{Probability of at least two students having the same birthday.}
    \label{fig:bdp-02}
\end{figure}
\end{center}
\end{example}

\subsection{Continuous probabilistic model}
A probabilistic model of a random experiment is continuous if the sample space of the random experiment is at least continuum in nature i.e. $\Omega\neq\{x_k\}_{k\in \mathbb{N}}$, but $\Omega$ is an interval or union of intervals. In this case, a probability measure of an event $A$ can be computed in the following two steps:
\begin{enumerate}
\item[a.] Come up with an assumption or information from the random experiment, where first instant assumption happens to be uniform assumption (or equally likely outcomes).
\item[b.] Computing the probability measure of an event using the assumption observed from the random experiment in step-01 and the axioms of the probability measure of the event.
\end{enumerate}

\begin{remark}[Uniform law]\label{rem:unlawcpm-01}
Consider a sample space $\Omega =[a,b]$ with a uniform partition $\{x_0=a,x_1,x_2,\ldots,x_n=b\}$ of width $\delta = x_i-x_{i-1}=\frac{b-a}{n}$ i.e. $[a,b]=[x_0,x_1]\cup [x_1,x_2]\cup \ldots \cup [x_{n-1},x_n]$ and the n sub-intervals are mutually disjoint except the common terminal points. 

\noindent As per uniform law or equally likely assumption we have
\begin{align*}
& P([x_0,x_1])=P([x_1,x_2])=\ldots=P([x_{n-1},x_n]) \ \mbox{and} \ P([x_{i-1},x_i]) \propto \delta \\
& \implies P([x_{i-1},x_i])=c_i\delta \ \mbox{for} \ i=1,2,\ldots,n \implies c_1\delta=c_2\delta=\ldots=c_n\delta \\
& \implies c_1=c_2=\ldots=c_n=c \ \mbox{so} \ P([x_{i-1},x_i])=c\delta \ \mbox{for} \ i=1,2,\ldots,n. 
\end{align*}
\noindent Now, using property of probability, we have
\begin{align*}
& P([x_0,x_1])+P([x_1,x_2])+\ldots+P([x_{n-1},x_n])=P(\Omega)=1 \implies nc\delta=1 \\
& \implies c=\frac{1}{n\delta}=\frac{1}{b-a}\implies P([x_{i-1},x_i])=\frac{\delta}{b-a} \ \mbox{for} \ i=1,2,\ldots,n.
\end{align*}
\noindent So, an equally likely or uniform law assumption for a continuous sample space is a valid one and produces a probability equal to the ratio of the length of the interval and the length of the sample space as follows:
\begin{align*}
P([a,b])=\frac{\ \mbox{length of} \ [a,b]}{\ \mbox{length of} \ \Omega}=\frac{b-a}{\ \mbox{length of} \ \Omega}.
\end{align*}
\noindent Since the probability of a point event occurring is zero, the probability of any interval 
\begin{align*}
P([a,b])=P((a,b])=P([a,b))=P((a,b))=\frac{b-a}{\ \mbox{length of} \ \Omega}.
\end{align*}
\end{remark}

\begin{remark}
In continuous probabilistic modelling, we define probability of a subset of $[0,1]$ as the length of the subset. It is easy to compute of a subset in an interval or union of intervals. In real analysis, the concept of length of a set is generalized to a measure of the set and it is always not possible to come up with an explicit probability measure of every subset of a continuum sample space. That is, we need to introduce some kind of finiteness pattern in the continuum set, so that, a concrete probability measure can be defined similar to computation of probabilities in the remark \eqref{rem:unlawcpm-01}. For the sack of simplicity, the finite pattern approach may lead to a well defined result that every measurable set in a continuum sample space is having a probability measure.
\end{remark}

\begin{example}[Linear dartboard]
Suppose one throws a dart at a linear dartboard as shown in Figure \eqref{fig:ldb-01} and measures the horizontal distance from the bullseye or center at $x=0$. We will then have a sample space $\Omega=\{x | -\frac{1}{2}\leq x \leq \frac{1}{2}\}$ which is continuum in nature i.e. not countable. A possible approach is to assign probabilities to intervals (i.e. events in $\Omega$) as opposed to sample points. If the dart is equally likely to land anywhere, then we could assign the interval $[a,b]$ a probability equal to the
length of the interval or
\begin{align*}
P([a,b])=\frac{b-a}{\frac{1}{2}+\frac{1}{2}}=b-a  \ \mbox{for} \ -\frac{1}{2} \leq a\leq b\leq \frac{1}{2}.
\end{align*}
\begin{center}
\begin{figure}[h]
    \centering
    \includegraphics[width=1.0\textwidth]{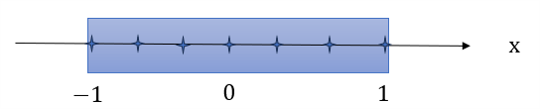}
    \caption{A linear dartboard.}
    \label{fig:ldb-01}
\end{figure}
\end{center}
\end{example}

\begin{example}
Question: a person always arrives at his job between 8:00 AM and 8:20 AM. He is equally likely to arrive any time within that period. What is the probability that he will arrive at 8:10 AM? What is the probability that he will arrive between 8:05 and 8:10 AM?

\noindent Solution: since arrival time between 08:00 AM to 08:20 AM. And the person is arriving in equally likely manner in the interval. So, we have
\begin{align*}
(a.) \ P(\mbox{he will arriave at 8:10 AM})=0.
\end{align*}
And
\begin{align*}
(b.) \ P(\mbox{he will arriave between 8:05 AM and 08:10 AM})&=\frac{08:10-08:05}{08:20-08:00} \\
                                                           &=\frac{5}{20}=\frac{1}{4}.
\end{align*}
\end{example}

\begin{example}[Arrival time]
A person always arrives at his job between 8:00 AM and 8:20 AM. He is equally likely to arrive any time within that period. What is the probability that he will arrive at 8:10 AM? What is the probability that he will arrive between 8:05 and 8:10 AM?

\noindent Solution: since arrival time between 08:00 AM to 08:20 AM. And the person is arriving in equally likely manner in the interval. So, we have
\begin{align*}
(a.) \ P(\mbox{he will arriave at 8:10 AM})=0.
\end{align*}
And
\begin{align*}
(b.) \ P(\mbox{he will arriave between 8:05 AM and 08:10 AM})&=\frac{08:10-08:05}{08:20-08:00} \\
                                                           &=\frac{5}{20}=\frac{1}{4}.
\end{align*}
\end{example}

\begin{example}[Romeo and Juliet meeting \citep{BT08}]
Romeo and Juliet have a date on given time and each will arrive at the meeting with a delay between 0 and 1 hour with all the pair of delays are equally likely. The first to arrive wait 15 minutes and will leave the place if the other has not arrived. Compute the probability of meeting.

\noindent The sample space of the experiment is as follows: $\Omega=\{(x,y)|0\leq x \leq 1, 0\leq y \leq 1\}$. Event of interest of the problem is $A=\{(x,y)||x-y|\leq \frac{1}{4}, 0\leq x \leq 1, 0\leq y \leq 1\}$ and is shaded in the figure \eqref{fig:rjcpm-01}.
\begin{center}
\begin{figure}[h]
    \centering
    \includegraphics[width=1.0\textwidth]{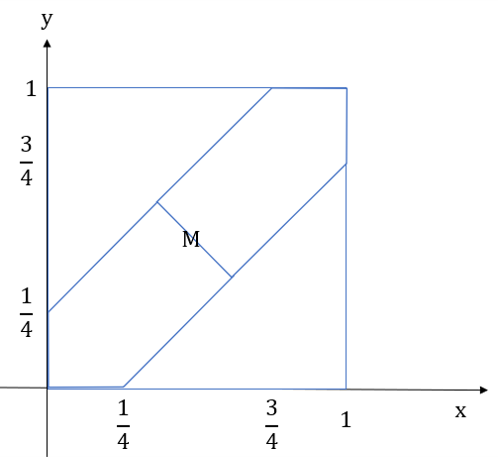}
    \caption{The event $A$ that Romeo and Juliet will arrive within 15 minutes of each other.}
    \label{fig:rjcpm-01}
\end{figure}
\end{center}
\noindent Thus, the desired probability is given by
\begin{align*}
P(A)=1-\frac{3}{4} \cdot \frac{3}{4}=1-\frac{9}{16}=\frac{7}{16}.
\end{align*}
\end{example}

\section{Conditional probability}
\noindent In many real-world experiments, an outcome may not be completely random since we have some prior knowledge. For instance, (a.) knowing that it has rained the previous 2 days might influence our assignment of the probability of sunshine for the following day, (b.) knowing that from a population, a person's height exceeds 6 ft might influence assignment of probability that he weighs more than $100$kg. It is the interaction between the original probabilities and the probabilities in light of prior knowledge that we wish to describe and quantify, leading to the concept of a conditional probability. So, it provides us with a way to reason about the outcome of an experiment, based on partial information.
\begin{definition}[Conditional probability]\label{def:conprob-01}
Suppose that we know outcome of an event $A$ is within some given event $B$, then conditional probability of $A$ given $B$ with $P(B)>0$ is defined as below:
\begin{align*}
P(A|B)=\frac{P(A\cap B)}{P(B)}.
\end{align*} 
\begin{center}
\begin{figure}[h]
    \centering
    \includegraphics[width=1.0\textwidth]{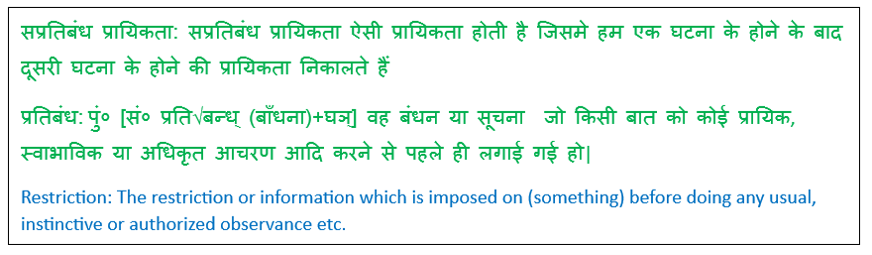}
    \caption{A self-explanatory meaning of conditional probability in Hindi with corresponding translation in English.}
    \label{fig:cond-prob-01}
\end{figure}
\end{center}
\end{definition}

\begin{example}
Question: In three successive tosses of a fair coin, compute the conditional probability $P(A|B)$ when $A$ and $B$ are the events 
\begin{align*}
A = \{\mbox{more heads than tails come up}\},\ B= \{\mbox{1st toss is a head}\}.
\end{align*} 

\noindent Answer: We have the sample space as a collection of all possible outcomes 
\begin{align*}
\Omega=\{HHH, HTH, HHT, THH, TTH, THT, HTT, TTT\}
\end{align*} 
and the events:
\begin{align*}
A = \{\mbox{more heads than tails come up}\},\ B= \{\mbox{1st toss is a head}\}.
\end{align*} 
Then
\begin{align*}
P(A|B)=\frac{P(A\cap B)}{P(B)}=\frac{3}{4}.
\end{align*}
\end{example}

\begin{example}[\citep{BT08}]
Question: a conservative design team, call it C, and an innovative design team, call it N, are asked to separately design a new product within a month. From past experience we know that:
\begin{enumerate}
\item[(a)] The probability that team C is successful is $2/3$. 
\item[(b)] The probability that team N is successful is $1/2$. 
\item[(c)] The probability that at least one team is successful is $3/4$.
\end{enumerate}
If both teams are successful, the design of team N is adopted. Assuming that exactly one successful design is produced, then compute the probability that it was designed by team N.

\noindent Answer: There are four possible outcomes here, corresponding to the four combinations of success and failure of the two teams as follows:
\begin{align*}
\Omega=\{SS,SF,FS,FF\} \ \mbox{where} \ SF \equiv C \ \mbox{succeeds and }\ N \mbox{fails}.
\end{align*}
\noindent 02: Events of interests and corresponding probabilities. 
\begin{enumerate}
\item[a.] $A_1 \to C$ is successful i.e.$A_1=\{SS, SF\}$, so $P(A_1)=2/3$.
\item[b.] $A_1 \to N$ is successful i.e. $A_2=\{FS, FF\}$, so $P(A_2)=1/2$.
\item[c.] $A_3 \to $ at least one team is successful i.e. $A_3=\{SS,SF,FS\}$, so $P(A_3)=3/4$.
\item[d.] From normalizing property of $\Omega$, we have $P(\{SS.SF,FS,FF\})=P(\Omega)=1$.
\end{enumerate}
So, we have
\begin{align*}
P(SS)=5/12,P(SF)=1/4,P(FS)=1/12,P(FF)=1/4.
\end{align*}
The the desired probability is given by
\begin{align*}
P(A)=P(\{SF\}|\{SF,FS\})=\frac{P(SF)}{P(SF)+P(FS)}=\frac{1}{4}.
\end{align*}
\end{example}

\begin{example}[bus arrival time]
Question: this morning, I caught the bus whose arrival time is based on an uniform assumption (i.e. it is equally likely to come at any time) that bus always comes within at most one hour. Compute the probability that whether bus came within the first 5 minutes in the following two situation (a.) not any further information, (b.) the bus actually came in the first 10 minutes.

\noindent Solution: In the absence of any further information, you would say this is quite unlikely as
\begin{align*}
P\left(t\in \left[0,\frac{1}{12} \right] \right)=\frac{1}{12}\approx 0.083.
\end{align*}
I give you a hint: I tell you that the bus actually came in the first 10 minutes. given this additional information, it seems much more likely that the bus came in the first 5 minutes than without this information in figure \eqref{fig:artime-01}.
\begin{align*}
P\left(t\in \left. \left[0,\frac{1}{12} \right] \right| \left[0,\frac{1}{6} \right] \right)=\frac{1}{2}=0.5.
\end{align*}
\begin{center}
\begin{figure}[h]
    \centering
    \includegraphics[width=1.0\textwidth]{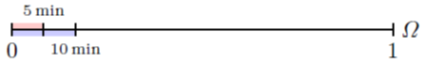}
    \caption{Bus arrival time conditioned on the bus actually came in the first 10 minutes.}
    \label{fig:artime-01}
\end{figure}
\end{center}
\end{example}

\subsection{Joint probability}
\begin{definition}[Joint probability \citep{BT08}]\label{def:jp-01}
In order to compute joint probability of two events A and B is a restatement of the definition of conditional probability \eqref{def:conprob-01} leads to a multiplication rule as:
\begin{align*}
P(A\cap B)=\left\lbrace\begin{array}{cc}
P(A)P(B|A) & \ \mbox{if A occurs first} \\
P(B)P(A|B) & \ \mbox{if B occurs first}.
\end{array} \right.
\end{align*}
Furthermore, an event A which occurs if and only if each one of several events $A_1,\ldots,A_n$ has occurred, i.e.
\begin{align*}
A=A_1\cap \ldots \cap A_n \to P(A)=P(A_1)P(A_2|A_1)\cdots P(A_n|A_1,\ldots,A_{n-1}).
\end{align*}
The occurrence of A can be viewed as an occurrence of $A_1$, followed by the occurrence of $A_2$, then of $A_3$, etc, and it is visualized as a path on the tree with n branches, corresponding to the events $A_1,\ldots,A_n$ in \eqref{fig:jpcr-02}.
\begin{center}
\begin{figure}[h]
    \centering
    \includegraphics[width=1.0\textwidth]{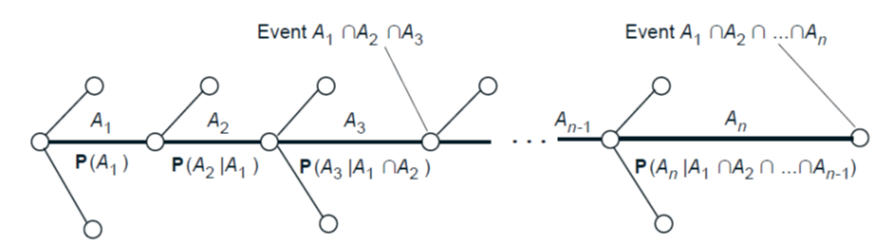}
    \caption{Visualizing occurrence of A through joint occurrence of $A_1,\ldots,A_n$ sequentially.}
    \label{fig:jpcr-02}
\end{figure}
\end{center}
\end{definition}

\begin{example}
Question: In a factory there are 100 units of a certain product, 5 of which are defective. When we pick three units from the 100 units at random, compute the probability that none of them are defective.

\noindent Solution: let $A_i$ is the event that the ith chosen unit is not defective, for $i=1,2,3$. Then, we need to compute $P(A_1\cap A_2 \cap A_3)$ as the desired probability of interest.

\noindent Now, we have 
\begin{align*}
P(A_1)=\frac{95}{100} \ \mbox{from uniform assumption}.
\end{align*}
Given that the first chosen item was good, the second item will be chosen from 94 good units and 5 defective units, thus
\begin{align*}
P(A_2|A_1)=\frac{94}{99} \ \mbox{from uniform assumtion conditioned on} \ A_1.
\end{align*} 
Given that the first and second chosen items were okay, the third item will be chosen from 93 good units and 5 defective units, thus
\begin{align*}
P(A_3|A_1,A_2)=\frac{93}{98} \ \mbox{from uniform assumtion conditioned on} \ A_1,A_2. 
\end{align*}
So, the desired probability is given by
\begin{align*}
P(A)&=P(A_1\cap A_2 \cap A_3)=P(A_1)P(A_2|A_1)P(A_3|A_1,A_2)\\
    &=\frac{95}{100} \cdot \frac{94}{99} \cdot \frac{93}{98} =0.8560.
\end{align*}
\end{example}

\begin{definition}[Independent events]
Consider two events A and B, if chance of occurrence of one first does effect the chance of occurrence of the other i.e. conditioned on the outcome of one event does not change the probability of the other event, then these two events are independent to each other. Mathematically,
\begin{align*}
P(B|A)=P(B) \ \mbox{or} \ P(A|B)=P(A) \implies P(A\cap B)=P(A)P(B).
\end{align*}
It can generalized to a sequence of events namely $A_1,A_2,\ldots,A_n$: the events $A_1,\ldots,A_n$ are independent if 
\begin{align*}
P(A_i|A_{j_1},\ldots,A_{j_k})=P(A_i) \ \mbox{for} \ i\neq j_1 \neq \ldots \neq j_k, 1\leq k\leq n.
\end{align*}
Furthermore, conditioned on an event C, the events A and B are called conditionally independent if
\begin{align*}
& P(B|A,C)=P(B|C) \ \mbox{or} \ P(A|B,C)=P(A|C) \implies  \\
& P(A\cap B|C)=P(A|C)P(B|C).
\end{align*}
\end{definition}

\begin{example}
Backdrop: Consider two independent fair coin tosses, in which all f a our possible outcomes are equally likely. Let 
\begin{align*}
& A_1= \{1st \ \mbox{toss is a head.}\}, A_2 =\{2nd \ \mbox{toss is head.}\},\\
& B=\{\mbox{the two tosses have different results.}\}
\end{align*}
Question: show that (a.) $A_1$ and $A_2$ are independent, (b.) but conditioned on B, $A_1$ and $A_2$ are dependent.

\noindent Answer: we have
\begin{align*}
& A_1=\{HT,HH\}, A_2=\{TH,HH\}, B=\{HT,TH\} \implies P(A_1)=\frac{1}{2}, \\
& P(A_2)=\frac{1}{2}, P(B)=\frac{1}{2}, P(A_1\cap A_2)=\frac{1}{4} \\
& P(A_1|B)=\frac{1}{1}, P(A_2|B)=\frac{1}{2}, P(A_1\cap A_2 |B)=0.
\end{align*}
Then
\begin{align*}
& P(A_1\cap A_2)=\frac{1}{4}=\frac{1}{2}\cdot \frac{1}{2}=P(A_1)P(A_2) \implies \\
& \ \mbox{$A_1$ and $A_2$ are indepedent}.
\end{align*}
Again
\begin{align*}
& P(A_1\cap A_2 |B)=0 \neq P(A_1|B)P(A_2|B)=\frac{1}{2}\cdot \frac{1}{2}=\frac{1}{4} \implies \\
& \ \mbox{$A_1|B$ and $A_2|B$ are depedent}.
\end{align*}
\end{example}

\subsection{Total probability}
\begin{definition}[Mutually disjoint events]
Consider any two sets $A$ and $B$. Then A and B are mutually disjoint if
\begin{align*}
A\cap B=\phi.
\end{align*}
Moreover, a list of events namely $E_1,E_2,E_3,\ldots$ are mutually disjoint if the intersection of any pair of events is empty, that is,
\begin{align*}
E_i\cap E_j =\phi \ \mbox{for} \ i\neq j, \ i,j=1,2,3,\ldots.
\end{align*}
\end{definition}

\begin{definition}[Mutually exclusive events]
Consider two events $A$ and $B$ in a probabilistic modelling of a random experiment. Then A and B are mutually exclusive if
\begin{align*}
P(A\cap B)=0 \ \mbox{i.e. joint occurance of A and B is impossible}.
\end{align*}
That is, happening of one in a single trial, kills the happening of the other (or others) in the same trial. For example, a coin cannot land with head and tail (or a die cannot land with several faces up), so they are exclusive outcomes.

\noindent Moreover, a list of events namely $E_1,E_2,E_3,\ldots$ are mutually exclusive if the joint occurrence of any pair of events is impossible, that is,
\begin{align*}
P(E_i\cap E_j) =0 \ \mbox{for} \ i\neq j, \ i,j=1,2,3,\ldots.
\end{align*}
And, for mutually exclusive events $E_1,E_2,E_3,\ldots$, we have
\begin{align*}
P(E_1)+P(E_2)+P(E_3)+\ldots  \leq 1=P(\Omega). 
\end{align*}
\end{definition}

\begin{definition}[collectively exhaustive]
A list of events $E_1,E_2,E_3,\ldots$ in a probabilistic modelling of a random experiment is collectively exhaustive if
\begin{align*}
E_1\cup E_2\cup E_3\cup \ldots = \Omega \implies P(E_1\cup E_2\cup E_3\cup \ldots)=1=P(\Omega).
\end{align*}
Moreover, if a list of events $E_1,E_2,E_3,\ldots$ are  mutually exclusive and collectively exhaustive, then
\begin{align*}
P(E_1) + P(E_2) + P(E_3)+ \ldots=1=P(\Omega).
\end{align*}
\end{definition}

\begin{example}
Consider a random experiment with the sample space: $\Omega=\{1,2,3,4\}$, and define a few events as:
\begin{align*}
A_1=\{1,2\}, \ A_2=\{3\}, \ B_1=\{1,2,3\}, \ B_2 =\{3,4\}.
\end{align*}
Then
\begin{enumerate}
\item[a.] $A_1=\{1,2\}$ and $A_2=\{3\}$ are mutually exclusive but not collectively exhaustive.
\item[b.] $B_1=\{1,2,3\}$ and $B_2=\{3,4\}$ are collectively exhaustive but not mutually exclusive.
\item[c.]$\{A_1,B_2\}=\{1,2,3,4\}$ is a collection of mutually exclusive and collectively exhaustive events, and is thus it forms a partition of the sample space $\Omega$.
\end{enumerate}
\end{example}

\begin{theorem}[Mutually disjoint vs mutually exclusive]
In a probabilistic modelling of a random experiments, mutually disjoint events are also mutually exclusive, but converse is not true.
\end{theorem}

\begin{proof}
Consider two mutually disjoint events A and B in a probabilistic modelling of a random experiment. Then
\begin{align*}
A \cap B = \phi \implies P(A\cap B)=P(\phi)=0.
\end{align*}
Since $P(A\cap B)$, so the mutually disjoint events A and B are also mutually exclusive.

\noindent A counter example of the converse: consider points in the square with each coordinate uniformly distributed from 0 to 1. Let A
 be the event where the x-coordinate is 0, and B be the event that the y-coordinate is 0. Then
 \begin{align*}
 A\cap B =\{(0,0)\}\neq \phi, \ \mbox{but} \ P(A\cap B)=0.
 \end{align*}
So A and B are not mutually disjoint, but they are mutually exclusive.

\noindent Another counter example: consider a sample space be $\Omega=\{x,y,z \}$ with probabilities $P(\{x \})=0, P(\{y \})=1/2$, and $P(\{z \})=1/2$. If $A=\{x,y \}$ and $B=\{x,z \}$, then $A\cap B=\{x \}$, but $P(A\cap B)=P(\{x \})=0$. That is, they are mutually exclusive but not disjoint.
\end{proof}

\begin{definition}[Total probability]\label{def:tp-03}
Consider a sample space $\Omega$ with a partition $\{B_1,\ldots,B_n\}$ i.e. 
\begin{enumerate}
\item[a.] mutually exclusive:
\begin{align*}
B_i \cap B_j =\phi \ \mbox{for any pair of} \ B_i,B_j, i\neq j. 
\end{align*}
\item[b.] collectively exhaustive:
\begin{align*}
B_1\cup B_2 \cup \ldots \cup B_n =\Omega.
\end{align*}
\end{enumerate}
Suppose $A\subset \Omega$ be an event in the experiment, then $\{B_1\cap A, \ldots,B_n\cap A\}$ introduces a partition of A as (in figure \eqref{fig:tpp-03}):
\begin{enumerate}
\item[a.] mutually exclusive:
\begin{align*}
(B_i\cap A) \cap B_j\cap A =\phi \ \mbox{for any pair of} \ B_i,B_j, i\neq j. 
\end{align*}
\item[b.] collectively exhaustive:
\begin{align*}
B_1\cap A \cup B_2\cap A \cup \ldots \cup B_n\cap A =A.
\end{align*}
\begin{center}
\begin{figure}[h]
    \centering
    \includegraphics[width=1.0\textwidth]{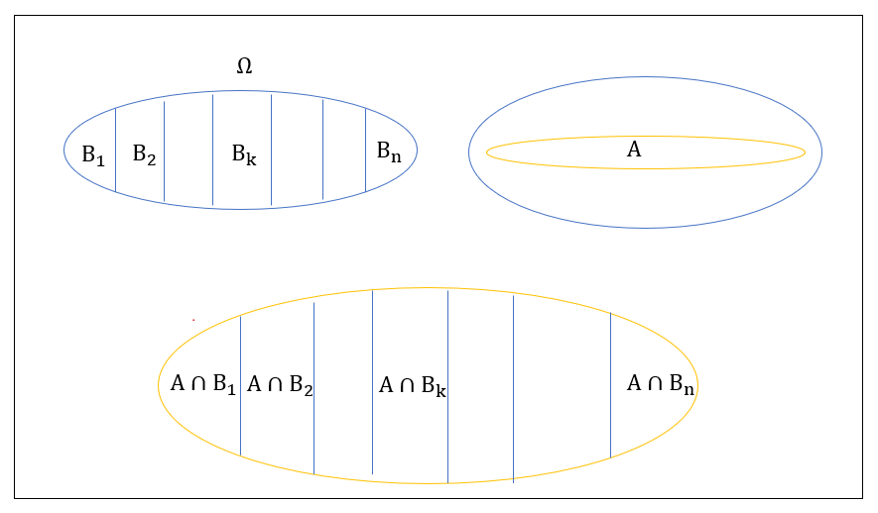}
    \caption{Visualizing a partition of event A through partition of the sample space $\Omega$.}
    \label{fig:tpp-03}
\end{figure}
\end{center}
Then,
\begin{align*}
P(A)&=P(A\cap B_1 \cup \ldots \cup A\cap B_n)=\sum_{i=1}^n P(A\cap B_i) \\
    &=\sum_{i=1}^n P(B_i)P(A|B_i).
\end{align*}
\end{enumerate}
\end{definition}

\begin{example}
Question: Three machines make parts at a factory. Suppose, we know the following about the manufacturing process that Machine 1 ($M_1$) makes $60\%$ of the parts, Machine 2 ($M_2$) makes $30\%$ of the parts, Machine 3 ($M_3$) makes $10\%$ of the parts. Of the parts $M_1$ makes $7\%$ are defective, of the parts $M_2$ makes $15\%$ are defective, of the parts $M_3$ makes $30\%$ are defective. Compute the probability of defective products produced by the factory.

\noindent Answer: the desired probability of defective product produced by the factory is given by
\begin{align*}
P(D)&=P(M_1)P(D|M_1)+P(M_2)P(D|M_2)+P(M_3)P(D|M_3) \\
    &=0.6\cdot 0.07 +0.3\cdot 0.15 + 0.1\cdot 0.3=0.117.
\end{align*}
\end{example} 

\subsection{Bayes rule}
\begin{definition}[Bayes rule]\label{def:brule-04}
Consider two events A and B with $P(A)>0, P(B)>0$, then Bayes rule is a restatement of the definition of conditional probability \eqref{def:conprob-01}  together with total probability \eqref{def:tp-03} as:
\begin{align*}
& P(A\cap B)=\left\lbrace\begin{array}{cc}
P(A)P(B|A) & \ \mbox{if A occurs first} \\
P(B)P(A|B) & \ \mbox{if B occurs first}.
\end{array} \right.\\
& \implies P(B|A)=\frac{P(A\cap B)}{P(A)}=\frac{P(A\cap B)}{\sum_{i=1}^nP(B_i)P(A|B_i)}
\end{align*}
where, events $\{B_1,\ldots,B_n\}$ together is a partition of the sample space $\Omega$, and hence it brings a partition of the event A as $\{A\cap B_1,\ldots,A\cap B_n\}$.

\noindent Furthermore, the probabilities of effect causes $\{B_1,\ldots,B_n\}$ conditioned on effect A, can be reverse calculated as
\begin{align*}
P(B_k|A)=\frac{P(A\cap B_k)}{P(A)}=\frac{P(B_k)P(A|B_k)}{\sum_{k=1}^nP(B_i)P(A|B_i)}.
\end{align*}
\end{definition}

\begin{example}[01: Medical diagnostics for a rare disease i.e. a false positive]
Le A be the event that the test comes back positive, and B be the event that you have the disease. Our modelling assumptions are: 1. $95\%$ of patients with the disease test positive i.e. $P(A|B)=0.95$, 2. $2\%$ of patients without the disease test positive i.e. $P(A|B^c)=0.02$, 3. one in a thousand people have the disease i.e. $P(B)=0.001$. Compute the probability $P(B|A)\to$ you have the disease given that the test comes back positive.

\noindent Answer: we have the following probabilistic data:
\begin{align*}
P(A|B)=0.95, P(A|B^c)=0.02, P(B)=0.001.
\end{align*}
Now, in order to compute the desired probability $P(B|A)$, we apply Bayes rule as :
\begin{align*}
P(B|A)&=\frac{P(B)P(A|B)}{P(B)P(A|B)+P(B^c)P(A|B^c)}=\frac{0.001\cdot 0.95}{0.001\cdot 0.95 + 0.999\cdot 0.02} \\
      &\approx 0.045.
\end{align*}
\end{example}

\begin{example}[Airport security]
Question: on an airport all passengers are checked carefully. Let T be the random variable indicating whether somebody is a terrorist (i.e.$t=1$) or not (i.e. $t=0$) and A be the variable indicating arrest (i.e. $A=a\in \{0,1\}$. A terrorist shall be arrested with probability $P(A= 1|T=1)=0.98$, a non-terrorist with probability $P(A=1|T=0)=0.001$. One in hundred thousand passengers is a terrorist, (i.e. $P(T=1)=0.00001$. Compute the probability that an arrested person actually is a terrorist.

\noindent Answer: using Bayes theorem, we have
\begin{align*}
P(T=1|A=1)&=\frac{P(T=1)P(A=1|T=1)}{P(A=1)} \\
          &=\frac{P(T=1)P(A=1|T=1)}{P(T=0)P(A=1|T=0)+P(T=1)P(A=1|T=1)} \\
          &=\frac{0.00001\cdot 0.98}{(1-0.00001)\cdot 0.001 + 0.00001\cdot 0.98} \approx 0.01
\end{align*}
\end{example}

\begin{remark}[Screening test for a disease]
One thousand in a 100 thousand people have the disease i.e. disease rate is $1\%$. $99\%$ accurate doesn't really give us information about the disease. So, we use the following terms:
\begin{enumerate}
\item[a.] Sensitivity - the odds that the test will be positive if you have the disease. A $98\%$ sensitive means, it will correctly identify 980 of 1000 people with the disease i.e. it also incorrectly claim $20=1000-980$ people with the disease don't have it,
\item[b.] Specificity - the odds that the test will be negative if you lack the disease. A $99\%$ specific means, it correctly identify 98,010 of 99,000 without the disease i.e. it also incorrectly claim $990=99000-98010$ people without the disease have it. 
\item[c.] Positive predictive value - the odds that the test will correctly predict you have the disease, if you test positive. So, out of $1970=980+990$ people (true positive and false positive) who test positive, 980 have the disease. Thus, our positive predictive value is $1000/1970=50.76\%$.
\item[d.] Negative predictive value - the odds that the test will correctly predict you lack the disease, if you test negative.So, Out of $98030=98010+20$ who test negative (false negative and true negative), 99000 do not have the disease. Thus, our negative predictive value is $98030/99000=99.02\%$.
\end{enumerate}
In this case, this test is first-rate for determining who lacks the disease. The 1970 (true positive and false positive), who test positive can be tested to confirm they do have the disease, whereas those who tested negative (false negative and true negative) need no further tests.
\end{remark}

\begin{remark}
The accuracy of a DNA test would be in its ability to predict a match vs non-match between two samples of blood. For example, I draw blood from one person into two tubes, then perform the test and see if it matches. Repeat a bunch and divide the positives by the total, that's the sensitivity. Then draw blood from two random donors, repeat, and divide the number of negatives over the total. That's the specificity. They are testing whether the machine can tell $A=A$ and $\neq B$, not whether a random sample of all the humans on the planet will match your blood specifically.
\end{remark}

\begin{example}[Drug test]
Question: a drug test (say T) has 1\% false positives (i.e. 1\% of those not taking drugs show positive in the test), and 5\% false negatives (i.e. 5\% of those taking drugs test negative). Suppose that 2\% of those tested are taking drugs. Compute the probability that somebody who tests positive is actually taking drugs (say D).

\noindent  Answer: we define here
\begin{enumerate}
\item[a.] $T=p \to $ Test positive,
\item[b.] $T=n \to $ Test negative,
\item[c.] $D=p \to $ person takes drug,
\item[d.] $D=n \to $ person does not take drugs
\end{enumerate}
Now, we have
\begin{enumerate}
\item[i.] false positive i.e. test positive given that person does not take drug: $P(T=p| D=n)=0.01$
\item[ii.] false negatives i.e. test negative given that person takes drug: $P(T=n|D=p)=0.05 \implies P(T=p| D=p)=0.95$ : true positives i.e. test positive given that person takes drug.
\item[iii.] drug consuming rate: $P(D=p)=0.02 \implies P(D=n)=0.98$. 
\end{enumerate}
Then, using Bayes theorem, the probability that somebody who tests positive is actually taking drugs computed as:
\begin{align*}
P(D=p| T=p)&=\frac{P(D=p)P(T=p| D=p)}{P(T=p)} \\
           &=\frac{P(D=p)P(T=p| D=p)}{P(D=n)P(T=p| D=n) + P(D=p)P(T=p| D=p)} \\
           &=\frac{0.02\cdot 0.95}{0.98\cdot 0.01 + 0.02\cdot 0.95} \approx 0.66.
\end{align*}
An alternative way to solve this exercise is using decision trees. Let's assume there are 1000 people tested. Then we have following decision tree:
\begin{center}
\begin{figure}[htb]
    \centering
    \includegraphics[width=1.0\textwidth]{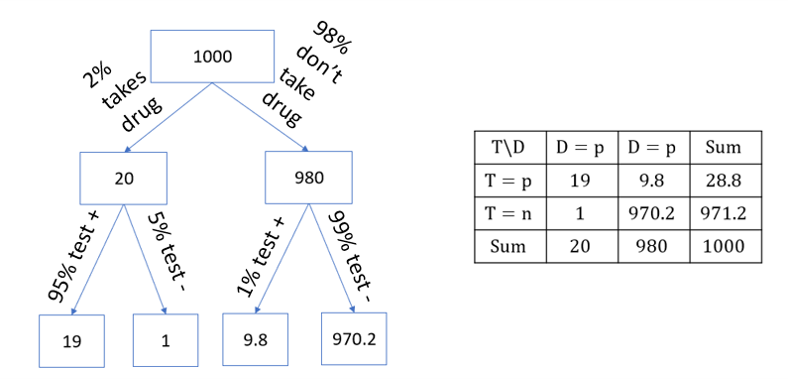}
    \caption{In left plot, there is a decision tree of 100 tested people and in the right plot there is a contingency table for tested people.}
    \label{fig:drug-test-bayes-01}
\end{figure}
\end{center}

So, from the contingency table \eqref{fig:drug-test-bayes-01}, we have
\begin{align*}
\frac{\mbox{taking drug and tests positive}}{\mbox{all positive tests}}=\frac{19}{28.8} \approx 0.66.
\end{align*}
There is a chance of only two thirds that someone with a positive test is actually taking drugs.
\end{example}

\begin{example}[Oral exam]
In an oral exam you have to solve exactly one problem, which might be one of three types, A, B, or C, which will come up with probabilities 30\%, 20\%, and 50\%, respectively. During your preparation you have solved 9 of 10 problems of type A, 2 of 10 problems of type B, and 6 of 10 problems of type C. Then, (a) compute the probability that you will solve the problem of the exam, (b.) given you have solved the problem, compute the probability that it was of type A.

\noindent Answer: the probability to solve the problem of the exam is the probability of getting a problem of a certain type times the probability of solving such a problem, summed over all types. This is known as the total probability.
\begin{align*}
P(solved) & = P(solved| A)P(A) + P(solved| B)P(B) + P(solved| C)P(C)  \\
          & = \frac{9}{10}\cdot 30\% + \frac{2}{10}\cdot 20\% + \frac{6}{10} \cdot 50\% = \frac{27}{100} + \frac{4}{100} + \frac{30}{100} = 0.61.
\end{align*}
Furthermore, using Bayes theorem, we have
\begin{align*}
P(A|solved)= \frac{P(A)P(solved| A)}{P(solved)}=\frac{30\% \cdot \frac{9}{10}}{0.61}= 0.442.
\end{align*}
\end{example}

%% file: dist.tex
Till now, we have seen in section \eqref{sec:bcpm} that in principle of probabilistic modelling of a random experiment in terms of outcomes and event and carried out through the following three basic concepts: 
\begin{enumerate}
\item[a.] Sample space $\to$ as a collection of all possible outcomes of a random experiment. 
\item[b.] Event $to$ as a subset of the sample (i.e. a collection of some possible outcomes) defined by an statement.
\item[c.] Probability measure $\to$ as a map $P:\Sigma \to [0,1]$ such that a. $P(\Phi)=0, P(\Omega)=1$, b. $0 \leq P(A) \leq 1$ for any event $A\in \Sigma$, c. $P(A\cup B)=P(A)+P(B)$ provided A and B are mutually disjoint (i.e. $A\cap B =\Phi$). In summary, compute the
probability of an event by breaking it into disjoint pieces whose probabilities are summed.
\end{enumerate} 
But in practice, everything we measure in real life is random in the sense (i.e. if we perform the experiment again we will not get the same measured value). So, every measurement is a random variable ($X:\Omega \to \Omega_X \subset \mathbb{R}, X^{-1}(B) \in \Sigma_{\Omega}$ for every $B\in \mathcal{B}$) that provides a concrete/computational representation $\Omega_X$ of the sample space $\Omega$ in a numeric/quantified platform.

\noindent We have already observed that an interesting and generic mathematical principle: to study, analyse and understand a space (with some structure) is extremely advantageous via study of functions from the space to a target space such that functions respect the structure of the space. For example:
\begin{enumerate}
\item An $m\times n$ matrix as a linear transformation from $\mathbb{R}^n$ to $\mathbb{R}^m$.
\item A linear representation of a group to a general linear group of matrices.
\item A rational function from a variety to an algebraic variety.
\item A random variable $X:\Omega \to \Omega_X \subset \mathbb{R}, X^{-1}(B) \in \Sigma_{\Omega}$ for every $B\in \mathcal{B}$.
\end{enumerate} 
This observation makes random variable much important for the computation of randomness in a real word situation.
  
\section{Random variable}
In Bernoulli trials, we are interested only in number of successes rather than successes or failure. In practice, everything in applied/computational probability is about random variables and their probability distributions.

\begin{definition}[Informal]\label{def:rv-01}
A random variable translates an outcome of nature in a given random experiment to some random number ( i.e. a number with a probability of occurrence). For example, distance is a function of a pair of points, the perimeter of a triangle is a function of triangles. In summary, a random variable translates non-numerical sample space $\Omega$ of a random experiment to numerical sample space $\Omega_X$. 
\end{definition}

\begin{definition}[Formal]\label{def:rv-02}
A random variable $X$ is a $(\mathbb{R},\mathcal{B})-$ valued map $X:\Omega \to \mathbb{R}_+$ such that $X^{-1}(B)=\{ \omega | X(\omega) \in B \} \subset \Sigma$ for every $B \in \mathcal{B}$ (i.e. $X^{-1}(B)$ is a measurable set w.r.t probability measure P), where $\Omega$ is a sample space of a given random experiment and $\mathbb{R}$ endures a Borel $\sigma -$ algebra $\mathcal{B}$ on it generated by open intervals. Alternatively, a random variable is a measurable function $X:(\Omega,\Sigma,P)\to (\mathbb{R},\mathcal{B})$.
\end{definition}

\begin{definition}[Probability distribution]\label{def:probdis-03}
Probability distribution of a random variable \eqref{def:rv-02} $X$ is defined as a probability measure $p_X=P\circ X^{-1}:\left(\mathbb{R},\mathcal{B}\right) \to [0,1]$ such that
\begin{align*}
P\circ X^{-1}(A)=P(X^{-1}(A))=P(X\in A)=P\left( \left\lbrace \omega | X(\omega)\in A \right\rbrace\right).
\end{align*}
It describes the likelihood of occurrence of the possible value of a random variable.
\begin{center}
\begin{figure}[h]
    \centering
    \includegraphics[width=1.0\textwidth]{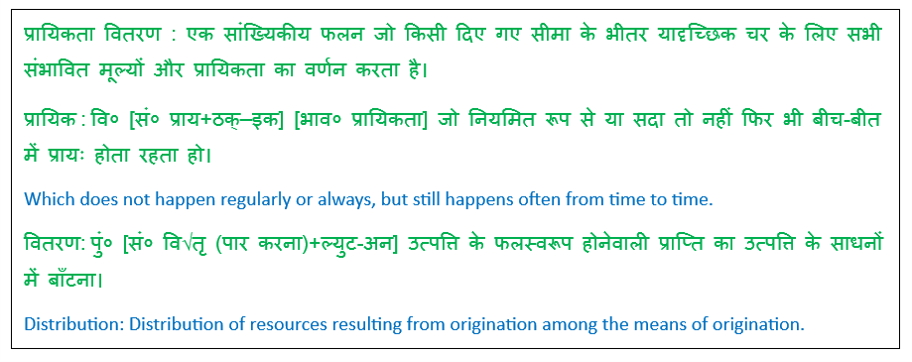}
    \caption{A self-explanatory meaning of probability distribution in Hindi with corresponding translation in English.}
    \label{fig:prob-hdef-01}
\end{figure}
\end{center}
\end{definition}

\begin{remark}
Once we get to more complicated examples, it can quickly become impractical to specify  explicitly. Although the concept of a probability space $(\Omega, \Sigma, P)$ underlies everything, in practice it will be rare that we think about itself, instead we will talk directly about events and their probabilities, and random variables and their distributions (and we can do that without assuming any particular structure for $\Omega$).
\end{remark}

\section{Discrete distribution} \label{sec:disdis-04}
\begin{definition}[Discrete random variable]\label{def:disrv-05}
A measurable function $X:\Omega =\{\omega_1, \omega_2, \ldots, \} \to \Omega_X=\{x_1,x_2,\ldots \} \subset \mathbb{R}$ with $X^{-1}(B)\in \Sigma_{\Omega} =\{X(\omega_1), X(\omega_2), \ldots, X(\omega_n)\} \ \forall \ B\in \mathcal{B}$ such that
\begin{enumerate}
\item[a.] $X^{-1}(x)=\{\omega | X(\omega)=x\} \in \Sigma_{\Omega}$,
\item[b.] $\Omega_X =\{x_k\}_{k\in\mathbb{N}}$ i.e. $\Omega_X$ is at most countable.
\item[a'.] $\Omega_X =\{x_k\}_{k\in\mathbb{N}} \implies B=\{x\} \implies X^{-1}(B)=X^{-1}(x) = \{\omega | X(\omega)=x\} \equiv \{X=x\}$.
\end{enumerate}
is called a discrete random variable and its probability distribution is known as probability mass function. An event of a discrete random variable $X$ is described as below:
\begin{align*}
X\in B = \{x\} \equiv \{X=x\} \equiv \{\omega | X(\omega)=x\}.
\end{align*}
Here, $\{X=x\} \to $ the event that the random variable $X$ takes the given value x.
\end{definition}

\begin{center}
\begin{figure}[h]
    \centering
    \includegraphics[width=1.0\textwidth]{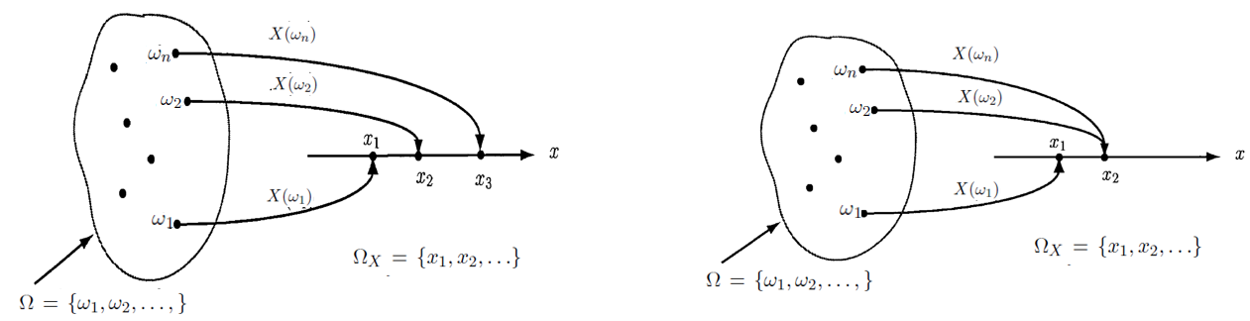}
    \caption{Discrete random variable as a mapping of at most countably sample space into a sequence of real numbers.}
    \label{fig:drv-01}
\end{figure}
\end{center}

\begin{definition}[Probability mass function]
Probability distribution of a discrete random variable (DRV) $X$ is known as a probability mass function (PMF) of the DRV and defined as follows:
\begin{align*}
p_X(x)=P(X=x)=P(\{\omega | X(\omega)=x\}).
\end{align*}
The PMF $p_X$ satisfies the following properties:
\begin{enumerate}
\item[a.] $p_X(x) \in [0,1]$,
\item[b.] $\sum_x P_X(x) =1$,
\item[c.] $P(X\in B) = \sum_{x\in B} p_X(x)$.
\end{enumerate}
\end{definition}

\subsection{Various discrete distribution}
\begin{example}[Bernoulli distribution]
A random variable $X$ (with an event $\{X=x\} \equiv \{\omega = H \ \mbox{or} \ T | X(\omega) =x\}$) has a Bernoulli distribution with a parameter $p\in [0,1]$ if
\begin{align*}
p_X(x) = P(X=x) = \left\lbrace \begin{array}{cc}
1-p & \ \mbox{if} \ x=0 \\
p  & \ \mbox{if} \ x=1
\end{array} \right.
= p^x(1-p)^{1-x} \ \mbox{for} \ x\in \{0,1\}.
\end{align*}
So, $X$ is a Bernoulli random variable and we write $X\sim Ber(p)$ (see in \eqref{fig:brpmf-01}). e.g. a. $X=I_A \to$ an indicator function for an event $A$, b. $X \to \#$ successes in a Bernoulli trial. 

\noindent The Bernoulli distribution is used to model, for example, the outcome of a Bernoulli trial with $1$ representing successes and $0$ representing failure. It is also a basic building block for other classical discrete distributions.

\noindent In the experiment of a single fair coin toss with a probability of getting head is $p$, if we are interested only in whether the outcome is head or tail, then we define
a Bernoulli random variable as follows:
\begin{align*}
X(\omega) =\left\lbrace\begin{array}{cc}
0 & \ \mbox{if} \ \omega=T \\
1 & \ \mbox{if} \ \omega=H.
\end{array} \right.
\end{align*}
With the following Bernoulli distribution as a PMF of the $X$:
\begin{align*}
p_X(0) & = p(X=0) =  \sum_{X^{-1}(0) \in \Sigma} P(\{\omega | X(\omega)=0\})=P(T)= 1-p, \\
p_X(1) & = p(X=1) = \sum_{X^{-1}(1) \in \Sigma} P(\{\omega | X(\omega)=1\})=P(H)= p.
\end{align*}

\begin{center}
\begin{figure}[h]
    \centering
    \includegraphics[width=0.7\textwidth]{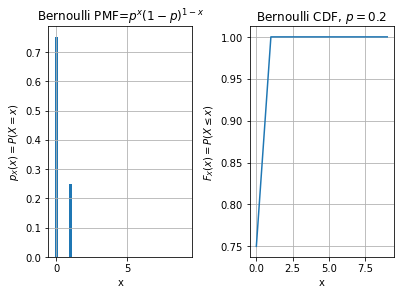}
    \caption{Bernoulli probability mass function for $p = 0.25$.}
    \label{fig:brpmf-01}
\end{figure}
\end{center}

\noindent In the experiment of a single fair die toss if we are interested only in whether the outcome is even or odd, then we define
a Bernoulli random variable as follows:
\begin{align*}
X(\omega) =\left\lbrace\begin{array}{cc}
0 & \ \mbox{if} \ \omega=1,3,5 \\
1 & \ \mbox{if} \ \omega=2,4,6.
\end{array} \right.
\end{align*}
With the following Bernoulli distribution as a PMF of the $X$:
\begin{align*}
p_X(0) & = p(X=0) = \sum_{X^{-1}(0) \in \Sigma} P(\{\omega | X(\omega)=0\})=3/6=1/2, \\
p_X(1) & = p(X=1) = \sum_{X^{-1}(1) \in \Sigma} P(\{\omega | X(\omega)=1\})=3/6=1/2.
\end{align*}
\end{example}

\noindent Here, we can draw a remark that once the PMF has been specified all subsequent probability calculations can be based on it, without referring back to the original sample space S. Specifically, for an event A defined on Sx the probability is given by
\begin{align*}
P(X\in A)=\sum_{i: x_i \in A} p_X(x_i).
\end{align*}
Consider a die whose sides have been labelled with two sides having 1 dot, two sides having 2 dots, and two sides having 3 dots. Hence, $\Omega=\{\omega_1=1,\omega_2=1,\omega_3=2,\omega_4=2,\omega_5=3,\omega_6=3\}$. Then if we are interested in the probabilities of the outcomes displaying either 1, 2, or 3 dots, we would define a random variable as

\begin{example}[Binomial distribution]
A random variable $X$ (with an event $\{X=x\} \equiv \{\omega \in \{H,T\}^n | X(\omega) =x = \# H\}$) has a Binomial distribution with a parameter $p\in [0,1]$ and $n \in \mathbb{N}$ if
\begin{align*}
p_X(x) = P(X=x) = \binom{n}{x} p^x(1-p)^{n-x} \ \mbox{for} \ x\in \{0,1,2,\ldots,n\}.
\end{align*}
So, $X$ is a Binomial random variable and we write $X\sim Bin(n,p)$ (see in \eqref{fig:bnpmf-02}). e.g. $X \to \#$ successes in n Bernoulli trials with $p$ as a probability of a success. Probability of a particular sequence of n Bernoulli trials with x successes is $p^x(1-p)^{n-x}$ and there are exactly $\left( \begin{array}{cc}
n \\
x
\end{array} \right)$ such sequences. 
\begin{center}
\begin{figure}[h]
    \centering
    \includegraphics[width=0.7\textwidth]{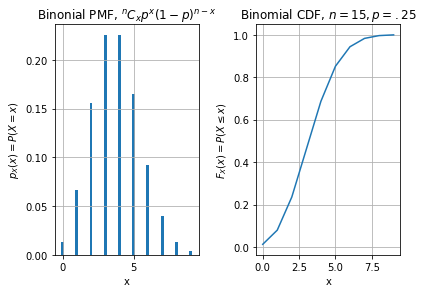}
    \caption{Binomial probability mass function with its cumulative distribution for $n=15,p = 0.25$.}
    \label{fig:bnpmf-02}
\end{figure}
\end{center} 
The location of the maximum of the PMF can be shown to be given by $\arg\max_{x}\left(\left( \begin{array}{cc}
n \\
x
\end{array} \right)
p^x(1-p)^{n-x} \right)=[(n+1)p]$ (see \citep{F68}), where $[x]$ denotes the largest integer less than or equal to x.
\end{example}

\begin{example}[Geometric distribution]
A random variable $X$ (with an event $\{X=x\} \equiv \{\omega \in \{H,T\}^{\infty} | X(\omega) =x = \# \ \mbox{Bernoulli trials till first} \ H\}$) has a Geometric distribution with a parameter $p\in [0,1]$ if
\begin{align*}
p_X(x) = P(X=x) = p(1-p)^{x-1} \ \mbox{for} \ x=1,2,\ldots.
\end{align*}
So, $X$ is a geometric random variable and we write $X\sim Geo(p)$ (see in \eqref{fig:gpmf-01}). e.g. $X \to \#$ Bernoulli trials till first success. 
\begin{center}
\begin{figure}[h]
    \centering
    \includegraphics[width=0.7\textwidth]{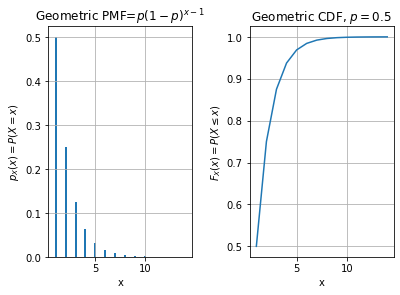}
    \caption{Geometric probability mass function for $p = 0.25$.}
    \label{fig:gpmf-01}
\end{figure}
\end{center}
\end{example}

\begin{example}[Poisson distribution]
A random variable $X$ has a Poisson distribution with a parameter $\lambda =np$, where n is possibly large and p is possibly very small, if
\begin{align*}
p_X(x) =P(X=x)& = \lim_{n \to \infty} \left( \left( \begin{array}{cc}
n \\
x
\end{array} \right)
p^x(1-p)^{n-x} \right) \\
& = \frac{\lambda^x e^{-\lambda}}{x!} \ \mbox{for} \ x=0,1,2,\ldots .
\end{align*}
So, $X$ is a Poisson random variable and we write $X\sim Pois(\lambda)$ (see in \eqref{fig:ppmf-02}). e.g. $X \to \#$ successes in Bernoulli trials per unit interval  (or within a given interval of time or space). 
\begin{center}
\begin{figure}[h]
    \centering
    \includegraphics[width=1.0\textwidth]{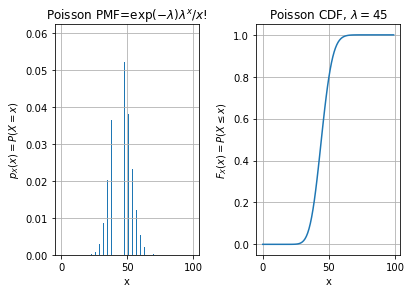}
    \caption{The Poisson probability mass function with its cumulative distribution for $\lambda=75$.}
    \label{fig:ppmf-02}
\end{figure}
\end{center}
\end{example}

\begin{remark}[Poisson approximation of binomial distribution]
\noindent If in a binomial PMF, we let $n \to \infty$ as $p \to 0$ such that the product $\lambda=np$ remains constant, then $Bin(n,p) \to Pois(\lambda)$. Note that $\lambda=np$ represents the expected or average number of successes in n Bernoulli trials, hence, by keeping the average number of successes fixed but assuming more and more trials with smaller and smaller probabilities of success on each trial, we are led to a Poisson PMF (see in \eqref{fig:bppmf-01}).
\begin{center}
\begin{figure}[h]
    \centering
    \includegraphics[width=1.0\textwidth]{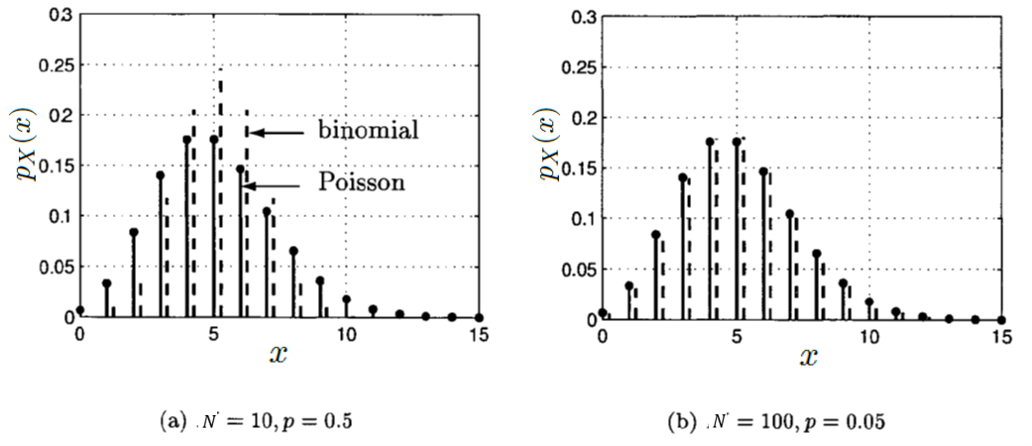}
    \caption{The Poisson approximation to the binomial probability mass function.}
    \label{fig:bppmf-01}
\end{figure}
\end{center}
\end{remark}

\begin{example}[Hyper-geometric distribution \citep{F68}]
In a population of n elements $n_1$ are red and $n_2$ are black. A group of r elements is chosen at random. We seek the probability $p_X(k)$ that the group so chosen will contain exactly k red elements. Here k can be any integer between zero and $n_1$ or r, whichever is smaller. We note that the chosen group contains k red and $r-k$ black elements. The red ones can be chosen in $\binom{n_1}{k}$ different ways and the black ones in $\binom{n-n_1}{r-k}$ ways. Since any choice of k red elements may be combined with any choice of black ones, so we have
\begin{align*}
p_X(k)=P(X=k;n,n_1,r)=\frac{\left( \begin{array}{cc} n_1 \\ k \end{array}\right) \left( \begin{array}{cc} n-n_1 \\ r-k \end{array}\right)}{\left( \begin{array}{cc} n \\ r \end{array}\right)}
\end{align*}
The hyper-geometric distribution can easily be generalized to the case where the original population of size n contains several classes of elements. For example, let the population contain three classes of sizes $n_1,n_2,$ and $n-n_1-n_2$ respectively. If a sample of size r is taken, the probability that it contains $k_1$ elements of the first, $k_2$ elements of the second, and $r-k_1-k_2$ elements of the last class is given by
\begin{align*}
p_X(k)=P(X=k;n,n_1,n_2,r)=\frac{\left( \begin{array}{cc} n_1 \\ k_1 \end{array}\right) \left( \begin{array}{cc} n_2 \\ k_2 \end{array}\right) \left( \begin{array}{cc} n-n_1-n_2 \\ r-k_1-k_2 \end{array}\right)}{\left( \begin{array}{cc} n \\ r \end{array}\right)}
\end{align*}
In Bridge game: the population of 52 cards consists of four classes, each of thirteen elements. The probability that a hand of thirteen
cards consists of five spades, four hearts, three diamonds, and one club is
\begin{align*}
p_X(k)=P(X=(5,4,3,1);52,13,13)=\frac{\left( \begin{array}{cc} 13 \\ 5 \end{array}\right)\left( \begin{array}{cc} 13 \\ 4 \end{array}\right)\left( \begin{array}{cc} 13 \\ 3 \end{array}\right) \left( \begin{array}{cc} 13 \\ 1 \end{array}\right)}{\left( \begin{array}{cc} 52 \\ 13 \end{array}\right)}
\end{align*}
\end{example}

\begin{example}[Pascal or negative binomial distribution \citep{F68}]
In a sequence of independent Bernoulli trials, let the random variable $X$ be the trial at which the $rth$ success occurs, where r is a fixed integer. Then X has a Pascal distribution as follows:
\begin{align*}
p_X(x)=\binom{x}{r-1}p(1-p)^{x-r}.
\end{align*}
Consider a random variable $Y \to $ number of failures before $rth$ success. This formulation is statistically equivalent to the above $X \to $ trial at which the $rth$ success occurs, since $Y=X-r$. And the alternative form of the Pascal distribution is
\begin{align*}
p_Y(y)= P(Y=y) & = \binom{r+y-1}{y} p^r(1-p)^y \\
& = \binom{-r}{y} p^r(p-1)^y.
\end{align*}
As we know negative binomial expansion follows:
\begin{align*}
(1+t)^{-r} = \sum_{k} \binom{-r}{k} t^k = \sum_{k}(-1)^k\binom{r+k-1}{k} t^k.
\end{align*}
The problem of Banach's match boxes \citep{F68}. A certain mathematician always carries one match box in his right pocket and one in
his left. When he wants a match, he selects a pocket at random, the successive choices thus constituting Bernoulli trials with $p=1/2$. Suppose that initially each box contained exactly N matches and consider the moment when, for the first time, our mathematician discovers that a box is empty. At that moment the other box may contain $0,1,2,\ldots, N$ matches, and we denote the corresponding probabilities by $u_r$. Let us identify success with choice of the left pocket. The left pocket will be found empty at a moment when the right pocket contains exactly r matches if, and only if, exactly $N-r$ failures precede the $(N+1)$st success. Therefore the required probability is
\begin{align*}
u_r = \binom{2N-r}{N} \left(\frac{1}{2} \right)^N \left(\frac{1}{2} \right)^{N-r}.
\end{align*}
Generalization of table tennis [Feller]. Suppose that Peter and Paul play a game which may be treated as a sequence of Bernoulli trials in which the probabilities p and q serve as measures for the players' skill. In ordinary table tennis the player who first accumulates
21 individual victories wins the whole game. For comparison with the preceding example we consider the general situation where $2\nu +1$ individual successes are required. The game lasts at least $2\nu +1$ and at most $4\nu +1$ trials. Denote by $a_r$ the probability that Peter wins at the trial number $4\nu +1-r$. This event occurs if, and only if, in the first $4\nu -r$ trials Peter has scored $2\nu$ successes and thereafter wins the $(2\nu +1)$st trial. So, the required probability is
\begin{align*}
a_r = \binom{4\nu -r}{2\nu} p^{2\nu +1}q^{2\nu - r}.
\end{align*}
\end{example}

\begin{example}[Multinomial distribution]
The binomial distribution can easily be generalized to the case of n repeated independent trials where each trial can have one of several outcomes. Denote the possible outcomes of each trial by $E_1,E_2,\ldots,E_r$, and suppose that the probability of the realization of $E_i$ in each trial is $p_i$ subject to the condition
\begin{align*}
p_1 + p_2 + \ldots + p_r =1.
\end{align*}
The result of n trials is a succession like $E_3E_1E_2\ldots$. The probability that in n trials $E_1$ occurs $k_1$ times, $E_2$ occurs $k_2$ times, etc., is
\begin{align*}
p_X(k)=P(X=k=(k_1,k_2,\ldots,k_r))=\frac{n!}{k_1! k_2! \ldots k_r!}p_1^{k_1}p_2^{k_2}\ldots p_r^{k_r},
\end{align*}
where, $k_1+k_2+\ldots+k_r=n$. If $r=2$, then the above multinomial distribution reduces to the binomial distribution with $p_1=p, p_2=1-p, k_1=k, k_2=n-k$.

\noindent For example: in rolling twelve dice, find the probability of getting each face twice. Here $E_1,E_2,\ldots, E_6$ represent the six faces, all $k_i$ equal 2, and all $p_i$ equal $1/6$. Therefore, the probability is $\frac{12!}{2^6}(1/6)^{12}=0.0034$.
\end{example}

\begin{example}[Uniform discrete distribution]
A random variable $X \in \{a,a+1,\ldots,b\}$ has a uniform discrete distribution if
\begin{align*}
p_X(x)=\left\lbrace\begin{array}{cc}
\frac{1}{b-a+1} & \ \mbox{if} \ x\in \{a,a+1,\ldots,b\}, \\
0 & \ \mbox{if} \ x\notin \{a,a+1,\ldots,b\}.
\end{array} \right.
\end{align*}
\end{example}

\subsection{Derived distribution of a discrete random variable}
\noindent Consider a die whose faces are labelled with the numbers $0,0,1,1,2,2$. We wish to find the PMF of the number observed when the die is tossed, assuming all sides are equally likely to occur. If the original sample space is composed of the possible cube sides that can occur, so that $\Omega_X=\{1,2,3,4,5,6\}$, then the transformation appears as follows:
\begin{align*}
Y= y = \left\lbrace\begin{array}{ccc} 
0 & \ \mbox{if} \ x= 1 \ \mbox{or} \ 2 \\
1 & \ \mbox{if} \ x= 3 \ \mbox{or} \ 4 \\
2 & \ \mbox{if} \ x= 5 \ \mbox{or} \ 6. 
\end{array} \right.
\end{align*}
So, PMF of Y is given by
\begin{align*}
p_Y(y)=\left\lbrace\begin{array}{ccc}
p_X(1)+p_X(2)=1/3 \ \mbox{if} \ y=0, \\
p_X(3)+p_X(4)=1/3 \ \mbox{if} \ y=1, \\
p_X(5)+p_X(6)=1/3 \ \mbox{if} \ y=2. \\
\end{array} \right.
\end{align*}
\begin{center}
\begin{figure}[h]
    \centering
    \includegraphics[width=1.0\textwidth]{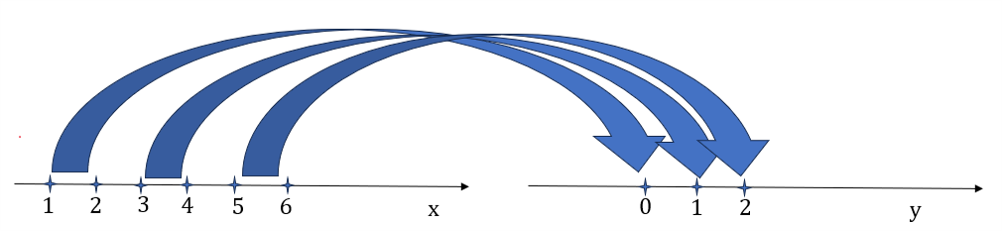}
    \caption{Transformation or function of a discrete random variable.}
    \label{fig:dddrv-01}
\end{figure}
\end{center}
If $Y=g(X)$ is a transformed or derived random variable of a given DRV $X$, then the derived distribution $p_Y(y)$ is given by
\begin{align*}
p_Y(y)=P(Y=y)=P(g(X)=y)=P(X=g^{-1}(y))=\sum_{\{x: g(x)=y\}} p_X(x),
\end{align*}
where $g^{-1}(y)$ is an inverse image of the observed value $Y=y$ in $\Omega_X$. 

\noindent So, computation of derived distribution $p_Y(y) $of a discrete random variable $X$ follows two steps:
\begin{enumerate}
\item Determining all possible observed value of the derived random variable $Y$ i.e. $\Omega_Y$,
\item Computing the derived distribution as follows:
\begin{align*}
p_Y(y)=\sum_{\{x: g(x)=y\}} p_X(x).
\end{align*}
\end{enumerate}

\begin{example}
If $X\sim Ber(p)$ and $Y=2X-1$, then derived distribution $p_Y(y)$ can be computed as follows:
\begin{enumerate}
\item As $X\sim Ber(p)$ and $Y=2X-1$, so $\Omega_X =\{0,1\}$ transformed to $\Omega_Y=\{-1,1\}$.
\item Now, we have
\begin{align*}
p_Y(-1) & =P(Y=-1)=P(2X-1=-1)=P(2X=0) \\
        & =P(X=0)=1-p, \\
p_Y(1) & =P(Y=1)=P(2X-1=1)=P(2X=2) \\
       & =P(X=1)=p.
\end{align*}
\end{enumerate} 
\end{example}

\begin{example}
If $X\sim Unif([-4,4]\cap \mathbb{Z})$ and $Y=|X|$, then derived distribution $p_Y(y)$ can be computed as follows:
\begin{enumerate}
\item As $X\sim Unif([-4,4]\cap \mathbb{Z})$ and $Y=|X|$, so $\Omega_X =\{-4,-3,-2,-1,0,1,2,3,4\}$ transformed to $\Omega_Y=\{0,1,2,3,4\}$.
\item Now, we have
\begin{align*}
p_Y(0) & =P(Y=0)=P(|X|=0)=P(X=0)=1/9, \\
p_Y(1) & =P(Y=1)=P(|X|=1)=P(X\in \{-1,1\})=2/9, \\
p_Y(2) & =P(Y=2)=P(|X|=2)=P(X\in \{-2,2\})=2/9 \\
p_Y(3) & =P(Y=3)=P(|X|=3)=P(X\in \{-3,3\})=2/9 \\
p_Y(4) & =P(Y=4)=P(|X|=4)=P(X\in \{-4,4\})=2/9.
\end{align*}
\end{enumerate} 
\begin{center}
\begin{figure}[h]
    \centering
    \includegraphics[width=1.0\textwidth]{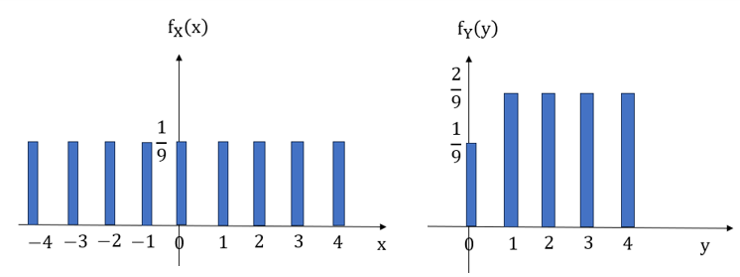}
    \caption{Transformation of $X\sim Unif([-4,4]\cap \mathbb{Z})$ to $Y=|X|$.}
    \label{fig:dddrv-02}
\end{figure}
\end{center}
\end{example}

\begin{example}
Let $Y=|X|$ and $X\sim Unif([-3,5]\cap \mathbb{Z})$ is a uniform random variable, then find PMF $p_Y$ and expectation $E(Y)$ of Y.

\noindent Solution: if $X\sim Unif([-4,4]\cap \mathbb{Z})$ and $Y=|X|$, then derived distribution $p_Y(y)$ can be computed as follows:
\begin{enumerate}
\item As $X\sim Unif([-3,5]\cap \mathbb{Z})$ and $Y=|X|$, so $\Omega_X =\{-3,-2,-1,0,1,2,3,4,5\}$ transformed to $\Omega_Y=\{0,1,2,3,4,5\}$.
\item Now, we have
\begin{align*}
p_Y(0) & =P(Y=0)=P(|X|=0)=P(X=0)=1/9, \\
p_Y(1) & =P(Y=1)=P(|X|=1)=P(X\in \{-1,1\})=2/9, \\
p_Y(2) & =P(Y=2)=P(|X|=2)=P(X\in \{-2,2\})=2/9 \\
p_Y(3) & =P(Y=3)=P(|X|=3)=P(X\in \{-3,3\})=2/9 \\
p_Y(4) & =P(Y=4)=P(|X|=4)=P(X\in \{4\})=1/9 \\
p_Y(5) & =P(Y=5)=P(|X|=5)=P(X\in \{5\})=1/9.
\end{align*}
\end{enumerate} 
\begin{center}
\begin{figure}[htb]
    \centering
    \includegraphics[width=1.0\textwidth]{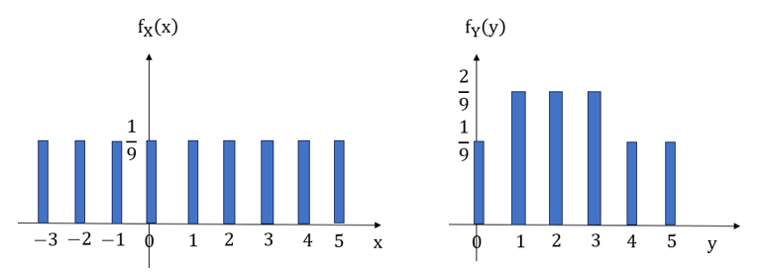}
    \caption{Transformation of $X\sim Unif([-3,5]\cap \mathbb{Z})$ to $Y=|X|$.}
    \label{fig:dddrv-03}
\end{figure}
\end{center}
\end{example}

\begin{example}
If $X\sim Unif([-2,2]\cap \mathbb{Z})$ and $Y=X^2$, then derived distribution $p_Y(y)$ can be computed as follows:
\begin{enumerate}
\item As $X\sim Unif([-2,2]\cap \mathbb{Z})$ and $Y=X^2$, so $\Omega_X =\{-2,-1,0,1,2\}$ transformed to $\Omega_Y=\{0,1,4\}$.
\item Now, we have
\begin{align*}
p_Y(0) & =P(Y=0)=P(X^2=0)=P(X=0)=1/5, \\
p_Y(1) & =P(Y=1)=P(X^2=1)=P(X\in \{-1,1\})=2/5, \\
p_Y(4) & =P(Y=4)=P(X^2=4)=P(X\in \{-2,2\})=2/5.
\end{align*}
\end{enumerate} 
\end{example}

\begin{example}
If $X\sim Pois(\lambda))$ and $Y=g(X)=\left\lbrace\begin{array}{cc} 1 & \ \mbox{if} \ x \mbox{is even} \\ -1 & \ \mbox{if} \ x \mbox{is odd} \end{array} \right.$, then derived distribution $p_Y(y)$ can be computed as follows:
\begin{enumerate}
\item As $X\sim Pois(\lambda))$ and $Y=g(X)$, so $\Omega_X =\{0,1,2,\ldots\}$ transformed to $\Omega_Y=\{-1,1\}$.
\item Now, we have
\begin{align*}
p_Y(1) & =P(Y=1)=P(X \ \mbox{is even})=\sum_{\{x \ \mbox{is even}\}} \frac{\lambda^x}{x!}e^{-\lambda}\\
       & =\frac{1}{2}\left(1+e^{2\lambda}\right), \\
p_Y(-1) & =P(Y=-1)=P(X \ \mbox{is odd})=\sum_{\{x \ \mbox{is odd}\}} \frac{\lambda^x}{x!}e^{-\lambda} \\
        & =\frac{1}{2}\left(1-e^{-2\lambda}\right).
\end{align*}
\end{enumerate} 
\end{example}

\newpage 
 
\subsection{Jointly discrete distribution}
In science and in real life, we are often interested in two (or more) random variables at the same time in a random experiment. In such situations the random variables have a joint distribution that allows us to compute probabilities of events involving both variables and understand the relationship between the variables. This is simplest when the variables are independent. When they are not, we use
covariance and correlation as measures of the nature of the dependence between them.

\begin{definition}[Joint probability mass function]
Suppose X and Y are two discrete random variables and that X takes values $\{x_1, x_2,\ldots , x_m\}$ and Y takes values $\{y_1, y_2,\ldots , y_n\}$. Then, the ordered pair $(X,Y)$ take values in the product $\{(x_1,y_1),(x_1, y_2), \ldots(x_m,y_n)\}$. So, the joint probability mass function (joint pmf) of X and Y can be defined as the function $p_{(X,Y)}(x_i,y_j)=P(X=x_i,Y=y_j)$ as the joint probability of the joint outcome $(X=x_i,Y=y_j)$ and satisfies the following:
\begin{enumerate}
\item $p_{(X,Y)}(x,y) \in [0,1]$ for any $(x,y)\in \Omega_X \times \Omega_Y$,
\item $\sum_{\{(x,y) \in \Omega_X \times \Omega_Y\}}p_{(X,Y)}(x,y) =1$,
\item $P((X,Y)\in B)=\sum_{\{(x,y)\in B\}}p_{(X,Y)}(x,y)$.
\end{enumerate}
More, explicitly, we can visualize the joint probability mass function for each joint point $\{(x_i,y_j)| 1\leq i \leq m, 1\leq j \leq n\}$ in the table \eqref{tab:jpmfdef-01}:
\begin{table}[htb]

    \centering 
    \caption{Joint probability table for joint probability mass function $p_{(X,Y)}$ of $X\in \Omega_X=\{x_1,x_2,\ldots,x_n\}$ and $Y\in \Omega_Y=\{y_1,y_2,\ldots, y_m\}$.}
    \scalebox{0.85}{ 
    \begin{tabular}{ | c | c | c | c | c | c | c | c | c | c |}
        \hline
        $X \backslash Y$ & $y_1$ & $y_2$ & $\ldots$ & $y_j$ & $\ldots$ & $y_n$  \\ 
        \hline
        $x_1$ & $P_{X,Y}(x_1,y_1)$ & $P_{X,Y}(x_1,y_2)$ & $\ldots$ & $P_{X,Y}(x_1,y_j)$ & $\ldots$& $P_{X,Y}(x_1,y_n)$ \\ 
         \hline
        $x_2$ & $P_{X,Y}(x_2,y_1)$ & $P_{X,Y}(x_2,y_2)$ & $\ldots$ & $P_{X,Y}(x_2,y_j)$ & $\ldots$& $P_{X,Y}(x_2,y_n)$ \\ 
         \hline
        $\vdots$ & $\vdots$ & $\vdots$ & $\ldots$ & $\vdots$ & $\ldots$& $\vdots$ \\ 
         \hline 
        $x_i$ & $P_{X,Y}(x_i,y_1)$ & $P_{X,Y}(x_i,y_2)$ & $\ldots$ & $P_{X,Y}(x_i,y_j)$ & $\ldots$& $P_{X,Y}(x_i,y_n)$ \\ 
         \hline
        $\vdots$ & $\vdots$ & $\vdots$ & $\ldots$ & $\vdots$ & $\ldots$& $\vdots$ \\  
         \hline 
        $x_m$ & $P_{X,Y}(x_m,y_1)$ & $P_{X,Y}(x_m,y_2)$ & $\ldots$ & $P_{X,Y}(x_m,y_j)$ & $\ldots$& $P_{X,Y}(x_m,y_n)$ \\ 
         \hline  
    \end{tabular}\label{tab:jpmfdef-01}}

\end{table}
\end{definition}

\begin{example}
Question: consider a roll two dice in together. Let X be the value on the first die and let Y be the value on the second die. Then both X and Y take values 1 to 6 and the joint pmf is $p(i,j)=1/36$ for all i and j between 1 and 6. Draw the probability table for the joint PMF for each observation of the jointly discrete random variable $(X,Y)$. Further, consider an event $B=Y-X\geq 2$, then compute its probability. 

\noindent Answer: It is given that X be the value on the first die and let Y be the value on the second die. So, both X and Y take values 1 to 6 with joint pmf is $p(i,j)=1/36$ for all i and j between 1 and 6. Then, the probability table of the joint PMF is given in the table \eqref{tab:jpmfdice-01}:
\begin{table}[htb]

    \centering 
    \caption{Joint probability table for joint probability mass function $p_{(X,Y)}$ of $X\in \Omega_X=\{1,2,\ldots, 6\}$ and $Y\in \Omega_Y=\{1,2,\ldots,6\}$.}
    \scalebox{1.5}{ 
    \begin{tabular}{ | c | c | c | c | c | c | c | c | c | c |}
        \hline
        $X \backslash Y$ & 1 & 2 & 3 & 4 & 5 & 6  \\ 
        \hline
        1 & $\frac{1}{36}$ & $\frac{1}{36}$ & $\frac{1}{36}$ & $\frac{1}{36}$ & $\frac{1}{36}$ & $\frac{1}{36}$ \\ 
         \hline
        2 & $\frac{1}{36}$ & $\frac{1}{36}$ & $\frac{1}{36}$ & $\frac{1}{36}$ & $\frac{1}{36}$ & $\frac{1}{36}$ \\
         \hline
        3 & $\frac{1}{36}$ & $\frac{1}{36}$ & $\frac{1}{36}$ & $\frac{1}{36}$ & $\frac{1}{36}$ & $\frac{1}{36}$ \\
         \hline 
        4 & $\frac{1}{36}$ & $\frac{1}{36}$ & $\frac{1}{36}$ & $\frac{1}{36}$ & $\frac{1}{36}$ & $\frac{1}{36}$ \\
         \hline
        5 & $\frac{1}{36}$ & $\frac{1}{36}$ & $\frac{1}{36}$ & $\frac{1}{36}$ & $\frac{1}{36}$ & $\frac{1}{36}$ \\ 
         \hline 
        6 & $\frac{1}{36}$ & $\frac{1}{36}$ & $\frac{1}{36}$ & $\frac{1}{36}$ & $\frac{1}{36}$ & $\frac{1}{36}$ \\ 
         \hline  
    \end{tabular}\label{tab:jpmfdice-01}}

\end{table}

\noindent Further, we have 
\begin{align*}
B & =Y-X\geq 2  \\
 &=\{(1,3),(1,4),(1,5),(1,6),(2,4),(2,5),(2,6),(3,5),(3,6),(4,6)\}. 
\end{align*}
So, the desired probability of the event B is given by
\begin{align*}
P(B)=P((X,Y)|Y-X\geq 2)=10\cdot \frac{1}{36}=\frac{5}{18}.
\end{align*}
\end{example}

\begin{example}
Question: consider a roll two dice in together. Let X be the value on the first die and let T be the total on both dice. Then X takes values 1 to 6, Y takes values 2 to 12 and the joint pmf is $p_{X,T}(i,t)=p_{X,Y}(i,t-i)$ for all $i\in \{1,2,3,4,5,6\}$ and $t\in \{2,3,4,5,6,7,8,9,10,11,12\}$. Draw a probability table for each possible observation of the jointly discrete random variable $(X,T)$.

\noindent Answer: It is given that X be the value on the first die and let T be the total on both dice. Then X takes values 1 to 6, Y takes values 2 to 12 and the joint pmf is $p_{X,T}(i,t)=p_{X,Y}(i,t-i)$ for all $i\in \{1,2,3,4,5,6\}$ and $t\in \{2,3,4,5,6,7,8,9,10,11,12\}$. Then, the probability table of the joint PMF is given in the table \eqref{tab:jpmfdice-02}:
\begin{table}[htb]

    \centering 
    \caption{Joint probability table for joint probability mass function $p_{(X,T)}$ of $X\in \Omega_X=\{1,2,\ldots, 6\}$ and $T\in \Omega_T=\{2,3,\ldots,12\}$.}
    \scalebox{1.0}{ 
    \begin{tabular}{ | c | c | c | c | c | c | c | c | c | c | c | c |}
        \hline
        $X \backslash T$ & 2 & 3 & 4 & 5 & 6 & 7 & 8 & 9 & 10 & 11 & 12  \\ 
        \hline
        1 & $\frac{1}{36}$ & $\frac{1}{36}$ & $\frac{1}{36}$ & $\frac{1}{36}$ & $\frac{1}{36}$ & $\frac{1}{36}$ & 0 & 0 & 0 & 0 & 0 \\ 
         \hline
        2 & 0 & $\frac{1}{36}$ & $\frac{1}{36}$ & $\frac{1}{36}$ & $\frac{1}{36}$ & $\frac{1}{36}$ & $\frac{1}{36}$ & 0 & 0 & 0 & 0 \\
         \hline
        3 & 0 & 0 & $\frac{1}{36}$ & $\frac{1}{36}$ & $\frac{1}{36}$ & $\frac{1}{36}$ & $\frac{1}{36}$ & $\frac{1}{36}$ & 0 & 0 & 0 \\
         \hline 
        4 & 0 & 0 & 0 & $\frac{1}{36}$ & $\frac{1}{36}$ & $\frac{1}{36}$ & $\frac{1}{36}$ & $\frac{1}{36}$ & $\frac{1}{36}$ & 0 & 0 \\
         \hline
        5 & 0 & 0 & 0 & 0 & $\frac{1}{36}$ & $\frac{1}{36}$ & $\frac{1}{36}$ & $\frac{1}{36}$ & $\frac{1}{36}$ & $\frac{1}{36}$ & 0 \\ 
         \hline 
        6 & 0 & 0 & 0 & 0 & 0 & $\frac{1}{36}$ & $\frac{1}{36}$ & $\frac{1}{36}$ & $\frac{1}{36}$ & $\frac{1}{36}$ & $\frac{1}{36}$ \\ 
         \hline  
    \end{tabular}\label{tab:jpmfdice-02}}

\end{table}
\end{example}


\begin{definition}[Marginal PMFs]
Marginalization of joint probability mass function: when X and Y are jointly-distributed random variables, we may want to consider only one of them, say X. In that case we need to find the pmf of X without Y. This is called a marginal pmf of X. For each observation/event $X=x$, the $p_X$ is computed by summing probability of points on $x --$ strip i.e.
\begin{align*}
p_X(x) = \sum_{y\in \Omega_Y} p_{X,Y}(x,y). \ \mbox{Similarly} \ p_Y(y) = \sum_{x\in \Omega_X} p_{X,Y}(x,y).
\end{align*} 
\begin{center}
\begin{figure}[h]
    \centering
    \includegraphics[width=1.0\textwidth]{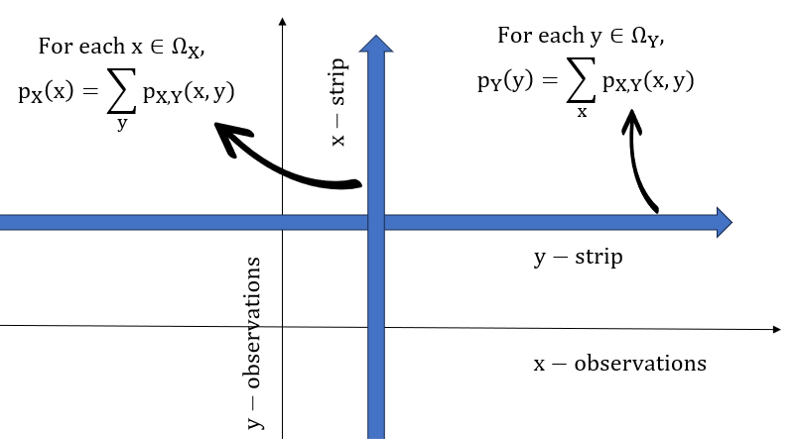}
    \caption{Visualizing marginal PMFs $p_X(x)$ and $p_Y(y)$ respectively.}
    \label{fig:mpmf-01}
\end{figure}
\end{center}
\end{definition}

\begin{example}
Question: consider a roll two dice in together. Let X be the value on the first die and let T be the total on both dice. Then X takes values 1 to 6, Y takes values 2 to 12 and the joint pmf is $p_{X,T}(i,t)=p_{X,Y}(i,t-i)$ for all $i\in \{1,2,3,4,5,6\}$ and $t\in \{2,3,4,5,6,7,8,9,10,11,12\}$. In the table each row represents a single value of X. So the event $X=3$ is the third row of the table. To find $P(X=3)$ we simply have to sum up the probabilities in this row. We put the sum in the right-hand margin of the table. Likewise $P(T=5)$ is just the sum of the column with $T=5$. We put the sum in the bottom margin of the table. Draw a probability table for each possible observation of the jointly discrete random variable $(X,T)$ along with marginal PMFs.

\noindent Answer: It is given that X be the value on the first die and let T be the total on both dice. Then X takes values 1 to 6, Y takes values 2 to 12 and the joint pmf is $p_{X,T}(i,t)=p_{X,Y}(i,t-i)$ for all $i\in \{1,2,3,4,5,6\}$ and $t\in \{2,3,4,5,6,7,8,9,10,11,12\}$. Then, the probability table of the joint PMF is given in the table \eqref{tab:jpmfdice-03}:
\begin{table}[htb]

    \centering 
    \caption{Joint probability table for joint probability mass function $p_{(X,T)}$ of $X\in \Omega_X=\{1,2,\ldots, 6\}$ and $T\in \Omega_T=\{2,3,\ldots,12\}$ along with marginal PMFs $p_X(x)$ and $p_T(t)$.}
    \scalebox{1.0}{ 
    \begin{tabular}{ | c | c | c | c | c | c | c | c | c | c | c | c | c |}
        \hline
        $X \backslash T$ & 2 & 3 & 4 & 5 & 6 & 7 & 8 & 9 & 10 & 11 & 12 & $p_X(x)$ \\ 
        \hline
        1 & $\frac{1}{36}$ & $\frac{1}{36}$ & $\frac{1}{36}$ & $\frac{1}{36}$ & $\frac{1}{36}$ & $\frac{1}{36}$ & 0 & 0 & 0 & 0 & 0 & $\frac{1}{6}$ \\ 
         \hline
        2 & 0 & $\frac{1}{36}$ & $\frac{1}{36}$ & $\frac{1}{36}$ & $\frac{1}{36}$ & $\frac{1}{36}$ & $\frac{1}{36}$ & 0 & 0 & 0 & 0 & $\frac{1}{6}$ \\
         \hline
        3 & 0 & 0 & $\frac{1}{36}$ & $\frac{1}{36}$ & $\frac{1}{36}$ & $\frac{1}{36}$ & $\frac{1}{36}$ & $\frac{1}{36}$ & 0 & 0 & 0 & $\frac{1}{6}$ \\
         \hline 
        4 & 0 & 0 & 0 & $\frac{1}{36}$ & $\frac{1}{36}$ & $\frac{1}{36}$ & $\frac{1}{36}$ & $\frac{1}{36}$ & $\frac{1}{36}$ & 0 & 0 & $\frac{1}{6}$ \\
         \hline
        5 & 0 & 0 & 0 & 0 & $\frac{1}{36}$ & $\frac{1}{36}$ & $\frac{1}{36}$ & $\frac{1}{36}$ & $\frac{1}{36}$ & $\frac{1}{36}$ & 0 & $\frac{1}{6}$ \\ 
         \hline 
        6 & 0 & 0 & 0 & 0 & 0 & $\frac{1}{36}$ & $\frac{1}{36}$ & $\frac{1}{36}$ & $\frac{1}{36}$ & $\frac{1}{36}$ & $\frac{1}{36}$ & $\frac{1}{6}$ \\ 
         \hline 
        $p_T(t)$ & $\frac{1}{36}$ & $\frac{2}{36}$ & $\frac{3}{36}$ & $\frac{4}{36}$ & $\frac{5}{36}$ & $\frac{6}{36}$ & $\frac{5}{36}$ & $\frac{4}{36}$ & $\frac{3}{36}$ & $\frac{2}{36}$ & $\frac{1}{36}$ & 1 \\ 
         \hline  
    \end{tabular}\label{tab:jpmfdice-03}}

\end{table}
\end{example}


\noindent As motivated by this example, marginal pmf's are obtained from the joint pmf by summing:
\begin{align*}
p_X(x)=\sum_{\{y\in \Omega_Y\}} p_{(X,Y)}(x,y) \ p_Y(y)=\sum_{\{x\in \Omega_X\}} p_{(X,Y)}(x,y).
\end{align*}
The term marginal refers to the fact that the values are written in the margins of the table.

\begin{definition}[Conditional probability mass function of a DRV conditioned on an event]
In a probabilistic model if a certain event A has already occurred, then we define conditional PMF of the possible observations of a random variable X conditioned on the event A as follows:
\begin{align*}
p_{X|A}(x|A)=P(X=x,A|A)=\frac{P(X=x,A}{P(A)}=\left\lbrace\begin{array}{cc}
\frac{p_X(x)}{P(A)} & \ \mbox{if} \ x\in A, \\
0 & \ \mbox{if} \ x\notin A. 
\end{array} \right.
\end{align*}
The conditional PMF is a legitimate discrete distribution and satisfies the following:
\begin{enumerate}
\item $p_{X|A}(x|A)\geq 0$ as $p_X(x)\geq 0$,
\item We have the following normalization property: 
\begin{align*} 
\sum_{\{x \in A\}} p_{X|A}(x|A) & =\frac{1}{P(A)}\sum_{\{x\in A\}}p_{X}(x)=\frac{1}{P(A)}\sum_{\{x\in A\}}P(X=x,A)\\
                                & =\frac{1}{P(A)} P(A)=1.
\end{align*}
Since for any discrete random variable $X$, we have 
\begin{align*}
& a. \ A_i \cap A_j =X^{-1}(x_i)\cap X^{-1}(x_j)=\{\omega|X(\omega)=x_i\} \cap \{\omega|X(\omega)=x_j\} \\
 & \quad \quad \quad \quad \quad  =\phi \ \mbox{for} \ i\neq j, \\
& b. \ \cup_{\{x_k \in \Omega_X\}} A_k =\cup_{\{x_k \in \Omega_X\}} X^{-1}(x_k) =\cup_{\{x_k \in \Omega_X\}} \{\omega|X(\omega)=x_k\} \\
& \quad \quad \quad \quad \quad \quad \quad = \Omega,  \\
& c. \ A_k = X^{-1}(x_k) = \{\omega| X(\omega)=x_k\}\neq \phi \ \mbox{for any} \ x_k \in \Omega_X.
\end{align*}
\end{enumerate}
\begin{center}
\begin{figure}[h]
    \centering
    \includegraphics[width=1.0\textwidth]{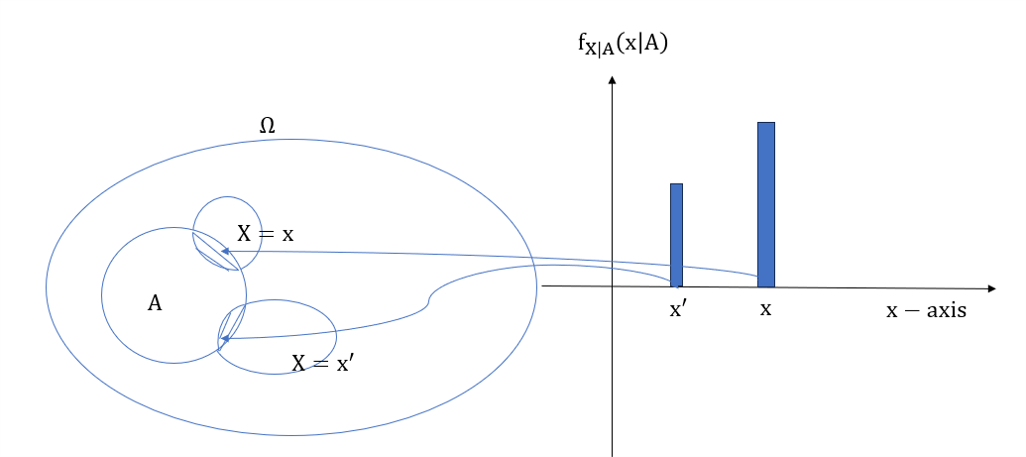}
    \caption{Conditional probability mass function of $X$ conditioned on an event $A$.}
    \label{fig:cpmf-01}
\end{figure}
\end{center}
\end{definition}

\begin{example}
Consider an experiment of rolling a dice, A be the event that roll is an even number. Let $X$ be the roll of a die, then conditional PMF of X conditioned A is given by
\begin{align*}
p_{X|A}(x|A) & =P(X=x|A)=\frac{P(X=x,A}{P(A)}=\left\lbrace\begin{array}{cc}
\frac{p_X(x)}{P(A)} & \ \mbox{if} \ x\in A \\
0  & | \mbox{if} \ x\notin A
\end{array} \right. \\
 & =\left\lbrace\begin{array}{cc}
1/3 & \ \mbox{if} \ x\in A=\{2,4,6\} \\
0  & | \mbox{if} \ x\notin A.
\end{array} \right.
\end{align*}
\end{example}

\subsection{Conditional PMF}
\begin{definition}[Conditional probability mass function]
When an experiment is a compound one, in which the second part of the experiment (described by a discrete random variable $Y$) depends upon the outcome of the first part (described another discrete random variable $X$), conditional probability mass function, conditioned on the event $Y=y$, is coming into picture as follows:
\begin{align*}
p_{X|Y}(x|y) & =P(X=x,Y=y|Y=y)=\frac{P(X=x,Y=y}{P(Y=y)}=\frac{p_{(X,Y)(x,y)}}{p_Y(y)} \\
             & =\frac{p_{X,Y}(x,y)}{\sum_{x}p_{X,Y}(x,y)} \\
             & \ \mbox{for a fixed observation} \ Y=y.
\end{align*}
\end{definition}

\begin{definition}[Conditional PMFs]
Conditioning of joint probability mass function: when X and Y are jointly-distributed random variables, we may want to consider only one of them conditioned on one observation the other, say $X|Y=y$. In that case we need to find the pmf of X conditioned on $Y=y$. This is called a conditional pmf of X given $Y=y$. For each observation/event $X=x$, the $p_{X|Y}$ is computed by a normalized probability of a point on $x \to$ strip, where the normalization is provided by the value of PMF $p_Y$ at $Y=y$ i.e. $p_Y(y)=P(Y=y)=\sum_{x}p_{X,Y}(x,y)$ as a total probability of the event $Y=y$ on $y-$ axis. So, we can have a concrete definition of the conditional PMF as follows:
\begin{align*}
p_{X|Y}(x|y) & =P(X=x|Y=y)=\frac{P(X=x,Y=y}{P(Y=y)}=\frac{p_{(X,Y)(x,y)}}{p_Y(y)} \\
             & =\frac{p_{X,Y}(x,y)}{\sum_{x}p_{X,Y}(x,y)} \\
             & \ \mbox{for a fixed observation} \ Y=y.
\end{align*} 
Similarly,
\begin{align*}
p_{Y|X}(y|x) & =P(Y=y|X=x)=\frac{P(X=x,Y=y}{P(X=x)}=\frac{p_{(X,Y)(x,y)}}{p_X(x)} \\
             & =\frac{p_{X,Y}(x,y)}{\sum_{y}p_{X,Y}(x,y)} \\
             & \ \mbox{for a fixed observation} \ X=x.
\end{align*}
\begin{center}
\begin{figure}[h]
    \centering
    \includegraphics[width=1.0\textwidth]{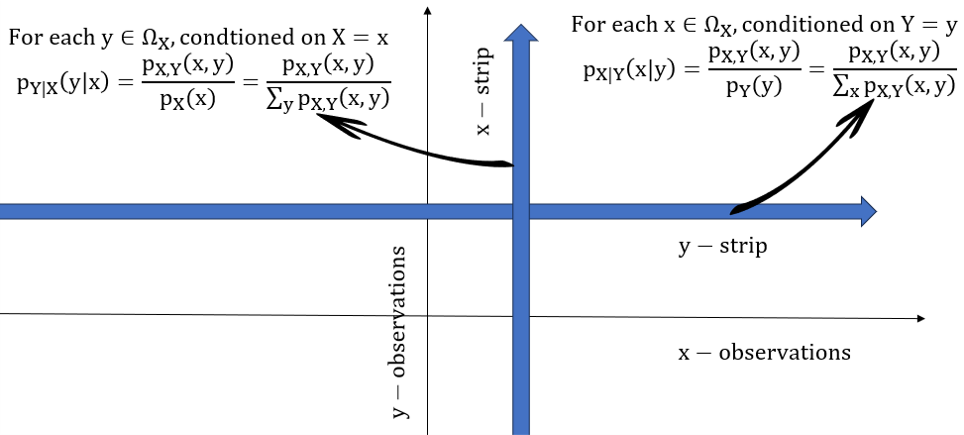}
    \caption{Visualizing conditional PMFs $p_{Y|X}(y|x)$ and $p_{X|Y}(x|y)$ respectively.}
    \label{fig:cpmf-03}
\end{figure}
\end{center}
\end{definition}

So, by using the definition of conditional probability for events we have
\begin{align*}
p_{(X,Y)}((x,y) & =P(X=x,Y=y)=P(Y=y)P(X=x|Y=y) \\
                & =p_Y(y)p_{X|Y}(x|y).
\end{align*}
It provides a computational framework to determine joint PMF of the joint discrete random variable $(X,Y)$.

\begin{example}
Two dice are tossed with all outcomes assumed to be equally likely. The number of dots observed on each die are added together. Then compute the conditional PMF of the sum if it is known that the sum is even.

\noindent Let Y be the sum and define $X=1$ if the sum is even and $X=0$ if the sum is odd. Thus, we wish to determine $p_{Y|X}(y|0)$ and $p_{Y|X}(y|1)$ for all $y\in \Omega_Y=\{2,3,\ldots,12\}$. Then  
\begin{align*}
p_{Y|X}(y|1)=\frac{p_{(Y,X)}(y,1)}{p_X(1)}=\left\lbrace\begin{array}{ccc}
2\cdot \frac{1}{36} & \ \mbox{if} \ y=2,12 \\
2\cdot \frac{3}{36} & \ \mbox{if} \ y=4,10 \\
2\cdot \frac{5}{36} & \ \mbox{if} \ y=6,8
\end{array} \right.
=\left\lbrace\begin{array}{ccc}
\frac{1}{18} & \ \mbox{if} \ y=2,12 \\
\frac{3}{18} & \ \mbox{if} \ y=4,10 \\
\frac{5}{18} & \ \mbox{if} \ y=6,8.
\end{array} \right.
\end{align*}
And
\begin{align*}
p_{Y|X}(y|0)=\frac{p_{(Y,X)}(y,0)}{p_X(0)}=\left\lbrace\begin{array}{ccc}
2\cdot \frac{2}{36} & \ \mbox{if} \ y=3,11 \\
2\cdot \frac{4}{36} & \ \mbox{if} \ y=5,9 \\
2\cdot \frac{6}{36} & \ \mbox{if} \ y=7
\end{array} \right.
=\left\lbrace\begin{array}{ccc}
\frac{2}{18} & \ \mbox{if} \ y=3,11 \\
\frac{4}{18} & \ \mbox{if} \ y=5,9 \\
\frac{6}{18} & \ \mbox{if} \ y=7.
\end{array} \right.
\end{align*}
\begin{center}
\begin{figure}[h]
    \centering
    \includegraphics[width=1.0\textwidth]{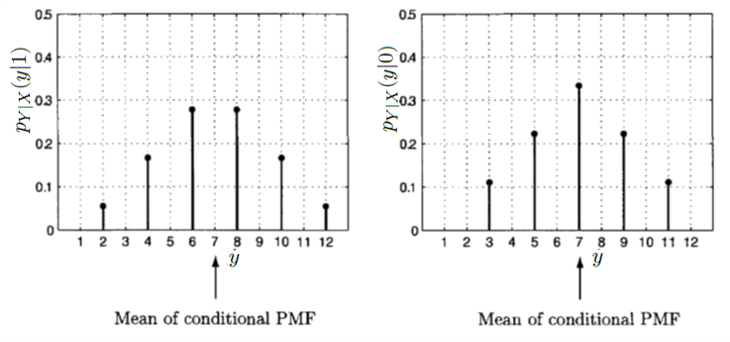}
    \caption{Conditional PMFs $p_{Y|X}(y|1)$ and $p_{Y|X}(y|0)$.}
    \label{fig:cpmf-02}
\end{figure}
\end{center}
\end{example}

\begin{theorem}[Bayes rule for probability mass function]
Definition of conditional probability mass function leads to computation of posterior distribution as follows:
\begin{align*}
p_{X|Y}(x|y) = \frac{p_{X,Y}(x,y)}{p_Y(y)}=\frac{p_X(x)p_{Y|X}(y|x)}{\sum_x p_X(x)p_{Y|X}(y|x)}.
\end{align*}
where $p_X(x) \to $ prior distribution of $X$, $p_{Y|X}(y|x)=f(x) \to $ is the data generating process as a likelihood of $Y=y$ conditioned on $X=x$ and $p_Y(y)=\sum_x p_X(x)p_{Y|X}(y|x) \to $ probability of observing $Y=y$ or evidence.
\begin{center}
\begin{figure}[h]
    \centering
    \includegraphics[width=1.0\textwidth]{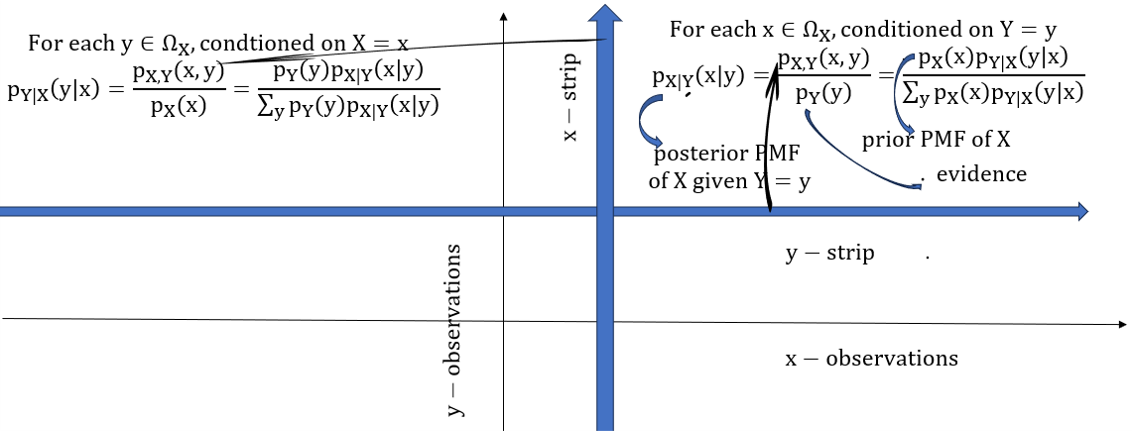}
    \caption{Bayes theorem for posterior PMF of Y conditioned on $X=x$ and of X conditioned on $Y=y$ respectively.}
    \label{fig:btpmf-01}
\end{figure}
\end{center}
\end{theorem}

\begin{example}[Computation of joint PMF]\label{ex:may-01}
\eqref{ex:may-01} Professor May often has her facts wrong, and answers each of her students' questions incorrectly with probability $1/4$, independently of other questions. In each lecture May is asked 0, 1, or 2 questions with equal probability $1/3$. Find the desired joint probability mass function.

\noindent Solution: Let X and Y be the number of questions May is asked and the number of questions she answers wrong in a given lecture, respectively. Now need to compute the joint PMF of X and Y as follows:
\begin{align*}
p_{(X,Y)}((x,y)= p_X(x)p_{Y|X}(y|x).
\end{align*}
We have $(a.) \ Y|X=0 \in \{0\}$ with PMF $p_{Y|X}(y|0)=1, \ (b.) \ Y|X=1 \in \{0,1\}$ with PMF $p_{Y|X}(y|0)=\left(\frac{1}{4}\right)^y\left(\frac{3}{4}\right)^{1-y}$, and $(c.) \ Y|X=2 \in \{0,1,2\}$ with PMF $p_{Y|X}(y|0)=\left(\begin{array}{cc}
2 \\
y
\end{array} \right)\left(\frac{1}{4}\right)^y\left(\frac{3}{4}\right)^{2-y}$.
Then
\begin{align*}
p_{(X,Y)}(0,0)& = p_X(0)p_{Y|X}(0|0)=1/3\cdot 1=1/3=16/48, \\
p_{(X,Y)}(0,1)& = p_X(0)p_{Y|X}(1|0)=1/3\times 0=0, \\
\vdots  & = \vdots \\
p_{(X,Y)}(2,2)& = p_X(2)p_{Y|X}(2|2)=1/3\times 1/16=1/48. 
\end{align*}
\begin{center}
\begin{figure}[h]
    \centering
    \includegraphics[width=1.0\textwidth]{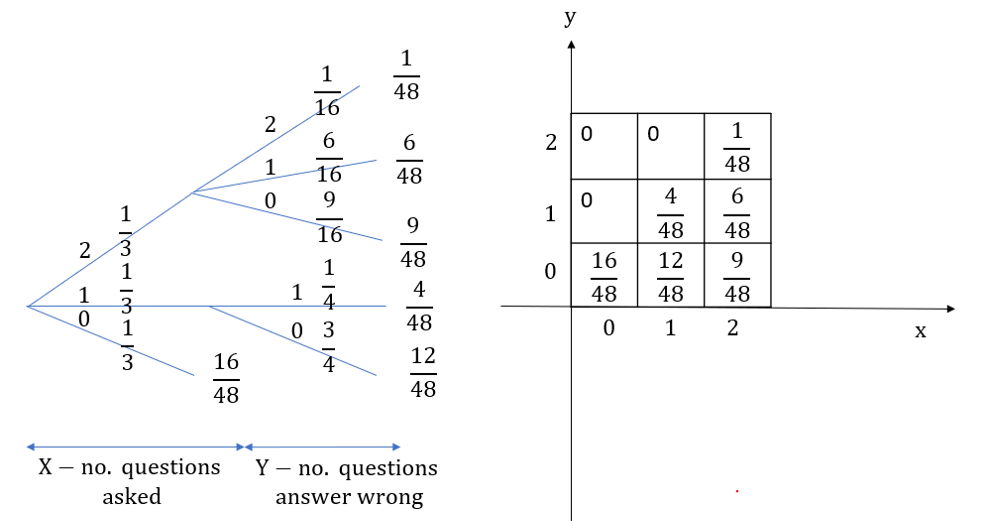}
    \caption{Computation of joint PMF of dependent discrete random variables $X$ and $Y$, where $p_{(X,Y)}(x,y)=p_X(x)p_{Y|X}(y|x)$.}
    \label{fig:jpmf-04}
\end{figure}
\end{center}
\end{example}

\begin{example}[Computation of posterior PMF]
Question: suppose $X$ is the number of $1$s that appear in a binary string of length $L$, each bit in the string is equal to zero or one with probability $1/2$, and the bits are independent. Suppose that the length of the string is also random and uniformly distributed between 1 and 10. We learn that the binary string contains 4 ones. By using the given information, compute the updated PMF for the length of the string L.

\noindent Answer: Given $L=\ell$, we know that X has the binomial distribution when $x\leq \ell$ (the probability is zero otherwise):
\begin{align*}
p_{X|L}(x|\ell)=\left\lbrace\begin{array}{cc}
\left(\begin{array}{cc}\ell \\ x \end{array} \right)\left(\frac{1}{2}\right)^x\left(\frac{1}{2}\right)^{\ell -x} & \ \mbox{if} \  x\leq \ell, \\
0 & \ \mbox{otherwise}.
\end{array} \right.
\end{align*}
And 
\begin{align*}
p_L(\ell)=\left\lbrace\begin{array}{cc}
\frac{1}{10} & \ \mbox{if} \  \ell =1,2,\ldots, 10, \\
0 & \ \mbox{otherwise}.
\end{array} \right.
\end{align*}
Given $L=\ell$, we have the binary string contains 4 ones (i.e. $X=4$), so likelihood of observing $X=4$ is given by 
\begin{align*}
p_{X|L}(4|\ell)= \left(\begin{array}{cc}\ell \\ 4 \end{array} \right)\left(\frac{1}{2}\right)^4\left(\frac{1}{2}\right)^{\ell -4}
               = \left(\begin{array}{cc}\ell \\ 4 \end{array} \right)\left(\frac{1}{2}\right)^{\ell} = f(\ell).
\end{align*}
And by suing the prior probability mass function of L, we get
\begin{align*}
p_X(4)=\sum_{L}p_{(X,L)}(x,\ell)=\sum_{L}p_L(\ell)p_{X|L}(4|\ell)=\sum_{\ell =1}^{10}\frac{1}{10}\left(\begin{array}{cc}\ell \\ 4 \end{array} \right)\left(\frac{1}{2}\right)^{\ell}
\end{align*}
Then, the updated probability mass function as the posterior PMF of L is given by
\begin{align*}
p_{L|X}(\ell|4)=\frac{p_L(\ell)p_{X|L}(4|\ell)}{p_X(4)}=\left\lbrace\begin{array}{cc}
\frac{\frac{1}{10} \times \left(\begin{array}{cc}\ell \\ 4 \end{array} \right)\left(\frac{1}{2}\right)^{\ell}}{p_X(4)} & \ \mbox{if} \ \ell =1,2,\ldots, 10, \\
0 & \ \mbox{otherwise}.
\end{array} \right.
\end{align*}
\end{example}

\begin{definition}[In-dependency or dependency of joint discrete random variables]
Jointly-distributed discrete random variables X and Y are independent if observation of one random variable before the other is going to affect their probability mass functions or their joint PMF is the product of the marginal PMF's:
\begin{align*}
p_{X|Y}(x|y)=p_X(x) \ \mbox{or} \ p_{X,Y}(x,y)=p_X(x)p_Y(y) \ \mbox{or} F_{X,Y}(x,y)=F_X(x)F_y(y),
\end{align*}
If $E(XY)\neq E(X)E(Y)$ or $Cov(X,Y) \neq 0$, then the jointly distributed random variables $X$ and $Y$ are dependent.
\end{definition}

\begin{example}
Roll two dice. Let X be the value on the first die and let Y be the value on the second die. Independent discrete variables means the probability in a cell must be the product of the marginal probabilities of its row and column. In the probability table below this is true: every marginal probability is $1/6$ and every cell contains $1/36$, i.e. the product of the marginals and $X,Y\sim Unif(\{1,2,3,4,5,6\})$. Therefore X and Y are independent.

Question: We rolled two dice and let X be the value on the first die and T be the total on both dice. First construct a probability table for joint probability mass function and then establish whether X and Y are independent or nor.

\noindent Answer: we rolled two dice and let X be the value on the first die and let Y be the value on the second die. In the probability table \eqref{tab:jpmfdice-04}, the probability in a cell is the product of the marginal probabilities of its row and column. Therefore X and Y are independent. 
\begin{table}[htb]

    \centering 
    \caption{Joint probability table for joint probability mass function $p_{(X,Y)}$ of $X\in \Omega_X=\{1,2,\ldots, 6\}$ and $Y\in \Omega_Y=\{1,2,\ldots,6\}$ along with marginal PMFs $p_X(x)$ and $p_Y(y)$.}
    \scalebox{1.5}{ 
    \begin{tabular}{ | c | c | c | c | c | c | c | c | c | c |}
        \hline
        $X \backslash Y$ & 1 & 2 & 3 & 4 & 5 & 6 & $p_X(x)$ \\ 
        \hline
        1 & $\frac{1}{36}$ & $\frac{1}{36}$ & $\frac{1}{36}$ & $\frac{1}{36}$ & $\frac{1}{36}$ & $\frac{1}{36}$ & $\frac{1}{36}$ \\ 
         \hline
        2 & $\frac{1}{36}$ & $\frac{1}{36}$ & $\frac{1}{36}$ & $\frac{1}{36}$ & $\frac{1}{36}$ & $\frac{1}{36}$ & $\frac{1}{36}$ \\
         \hline
        3 & $\frac{1}{36}$ & $\frac{1}{36}$ & $\frac{1}{36}$ & $\frac{1}{36}$ & $\frac{1}{36}$ & $\frac{1}{36}$ & $\frac{1}{36}$ \\
         \hline 
        4 & $\frac{1}{36}$ & $\frac{1}{36}$ & $\frac{1}{36}$ & $\frac{1}{36}$ & $\frac{1}{36}$ & $\frac{1}{36}$ & $\frac{1}{36}$ \\
         \hline
        5 & $\frac{1}{36}$ & $\frac{1}{36}$ & $\frac{1}{36}$ & $\frac{1}{36}$ & $\frac{1}{36}$ & $\frac{1}{36}$ & $\frac{1}{36}$ \\ 
         \hline 
        6 & $\frac{1}{36}$ & $\frac{1}{36}$ & $\frac{1}{36}$ & $\frac{1}{36}$ & $\frac{1}{36}$ & $\frac{1}{36}$ & $\frac{1}{36}$ \\ 
         \hline 
        $p_Y(y)$ & $\frac{1}{36}$ & $\frac{1}{36}$ & $\frac{1}{36}$ & $\frac{1}{36}$ & $\frac{1}{36}$ & $\frac{1}{36}$ & 1 \\ 
         \hline  
    \end{tabular}\label{tab:jpmfdice-04}}

\end{table}
\end{example}

\begin{example}
Question: We rolled two dice and let X be the value on the first die and T be the total on both dice. First construct a probability table for joint probability mass function and then establish whether X and T are independent or nor.

\noindent Answer: we rolled two dice and let X be the value on the first die and T be the total on both dice. In the probability table \eqref{tab:jpmfdice-05}, below most of the cell probabilities are not the product of the marginal probabilities. For example, none of marginal probabilities are 0, so none of the cells with 0 probability can be the product of the marginals.
\begin{table}[htb]

    \centering 
    \caption{Joint probability table for joint probability mass function $p_{(X,T)}$ of $X\in \Omega_X=\{1,2,\ldots, 6\}$ and $T\in \Omega_T=\{2,3,\ldots,12\}$ along with marginal PMFs $p_X(x)$ and $p_T(t)$.}
    \scalebox{1.0}{ 
    \begin{tabular}{ | c | c | c | c | c | c | c | c | c | c | c | c | c |}
        \hline
        $X \backslash T$ & 2 & 3 & 4 & 5 & 6 & 7 & 8 & 9 & 10 & 11 & 12 & $p_X(x)$ \\ 
        \hline
        1 & $\frac{1}{36}$ & $\frac{1}{36}$ & $\frac{1}{36}$ & $\frac{1}{36}$ & $\frac{1}{36}$ & $\frac{1}{36}$ & 0 & 0 & 0 & 0 & 0 & $\frac{1}{6}$ \\ 
         \hline
        2 & 0 & $\frac{1}{36}$ & $\frac{1}{36}$ & $\frac{1}{36}$ & $\frac{1}{36}$ & $\frac{1}{36}$ & $\frac{1}{36}$ & 0 & 0 & 0 & 0 & $\frac{1}{6}$ \\
         \hline
        3 & 0 & 0 & $\frac{1}{36}$ & $\frac{1}{36}$ & $\frac{1}{36}$ & $\frac{1}{36}$ & $\frac{1}{36}$ & $\frac{1}{36}$ & 0 & 0 & 0 & $\frac{1}{6}$ \\
         \hline 
        4 & 0 & 0 & 0 & $\frac{1}{36}$ & $\frac{1}{36}$ & $\frac{1}{36}$ & $\frac{1}{36}$ & $\frac{1}{36}$ & $\frac{1}{36}$ & 0 & 0 & $\frac{1}{6}$ \\
         \hline
        5 & 0 & 0 & 0 & 0 & $\frac{1}{36}$ & $\frac{1}{36}$ & $\frac{1}{36}$ & $\frac{1}{36}$ & $\frac{1}{36}$ & $\frac{1}{36}$ & 0 & $\frac{1}{6}$ \\ 
         \hline 
        6 & 0 & 0 & 0 & 0 & 0 & $\frac{1}{36}$ & $\frac{1}{36}$ & $\frac{1}{36}$ & $\frac{1}{36}$ & $\frac{1}{36}$ & $\frac{1}{36}$ & $\frac{1}{6}$ \\ 
         \hline 
        $p_T(t)$ & $\frac{1}{36}$ & $\frac{2}{36}$ & $\frac{3}{36}$ & $\frac{4}{36}$ & $\frac{5}{36}$ & $\frac{6}{36}$ & $\frac{5}{36}$ & $\frac{4}{36}$ & $\frac{3}{36}$ & $\frac{2}{36}$ & $\frac{1}{36}$ & 1 \\ 
         \hline  
    \end{tabular}\label{tab:jpmfdice-05}}

\end{table}

\end{example}

\begin{example}
Consider a tossing of a biased coin (with probability of success is 1/4) in the following 2 steps: in the first step we flip the coin until we get a heads and let X denote the trial on which the first heads occurs. In the second step we flip the coin X more times and let Y be the number of heads in the second step. Compute joint PMF of X and Y.

\noindent Here, the random variable $X$ is geometrically distributed with parameter $p=1/4$, so
\begin{align*}
p_X(x)=p(1-p)^{x-1}=\left(\frac{3}{4}\right)^{x-1}\frac{1}{4} \ \mbox{for} \ x=1,2,\ldots.
\end{align*}
And, the conditional PMF of $Y$ conditioned on $X=x$ is given by
\begin{align*}
p_{Y|X}(y|x)=\left(\begin{array}{cc}
x \\
y
\end{array} \right)\left(\frac{1}{4}\right)^y\left(\frac{3}{4} \right)^{x-y} \ \mbox{for} \ y=0,1,2,\ldots ,x. 
\end{align*}
And hence, the joint PMF can be computed as follows:
\begin{align*}
p_{X,Y}(x,y)=p_X(x)p_{Y|X}(y|x)=\left(\begin{array}{cc}
x \\
y
\end{array} \right)\left(\frac{1}{4}\right)^y\left(\frac{3}{4} \right)^{x-y}
\left(\frac{3}{4}\right)^{x-1}\frac{1}{4}
\end{align*}
\end{example}

\subsection{Expectation and variance of discrete random variable}
\begin{definition}[Expectation of a discrete random variable]
Consider a discrete random variable, if we observe $n$ random values of X, then the mean of the $n$ values will be approximately equal to expectation of X, $E(X)$, for large $n$. The expectation, $E(X)$ is defined as follows:
\begin{align*}
E(X)=\sum_{\{x \in \Omega_X\}}xp_X(x).
\end{align*}
\end{definition}

\begin{definition}[Expected value rule]
Consider a discrete random variable and let $Y=g(X)$ be a function of X. If we observe $X$ many times ($n$ times) to give results $x_1,x_2, \ldots, x_n$ and hence $g$ provides the observations $g(x_1),g(x_2),\ldots, g(x_n)$. The mean of $g(x_1),g(x_2),\ldots, g(x_n)$ approaches $E(g(X))$, for large $n$. The expectation, $E(g(X))$ is defined as follows:
\begin{align*}
E(g(X))& = E(Y)=\sum_{\{y \in \Omega_Y\}} yp_Y(y) = \sum_{\{y \in \Omega_Y\}} y \sum_{g^{-1}(y)=\{x|g(x)=y\}}p_X(x) \\
       & = \sum_{\{x \in \Omega_X\}}g(x)p_X(x).
\end{align*}
\end{definition}

\noindent The expectation tells us important information about the average value of a random variable, but there is a lot of information that it doesn't provide. If you are investing in the stock market, the expected rate of return of some stock is not the only quantity you will be interested in, you would also like to get some idea of the riskiness of the investment. The typical size of the 
fluctuations of a random variable around its expectation is described by another summary statistic, the variance.
\begin{definition}[Variance of a discrete random variable]
For a discrete random variable $X$, the variance, $Var(X)$, of $X$ is defined by
\begin{align*}
Var(X) = E\left(\left(X-E(X)\right)^2\right)=E(X^2)-(E(X))^2.
\end{align*}
provided that this quantity exists. We can also treat variance, $Var(X)$, of X as an expected value $E(g(X))$ of the function $g(X)=(X-E(X))^2$.

\noindent That is variance is a measure of how much the distribution of X is spread out about its mean and only randomness gives rise to variance. Statisticians often prefer to use the standard deviation (in practice for interpretation and visualization) rather than variance (in principle for theoretical analysis and understanding) as a measure of spread. 
\end{definition}

\begin{definition}[Covariance of two a joint discrete random variable]
Consider a joint discrete random variable $(X,Y)$, where $X$ and $Y$ are associated with the same random experiment, then covariance of $(X,Y)$ is defined as follows:
\begin{align*}
& Cov(X,Y) =E\left((X-E(X))(Y-E(Y))\right) \\
& =E\left(XY-XE(Y)-YE(X)+E(X)E(Y)\right)=E(XY)-E(X)E(Y).
\end{align*}
\end{definition}

\begin{theorem}\label{th:indrvexp-01}
If X and Y are independent discrete random variables whose expectations exist, then
\begin{align*}
(a.) \ E(XY)=E(X)E(Y), \ (b.) \ E(g(X)h(Y))=E(g(X))E(h(Y)).
\end{align*}
\end{theorem}

\begin{example}[A counter-example for converse of \eqref{th:indrvexp-01}]
Consider a uniformly distributed discrete random variable, $X\in \{-1,0,1\}$ i.e. $X\sim Unif([-1,1]\cap \mathbb{Z})$, then
\begin{align*}
& a. \ E(X)=0, \ b. \ E(X^2)=\frac{2}{3}, \ c. \ E(X^3)=0, \\ & d. \ Cov(X,X^2)=E(X^3)-E(X)E(X^2)=0.
\end{align*}
For the observation $X^2=t$,
\begin{align*}
p_{X|X^2}(x|t)=\frac{P(X=x,X^2=t)}{P(X^2=t)}=\frac{P(X=x, X\in \{-\sqrt{t},\sqrt{t}\})}{P(X \in \{-\sqrt{t},\sqrt{t}\})}.
\end{align*}
So,
\begin{align*}
p_{X|X^2}(\sqrt{t}|t)=\frac{p_X(\sqrt{t})}{p_X(-\sqrt{t})+p_X(\sqrt{t})} \neq p_X(\sqrt{t}).
\end{align*}
Also,
\begin{align*}
p_{X|X^2}(-\sqrt{t}|t)=\frac{p_X(-\sqrt{t})}{p_X(-\sqrt{t})+p_X(\sqrt{t})} \neq p_X(-\sqrt{t})
\end{align*}
Therefore $X$ and $X^2$ are dependent discrete random variable despite of having zero covariance.
\end{example}

\begin{theorem}
If $X$ is a non-negative and non-constant discrete random variable, then $X$ and $X^2$ are positively correlated, and hence dependent.
\end{theorem}

\begin{proof}
We have $X\geq 0$ and $X^2\geq 0$, then
\begin{align*}
Cov(X,X^2) & =E\left((X-E(X))(X^2-E(X^2))\right) \\
           & =E\left((X-E(X))(X^2-(E(X))^2+(E(X))^2-E(X^2))\right) \\
           & =E\left((X-E(X))(X^2-(E(X))^2)\right) \\
           & + E\left((X-E(X))((E(X))^2-E(X^2))\right) \\
           & =E\left((X-E(X))(X^2-(E(X))^2)\right) \\
           & =E\left((X-E(X))^2(X+E(X))\right) > 0.
\end{align*}
\end{proof}

\begin{example}
There are 5 candidates for 2 job openings, 3 of the candidates are women and 2 are men. Let X be the number of women hired. Find PMF, expected value, and variance of X.

\noindent Solution: the count of total possible outcomes of hiring 2 person from 5 candidates as a size of the sample is given by 
\begin{align*}
|\Omega|=\binom{5}{2}=10.
\end{align*}
Suppose, X be the number of women hired, then probability mass function of X is computed as:
\begin{align*}
& \{X=0\}\implies |\{X=0\}|=\binom{3}{0}\cdot \binom{2}{2}=1 \implies p_X(0)=P(X=0)=\frac{1}{10}. \\
& \{X=1\}\implies |\{X=1\}|=\binom{3}{1}\cdot \binom{2}{1}=6 \implies p_X(1)=P(X=1)=\frac{6}{10}. \\
& \{X=2\}\implies |\{X=2\}|=\binom{3}{2}\cdot \binom{2}{0}=3 \implies p_X(2)=P(X=2)=\frac{3}{10}.
\end{align*}
Then
\begin{align*}
E(X)=\sum_{x=0}^2xp_X(x)=\frac{12}{10}, \ E(X^2)=\sum_{x=0}^2x^2p_X(x)=\frac{18}{10} \implies Var(X)=\frac{36}{100}.
\end{align*}
\end{example}

\begin{example}[Expectation and variance of various discrete distributions]
Determine the expectation and variance of the following discrete random variables:
\begin{enumerate}
\item[(a.)] $X\sim Ber(p)$,
\begin{solution}
As $X\sim Ber(p)$, so $p_X(x)=p^x(1-p)^{1-x}$ for $x\in \{0,1\}$. Then
\begin{align*}
E(X)=\sum_{x\in \{0,1\}}xp_X(x) = 0\cdot (1-p) + 1\cdot p =p.
\end{align*}
And,
\begin{align*}
E(X^2)=\sum_{x\in \{0,1\}}x^2p_X(x) = 0\cdot (1-p) + 1\cdot p =p.
\end{align*}
Then,
\begin{align*}
Var(X)=E(X^2)-(E(X))^2=p-p^2=p(1-p).
\end{align*}
\end{solution}
\item[(b.)] $Y\sim Bin(n,p)$,
\begin{solution}
As $Y\sim Ber(p)$, $p_Y(y)=\left(\begin{array}{cc} n \\ y \end{array} \right) p^y(1-p)^{n-y}$ for $y\in \{0,1,2,\ldots, n \}$. Then
\begin{align*}
E(Y) & =\sum_{y\in \{0,1,\ldots,n\}}yp_Y(y) \\
     & =\sum_{y\in \{0,1,\ldots,n\}}y\left(\begin{array}{cc} n \\ y \end{array} \right) p^y(1-p)^{n-y} \\
     & =\sum_{y\in \{0,1,\ldots,n\}}n\left(\begin{array}{cc} n-1 \\ y-1 \end{array} \right) p^y(1-p)^{n-y}=np.
\end{align*}
And,
\begin{align*}
E(Y^2) & =\sum_{y\in \{0,1,\ldots,n\}}y^2p_Y(y) \\
       & =\sum_{y\in \{0,1,\ldots,n\}}y^2\left(\begin{array}{cc} n \\ y \end{array} \right) p^y(1-p)^{n-y} \\
     & =\sum_{y\in \{0,1,\ldots,n\}}npy\left(\begin{array}{cc} n-1 \\ y-1 \end{array} \right) p^{y-1}(1-p)^{n-y} \\
     & =np\left((n-1)p+1 \right)=(np)^2+np(1-p).
\end{align*}
Then,
\begin{align*}
Var(Y)& =E(Y^2)-(E(Y))^2=(np)^2+np(1-p)-(np)^2\\
      & =np(1-p).
\end{align*}
\end{solution}
\item[(c.)] $Z\sim Geo(p)$,
\begin{solution}
As $Z\sim Geo(p)$, so $p_Z(z)=p(1-p)^{z-1}$ for $x\in \{1,2,\ldots\}$. Then
\begin{align*}
E(Z)=\sum_{z\in \{1,2,\ldots\}}zp_Z(z) =\frac{1}{p}.
\end{align*}
And,
\begin{align*}
E(Z(Z-1))=\sum_{z\in \{1,2,\ldots\}}z(z-1)p_Z(z) = \frac{2(1-p)}{p^2}.
\end{align*}
Then,
\begin{align*}
Var(Z)& =E(Z^2)-(E(Z))^2=E(Z(Z-1))+E(Z)-(E(Z))^2 \\
      & =\frac{1-p}{p^2}.
\end{align*}
\end{solution}
\item[(d.)] $W\sim Poiss(\lambda)$.
\begin{solution}
As $W$, we have $p_W(w)=\frac{\lambda^w e^{-\lambda}}{w!} \ \mbox{for} \ w=0,1,2,\ldots$ with $E(W)=\lambda=Var(W)$.
\end{solution}
\end{enumerate}
\end{example}

\begin{definition}[Conditional expectation]
In a probabilistic model if a certain event A has already occurred, then we define conditional expectation of the possible observations of a random variable X conditioned on the event A as follows:
\begin{align*}
E(X|A)=\sum_{\{x \in \Omega_X\}}xp_{X|A}(x|A).
\end{align*}
Moreover, if the event A is characterized by another discrete random variable $Y$ in  the same experiment as $Y=y$ and the above summation is absolutely convergent, then the conditional expectation of X conditioned on $Y=y$ is defined as follows:
\begin{align*}
E(X|Y=y)=\sum_{\{x\in \Omega_X\}}xp_{X|Y}(x|y)=g(y).
\end{align*}
That is for each possible observation $Y=y, E(X|Y=y)=g(y)$, and hence conditional expectation defines a random variable (as a function of random variable) as follows:
\begin{align*}
g(Y)=E(X|Y).
\end{align*}
\end{definition}

\begin{theorem}[Law of iterated expectation]
If $\{B_1,B_2,\ldots\}$ be a partition of $\Omega$ such that $P(B_i)> 0 \ \forall \ i$, and $X$ is a discrete random variable in the same experiment, then
\begin{align*}
E(X)=\sum_{i\geq 1} E(X|B_i)P(B_i).
\end{align*} 
Moreover, if $Y$ is a discrete random variable such that $Y^{-1}(y_i)=B_i=\{\omega | Y(\omega)=y_i\}$, then
\begin{align*}
E(X) & =\sum_{\{y\in \Omega_Y\}}E(X|Y=y)P(Y=y)=\sum_{\{y \in \Omega_y\}}E(X|Y=y)p_Y(y)\\
     & =E(E(X|Y)).
\end{align*}
\end{theorem}

\begin{proof}
We have
\begin{align*}
E(X)& =\sum_{\{x \in \Omega_X\}}xp_X(x)=\sum_x xP(X=x)=\sum_x x\sum_iP(X=x, B_i) \\
    & =\sum_x x\sum_iP(B_i)P(X=x|B_i)=\sum_i P(B_i)\sum_x xP(X=x|B_i) \\
    & =\sum_i E(X|B_i)P(B_i).
\end{align*}
Further, We have
\begin{align*}
LHS & =E(E(X|Y)) =\sum_{\{y \in \Omega_Y\}}E(X|Y=y)p_Y(y)\\
    & =\sum_{\{y \in \Omega_Y\}}\sum_{\{x \in \Omega_X\}}xp_{X|Y}(x|y)p_Y(y) =\sum_{\{y \in \Omega_Y\}}\sum_{\{x \in \Omega_X\}}xp_{(X,Y)}(x,y) \\
    & = \sum_{\{x \in \Omega_X\}}x\sum_{\{y \in \Omega_Y\}}p_{(X,Y)}(x,y) = \sum_{\{x \in \Omega_X\}}xp_X(x)=E(X)=RHS.
\end{align*}
\end{proof}

\begin{theorem}
Suppose X and Y are discrete random variables and $a,b \in \mathbb{R}$ are constants. Then
\begin{align*}
E(aX+bY)=aE(X)+bE(Y)
\end{align*}
provided that both $E(X)$ and $E(Y)$ exist.
\end{theorem}

\begin{proof}
\begin{align*}
E(aX+bY)& =\sum_x\sum_y (ax+by)p_{(X,Y)}(x,y)=a\sum_x x \sum_y p_{(X,Y)}(x,y) \\
        & + b\sum_y y \sum_x p_{(X,Y)}(x,y) = a\sum_x xp_X(x) + b\sum_y yp_Y(x) \\
        & = aE(X)+bE(Y).
\end{align*}
\end{proof}

\begin{example}
Let $X$ be the number of rolls of a fair die required to get the first 6. (So X is geometrically distributed with parameter $1/6$.) Now, computation of expectation, $E(X)$, and variance, $Var(X)$, of $X$ as below.

\noindent Let $B_1$ be the event that the first roll of the die gives a 6 (i.e. $X|B_1 \in \{1\}$ or $\Omega_{X|B_1}=\{1\}$, so that $B_1^c$ is the event that it does not (so $X|B_1^c \in \{2,3,\ldots\}$ or $\Omega_{X|B_1^c}=\{2,3,\ldots\}$ or number of rolls till first 6 given the first roll is nor 6 or $1+X$ as $\underline{1} \ \underrightarrow{\ \ X \ \ }$). Then
\begin{align*}
E(X) & =E(X|B_1)P(B_1)+E(X|B_1^c)P(B_1^c)=\frac{1}{6}+\frac{5}{6} E(X|B_1^c)\\
     & =\frac{1}{6}+\frac{5}{6}E(X+1)=1+\frac{5}{6}E(X) \implies E(X)=6.
\end{align*}
Above computation is much simpler as compared a direct calculation using the probability mass function:
\begin{align*}
E(X)=\sum_{x=1}^{\infty}x\left(\frac{1}{6} \right)\left(\frac{5}{6}\right)^{x-1}=\frac{1}{1/6} =6.
\end{align*}
\begin{align*}
E(X^2)& =E(X^2|B_1)+E(X^2|B_1^c)P(B_1^c)=\frac{1}{6}+\frac{5}{6}E(X^2|B_1^c) \\
      & =\frac{1}{6}+\frac{5}{6}E((1+X)^2)=\frac{1}{6}+\frac{5}{6}(1+2E(X)+E(X^2)) \\
      & =11+\frac{5}{6}E(X^2) \implies E(X^2)=66.
\end{align*}
And hence variance $Var(X)=E(X^2)-(E(X))^2=66-36=30$. Alternatively, we compute the variance as follows:
\begin{align*}
E(X^2)=\sum_{x=1}^{\infty}x^2\left(\frac{1}{6} \right)\left(\frac{5}{6}\right)^{x-1}=66 \implies Var(X)=30.
\end{align*} 
\end{example}

\begin{example}
Your spaghetti bowl contains $n$ strands of spaghetti. You repeatedly choose $2$ ends at random, and join them together. Compute the average number of loops in the bowl, once no ends remain.


\noindent Answer: 
\begin{enumerate}
\item[a.] On the first try you have n noodles and $2n$ ends. You pick an end  and the probability that you pick up the other end of the same noodle and make a loop is 1 out of $2n-1$ i.e $\frac{1}{2n-1}$. 
\item[b.] On the next try you have n-1 noodles and $2n-2$ ends. You pick an end  and the probability that you pick up the other end of the same noodle and make a loop is 1 out of $2n-3$ i.e $\frac{1}{2n-3}$. 
\item[c.] Since, formation of each loop decreases the number of ends by 2. So, on the ith try you $(n-i)$ noodles and $2(n-i)+1)$ ends. You pick an end and the probability that you pick up the other end of the same noodle and make a loop is 1 out of $2(n-i)+1$.
\end{enumerate}
Let $X_i$ be the indicator function of the event that we form a loop at the ith step. Then $E(X_i)=1/(2(n-i)+1)$. Suppose T be the total number of loops formed, then $T=E(X_1+X_2+\ldots+X_n)$. So, using linearity of expectation, we have
\begin{align*}
E(T)& =E(X_1+X_2+\ldots+X_n)=\frac{1}{2n-1}+\frac{1}{2n-3}+\ldots+1/3+1 \\
    & \approx log(n) \ \mbox{for large} \ n.
\end{align*}
Alternatively, let the expected number of loops from n noodles be $E(n)$. Then, obviously, $E(1)=1$. Now, for $n>1$, if you pick up two ends, the possibility for those two ends belongs to the same noodle will be $\frac{1}{2n-1}$. Then you have one loop now and there are $n-1$ noodles to keep on going. Otherwise, you get no loop and still have $n-1$ noodles to go (the two you pick before are now connected as one). Pick up two ends. They either belong to the same noodle (S) or they don't (S' i.e. no loop addition from the two end). So,
\begin{align*}
E(n)&=E(n|S)P(S)+E(n|S')P(S')=\left(1+ E(n-1)\right)\cdot \frac{1}{2n-1} \\
     &+ E(n-1)\cdot \left(1 - \frac{1}{2n-1}\right)=E(n-1)+\frac{1}{2n-1}.
\end{align*}
Then, inductively you may expect that 
\begin{align*}
E(n)=1 + \frac{1}{3} + \frac{1}{5}+\ldots+\frac{1}{2n-1} \to \ln(n) \ \mbox{for large n}.
\end{align*}
\end{example}

\subsection{Function of jointly discrete random variables}
The purpose of assessing the student's selection in a campus placement, we might wish to know his/her cgp, numbers of desired courses, number of projects, etc. All these measurements and others are used by an interview panel/recruiter agency for an offer assessment of a candidate.

\begin{theorem}[An another form of total probability]
In general to compute probabilities of events it is advantageous to use conditioning arguments whether or not X and Y are independent.
\begin{align*}
P(Y\in A) = \sum_x P\left( Y\in A|X=x\right)p_X(x)dx=\sum_x \sum_A p_{Y|X}(y|x)p_X(x).
\end{align*}
\end{theorem} 

\begin{corollary}[PDF of a function of jointly discrete random variables $(X,Y)$]
Consider $X$ and $Y$ are two jointly discrete and independent random variables and a function $Z=\frac{Y}{X}$, then PMF of $Z$ is computed as follows:
\begin{align*}
p_Z(z)=\sum_x p_X(x)p_Y(xz).
\end{align*}
\end{corollary}
\begin{proof}
Derivation as follows:
\begin{align*}
p_Z(z)&=P(Z= z)=P(Y=zX)=\sum_x P\left(Y=zx|X=x\right)p_X(x) \\
      &=\sum_x p_{Y|X}(zx|x)p_x(x)=\sum_x p_x(x)p_Y(zx).
\end{align*}
\end{proof}

\begin{corollary}[PDF of a function of jointly discrete random variables $(X,Y)$]
Consider $X$ and $Y$ are two jointly discrete and independent random variables and a function $Z=XY$, then PMF of $Z$ is computed as follows:
\begin{align*}
p_Z(z)=\sum_x p_X(x)p_Y(z/x).
\end{align*}
\end{corollary}
\begin{proof}
Derivation as follows:
\begin{align*}
p_Z(z)&=P(Z=z)=P\left(Y=\frac{z}{X}\right)=\sum_x P\left(Y=\frac{z}{x}|X=x\right)p_X(x)dx \\
      &=\sum_x p_X(x)p_{Y|X}(z/x|x)=\sum_x p_X(x)p_Y(z/x).
\end{align*}
\end{proof}

\noindent Conditional PMFs can be used to simplify probability calculations. For example, consider the determination of the PMF for $Z=X+Y$, where X and Y are jointly discrete and independent random variables as $p_Z(z)=p_X\* p_Y (z)$. To solve this problem using conditional PMFs, we ask ourselves the question: Could I find the PMF of Z if X were known? If so, then we should be able to use conditioning arguments to first find the conditional PMF of Z given X, and then uncondition the result to yield the PMF of Z.

\begin{corollary}[PDF of a function of jointly discrete random variables $(X,Y)$]
Consider $X$ and $Y$ are two jointly discrete and independent random variables and a function $Z=X+Y$, then PMF of $Z$ is computed as follows:
\begin{align*}
p_Z(z)=\sum_x p_X(x)p_Y(z-x).
\end{align*}
\end{corollary}
\begin{proof}
Derivation as follows:
\begin{align*}
p_Z(z)&=P(Z= z)=P(X+Y=z)=\sum_x P\left(Y=z-x|X=x\right)p_X(x) \\
      &=\sum_x p_{Y|X}(z-x|x)p_x(x)=\sum_x p_x(x)p_Y(z-x).
\end{align*}
\end{proof}


\section{Continuous distribution} \label{sec:crv}
In the previous section, we discussed discrete random variables and the relevant methods employed to describe them probabilistically. The principal assumption necessary in order to do so is that the range of sample space $\Omega_X$, which is the set of all possible outcomes, is finite or at most countably infinite. It followed then that a probability mass function (PMF) could be defined as the probability of each sample point and used to calculate the probability of all possible events (which are subsets of the sample space). Most physical measurements, however, do not produce a discrete set of values but rather a continuum of values such as the rainfall measurement data, maximum temperature measured during the day. The number of possible rainfalls or temperatures take values, in an interval (e.g. $[20,60]$, is in continuum (i.e. infinite and uncountable). Of course, we could always choose to round off the measurement to the nearest unit so that the possible outcomes would then become $\{20,21,\ldots,60\}$. Then, many valid PMFs could be
assigned. But this approach compromises the measurement precision and so is to be avoided if possible. What we are ultimately interested in is the probability of any interval, to do so we must extend our previous approaches to be able to handle this new case. And if we later decide that less precision is warranted, such that the rounding of 20.6 to 21 is acceptable, we will still be able to determine the probability of observing 21. To do so we can regard the rounded temperature of 21 as having arisen from all temperatures in the interval $A = [20.5,21.5)$. Then, $P(\mbox{rounded temperature} = 21) = P(A)$, so that we have lost nothing by considering a continuum of outcomes.
\begin{definition}[Continuous random variable]
A measurable function $X:\Omega \to \Omega_X \subset \mathbb{R}$ such that $X^{-1}(B)\in \Sigma_{\Omega} \ \forall \ B\in \mathcal{B}$ such that

\noindent (a.) $X^{-1}(B)=\{\omega | X(\omega) \in B\} \in \Sigma_{\Omega} \ \mbox{for every} \ B \in \mathcal{B}$, (b.) $\Omega_X \subset \mathbb{R}$ is at least continuum in nature.

\noindent (c.) $\Omega_X \neq \{x_k\}_{k\in\mathbb{N}} \implies B$ is an interval or union of interval and we bring at most countable scenario using discretization or $\epsilon -$ covering.

\noindent Then, the measurable function X is known as a continuous random variable and its probability distribution is characterized by a probability density function. An event of a continuous random variable $X$ is described as below:
\begin{align*}
X\in B & = \{X=x \in [a,b]\} \equiv \{X\in [a,b]\} \equiv \{\omega | X(\omega)\in [a,b]\} \\
       & \ \mbox{for some real numbers} \ a,b \in \mathbb{R}.
\end{align*}
Here, $\{X \in B\} \to $ the event that the random variable $X$ takes the given value x from an interval $B=[a,b]$ in continuum manner.
\end{definition}

\begin{center}
\begin{figure}[h]
    \centering
    \includegraphics[width=1.0\textwidth]{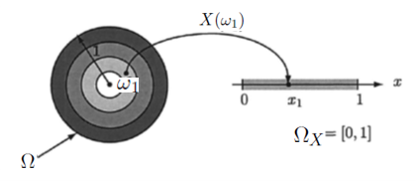}
    \caption{Continuous random variable as a mapping of an at least continuum sample space into an at least continuum set in $\mathbb{R}$.}
    \label{fig:crv-01}
\end{figure}
\end{center}

\begin{definition}[Probability density function]
Distribution pattern of a continuous random variable is characterized by a probability density function (pdf) (see in \eqref{fig:chardrv-01}), $f:\Omega_X \to \mathbb{R_+}$ as a probability per unit length and defined as follows:
\begin{align*}
f_X(x)= \lim_{\delta \to 0}\frac{P(x \leq X \leq x+\delta)}{\delta} \implies P(x \leq X \leq x+\delta) \approx f_X(x)\delta.
\end{align*}
\begin{center}
\begin{figure}[h]
    \centering
    \includegraphics[width=1.0\textwidth]{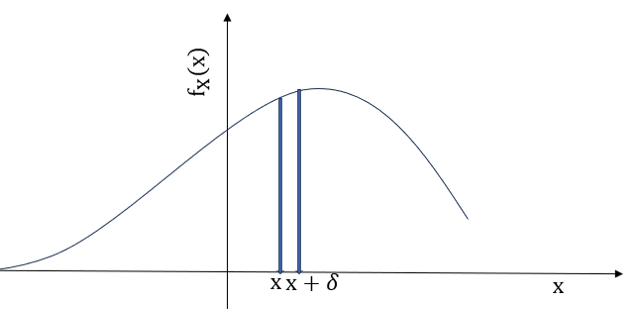}
    \caption{Characterization probability density function as a probability per unit length.}
    \label{fig:chardrv-01}
\end{figure}
\end{center}
The PDF $p_X$ satisfies the following properties:

\noindent (a.) $f_X(x) \geq 0$, (b.) $\int_{\Omega_X}f_X(x)dx =1$, (c.) for any event $X\in B$, we have $P(X\in B) = \int_{x\in B} f_X(x)dx$.
\end{definition}
\noindent Nature of an event in term of continuous random variable: if $X$ is a continuous random variable, then $\{\omega|X(\omega) \leq \} = \{X\leq x\}= (-\infty, x]$ is a basic event for every x. If we have an event $\{a < X \leq b\}$, then it is characterized as $\{a < X \leq b\}=\{X >a\} \cap \{ X\leq b\}$, where $\{X>a\} = \{X\leq a\}^c$. Furthermore, $\{a-\frac{1}{n} < X \leq a + \frac{1}{n}\}$ is an event for every $n$ and consequently implies $\cup_{n=1}^{\infty}\{a-\frac{1}{n} < X \leq a + \frac{1}{n}\}=\{X=a\}$ is also an event. And, probability of the above events are computed as follows 
\begin{align*}
P(X \in (a,b]) & =P(\{a < X \leq b\})= \lim_{n \to \infty} \sum_{i=1}^n f_X(x_i)\delta x_i =\int_a^b f_X(x)dx, \\
P(\{X=a\}) & = \lim_{n \to \infty}f_X(a)\frac{1}{n}=0 \ \mbox{as} \ P(x_{i-1} < X \leq x_i) \approx f_X(a)\delta x_i.
\end{align*}

\subsection{Various continuous distribution}
\begin{example}[Uniform distribution]
If a continuous random variable $X$ takes values in an interval $[a,b]$, and all subintervals of the same length (with respect to uniform partition of the interval) are equally likely. That is, $X$ has a uniform continuous distribution (i.e. a uniform PDF) if 
\begin{align*}
f_X(x) = \left\lbrace \begin{array}{cc}
\frac{1}{b-a} & \ \mbox{if} \ x \in [a,b] \\
0  & \ \mbox{if} \ x \notin [a,b].
\end{array} \right.
\end{align*}
So, $X$ is a uniformly continuous random variable and we write $X\sim Unif([a,b])$. e.g. in a wheel of fortune game, a gambler spins a wheel of fortune, continuously calibrated between 0 and 1, and observes the resulting number. Then, all subintervals of $[0,1]$ of the same length (under a uniform discretization scheme) are equally likely, and hence
\begin{align*}
f_X(x) = \left\lbrace \begin{array}{cc}
1 & \ \mbox{if} \ x \in [0,1] \\
0  & \ \mbox{if} \ x \notin [0,1].
\end{array} \right.
\end{align*}

\begin{center}
\begin{figure}[h]
    \centering
    \includegraphics[width=1.0\textwidth]{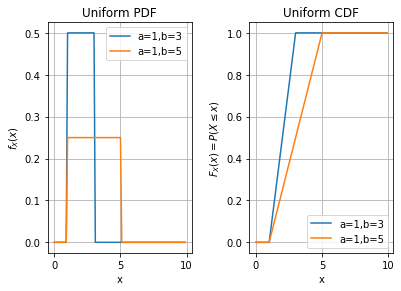}
    \caption{Probability density function as a uniformly distributed random variable on $[a,b]$.}
    \label{fig:unpdf-01}
\end{figure}
\end{center}
\end{example}

\begin{remark}[Characterization of PDF of continuous random variable]
The uniform distribution is the simplest continuous random variable and that obeys uniform law (or equally likely principle) as follows (for a small incremental length $\delta$):

\noindent for any two partition points $x_i,x_j$ of $[a,b]$,
\begin{align*}
& P(x_i < X \leq x_i +\delta)=P(x_j < X \leq x_j +\delta) \implies f_X(x_i)\delta = f_X(x_j)\delta  \\ 
& \implies f_X(x_i)=f_X(x_j). 
\end{align*} 
So $f_X(x)=\frac{1}{b-a} \to$ is a constant and hence defines a uniform continuous distribution.

\noindent For other than uniform continuous random variable, the corresponding PDF is non uniform is characterized as below: 

\noindent For small values of $\delta$ and for any two partition points $x_i,x_j$ of $[a,b]$, we have
\begin{align*}
& P(x_i < X \leq x_i +\delta)\neq P(x_j < X \leq x_j +\delta) \implies f_X(x_i)\delta \neq f_X(x_j)\delta  \\ 
& \implies f_X(x_i)\neq f_X(x_j).
\end{align*}
So $f_X(x)$ is a non-constant and hence defines a non-uniform continuous distribution.
\end{remark}

\begin{example}[A piecewise continuous PDF]
Alvin's driving time to work is between $15$ and $20$ minutes if the day is sunny, and between $20$ and $25$ minutes if the day is rainy, with all times being equally likely in each case. Assume that a day is sunny with probability $2/3$ and rainy with probability $1/3$. Compute the PDF of the driving time.

\noindent Here, all times are equally likely in each case means the PDF of $X$ is constant in each of the two subintervals $[15,20]$ and $[20,25]$. And hence
\begin{align*}
f_X(x)=\left\lbrace\begin{array}{ccc}
c_1 & \ \mbox{if} \ x \in [15,20], \\
c_2 & \ \mbox{if} \ x \in [20,25], \\
0 & \ \mbox{otherwise}.
\end{array} \right.
\end{align*}
\noindent Further, we have
\begin{align*}
\frac{2}{3}=P(sunny\ day)=\int_{15}^{20}f_X(x)dx=5c_1 \implies c_1=\frac{2}{15}, \\
\frac{1}{3}=P(rainy\ day)=\int_{20}^{25}f_X(x)dx=5c_2 \implies c_2=\frac{1}{15}.
\end{align*}
\end{example}

\begin{example}[Exponential distribution]
A continuous random variable $X$ takes values in an interval $[0,\infty)$, and has an exponential continuous distribution (i.e. an exponential PDF) with a parameter $\lambda$ if 
\begin{align*}
f_X(x) = \left\lbrace \begin{array}{cc}
\lambda e^{-\lambda x} & \ \mbox{if} \ x \geq 0, \\
0  & \ \mbox{if} \ x<0.
\end{array} \right.
\end{align*}
\begin{center}
\begin{figure}[h]
    \centering
    \includegraphics[width=1.0\textwidth]{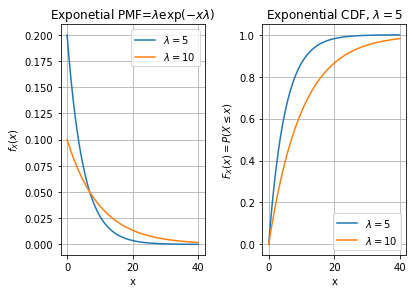}
    \caption{Probability density function as a exponentially distributed random variable on $[0,\infty)$.}
    \label{fig:epdf-01}
\end{figure}
\end{center}
\noindent So, $X$ is an exponentially continuous random variable with a discontinuous at $x = 0$ and we write $X\sim Exp(\lambda)$. One important property of exponential distribution is that it provides a formula for right tail probability in order to model life-time or decay-time of a product or object as follows
\begin{align*}
P(X>a) &=\int_a^{\infty}f_X(x)dx=\int_a^{\infty}\lambda e^{-\lambda x}dx =\lim_{b \to \infty}\int_a^b \lambda e^{-\lambda x}dx \\
       &=\lim_{b \to \infty}\left[-e^{-\lambda x}\right]_a^b=e^{-\lambda a} - \lim_{b \to \infty}e^{-\lambda b}= e^{-\lambda a}.
\end{align*}
\noindent For example, if $X$ is the failure time in days of a light-bulb, then $P(X>100)$ is the probability that the light-bulb will fail after 100 days or it will last for at least 100 days. This can be computed as follows:
\begin{align*}
P(X>100) =e^{-100\lambda}=\left\lbrace\begin{array}{cc}
0.367 & \ \mbox{if} \ \lambda=0.01, \\
0.904 & \ \mbox{if} \ \lambda=0.001.
\end{array} \right.
\end{align*}
\end{example}

\begin{example}[Normal distribution]
A continuous random variable $X$ takes values in an interval $(-\infty,\infty)$, and has a normal distribution (i.e. a bell-shaped Gaussian PDF) with a pair of parameters $(\mu, \sigma^2), \mu \in (-\infty,\infty), \sigma \in [0,\infty)$ if 
\begin{align*}
f_X(x) = \frac{1}{\sqrt{2\pi\sigma^2}}\exp\left(-\frac{(x-\mu)^2}{2\sigma^2} \right) \ \mbox{for} \ x\in (-\infty,\infty).
\end{align*}
\begin{center}
\begin{figure}[h]
    \centering
    \includegraphics[width=1.0\textwidth]{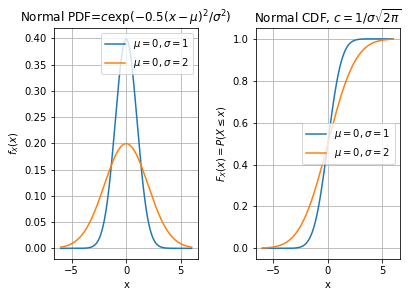}
    \caption{Probability density function as a normally distributed random variable on $(-\infty,\infty)$.}
    \label{fig:npdf-01}
\end{figure}
\end{center}
\noindent So, $X$ is an normally continuous random variable characterized by the pair of parameters $(\mu, \sigma^2), \mu \in (-\infty,\infty), \sigma \in [0,\infty)$ and we write $X\sim Exp(\lambda)$. 

\noindent The parameter $\mu$ indicates the center of the PDF with a depiction of the average value of the observations made by the random variable $X$ as a mean or expectation of $X$.

\noindent The parameter $\sigma^2$ indicates the width of the PDF with a depiction of variability of the observations made by the random variable $X$ as a variance of $X$. 
\end{example}

\begin{example}[Laplace distribution]
A continuous random variable $X$ takes values in an interval $(-\infty,\infty)$, and has a Laplace distribution (similar to normal distribution, but it does not decrease as rapidly from its maximum value) with a pair of parameters $(\mu, \sigma^2), \mu \in (-\infty,\infty), \sigma \in [0,\infty)$ if 
\begin{align*}
f_X(x) = \frac{1}{\sqrt{2\pi\sigma^2}}\exp\left(-\sqrt{\frac{2}{\sigma^2}}|x| \right) \ \mbox{for} \ x\in (-\infty,\infty).
\end{align*}
\begin{center}
\begin{figure}[h]
    \centering
    \includegraphics[width=1.0\textwidth]{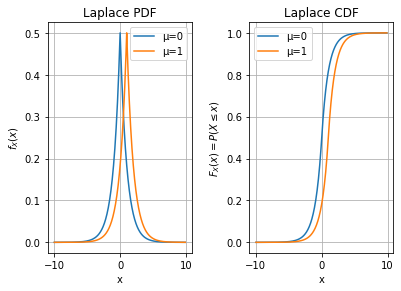}
    \caption{Probability density function as a Lapalcianally distributed random variable on $(-\infty,\infty)$.}
    \label{fig:lpdf-01}
\end{figure}
\end{center}
\noindent So, $X$ is an Laplacianally continuous random variable characterized by the pair of parameters $(\mu, \sigma^2), \mu \in (-\infty,\infty), \sigma \in [0,\infty)$ and we write $X\sim Exp(\lambda)$. 

\noindent The parameter $\mu=0$ indicates the center of the PDF with a depiction of the average value of the observations made by the random variable $X$ as a mean or expectation of $X$.

\noindent The parameter $\sigma^2$ indicates the width of the PDF with a depiction of variability of the observations made by the random variable $X$ as a variance of $X$. Here, the outcomes are larger due to the larger probability in the tails of the PDF (i.e.. the tail region of the PDF is that for which $|x|$ is large). This PDF is used as a model for speech amplitudes.
\end{example}

\begin{example}[Cauchy distribution]
A continuous random variable $X$ takes values in an interval $(-\infty,\infty)$, and has a Cauchy distribution (as a ratio of two standard normal distributions) if 
\begin{align*}
f_X(x) = \frac{1}{\pi(1+x^2)} \ \mbox{for} \ x\in (-\infty,\infty).
\end{align*}
\begin{center}
\begin{figure}[h]
    \centering
    \includegraphics[width=1.0\textwidth]{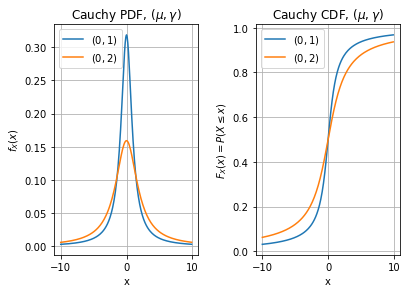}
    \caption{Probability density function as a Cauchy distributed random variable on $(-\infty,\infty)$.}
    \label{fig:capdf-01}
\end{figure}
\end{center}
\noindent So, $X$ is an Cauchy continuous random variable characterized by the pair of parameters $(\mu, \sigma^2), \mu \in (-\infty,\infty), \sigma \in [0,\infty)$ and we write $X\sim Exp(\lambda)$. 

\noindent The parameter $\mu=0$ indicates the center of the PDF with a depiction of the average value of the observations made by the random variable $X$ as a mean or expectation of $X$.

\noindent The parameter $\sigma^2$ indicates the width of the PDF with a depiction of variability of the observations made by the random variable $X$ as a variance of $X$. 
\end{example}

\begin{example}[Gamma distribution]
A continuous random variable $X$ takes values in an interval $[0,\infty)$, and has a gamma distribution (as a generalization of many continuous distributions) if 
\begin{align*}
f_X(x) = \left\lbrace\begin{array}{cc}
\frac{\lambda^{\alpha}}{\Gamma(\alpha)}x^{\alpha -1}\exp\left(-\lambda x \right) & \ \mbox{for} \ x\geq 0,\\
0 & \ \mbox{for} \ x<0.
\end{array} \right.
\end{align*}
\noindent where $\lambda >0, \alpha >0$, and $\Gamma(\alpha)$ is the Gamma function which is defined as below:
\begin{align*}
\Gamma(\alpha)=\int_0^{\infty}x^{\alpha -1}\exp(-x)dx.
\end{align*}
\begin{center}
\begin{figure}[h]
    \centering
    \includegraphics[width=1.0\textwidth]{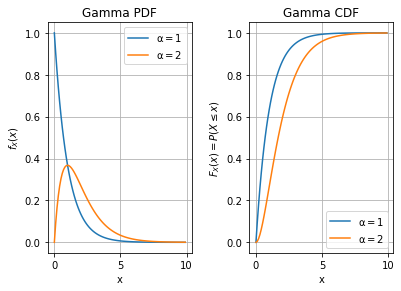}
    \caption{Probability density function as a gamma distributed random variable on $[0,\infty)$.}
    \label{fig:npdf-01}
\end{figure}
\end{center}
\noindent Few properties of the gamma distributed random variable are as follows:
\begin{enumerate}
\item[a.] $\Gamma(\alpha +a)=\alpha \Gamma(\alpha),$
\item[b.] $Gamma(n)=(n-1)!$ for a natural number $n\in \mathbb{N}$,
\item[c.] $\Gamma(1/2)=\sqrt{\pi}$.
\end{enumerate}
\noindent The Gamma PDF reduces to many well known PDFs for appropriate choices of the parameters $\alpha$a and $\lambda$ as follows:
\begin{enumerate}
\item[a.] If $\alpha=1$, the gamma distribution reduces to an exponential distribution as follows:
\begin{align*}
f_X(x) = \left\lbrace\begin{array}{cc}
\frac{\lambda}{\Gamma(1)}\exp\left(-\lambda x \right) & \ \mbox{for} \ x\geq 0,\\
0 & \ \mbox{for} \ x<0.
\end{array} \right.
\end{align*}
\item[b.] If $\alpha =n/2$ and $\lambda=1/2$, the the gamma distribution reduces to a Chi-squared PDF with n degrees of freedom (i.e. the sum of the squares of n independent standard normal random variables all with the same PDF $N(0,1)$) as follows:
\begin{align*}
f_X(x) = \left\lbrace\begin{array}{cc}
\frac{1}{2^n\Gamma(n/2)}x^{n/2 -1}\exp\left(- x/2 \right) & \ \mbox{for} \ x\geq 0,\\
0 & \ \mbox{for} \ x<0.
\end{array} \right.
\end{align*}
\item[c.] If $\alpha =n$, the the gamma distribution reduces to an Erlang PDF with n degrees of freedom (i.e. the sum of n independent exponential random variables all with the same $\lambda$) as follows:
\begin{align*}
f_X(x) = \left\lbrace\begin{array}{cc}
\frac{\lambda^n}{\Gamma(n)}x^{n -1}\exp\left(- \lambda x \right) & \ \mbox{for} \ x\geq 0,\\
0 & \ \mbox{for} \ x<0.
\end{array} \right.
\end{align*}
\end{enumerate}
\end{example}

\begin{example}[Rayleigh distribution]
A continuous random variable $X$ takes values in an interval $[0,\infty)$, and has a Rayleigh distribution (as a square-root of the sum of two independent standard normal distributions with the same PDF $N(0,1)$) if 
\begin{align*}
f_X(x) = \left\lbrace\begin{array}{cc}
\frac{x}{\sigma^2} \exp\left(-\frac{x^2}{2\sigma^2} \right) & \ \mbox{if} \ x\geq 0, \\
0, & \ \mbox{if} \ x<0.
\end{array} \right.
\end{align*}
\begin{center}
\begin{figure}[h]
    \centering
    \includegraphics[width=1.0\textwidth]{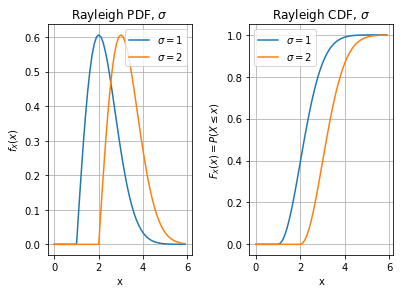}
    \caption{Probability density function of a CRV with Rayleigh distribution on $[0,\infty)$.}
    \label{fig:rpdf-01}
\end{figure}
\end{center}
\noindent So, $X$ is an Cauchy continuous random variable characterized by the pair of parameters $(\mu, \sigma^2), \mu \in (-\infty,\infty), \sigma \in [0,\infty)$ and we write $X\sim Exp(\lambda)$. 

\noindent The parameter $\mu=0$ indicates the center of the PDF with a depiction of the average value of the observations made by the random variable $X$ as a mean or expectation of $X$.

\noindent The parameter $\sigma^2$ indicates the width of the PDF with a depiction of variability of the observations made by the random variable $X$ as a variance of $X$. 
\end{example}

\begin{remark}
Finally, note that many of these PDFs arise as the PDFs of transformed Gaussian random variables. Therefore, realizations of the random variable may be obtained by first generating multiple realizations of independent standard normal or $N(0,1)$ random variables, and then performing the appropriate transformation. An alternative and more general approach to generating realizations of a random variable, once the PDF is known, is via the probability integral transformation.
\end{remark}

\begin{example}[Multi-valued function of a CRV]
In general, if $y=g(x)$ has k solutions namely $y_i=g_i^{-1}(y)$ for $i=1,2,\ldots, k$, then we have
\begin{align*}
p_Y(y)=\sum_{i=1}^k p_X(g^{-1}_i(y))\left|\frac{dg^{-1}_i(y)}{dy}\right|.
\end{align*}
\end{example}

\begin{theorem}
Consider a continuous random variable X with a continuous and monotonically increasing cumulative distribution function $F_X$, then $Y=F_X(X)$ is a uniformly distributed random variable on $[0,1]$.
\end{theorem}

\begin{proof}
For $y\in [0,1]$, we have
\begin{align*}
F_Y(y) & =P(Y\leq y)=P(F_X(X)\leq y)=P(X\leq F_X^{-1}(y)) =F_X(F_X^{-1}(y))\\
       & =y \implies  f_Y(y) =\frac{dF_Y(y)}{dy}=1.
\end{align*}
So, we have
\begin{align*}
f_y(y)=\left\lbrace\begin{array}{cc}
1 & \ \mbox{if} \ y\in [0,1] \\
0 & \ \mbox{if} \ y\notin [0,1].
\end{array} \right.
\end{align*}
\end{proof}

\begin{theorem}
Consider a continuous random variable X such that $X=F_X^{-1}(U)$, where $U\sim \mathcal{U}(0,1)$, then X has the probability density function $f_X(x)=\frac{dF_X(x)}{dx}$.
\end{theorem}

\begin{proof}
\begin{align*}
F_X(x) & =P(X\leq x)=P(F_X^{-1}(U)\leq x)=P(U\leq F_X(x)) =F_U(F_X(x))\\
       & =F_X(x) \implies  f_X(x) =\frac{dF_X(x)}{dx}.   
\end{align*}
\end{proof}

\begin{remark}
In simulating the outcome of a discrete random variable $X$, first an outcome of a $U\sim \mathcal{U}(0,1)$ random variable is generated and then mapped into a value of $X$ under the inverse of the CDF. This result is also valid for a continuous random variable so that $X=F_X^{-1}(U)$ is a random variable with CDF $F_X(x)$.
\end{remark}

\begin{example}[Laplace distribution]
Consider a continuous random variable with the following Laplace distribution:
\begin{align*}
f_X(x) = \frac{1}{\sqrt{2\pi\sigma^2}}\exp\left(-\sqrt{\frac{2}{\sigma^2}}|x| \right) \ \mbox{for} \ x\in (-\infty,\infty).
\end{align*}
Then, its cumulative distribution can be computed as follows:
\begin{align*}
F_X(x)  =\int_{-\infty}^{\infty} \frac{1}{\sqrt{2\pi\sigma^2}}\exp\left(-\sqrt{\frac{2}{\sigma^2}}|x| \right)dx.
\end{align*}


For $x<0$, we have
\begin{align*}
F_X(x)  =\int_{-\infty}^x \frac{1}{\sqrt{2\pi\sigma^2}}\exp\left(\sqrt{\frac{2}{\sigma^2}}x \right)dx = \frac{1}{2}\exp\left(\sqrt{\frac{2}{\sigma^2}}x \right)      
\end{align*}
For $x\geq 0$, we have
\begin{align*}
F_X(x)& =\int_{-\infty}^0 \frac{1}{\sqrt{2\pi\sigma^2}}\exp\left(\sqrt{\frac{2}{\sigma^2}}x \right)dx + \int_0^x \frac{1}{\sqrt{2\pi\sigma^2}}\exp\left(-\sqrt{\frac{2}{\sigma^2}}x \right)dx \\
      & = 1-\frac{1}{2}\exp\left(-\sqrt{\frac{2}{\sigma^2}}x \right).
\end{align*}
Now, we have $y=F_X(x)$ for each $x$, so
\begin{align*}
y=\left\lbrace\begin{array}{cc}
\frac{1}{2}\exp\left(\sqrt{\frac{2}{\sigma^2}}x \right) & \ \mbox{if} \ x<0 \\
1-\frac{1}{2}\exp\left(-\sqrt{\frac{2}{\sigma^2}}x \right) & \ \mbox{if} \ x\geq 0.
\end{array} \right.
\end{align*}
So, $0<y<1/2$ for $x<0$ and $1/2\leq y <1$ for $x\geq 0$, then we have the inverse map
\begin{align*}
x=\left\lbrace\begin{array}{cc}
\sqrt{\sigma^2/2}\ln(2y) & \ \mbox{for} \ y \in (0,1/2) \\
\sqrt{\sigma^2/2}\ln(1-2y) & \ \mbox{for} \ y \in [1/2,1). 
\end{array} \right.
\end{align*}
Finally to generate the outcome of a Laplacian random variable we can use
\begin{align*}
x=\left\lbrace\begin{array}{cc}
\sqrt{\sigma^2/2}\ln(2u) & \ \mbox{for} \ u \in (0,1/2) \\
\sqrt{\sigma^2/2}\ln(1-2u) & \ \mbox{for} \ u \in [1/2,1), 
\end{array} \right.
\end{align*}
where $u$ is a realization of a uniformly distributed random variable $\mathcal{U}(0,1)$.
\begin{center}
\begin{figure}[h]
    \centering
    \includegraphics[width=1.0\textwidth]{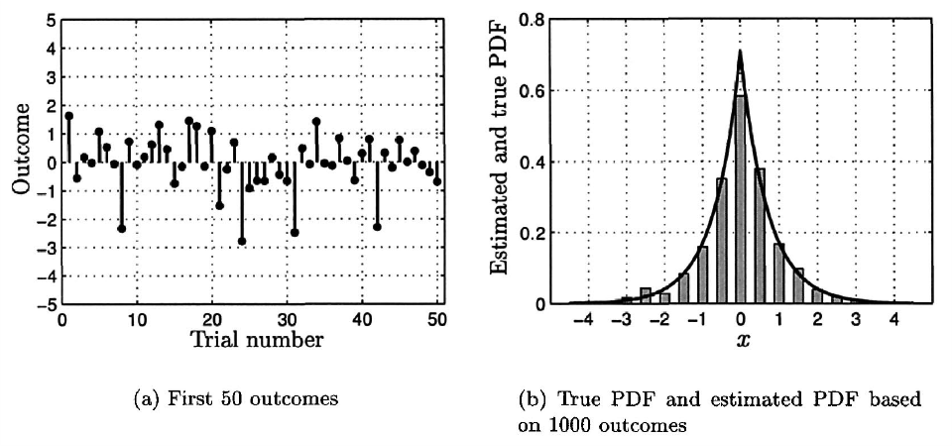}
    \caption{An example of the outcomes of a Laplacian random variable with $\sigma^2=1$, along with the true PDF (the solid curve) and the estimated PDF (the bar plot).}
    \label{fig:lpdf-02}
\end{figure}
\end{center}
\end{example}

\newpage 
 
\subsection{Jointly continuous distribution}
In science and in real life, we are often interested in two (or more) random variables at the same time in a random experiment. In such situations the random variables have a joint distribution that allows us to compute probabilities of events involving both variables and understand the relationship between the variables. This is simplest when the variables are independent. When they are not, we use
covariance and correlation as measures of the nature of the dependence between them.
\begin{example}[Health of a student at IIIT-DWD]
Consider an experiment of investigate health of a student at III-TDWD, then we need to measure a list of possible attributes (jointly) such as:
\begin{enumerate}
\item[a.] health of a student $\to$ varies randomly in continuum i.e. it may be characterized as a random variable $H$.
\item[b.] weight of a student $\to$ varies randomly in continuum i.e. it may be characterized as a random variable $W$.
\item[c.] average blood pressure of a student $\to$ varies randomly in continuum i.e. it may be characterized as a random variable $B$.
\item[d.] body temperature of a student $\to$ varies randomly in continuum i.e. it may be characterized as a random variable $T$.
\item[e.] and a few more.
\end{enumerate}
So, using some statistical method, we can come up with a conclusion that
\begin{align*}
\mbox{helth of a student at IIITDWD}\ = g(H,W,B,T,O).
\end{align*}
\end{example}

\begin{definition}[Jointly continuous random variables]\label{def:jcrv-01}
Suppose X and Y are two continuous random variables and that X takes values $\Omega_X$ and Y takes values $\Omega_Y$. The ordered pair $(X, Y)$ take values in the product $\Omega_X \times \Omega_Y$. The joint probability density function (joint pdf) of X and Y, $f_{(X,Y)}(x,y)$ is defined as the probability per unit as follows:
\begin{align*}
f_{X,Y}(x,y) = \lim_{(\delta x, \delta y) \to (0^+,0^+)} \frac{P(x\leq X\leq x+\delta x,y\leq Y\leq y+\delta y)}{\delta x \delta y}.
\end{align*}
And satisfies the following properties:
\begin{enumerate}
\item $f_{(X,Y)}(x,y) \geq 0 $ for any $(x,y)\in \Omega_X \times \Omega_Y$,
\item $\int_{\{(x,y) \in \Omega_X \times \Omega_Y\}}f_{(X,Y)}(x,y)dxdy =1$,
\item $P((X,Y)\in B)=\int_{\{(x,y)\in B\}}f_{(X,Y)}(x,y)dxdy$.
\end{enumerate}
Computationally visualizing of joint PDF over the red region in figure \eqref{fig:jpdf-01}.
\begin{center}
\begin{figure}[h]
    \centering
    \includegraphics[width=1.0\textwidth]{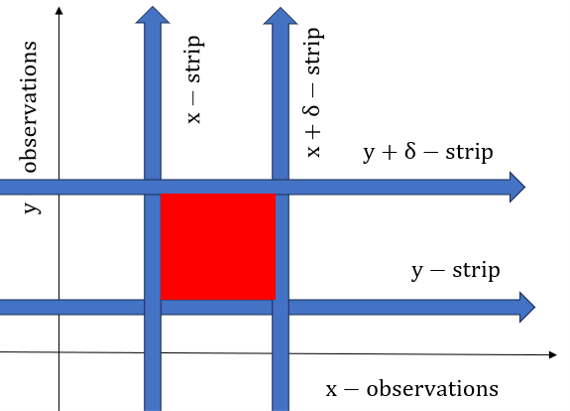}
    \caption{A visualizing of computation of joint PDF over the red region.}
    \label{fig:jpdf-01}
\end{figure}
\end{center}
\end{definition}

\begin{definition}[Conditional PDF for a CRV X conditioned on $Y=y$]\label{def:cpdf-01}
Consider X and Y are jointly-distributed and continuous random variables with a statistically dependent, we may want to consider only one of them say $X$ conditioned on one observation the other, say $X|Y=y$. Then, conditional CDF of X, conditioned on $Y=y$ is defined as
\begin{align*}
f_{X|Y}(x|y) & = \lim_{\delta \to 0}  \frac{P(x\leq X \leq x+\delta|Y=y)}{\delta} =  \frac{f_{X,Y}(x,y)}{f_Y(y)} \\
             & = \frac{f_{X,Y}(x,y)}{\int_{x\in \mathbb{R}}f_{X,Y}(x,y)dx}.
\end{align*}

\noindent Where derivation of conditional probability $P(x\leq X \leq x+\delta|Y=y)$ performed as:
\begin{align*}
P(x\leq X \leq x+\delta|Y=x) & =\lim_{\epsilon \to 0}P\left(x\leq X \leq x+\delta \leq x|Y \in (y-\epsilon,y+\epsilon)\right) \\
                              & = \lim_{\epsilon \to 0}\frac{P(x\leq X \leq x+\delta, y \leq Y \leq y +\epsilon)}{P(y\leq Y \leq y+ \epsilon)} \\
                              & \approx \lim_{\epsilon \to 0} \frac{f_{X,Y}(x,y)\delta \epsilon}{f_Y(y)\epsilon} = \frac{f_{X,Y}(x,y)\delta}{f_Y(y)}.                       
\end{align*}

Similarly, 
\begin{align*}
f_{Y|X}(y|x) & = \lim_{\delta \to 0}  \frac{P(y\leq Y \leq y+\delta|X=x)}{\delta} =  \frac{f_{X,Y}(x,y)}{f_X(x)} \\
             & = \frac{f_{X,Y}(x,y)}{\int_{y\in \mathbb{R}}f_{X,Y}(x,y)dy}.
\end{align*}
More explicitly, computational visualization of conditional PDFs in figure \eqref{fig:cpdf-02} as:
\begin{center}
\begin{figure}[h]
    \centering
    \includegraphics[width=1.0\textwidth]{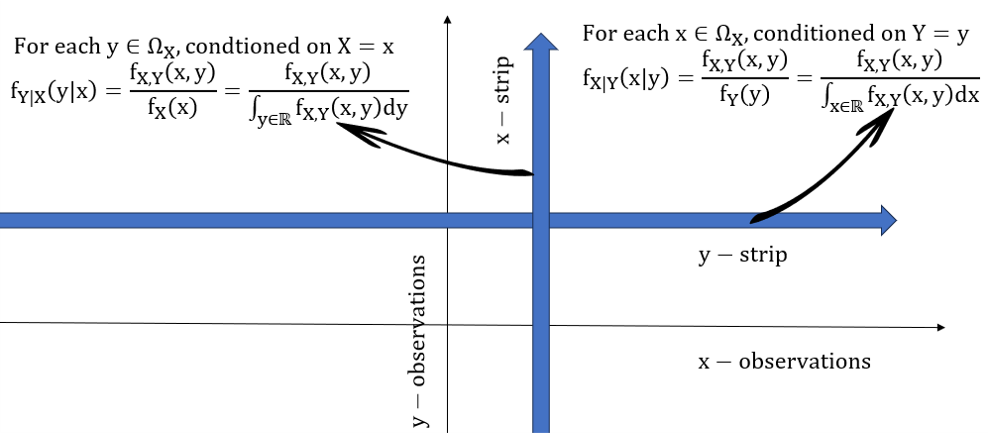}
    \caption{Computationally visualizing conditional PDFs $f_{Y|X}(y|x) \&  f_{X|Y}(x|y)$.}
    \label{fig:cpdf-01}
\end{figure}
\end{center}
A geometrical visualization of conditional PDFs including joint and marginal PDFs in figure \eqref{fig:cpdf-02-01} as:
\begin{center}
\begin{figure}[h]
    \centering
    \includegraphics[width=1.0\textwidth]{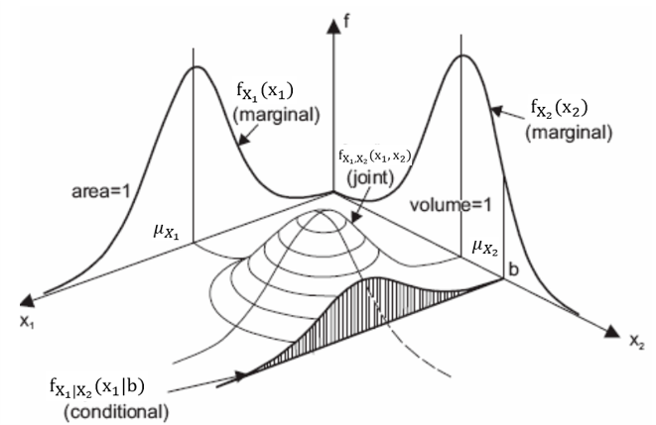}
    \caption{A geometrical visualizing conditional PDF $f_{X_1|X_2}(x_1|b)$ including joint and marginal PDFs.}
    \label{fig:cpdf-02-01}
\end{figure}
\end{center}
\end{definition}

\begin{example}[Conditional PDF from standard bivariate PDF]
Consider a joint continuously distributed bivariate $(X,Y)$ with a standard bivariate Gaussian PDF as follows:
\begin{align*}
f_{X,Y}(x,y)=\frac{1}{2\pi\sqrt{(1-\rho^2)}}\exp\left(-\frac{x^2-2\rho xy +y^2}{2(1-\rho^2)} \right).
\end{align*}
and the marginal PDF for X is computed as 
\begin{align*}
f_X(x)=\int_{y\in \mathbb{R}}f_{X,Y}(x,y)dy=\frac{1}{\sqrt{2\pi}}\exp\left(-\frac{x^2}{2} \right).
\end{align*}
Now, the condition PDF of Y conditioned on $X=x$ can be computed as follows:
\begin{align*}
f_{Y|X}(y|x)=\frac{\frac{1}{2\pi\sqrt{(1-\rho^2)}}\exp\left(-\frac{x^2-2\rho xy +y^2}{2(1-\rho^2)} \right)}{\frac{1}{\sqrt{2\pi}}\exp\left(-\frac{x^2}{2} \right)} 
=\frac{1}{2\pi\sqrt{(1-\rho^2)}}\exp\left(-\frac{Q}{2}\right),
\end{align*}
where
\begin{align*}
Q=-\frac{x^2-2\rho xy +y^2}{2(1-\rho^2)} + \frac{x^2}{2} =
\frac{(y-\rho x)^2}{(1-\rho^2)}.
\end{align*}
So, we can write $Y|X=x \sim N\left(x\rho, 1-\rho^2 \right)$ and visualize in figure \eqref{fig:cpdf-05} as follows:
\begin{center}
\begin{figure}[h]
    \centering
    \includegraphics[width=1.0\textwidth]{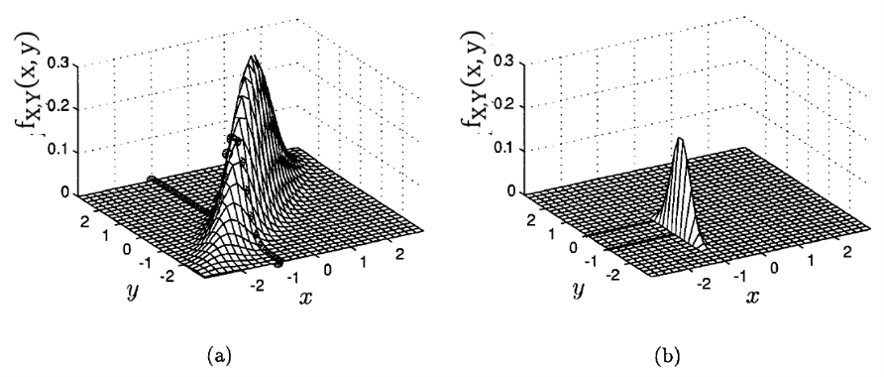}
    \caption{Standard bivariate Gaussian PDF with $\rho=0.9$ and its cross-section at $x=-1$. The
normalized cross-section is the conditional PDF.}
    \label{fig:cpdf-05}
\end{figure}
\end{center}
And, visualizing joint marginal and conditional PDFs together in figure \eqref{fig:jmcpdf-03} as follows:
\begin{center}
\begin{figure}[h]
    \centering
    \includegraphics[width=0.8\textwidth]{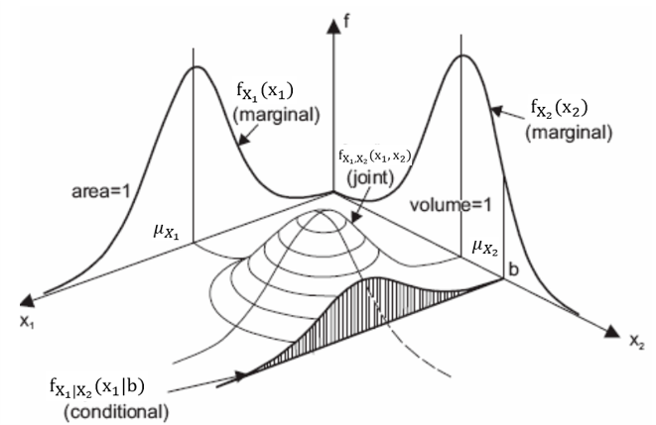}
    \caption{Visualizing joint marginal and conditional PDFs in together.}
    \label{fig:jmcpdf-03}
\end{figure}
\end{center}
\end{example}

\begin{example}
Let S denote the height (in inches) of a randomly chosen father, and T denote the height (in inches) of his son at maturity. Suppose each of S and T has a $N(\mu, \sigma^2)$ distribution with $\mu =69$ and $\sigma =2$. Suppose also that the standardized variables $\frac{S-\mu}{\sigma}$ and $\frac{T-\mu}{\sigma}$ have a standard bivariate normal distribution with correlation $\rho =0.3$. If Sam has a height of $S=74$ inches, the predict the ultimate height T of his young son Tom.

\noindent In standardized units, Sam has height $X=\frac{S-\mu}{\sigma}$, which we are given to equal $2.3$. Tom's ultimate standardized height is $Y=\frac{T-\mu}{\sigma}$. By assumption, before the value of X was known, the pair $(X,Y)$ has a standard bivariate normal distribution with correlation $\rho$. The conditional distribution of Y given that $X=2.5$ is
\begin{align*}
Y|X=2.5 \sim N(2.5\rho, 1-\rho^2).
\end{align*}
In the original units, the conditional distribution of T given $S=74$ is normal with mean $\mu +2.5\rho$ and variance $(1-\rho^2)\sigma^2$, that is, Tom's ultimate height | Sam's height = 74 inches $\sim N(70.5,3.64)$. If I had to make a guess, I would predict that Tom would ultimately reach a height of $70.5$ inches.
\end{example}

\begin{definition}[Conditional CDF for a CRV X conditioned on $Y=y$]
Consider X and Y are jointly-distributed and continuous random variables with a statistically dependent, we may want to consider only one of them say $X$ conditioned on one observation the other, say $X|Y=y$. Then, conditional CDF of X, conditioned on $Y=y$ is defined as
\begin{align*}
F_{X|Y}(x|y)=\lim_{\epsilon \to 0}P\left(X\leq x|Y \in (y-\epsilon,y+\epsilon)\right)=\int_{-\infty}^x f_{X|Y}(x|y)dx.
\end{align*}
\end{definition}

\begin{example}
Consider a joint continuously distributed bivariate $(X,Y)$ with a standard bivariate Gaussian PDF as follows:
\begin{align*}
f_{X,Y}(x,y)=\frac{1}{2\pi\sqrt{(1-\rho^2)}}\exp\left(-\frac{x^2-2\rho xy +y^2}{2(1-\rho^2)} \right).
\end{align*}
and the marginal PDF for X is computed as 
\begin{align*}
f_X(x)=\int_{y\in \mathbb{R}}f_{X,Y}(x,y)dy=\frac{1}{\sqrt{2\pi}}\exp\left(-\frac{x^2}{2} \right).
\end{align*}
Now, the condition PDF of Y conditioned on $X=x$ can be computed as follows:
\begin{align*}
f_{Y|X}(y|x)=\frac{\frac{1}{2\pi\sqrt{(1-\rho^2)}}\exp\left(-\frac{x^2-2\rho xy +y^2}{2(1-\rho^2)} \right)}{\frac{1}{\sqrt{2\pi}}\exp\left(-\frac{x^2}{2} \right)} 
=\frac{1}{2\pi\sqrt{(1-\rho^2)}}\exp\left(-\frac{Q}{2}\right),
\end{align*}
where
\begin{align*}
Q=-\frac{x^2-2\rho xy +y^2}{2(1-\rho^2)} + \frac{x^2}{2} =
\frac{(y-\rho x)^2}{(1-\rho^2)}.
\end{align*}
So, we can write $Y|X=x \sim N\left(x\rho, 1-\rho^2 \right)$. Then
\begin{align*}
F_{Y|X}(y|x)=\int_{-\infty}^y f_{Y|X}(y|x)dy=1-\Phi\left( \frac{y-\rho x}{\sqrt{1-\rho^2}}\right).
\end{align*}
\end{example}

\begin{example}
Consider two continuously distributed random variables $X,Y\sim \mathcal{U}(0,1)$ and distributed jointly with the following joint PDF:
\begin{align*}
f_{X,Y}(x,y)=\left\lbrace\begin{array}{cc}
cxy & \ \mbox{if} 0\leq y \leq x \leq 1\\
0 & \ \mbox{otherwise}.
\end{array} \right.
\end{align*}
Then compute (a.) $c$, (b.) $P\left( Y\leq \frac{X}{2}\right)$ and (c.) $P\left( Y\leq \frac{X}{4}\right)$. And, visualization of domain of joint PDF and surface of the joint PDF in figure \eqref{fig:pc-jpdf} as:
\begin{center}
\begin{figure}[h]
    \centering
    \includegraphics[width=0.8\textwidth]{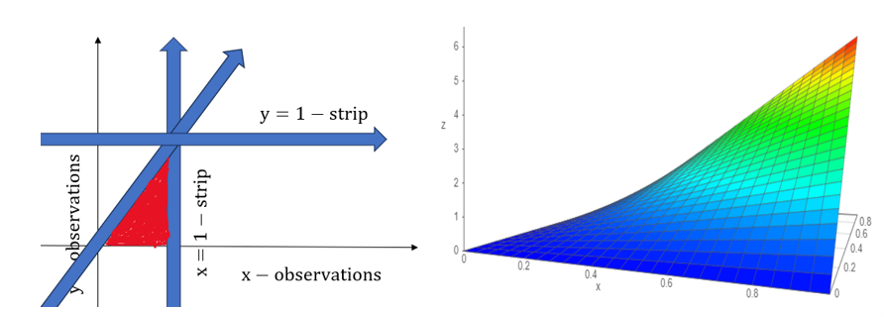}
    \caption{Visualizing domain of joint PDF in left plot and surface of the joint PDF in right plot.}
    \label{fig:pc-jpdf-01}
\end{figure}
\end{center}

\noindent Answer: to compute $c$, we have
\begin{align*}
& 1= \int_0^1\int_0^1 f_{X,Y}dxdy =\int_0^1\int_y^1 cxydxdy=\frac{c}{2}\int_0^1(1-y^2)ydy=\frac{c}{8}  \\
& \implies c=8.
\end{align*}
Now, to compute the probability $P(Y\leq \frac{X}{2})$, we have
\begin{align*}
P(Y\leq \frac{X}{2})=\int_0^1\int_0^{\frac{x}{2}}8xydydx=\int_0^1 x^3dx =\frac{1}{4}.
\end{align*}
Finally, to compute the probability $P(Y\leq \frac{X}{4})$, we have
\begin{align*}
P(Y\leq \frac{X}{4})=\int_0^1\int_0^{\frac{x}{4}}8xydydx=\frac{1}{4}\int_0^1 x^3dx =\frac{1}{16}.
\end{align*}
\end{example}

\begin{theorem}[Computation of joint PDF]
So, by using the definition of conditional probability for events we have
\begin{align*}
f_{(X,Y)}((x,y) =f_Y(y)f_{X|Y}(x|y) =f_X(x)f_{Y|X}(y|x).
\end{align*}
It provides a computational framework to determine joint PMF of the joint discrete random variable $(X,Y)$.
\end{theorem}

\begin{example}[Life-time of a spare bulb]
Question: a professor uses the overhead projector for his class. The time to failure of a new bulb $X$ has the exponential PDF $f_X(x)=\lambda\exp(-\lambda x)u(x)$, where x is in hours. A new spare bulb also has a time to failure $Y$ that is modelled as an exponential PDF.
However, the time to failure of the spare bulb depends upon how long the spare bulb sits unused. Assuming the spare bulb is activated as soon as the original bulb fails, the time to activation is given by X. As a result, the expected time to failure of the spare bulb is decreased as
\begin{align*}
\frac{1}{\lambda_Y} = \frac{1}{\lambda (1+\alpha x)} \ \mbox{for} \ \alpha \in (0,1),
\end{align*}
where $\alpha$ is some factor that indicates the degradation of the unused bulb with storage time. Then compute the PDF of $Y$.

\noindent Answer: the expected time to failure of the spare bulb decreases as the original bulb is used longer (and hence the spare bulb must sit unused longer). Thus, we model the time to failure of the spare bulb as
\begin{align*}
f_Y(y|x)=\lambda_Y \exp(-\lambda_Y y)u(y)=\lambda (1+\alpha x) \exp\left(-\lambda (1+\alpha x) y\right)u(y).
\end{align*}
Then, we compute the desired PDF of $Y$ as follows:
\begin{align*}
f_Y(y) & =\int_x f_X(x)f_{Y|X}(y|x)dx=\int_x \lambda^2 (1+\alpha x) e^{\left(-\lambda (x+(1+\alpha x)y)\right)}dx \\
       & =\lambda^2 \exp(-\lambda y)\int_x (1+\alpha x) \exp(ax)dx, \ a=-\lambda (\alpha y +1) \\
       & =\lambda^2 \exp(-\lambda y) \left(\int_0^{\infty} \exp(ax)dx +\alpha \int_0^{\infty} x\exp(ax)\right) \\
       & =\lambda^2 \exp(-\lambda y)\left(-\frac{1}{a} +\frac{\alpha}{a^2} \right), \ a=-\lambda (\alpha y +1) \implies \\
f_Y(y) & = \lambda^2 \exp(-\lambda y) \left(\frac{1}{(\lambda y+1)} +\frac{\alpha}{(\alpha y +1)^2} \right)u(y).
\end{align*}
\begin{center}
\begin{figure}[h]
    \centering
    \includegraphics[width=1.0\textwidth]{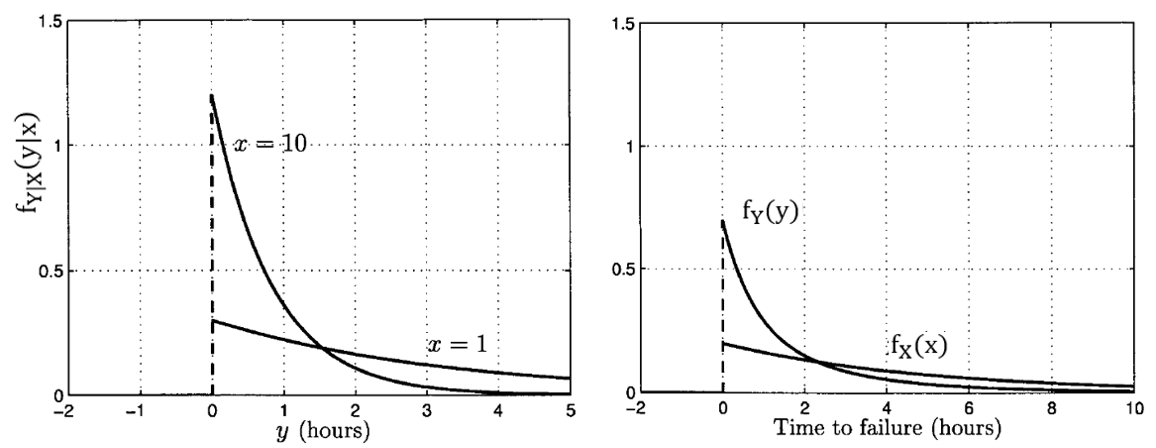}
    \caption{Conditional PDF for lifetime of spare bulb with dependence on time to failure x of original bulb, and PDFs for time to failure of original bulb $X$ \& spare bulb $Y$.}
    \label{fig:cpdf-06}
\end{figure}
\end{center}
\end{example}

\begin{theorem}[An another form of total probability]
In general to compute probabilities of events it is advantageous to use conditioning arguments whether or not X and Y are independent.
\begin{align*}
P(Y\in A) = \int_x P\left( Y\in A|X=x\right)f_X(x)dx=\int_x \int_A f_{Y|X}(y|x)dyf_X(x)dx.
\end{align*}
\end{theorem}

\begin{theorem}[Bayes rule for probability density function]
Consider two jointly continuous random variables $X$ with PDF $f_X(x)$ and $Y$ with PDF $f_Y(y)$ having joint PDF $f_{X,Y}(x,y)$, then restatement of the definition of conditional probability density function leads to computation of a posterior distribution as:
\begin{align*}
f_{(X,Y)}(x,y) & =f_X(x)f_{Y|X}(y|x)=f_Y(y)f_{X|Y}(x|y) \implies \\
f_{X|Y}(x|y) & = \frac{f_X(x)f_{Y|X}(y|x)}{f_Y(y)} = \frac{f_X(x)f_{Y|X}(y|x)}{\sum_{\{x\in \Omega_X\}}f_X(x)f_{Y|X}(y|x)},
\end{align*}
where $f_X(x) \to $ prior distribution of $X$, $f_{Y|X}(y|x)=f(x) \to $ is the data generating process as a likelihood of $Y=y$ conditioned on $X=x$ and $f_Y(y)=\sum_{\{x\in \Omega_X\}}f_X(x)f_{Y|X}(y|x) \to $ probability of observing $Y=y$ or evidence.
\end{theorem}

\begin{definition}[Independent jointly continuous random variables]
Jointly-distributed continuous random variables X and Y are independent if observation of one random variable before the other is not going to affect their probability density functions or their joint PDF is the product of the marginal PDF's:
\begin{align*}
f_{X|Y}(x|y)=f_X(x) \ \mbox{or} \ f_{Y|X}(y|x)=f_Y(y),
\end{align*}
or
\begin{align*}
f_{(X,Y)}(x,y)=f_X(x)f_Y(y) \ \mbox{for any} \ (x,y) \in \Omega_X \times \Omega_Y. 
\end{align*}
\end{definition}

%

\subsection{Derived distribution of a continuous random variable}
\begin{definition}[Derived distribution of a continuous random variable]\label{def:ddcrv-01}
Consider a continuous random variable $X$ with its PDF $f_X(x)$ and a transformation $Y=g(X)$ as a function of $X$ such that $\Omega_Y$ is continuum in nature. Then probability density function of $Y$ can be computed as follows:
\begin{align*}
& F_Y(y)=P(Y\leq y)=P(g(X)\leq y)=P(X \leq g^{-1}(y))=F_X(g^{-1}(y)) \\ 
& \implies f_Y(y)=\frac{dF_Y(y)}{dy}=\frac{dF_X(g^{-1}(y))}{dy}=\frac{dF_X(g^{-1}(y))}{dg^{-1}(y)}\left|\frac{dg^{-1}(y)}{dy} \right|.
\end{align*}
\end{definition} 

\noindent So, computation of derived distribution $p_Y(y)$ of a continuous random variable $X$ (with $\Omega_Y$ as continuum in nature) follows two steps:
\begin{enumerate}
\item Determining all possible observed value of the derived random variable $Y$ i.e. $\Omega_Y$,
\item Computing the derived distribution as follows:
\begin{align*}
f_Y(y)=\frac{dF_Y(y)}{dy}=\frac{dF_X(g^{-1}(y))}{dy}=\frac{dF_X(g^{-1}(y))}{dg^{-1}(y)}\left|\frac{dg^{-1}(y)}{dy} \right|.
\end{align*}
\end{enumerate}

\begin{example}[Derived distribution of a CRV $X\sim Unif(0,1)$]
If $X\sim Unif(0,1)$ and $Y=\sqrt{X}$, then derived PDF $f_Y(y)$ can be computed as:
\begin{enumerate}
\item As $X\sim Unif(0,1)$ and $Y=\sqrt{X}$, so $\Omega_X =[0,1]$ transformed to $\Omega_Y=[0,1]$ and for each $x \in \Omega_X, y=\sqrt{x} \implies x=g^{-1}(y)=y^2$.
\item Now, we have
\begin{align*}
& F_Y(y) =P(Y\leq y)=P(\sqrt{X}\leq y)=P(X\leq y^2) =\int_0^{y^2} dx = y^2, \\ 
& \implies f_y(y) =2y \ \mbox{for} \ y \in [0,1].   .
\end{align*}
Alternatively
\begin{align*}
F_Y(y) & =P(Y\leq y)=P(g(X)\leq y)=P(X\leq g^{-1}(y)) \\
       & =F_X(g^{-1}(y))= P(X\leq y^2)=F_X(y^2), \implies \\
f_y(y) & =\frac{dF_X(g^{-1}(y))}{dg^{-1}(y)}\left|\frac{dy^2}{dy}\right| = 1\times 2y =2y. 
\end{align*}
\end{enumerate} 
\end{example}

\begin{example}
Question: if $X\sim Unif(-1,1)$ and $Y=|X|$, then compute the derived distribution $f_Y(y)$  of $Y=g(X)=|X|$c.

\noindent Answer: derived distribution $f_Y(y)$  of $Y=g(X)=|X|$can be computed as:
\begin{enumerate}
\item As $X\sim Unif(-1,1)$ and $Y=|X|$, so $\Omega_X =[-1,1]$ transformed to $\Omega_Y=[0,1]$ and for each $x \in \Omega_X, y=|x| \implies x=g^{-1}(y)=\{-y,y\}$.
\item Now, we have
\begin{align*}
F_Y(y) & =P(Y\leq y)=P(|X|\leq y)=P(X\leq y \in \{-y,y\}) \\
       & =P(-y\leq X\leq y) =P(X\leq y) - P(X\leq -y), \\
\implies & f_y(y) =1 \ \mbox{for} \ y \in [0,1].   .
\end{align*}
Alternatively
\begin{align*}
F_Y(y) & =P(Y\leq y)=P(|X|\leq y)=P(X\leq y \in \{-y,y\}) \\
       & = P(-y\leq X \leq y)=F_X(y)-F_Y(-y),  \\
\implies  & f_y(y) =\frac{dF_X(g^{-1}(y))}{dg^{-1}(y)}\left|\frac{dg^{-1}(y)}{dy}\right| = \frac{1}{2}\times 1 + \frac{1}{2}\times 1=1. 
\end{align*}
\end{enumerate} 
\end{example}

\begin{example}
Question: If $X\sim Unif(-1,1)$ and $Y=X^2$, then compute the derived distribution $f_Y(y)$ of $Y=g(X)=X^2$.

\noindent Answer: derived distribution $f_Y(y)$ of $Y=g(X)=X^2$ can be computed as:
\begin{enumerate}
\item As $X\sim Unif(-1,1)$ and $Y=X^2$, so $\Omega_X =[-1,1]$ transformed to $\Omega_Y=[0,1]$ and for each $x \in \Omega_X, y=x^2 \implies x=g^{-1}(y)=\{-\sqrt{y},\sqrt{y}\}$.
\item Now, we have
\begin{align*}
F_Y(y) & =P(Y\leq y)=P(X^2\leq y)=P(X\leq y \in \{-\sqrt{y},\sqrt{y}\}) \\ 
       & =P(-\sqrt{y}\leq X \leq \sqrt{y}) =F_X(\sqrt{y}) - F_X(-\sqrt{y}), \\
\implies & f_y(y) =\frac{1}{4\sqrt{y}}+\frac{1}{4\sqrt{y}}.   
\end{align*}
Alternatively
\begin{align*}
F_Y(y) & =P(Y\leq y)=P(X^2\leq y)=P(X\leq y \in \{-\sqrt{y},\sqrt{y}\}) \\ 
       & =P(-\sqrt{y}\leq X \leq \sqrt{y}) =F_X(\sqrt{y}) - F_X(-\sqrt{y})  \\
\implies  & f_y(y) =\frac{dF_X(g^{-1}(y))}{dg^{-1}(y)}\left|\frac{dg^{-1}(y)}{dy}\right| ==\frac{1}{2}\left(\frac{1}{2\sqrt{y}}+\frac{1}{2\sqrt{y}} \right)=\frac{1}{2\sqrt{y}}. 
\end{align*}
\end{enumerate} 
\end{example}

\begin{example}
Question: if $X\sim N(0,1)$ and $Y=g(X)=X^2$, then compute the derived distribution $f_Y(y)$ of $Y=g(X)=X^2$.

\noindent Answer: derived distribution $f_Y(y)$ of $Y=g(X)=X^2$ can be computed as:
\begin{enumerate}
\item As $X\sim N(0,1)$ and $Y=g(X)=X^2$, so $\Omega_X = (-\infty,\infty)$ transformed to $\Omega_Y=[0,\infty)$.
\item Now, for $y\in [0,\infty)$, we have
\begin{align*}
F_Y(y) & =P(Y\leq y)=P(X^2\leq y)=P(X\leq y \in \{-\sqrt{y},\sqrt{y}\})  \\ 
         & =P(X\leq -\sqrt{y}) + P(X\leq \sqrt{y}) \\
\implies & f_y(y) =f_X(x)\left(\frac{1}{2\sqrt{y}}+\frac{1}{2\sqrt{y}}\right) =\frac{1}{\sqrt{2\pi y}}\exp(-y/2).   
\end{align*}
Alternatively
\begin{align*}
F_Y(y) & =P(Y\leq y)=P(g(X)\leq y)=P(X\leq g^{-1}(y)) \\
       & =F_X(g^{-1}(y))= P(X\leq y \in \{-y,y\})=F_X(-y)+F_Y(y),  \\
\implies  & f_y(y) =\frac{dF_X(g^{-1}(y))}{dg^{-1}(y)}\left|\frac{dy^2}{dy}\right| =f_X(x)\left(\frac{1}{2\sqrt{y}}+\frac{1}{2\sqrt{y}} \right) \\
& =\frac{1}{\sqrt{2\pi y}}\exp(-y/2). 
\end{align*}
And, $Y$ does not observe any value in $(-\infty,0)$, so, we have
\begin{align*}
f_Y(y)=\left\lbrace\begin{array}{cc}
\frac{1}{\sqrt{2\pi y}}\exp(-y/2) & \ \mbox{if} \ y\geq 0 \\
0, & \ \mbox{if} \ y<0.
\end{array} \right.
\end{align*}
\end{enumerate}
 \begin{center}
\begin{figure}[h]
    \centering
    \includegraphics[width=1.0\textwidth]{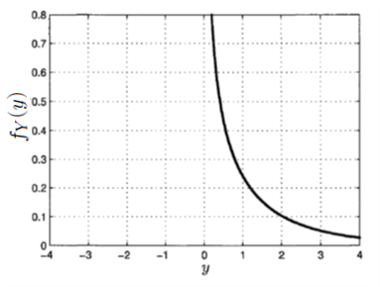}
    \caption{Probability density function for $Y=X^2, X\sim N(0,1)$.}
    \label{fig:jpmf-01}
\end{figure}
\end{center}
\end{example}

\begin{example}
If $X\sim N(0,1)$ and $Y=g(X)=\exp(X)$, then compute the derived distribution $f_Y(y)$ of $Y=g(X)=\exp(X)$.

\noindent Answer: derived distribution $f_Y(y)$ of$Y=g(X)=\exp(X)$ can be computed as:
\begin{enumerate}
\item As $X\sim N(0,1)$ and $Y=g(X)=\exp(X)$, so $\Omega_X = (-\infty,\infty)$ transformed to $\Omega_Y=(0,\infty)$.
\item Now, for $y\in (0,\infty)$, we have
\begin{align*}
F_Y(y) & =P(Y\leq y)=P(\exp(X)\leq y)=P(X\leq \ln(y)) =F_X(\ln(y)) \\
\implies & f_y(y) =\frac{dF_X(\ln(y))}{d\ln(y)}\left|\frac{d\ln(y)}{dy}\right| =f_X(\ln(y))\frac{1}{y} \\
         & =\frac{1}{\sqrt{2\pi} y}\exp(-(\ln(y))^2/2).   
\end{align*}
Alternatively
\begin{align*}
F_Y(y) & =P(Y\leq y)=P(g(X)\leq y)=P(X\leq g^{-1}(y)) \\
       & =F_X(g^{-1}(y))= P(X\leq y \in \{-y,y\})=F_X(-y)+F_Y(y),  \\
\implies  & f_y(y) =\frac{dF_X(\ln(y))}{d\ln(y)}\left|\frac{d\ln(y)}{dy}\right| =f_X(\ln(y))\frac{1}{y} \\
& =\frac{1}{\sqrt{2\pi} y}\exp\left(-\frac{(\ln(y))^2}{2}\right). 
\end{align*}
And, $Y$ does not observe any value in $(-\infty,0)$, so, we have
\begin{align*}
f_Y(y)=\left\lbrace\begin{array}{cc}
\frac{1}{\sqrt{2\pi} y}\exp\left(-\frac{(\ln(y))^2}{2}\right) & \ \mbox{if} \ y> 0 \\
0, & \ \mbox{if} \ y\leq 0.
\end{array} \right.
\end{align*}
\end{enumerate}
 \begin{center}
\begin{figure}[h]
    \centering
    \includegraphics[width=1.0\textwidth]{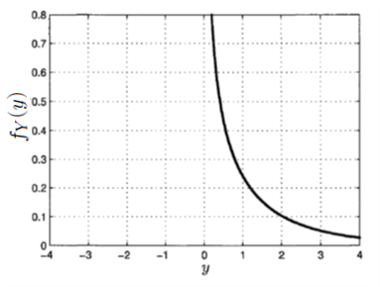}
    \caption{Probability density function for $Y=X^2, X\sim N(0,1)$.}
    \label{fig:jpmf-01}
\end{figure}
\end{center}
\noindent This PDF is called the log-normal PDF. It is frequently used as a model for a quantity that is measured in decibels (dB) and which has a normal PDF in dB quantities.
\end{example}

\begin{corollary}[PDF of a function of jointly continuous random variables $(X,Y)$]
Consider $X$ and $Y$ are two jointly continuous and independent random variables and a function $Z=\frac{Y}{X}$, then PDF is $Z$ is computed as follows:
\begin{align*}
f_Z(z)=\int_x f_X(x)f_Y(xz)|x|dx.
\end{align*}
\end{corollary}
\begin{proof}
Derivation as follows:
\begin{align*}
F_Z(z)&=P(Z\leq z)=P(Y\leq zX)=\int_x P\left(Y\leq zx|X=x\right)f_X(x)dx \\
\implies & f_Z(z)=\frac{\partial F_Z(z)}{\partial z}=\int_x \frac{\partial F_{Y|X}(zx|X=x)}{\partial z}f_X(x)dx \\
         & =\int_x f_{Y|X}(zx|X=x)|x|f_X(x)dx=\int_x f_X(x)f_Y(zx)|x|dx.
\end{align*}
\end{proof}

\begin{corollary}[PDF of a function of jointly continuous random variables $(X,Y)$]
Consider $X$ and $Y$ are two jointly continuous and independent random variables and a function $Z=XY$, then PDF is $Z$ is computed as follows:
\begin{align*}
f_Z(z)=\int_x f_X(x)f_Y(z/x)\frac{1}{|x|}dx.
\end{align*}
\end{corollary}
\begin{proof}
Derivation as follows:
\begin{align*}
F_Z(z)&=P(Z\leq z)=P\left(Y\leq \frac{z}{X}\right)=\int_x P\left(Y\leq \frac{z}{x}|X=x\right)f_X(x)dx \\
\implies & f_Z(z)=\frac{\partial F_Z(z)}{\partial z}=\int_x \frac{\partial F_{Y|X}(z/x|X=x)}{\partial z}f_X(x)dx \\
         & =\int_x f_{Y|X}(z/x|X=x)\frac{1}{|x|}f_X(x)dx=\int_x f_X(x)f_Y(z/x)\frac{1}{|x|}dx.
\end{align*}
\end{proof}

\begin{theorem}[PDF of a sum of jointly continuous random variables $(X,Y)$]
Consider $X$ and $Y$ are two jointly continuous and independent random variables and a function $Z=X+Y$, then PDF of $Z$ is computed as follows:
\begin{align*}
f_Z(z)=\int_x f_X(x)f_Y(z-x)dx.
\end{align*}
\end{theorem}
\begin{proof}
Derivation as follows:
\begin{align*}
F_Z(z)&=P(Z\leq z)=P\left(X+Y\leq z\right) = \int_x P\left(Y\leq z-x|X=x\right)f_X(x)dx \implies \\ 
f_Z(z) &=\frac{\partial F_Z(z)}{\partial z}=\int_x f_X(x) \frac{\partial F_Y(z-x)}{\partial z}dx =\int_x f_X(x)f_{Y|X}(z-x|X=x)dx \\
          &=\int_x f_X(x)f_Y(z-x)dx.
\end{align*}
\end{proof}

\begin{example}
Suppose X and Y are uniformly distributed on $[0,1]$ and independent random variables. Find the PDF of $Z=X+Y$. 

\noindent Solution: We have $X,Y \sim \mathcal{U}(0,1)$, then
\begin{align*}
f_X(x)=\left\lbrace\begin{array}{cc}
1 & x \in [0,1], \\
0 & x \notin [0,1].
\end{array} \right.
\end{align*}
and
\begin{align*}
f_Y(y)=\left\lbrace\begin{array}{cc}
1 & y \in [0,1], \\
0 & y \notin [0,1].
\end{array} \right.
\end{align*}
Now, $Z=X+Y$, so $f_Z(z)=\int_{-\infty}^{\infty}f_X(x)f_Y(z-x)dx$ can be computed case-wise as follows:
\begin{enumerate}
\item[] Case-I: If $0\leq z \leq 1$, then $0\leq x\leq z$, so we have
\begin{align*}
f_Z(z)=\int_0^z f_X(x)f_Y(z-x)dx =\int_0^z 1\cdot 1\cdot dx =z
\end{align*}
\item[] Case-II: If $1< z \leq 2 \equiv 0<z-1 \leq 1$ then $z-1 \leq x\leq 1$, so, we have
\begin{align*}
f_Z(z)=\int_{z-1}^1 f_X(x)f_Y(z-x)dx =\int_{z-1}^1 1\cdot 1\cdot dx =2-z
\end{align*}
\end{enumerate}
 \begin{center}
\begin{figure}[h]
    \centering
    \includegraphics[width=1.0\textwidth]{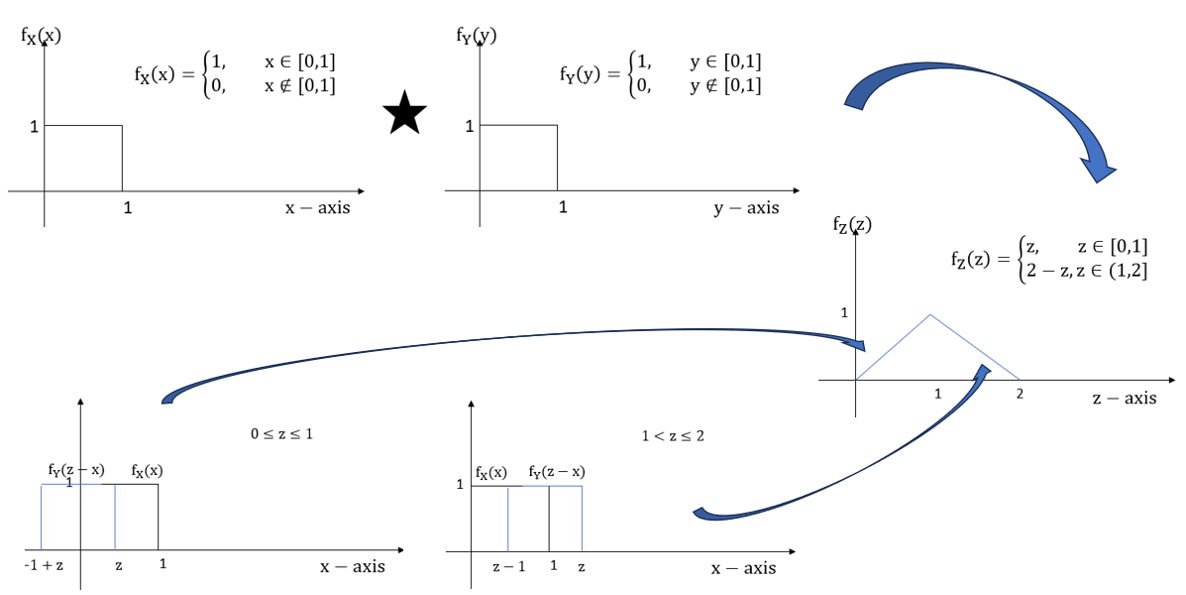}
    \caption{Convolution of two uniform and independent distributions on $[0,1]$.}
    \label{fig:cuconv-01}
\end{figure}
\end{center}
 \begin{center}
\begin{figure}[h]
    \centering
    \includegraphics[width=1.0\textwidth]{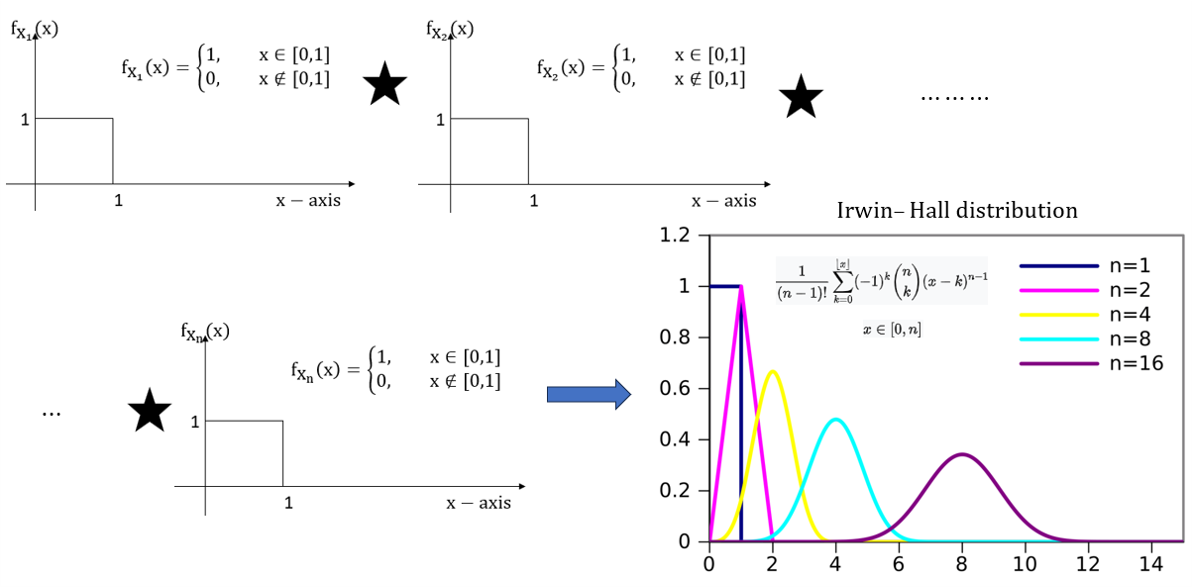}
    \caption{Convolution of $n$ uniform and independent distributions on $[0,1]$.}
    \label{fig:cunconv-02}
\end{figure}
\end{center}
\end{example}

\begin{example}
Backdrop: consider two jointly continuous and independent random variables X and Y uniformly distributed in $[0,1]$. 

\noindent Question: Let $Z=XY$, then compute CDF and PDF of $Z$.

\noindent Answer: we have $X\sim \mathcal{U}(0,1)$ and $Y\sim \mathcal{U}(0,1)$ and independent to each other. So, we get
\begin{align*}
& f_X(x)  =\left\lbrace\begin{array}{cc}
1 & \ \mbox{if} \ x\in [0,1] \\
0 & \ \mbox{if} \ x \notin [0,1].
\end{array} \right.
\ \mbox{and} \
f_Y(y)=\left\lbrace\begin{array}{cc}
1 & \ \mbox{if} \ y\in [0,1] \\
0 & \ \mbox{if} \ y \notin [0,1].
\end{array} \right.
\implies \\
& f_{X,Y}(x,y) =f_X(x)f_Y(y)=\left\lbrace\begin{array}{cc}
1 & \ \mbox{if} \ (x,y)\in [0,1]\times [0,1] \\
0 & \ \mbox{if} \ (x,y) \notin [0,1]\times [0,1].
\end{array} \right.
\end{align*}
Now, cumulative distribution function of $Z$ is given by
\begin{align*}
F_Z(z)&=P(Z\leq z)=P(XY\leq z)=P\left(X\leq \frac{z}{Y}\right)=\int_0^1\int_0^{\min\left(1, \frac{z}{y}\right)}dxdy \\
      &=\int_0^1\int_0^{\min\left(1, \frac{z}{y}\right)}dxdy=\int_0^1 \min\left(1, \frac{z}{y}\right)dy \\
      &=\int_0^zdy +\int_z^1\frac{z}{y}dy=z-z\ln(z) \implies f_Z(z)=-\ln(z)=\ln\left(\frac{1}{z} \right).
\end{align*}
\noindent Alternatively, using conditioning and law of total probability, we get
\begin{align*}
F_Z(z)&=P(Z\leq z)=P(XY\leq z)=P\left(X\leq \frac{z}{Y}\right) \\
      &=\int_0^1 P\left(X\leq \frac{z}{y}|Y=y\right)f_Y(y)dy =\int_0^1 \min\left(1, \frac{z}{y}\right)dy \\
      &=\int_0^zdy +\int_z^1\frac{z}{y}dy=z-z\ln(z) \implies f_Z(z)=-\ln(z)=\ln\left(\frac{1}{z} \right).
\end{align*}
\begin{center}
\begin{figure}[h]
    \centering
    \includegraphics[width=1.0\textwidth]{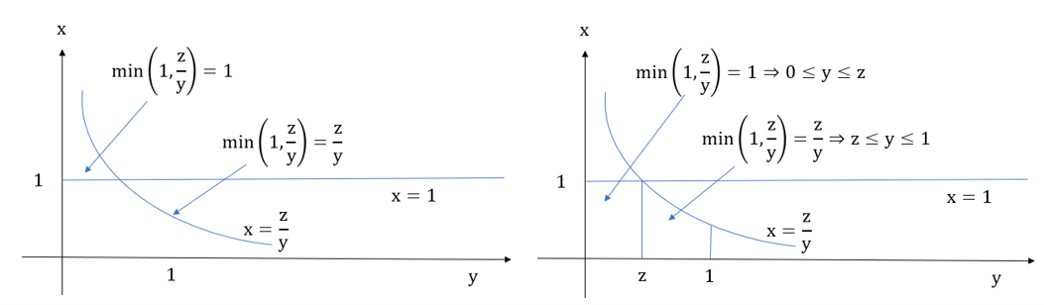}
    \caption{In left plot there are depiction of limits for inner integrating element $dx$ and in the right plot there are depiction of limits for outer integrating element $dy$.}
    \label{fig:ddjcrv-01}
\end{figure}
\end{center}
\end{example}

\subsection{Expectation and variance of continuous random variables}
\begin{definition}[Expectation of a continuous random variable]
Consider a continuous random variable with a continuous distribution or probability density function $f_X(x)$. The expectation, $E(X)$ is defined as:
\begin{align*}
E(X)=\int_{\{x \in \Omega_X\}}xf_X(x)dx.
\end{align*}
\end{definition}

\begin{definition}[Expected value rule]
Consider a continuous random variable X with PDF $f_X(x)$ and let $Y=g(X)$ be a function of X. The expectation, $E(g(X))$ is defined as:
\begin{align*}
E(g(X))& =\int_{\{y \in \Omega_Y\}} yf_Y(y)dy = \int_{\{y \in \Omega_Y\}} y \left( \int_{g^{-1}(y)=\{x|g(x)=y\}}f_X(x)dx \right)dy \\
       & = \int_{\{x \in \Omega_X\}}g(x)f_X(x)dx.
\end{align*}
\end{definition}

\noindent The expectation tells us important information about the average value of a random variable, but there is a lot of information that it doesn't provide. If you are investing in the stock market, the expected rate of return of some stock is not the only quantity you will be interested in, you would also like to get some idea of the riskiness of the investment. The typical size of the 
fluctuations of a random variable around its expectation is described by another summary statistic, the variance.
\begin{definition}[Variance of a continuous random variable]
For a continuous random variable $X$ with a PDF $f_X(x)$, the variance, $Var(X)$ is defined as:
\begin{align*}
Var(X) = E\left(\left(X-E(X)\right)^2\right)=E(X^2)-(E(X))^2.
\end{align*}
provided that this quantity exists. We can also treat variance, $Var(X)$, of X as an expected value $E(g(X))$ of the function $g(X)=(X-E(X))^2$.

\noindent That is variance is a measure of how much the distribution of X is spread out about its mean and only randomness gives rise to variance. Statisticians often prefer to use the standard deviation (in practice for interpretation and visualization) rather than variance (in principle for theoretical analysis and understanding) as a measure of spread. 
\end{definition}

\begin{definition}[Covariance of two jointly continuous random variables]
Consider a jointly continuous random variable $(X,Y)$ with joint PDF $f_{X,Y}(x,y)$, where $X$ and $Y$ are associated with the same random experiment, then covariance of $(X,Y)$ is defined as:
\begin{align*}
Cov(X,Y) & =E\left((X-E(X))(Y-E(Y))\right)= E\left(XY-XE(Y)-YE(X)\right.\\
         & + \left. E(X)E(Y)\right)=E(XY)-E(X)E(Y).
\end{align*}
Furthermore, the variance-covariance matrix is given by 
\begin{align*}
C_{X,Y}=\left(\begin{array}{cc}
Var(X) & Cov(X,Y) \\
Cov(X,Y) & Var(Y)
\end{array} \right).
\end{align*}
\end{definition}

\begin{theorem}\label{th:indrvexp-01}
If X and Y are independent and jointly continuous random variables whose expectations exist, then
\begin{align*}
a. \ E(XY)=E(X)E(Y) \ \mbox{and} \ b. \ Cov(X,Y)=0.
\end{align*}
\end{theorem}

\begin{definition}[Conditional expectation]
In a probabilistic model if a certain event A has already occurred, then we define conditional expectation of the possible observations of a continuous random variable X conditioned on the event A as follows:
\begin{align*}
E(X|A)=\int_{\{x \in \Omega_X\}}xf_{X|A}(x|A)dx.
\end{align*}
Moreover, if the event A is characterized by another continuous random variable $Y$ in  the same experiment as $Y=y$, then the conditional expectation of X conditioned on $Y=y$ is defined as follows:
\begin{align*}
E(X|Y=y)=\int_{\{x\in \Omega_X\}}xf_{X|Y}(x|y)dx=g(y).
\end{align*}
That is for each possible observation $Y=y, E(X|Y=y)=g(y)$, and hence conditional expectation defines a random variable (as a function of random variable Y) as:
\begin{align*}
g(Y)=E(X|Y).
\end{align*}
\end{definition}

\begin{example}
Question: consider two jointly continuous random variables $(X,Y)$ with a joint PDF $f_{X,Y}=2$ over a triangle region $X+Y\leq 1, X\geq 0, Y\geq 0$. Compute the conditional expectation $E(X|Y=y)$ conditioned on the observation $Y=y$.

\noindent Answer: we have the joint PDF as 
\begin{align*}
f_{X,Y}(x,y)=\left\lbrace\begin{array}{cc}
2 & x+y\leq 1, x\geq 0, y\geq 0 \\
0 & \ \mbox{otherwise}.
\end{array} \right.
\end{align*}
\begin{center}
\begin{figure}[h]
    \centering
    \includegraphics[width=1.0\textwidth]{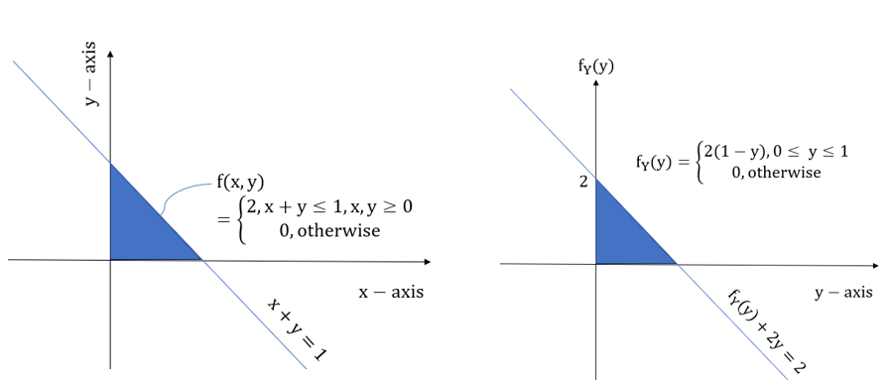}
    \caption{In left plot there is a joint region of the joint $f_{X,Y}=2$, in right plot its marginal PDF $f_Y(y)$.}
    \label{fig:jmpdf-01}
\end{figure}
\end{center}
Now, we compute marginal PDF $f_Y(y)$ (for $0\leq y \leq 1$) from the given joint PDF $f_{X,Y}(x,y)$ as:
\begin{align*}
f_Y(y)=\int_{-\infty}^{\infty}f_{X,Y}(x,y)dx=\int_0^{1-y} 2dx=2(1-y).
\end{align*} 
And the condition PDF $f_{X|Y}(x|y)$, conditioned on $Y=y$, can be computed as:
\begin{align*}
f_{X|Y}(x|y)=\frac{f_{X,Y}(x,y)}{f_Y(y)}=\frac{2}{2(1-y)}=\frac{1}{1-y}, 0\leq y\leq 1.
\end{align*}
Then, we compute the conditional expectation $E(X|Y=y)$, conditioned on $Y=y$, can be computed as:
\begin{align*}
E(X|Y=y)=\int_{-\infty}^{\infty}xf_{X|Y}(x|y)dx=\int_0^{1-y}x\frac{1}{(1-y)}dx=\frac{1-y}{2}.
\end{align*}
As $y$ varies, so does the conditional expectation $E(X|Y=y)$, and hence, we have the conditional expectation as a random variable as:
\begin{align*}
E(X|Y)& =\frac{1-Y}{2} \implies E(X)=\frac{1}{2}(1-E(Y)), E(X)=E(Y) \implies \\
E(X)  & =\frac{1}{3}.
\end{align*}
\begin{center}
\begin{figure}[h]
    \centering
    \includegraphics[width=1.0\textwidth]{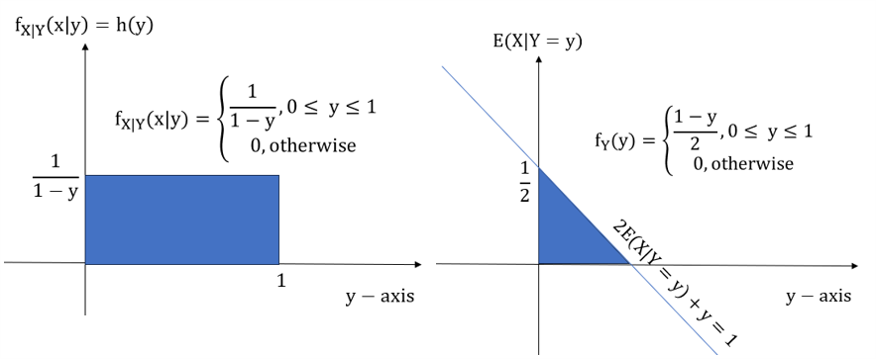}
    \caption{In left plot there is a conditional PDF $f_{X|Y}(x|y)$, in right plot its conditional expectation $E(X|Y=y)$.}
    \label{fig:cpdf-cexp-01}
\end{figure}
\end{center}
\end{example}

\begin{theorem}[Law of iterated expectation]
If $\{B_1,B_2,\ldots\}$ be a partition of $\Omega$ such that $P(B_i)> 0 \ \forall \ i$, and $X$ is a continuous random variable in the same experiment, then
\begin{align*}
E(X)=\sum_{i\geq 1} E(X|B_i)P(B_i).
\end{align*} 
Moreover, if $Y$ is a continuous random variable such that $Y^{-1}(y_i)=B_i=\{\omega | Y(\omega)=y_i\}$, then
\begin{align*}
E(X) =\int_{\{y\in \Omega_Y\}}E(X|Y=y)f_Y(y)dy =E(E(X|Y)).
\end{align*}
\end{theorem}

\begin{theorem}
Consider two jointly continuous random variables $(X,Y)$. Then, the variance of conditional expectation and expected value of conditional variance are related as:
\begin{align*}
Var(X)=Var(E(X|Y))+E(Var(X|Y)).
\end{align*}
\end{theorem}

\begin{proof}
We have
\begin{align*}
Var(X)& =E((X-E(X))^2)=E((X-E(X|Y)+E(X|Y)-E(X))^2) \\
      & =E((X-E(X|Y))^2)+2E((X-E(X|Y))(E(X|Y)-E(X)))\\
      & +E((E(X|Y)-E(X))^2)
\end{align*}
\end{proof}

\begin{example}[\citep{BT08}]
Question: we start with a stick of length $\ell$. We break it at a point which is chosen randomly and uniformly over its length, and keep the piece that contains the left end of the stick. We then repeat the same process on the stick that we were left with. Compute the expected length of the stick that we are left with, after breaking twice.

\noindent Answer: Let Y be the length of the stick after we break for the first time. Let X be the length after the second time. We have $E(X|Y)=\frac{Y}{2}$, since the breakpoint is chosen uniformly over the length Y of the remaining stick. For a similar reason, we also have $E(Y)=\frac{\ell}{2}$. Then, the expected value of $X$ can be computed as:
\begin{align*}
E(X)=E(E(X|Y))=E(Y/2)=\frac{\ell}{4}.
\end{align*}
\end{example}

\section{General random variable}
It may be of interest to note that discrete and continuous random variables can be subsumed under the topic of a general random variable. There exists the mathematical machinery to analyse both types of random variables simultaneously. This theory is called measure theory.

%% file: lsample.tex
Till now, we have described and computed the exact probability of events based on known probability mass functions (PMFs)/probability density functions (PDFs) and the implementation of its summation/integration. Also of importance were the methods to determine the moments/statistic/parameters of these random variables. In many practical situations the PMF/PDF may be unknown or the summation/integration may not be easily carried out. It would be of great utility, therefore, to be able to approximate the desired quantities using much simpler methods e.g. random sampling and point estimation. Law of large numbers asserts that the sample mean (the average of IID copies) converges to the expected value, a number, of each random variable in the average. Also, it provides a justification for the relative frequency interpretation of probability. Central limit theorem asserts that a properly normalized sum of I.I.D. random variables converges to a Gaussian random variable. Alternatively, it can be demonstrated that the repeated convolution of PDFs produces a Gaussian PDF.

\noindent So, in this section, we discuss large sample theory of sample statistic, when distribution of a measurement (or a random variable) is unknown. It discuss one (i.e. large sample size) of two important concepts to understand, obverse and utilize Big-Data. Another one (i.e. high-dimension) will be discussed in the next chapter \eqref{statm}.

\section{Sample statistic}\label{sec:ss-01}
Consider a random variable (or a measurement) $X$ with unknown distribution, then it is very difficult to compute statistic or parameter of the observed values under the unknown distribution. Now, we need to define sample statistic instead with respect to a random sample of size $n$ (i.e. $(X_1,X_2,\ldots,X_n) \to n$ independent and identically distributed (I.I.D.) copies of X).

\begin{definition}[Sample mean]\label{def:sm-01}
Consider a random variable (or a measurement) $X$ with unknown distribution. Assume a random sample of size $n$ (i.e. $(X_1,X_2,\ldots,X_n) \to n$ independent and identically distributed (I.I.D.) copies of X), then sample mean is defined as follows:
\begin{align*}
\bar{X}_n =\frac{X_1+X_2+\ldots +X_n}{n}.
\end{align*}
The sample mean is an estimator of the mean $\mu$ of the measurement X and its bias is defined as follows:
\begin{align*}
Bias(\bar{X})=E(\bar{X})-\mu.
\end{align*}
Consider a realization/observation of the random sample of size n, $(x_1, x_2,\ldots, x_n)$, then the corresponding realization of the sample mean as $\bar{x}$ returns the sample average
\begin{align*}
\bar{x}=\frac{x_1+x_2+\ldots +x_n}{n}.
\end{align*}
\end{definition}

\begin{definition}[Sample variance]\label{def:sv-02}
Consider a random variable (or a measurement) $X$ with unknown distribution. Assume a random sample of size $n$ (i.e. $(X_1,X_2,\ldots,X_n) \to n$ independent and identically distributed (I.I.D.) copies of X), then sample variance is defined as follows:
\begin{align*}
S_n^2 =\frac{1}{n-1}\sum_{i=1}^n \left(X_i-\bar{X}_n \right)^2=\frac{1}{n-1}\left(\sum_{i=1}^n X_i^2 -n\bar{X}^2 \right).
\end{align*}
The sample variance is an estimator of the variance $\sigma^2$ of the measurement X and its bias is defined as follows:
\begin{align*}
Bias(S^2)=E(S^2)-\sigma^2.
\end{align*}
Consider a realization/observation of the random sample of size $n, (x_1, x_2,\ldots, x_n)$, then the corresponding realization of the sample variance as $S^2$ returns the sample variance
\begin{align*}
S^2=\frac{1}{n-1}\sum_{i=1}^n \left(x_i-\bar{x} \right)^2=\frac{1}{n-1}\left(\sum_{i=1}^n x_i^2 -n\bar{x}^2 \right).
\end{align*}
Alternatively, we can have another definition of sample variance as follows:
\begin{align*}
S^2=\hat{Var(X)} &=\bar{X^2}-\left(\bar{X}\right)^2=\frac{1}{n}\sum_{i=1}^n X_i^2-\left(\frac{1}{n}\sum_{i=1}^n X_i \right)^2 \\
                 &=\frac{1}{n}\left(\sum_{i=1}^n X_i^2 -n\bar{X}^2 \right)=\frac{1}{n}\sum_{i=1}^n\left(X_i^2 -\bar{X} \right)^2.
\end{align*}
\end{definition}

\begin{definition}[Sample standard deviation]\label{def:sd-03}
In the above definition \eqref{def:sv-02}, $S^2$ is a sample variance of X, so the sample standard deviation is defined as
\begin{align*}
S=\sqrt{S^2}=\sqrt{\frac{1}{n-1}\sum_{i=1}^n \left(X_i-\bar{X}_n \right)^2}.
\end{align*}
Further, we have
\begin{align*}
0<Var(S^2)=E(S^2)-\left(E(S) \right)^2=\sigma^2 -\left(E(S) \right)^2 \implies E(S)<\sigma.
\end{align*}
So, the sample standard deviation $S$ is a biased estimator of the standard deviation $\sigma$ with a bias
\begin{align*}
Bias(S)=E(S)-\sigma <0.
\end{align*}
Consider a realization/observation of the random sample of size $n, (x_1, x_2,\ldots, x_n)$, then the corresponding realization of the sample standard deviation as $S$ returns the sample standard error
\begin{align*}
S=\sqrt{S^2}=\sqrt{\frac{1}{n-1}\sum_{i=1}^n \left(x_i-\bar{x} \right)^2}.
\end{align*}
\end{definition}

\begin{example}[Average mark in the course MA201]
Consider $X_1, X_2, X_3, X_4, X_5$ are the marks in the different evaluation components (e.g. quiz-1,quiz-2, assignment, mid-term, end-term respectively) i.e. we have a random sample of size 5 (upto normalization). Then, the sample mean $\bar{X}=\frac{X_1+X_2+\ldots + X_5}{5}$ with $70$ realization in section-A and 71 realization in section-B and $\bar{X}$ is a unbiased estimator of the average mark in the course MA201. 
\end{example}

\begin{example}
Backdrop: let $T$ be the time that is needed for a specific task in a factory to be completed. In order to estimate the mean and variance of $T$, we observe a random sample $(T_1,T_2,\ldots,T_6)$. Thus, $T_i$'s are i.i.d. and have the same distribution as $T$. 

\noindent Question: we obtain the following values (in minutes) as a realization/observation of the random sample of size $6$:
\begin{align*}
(t_1,t_2,t_3,t_4,t_5,t_6)=(18,21,17,16,24,20).
\end{align*}
Then compute the values of the sample mean, the sample variance, and the sample standard deviation for the observed sample.

\noindent Answer: as we have 
\begin{align*}
\bar{T}=\frac{T_1+T_2+\ldots +T_6}{6} \ S^2=\frac{1}{5}\left(\sum_{i=1}^n T_i^2 -6\bar{T}^2 \right)
\end{align*}
So, with respect to the realization/observation $(t_1,t_2,t_3,t_4,t_5,t_6)=(18,21,17,16,24,20)$ of the random sample of size 6, we have
\begin{align*}
\bar{t}=\frac{t_1+t_2+\ldots +t_6}{6} \ S^2=\frac{1}{5}\left(\sum_{i=1}^n t_i^2 -6\bar{t}^2 \right)
\end{align*}
So, we have the following computations:
\begin{align*}
\bar{T}=\bar{t}=19.33, \ S^2=8.67, \ S=\sqrt{S^2}=2.94.
\end{align*}
\end{example}

\subsection{Methods of estimation for sample statistic}
Various methods of estimation for sample statistic such as:
\begin{enumerate}
\item[a.] maximum likelihood estimation, 
\item[b.] method of moments,
\item[c.] max a-posteriori estimation,
\item[d.] least squared estimation
\end{enumerate}
will be discussed in detail in next chapter \eqref{statm}.

\subsection{Method of sampling}
In practice (e.g. in everyday activity or scientific research), our understanding and applicability are based on observation of some relevant samples. For example, a person (with several years of experience with an institute) can often determine an opinion, of an institute that conducts thousands of transactions everyday, by one or two encounters. In other scenario, a traveller who spends 10 days in a country and then proceeds to write a book telling the inhabitants how to revive their industries, reform their political system, balance their budget, and improve the food in their hotels is a familiar figure of fun. But in reality he differs from the political scientist who devotes 20 years to living and studying in the country only in that he bases his conclusions on a much smaller sample of experience and is less likely to be aware of the extent of his ignorance. In every branch of science we lack the resources to "study more than a fragment of the phenomena that might advance our knowledge. A list of the principal advantages of sampling as compared with complete enumeration.
\begin{enumerate}
\item[a.] Reduced cost: if data are secured from only a small fraction of the aggregate, expenditures may be expected to be smaller than if a complete census is attempted.
\item[b.] Greater speed: for the same reason, the data can be collected and summarized more quickly with a sample than with a complete
count. This may be a vital consideration when the information is urgently needed.
\item[c.] Greater scope: in certain types of inquiry, highly trained personnel or specialized equipment, limited in availability, must be used to obtain the data. A complete census may then be impracticable: the choice lies between obtaining the information by sampling or not at all. Thus surveys which rely on sampling have more scope and flexibility as to the types of information that can be obtained. On the other hand, if information is wanted for many subdivisions or segments of the population, it may be found that a complete enumeration offers the best solution.
\item[d.] Greater accuracy: because personnel of higher quality can he employed and can be given intensive training, a sample may actually produce more accurate results than the kind of complete enumeration that it is feasible to take.
\end{enumerate}

\noindent There are various methods of sampling:
\begin{enumerate}
\item[a.] Random sampling
\item[b.] Other sampling 
\item[c.] Issues with sampling such as sampling bias.
\end{enumerate}

\section{Probabilistic inequalities}
There are three fundamental probabilistic inequalities:
\begin{enumerate}
\item[a.] Markov inequality
\item[b.] Chebushev inequality
\item[c.] Chernoff inequality
\end{enumerate}

\subsection{Probabilistic inequalities and non-asymptotic sample theory}
The mean and variance of a random variable indicate the average value and variability of the outcomes of a repeated experiment. As such, they summarize important information about the distribution (i.e. PMF/PDF). However, they are not sufficient to determine probabilities
of events. For example, consider the PDFs:
\begin{align*}
f_X(x)=\frac{1}{\sqrt{2\pi}}\exp\left(-\frac{x^2}{2} \right), \ f_Y(y)=\frac{1}{\sqrt{2}}\exp\left(-\sqrt{2}|y| \right).
\end{align*}
both have $E(X)=0=E(Y)$ (due to symmetry about origin) and $var(X) = 1 = Var(Y)$. Yet, the probability of a given interval can be very different. Although the relationship between the mean and variance, and the probability of an event is not a direct one, we can still obtain some information about the probabilities based on the mean and variance. In particular, it is possible to bound the probability or to be able to assert that
\begin{align*}
P(|X-E(X)|\geq \epsilon) \leq \delta <1.
\end{align*}
For example, if the probability of a speech signal of mean 0 and variance 1 exceeding a given magnitude $\epsilon$ is to be no more than
$1\%$, then we would be satisfied if we could determine a $\epsilon$ so that
\begin{align*}
P(|X-E(X)|\geq \epsilon) \leq \delta <0.01.
\end{align*}
We now show that the probability for the event $|X-E(X)|\geq \epsilon$ can be bounded if we know the mean and variance. Computation of the probability is not required and therefore the distribution does not need to be known. Estimating the mean and variance is much easier than the entire distribution.

\begin{theorem}[Chebyshev inequality]
Consider a random variable X, then we have
\begin{align*}
P(|X-E(X)|> \epsilon) \leq \frac{Var(X)}{\epsilon^2}.
\end{align*}
Hence, the probability that a random variable deviates from its mean by more than $\epsilon$ (in either direction) is less than or equal to $\frac{Var(X)}{\epsilon^2}$. This agrees with the intuition that the probability of an outcome departing from the mean must become smaller as the width of the distribution decreases or equivalently as the variance decreases.
\end{theorem}

\begin{proof}
We have
\begin{align*}
Var(X)& =E((X-E(X))^2)=\int_{-\infty}^{\infty}\left(x-E(X) \right)^2f_X(x)dx \\
      & =\int_{|X-E(X)|>\epsilon \cup |X-E(X)|\leq \epsilon}\left(x-E(X) \right)^2f_X(x)dx \\
      & \geq \int_{|X-E(X)|> \epsilon}\left(x-E(X) \right)^2f_X(x)dx \\
      & \geq \epsilon^2 \int_{|X-E(X)|> \epsilon}f_X(x)dx =\epsilon^2 P(|X-E(X)|\geq \epsilon) \\
\implies & P\left(|X-E(X)|> \epsilon\right) \leq \frac{Var(X)}{\epsilon^2}.       
\end{align*}
\end{proof}

\begin{theorem}[Chebyshev inequality and confidence interval]
Consider a random sample of size n i.e. $(X_1,X_2,\ldots, X_n) \to n$ I.I.D. copies of a random variable $X$ with unknown distribution and $\mu=E(X), \sigma^2 =Var(X)<\infty$. Then, for any small positive real number $\epsilon$, we have
\begin{align*}
& P\left(|\bar{X}-\mu|>\epsilon \right) \leq \frac{\sigma^2}{n\epsilon^2}\leq \delta \equiv P\left(|\bar{X}-\mu|\leq \epsilon \right) \geq 1-\frac{\sigma^2}{n\epsilon^2} \geq 1-\delta \\
& \implies \bar{X} \in \left[\mu -\sigma \frac{1}{\sqrt{n\delta}}, \mu + \sigma \frac{1}{\sqrt{n\delta}}\right] \ \mbox{with high probability} \ 1-\delta. 
\end{align*}
\end{theorem}

\begin{example}
Consider a random variable X with $E(X)=0$ and $Var(X)=1$, then for an $\epsilon$, we have
\begin{align*}
P(|X|> \epsilon) \leq \frac{1}{\epsilon^2}.
\end{align*}
If $\epsilon =3$, then 
\begin{align*}
P(|X| > 3) \leq \frac{1}{9} \approx 0.11.
\end{align*} 
This is a rather loose bound as follows: 
\begin{enumerate}
\item[a.] if $X\sim N(0,1)$, then the actual value of this probability is 
\begin{align*}
P(|X|>3)&= 1-P(|X|\leq 3)=1-\left(P(3)-P(-3) \right) \\
        &=1-Q(3)+(1+Q(3))=2(1-Q(3))\approx 0.0026.
\end{align*}
\item[b.] if $X\sim L(0,1)$, then the actual value of this probability is 
\begin{align*}
P(|X|>3)&= \int_{x <-3 \cup x>3} \frac{1}{\sqrt{2}}\exp\left(-\sqrt{2}|x| \right)dx \\
        &=2\int_3^{\infty} \frac{1}{\sqrt{2}}\exp\left(-\sqrt{2}|x| \right) \\
        &=exp\left(-3\sqrt{2} \right) \approx 0.0144.
\end{align*}
\end{enumerate}
Hence, the actual probability is indeed less than or equal to the bound of $0.11$, but quite a bit less i.e. the bound provides a gross overestimation of the probability. A graph of the Chebyshev bound as well as the actual probabilities of $P(|X|> \epsilon)$ versus $\epsilon$ is shown in the below figure.
\begin{center}
\begin{figure}[h]
    \centering
    \includegraphics[width=1.0\textwidth]{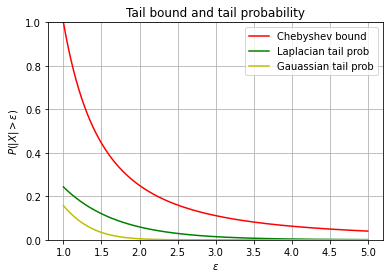}
    \caption{Probabilities $P(|X|> \epsilon)$ for Gaussian and Laplacian random variables with zero mean and unity variance compared to Ghebyshev inequality.}
    \label{fig:cbtb-01}
\end{figure}
\end{center}
\end{example}

\begin{theorem}[Markov inequality]
Consider a non-negative random variable X, then we have
\begin{align*}
P(X> \epsilon) \leq \frac{E(X)}{\epsilon}.
\end{align*}
Furthermore, if X is any random variable and $g$ is a non-negative function, then
\begin{align*}
P(g(X)> \epsilon) \leq \frac{E(g(X))}{\epsilon}.
\end{align*}
\end{theorem}

\begin{proof}
We have
\begin{align*}
E(g(X)) &=\int_{-\infty}^{\infty} g(x)f_X(x)dx =\int_{\{x:g(x)\geq \epsilon \cup g(x) < \epsilon\}}g(x)f_X(x)dx \\
        &\leq \int_{\{x:g(x)\geq \epsilon \}}g(x)f_X(x)dx \geq \epsilon \int_{\{x:g(x)\geq \epsilon \}}f_X(x)dx \\
        & =\epsilon P(g(X)\geq \epsilon) \implies P(g(X)> \epsilon) \leq \frac{E(g(X))}{\epsilon}.        
\end{align*}
\end{proof}

\begin{example}[Bound for QuickSort run \citep{T20}]
Question: Suppose the expected runtime of QuickSort is $2n\ln(n)$ operations/comparisons to sort an array of size n (we can show this using linearity of expectation with dependent indicator variables). Use Markov's inequality to bound the probability that QuickSort runs for longer than $20n\ln(n)$ time.

\noindent Answer: Let X be the runtime of QuickSort, with $E(X)=2n\ln(n)$. Then, since X is non-negative, we can use Markov's inequality as:
\begin{align*}
E(X>20n\ln(n))\leq \frac{E(X)}{20n\ln(n)}=\frac{2n\ln(n)}{20n\ln(n)}=\frac{1}{10}.
\end{align*}
So we know there's at most $10\%$ probability that QuickSort takes this long to run. Again, we can get this bound despite not knowing anything except its expectation.
\end{example}

\begin{example}[coin toss]
Question: A coin is weighted so that its probability of landing on heads is $20\%$, independently of other flips. Suppose the coin is  flipped 20 times. Use Markov's inequality to bound the probability it lands on heads at least 16 times.

\noindent Answer: We actually do know this distribution; the number of heads is $X\sim Bin(n=20, p=0.2)$. Thus, $E(X)=np=20\cdot 0.2=4$. By Markov's inequality: we have
\begin{align*}
& P(X\geq 16)\leq \frac{E(X)}{16}=\frac{4}{16}=0.25, \ \mbox{but} \\
& P(X\geq 16)=\sum_{k=16}^{20}\binom{n}{k}(0.2)^k(0.8)^{20-k}\approx 1.38\times 10^{-8}.
\end{align*}
\end{example}

\begin{corollary}[Chebyshev inequality as a corollary of the Markov inequality]
Consider a function of $X$ as $g(X)=\left(X-E(X)\right)$, then we have
\begin{align*}
P(|X-E(X(X)| &\geq \epsilon) =P\left((X-E(X))^2 \geq \epsilon^2 \right) \leq \frac{E\left(X-E(X))^2 \right)}{\epsilon^2} \\
             & =\frac{Var(X)}{\epsilon^2}.
\end{align*} 
\end{corollary}

\begin{example}[coin toss]
Question: Consider toss of a weighted coin independently with probability of landing heads $p= 0.2$. Upper bound the probability it lands on heads at least $16$ times out of $20$ flips using Chebyshev's inequality.

\noindent Answer: We actually do know this distribution; the number of heads is $X\sim Bin(n=20, p=0.2)$. Thus, $E(X)=np=20\cdot 0.2=4$.  and variance $Var(X)=np(1-p)=20\cdot 0.2\cdot 0.8=3.2$. By Chebyshev's inequality: we have
\begin{align*}
P(X\geq 16) &=P(X-4\geq 12)\leq \frac{1}{2}\cdot P(|X-4|\geq 12 ) \leq \frac{1}{2}\cdot \frac{Var(X)}{144} \\
           &=\frac{1}{2}\cdot \frac{3.2}{144}\approx 0.011, \ \mbox{but} \ P(X\geq 16)=\sum_{k=16}^{20}\binom{n}{k}(0.2)^k(0.8)^{20-k} \\
           &\approx 1.38\times 10^{-8}.
\end{align*}
\end{example}

\begin{remark}
As we are not able to improve Markov’s Inequality and Chebyshev’s Inequality in general, it is worth to consider whether we can say something stronger for a more restricted, yet interesting, class of random variables. This idea brings us to consider the case of a random variable that is the sum of a number of independent random variables.

\noindent This scenario is particularly important and ubiquitous in statistical applications. Examples of such random variables are the number of heads in a sequence of coin tosses, or the average support obtained by a political candidate in a poll.

\noindent Can Markov’s and Chebyshev’s Inequality be improved for this particular kind of random variable? Before confronting this question, let us check what Chebyshev’s Inequality (the stronger of the two) gives us for a sum of independent random variables.

\noindent Consider a sequence of random variables $X_1,X_2,\ldots,X_n,\ldots$ (i.e. $\{X_n\}$) with $E(X_i)=\mu_i$ and $Var(X_i)=\sigma^2_i$. Then for any real positive number $\epsilon$, we have
\begin{align*}
P\left(|\sum_{i=1}^n X_i-\sum_{i=1}^n \mu_i|>\epsilon\right) \leq \frac{Var(\sum_{i=1}^n X_i)}{\epsilon^2}=\frac{\sum_{i=1}^n \sigma^2_i}{\epsilon^2} 
\end{align*}
In particular, for identically distributed random variables with expectation $\mu$ and variance $\sigma^2$, we obtain
\begin{align*}
P\left(|\frac{1}{n}\sum_{i=1}^n X_i -\mu|>\epsilon \right) \leq \frac{\sigma^2}{n\epsilon^2}.
\end{align*}
\end{remark}

\begin{theorem}[Chernoff inequality]
Consider a random variable X, then we have
\begin{align*}
P(|X-E(X)|> \epsilon) \leq 2 \exp\left(-\frac{\epsilon^2}{2\sigma^2}\right).
\end{align*}
Hence, the probability that a random variable deviates from its mean by more than $\epsilon$ (in either direction) is less than or equal to $2 \exp\left(-\frac{\epsilon^2}{2\sigma^2}\right)$. This agrees with the intuition that the probability of an outcome departing from the mean must become smaller as the width of the distribution decreases or equivalently as the variance decreases and the sample size increases.
\end{theorem}
%

\begin{remark}
To improve a point estimation to a high accuracy (i.e. a small confidence interval), we need to improve either assumptions for the tight bound or sample size, because, confidence intervals produced by typical concentration inequalities have size $O\left(\frac{1}{\sqrt{n}} \right)$.
\end{remark}

\begin{example}[Chernoff inequality for a sample mean]
Consider a random variable X, then we have the Chernoff inequality as:
\begin{align*}
P(|X-E(X)|> \epsilon) \leq 2 \exp\left(-\frac{\epsilon^2}{2\sigma^2}\right).
\end{align*}
Now, consider a realization $\bold{x}=(x_1,x_2,\ldots,x_n)$ of an I.I.D. random sample of size n from $X\sim Ber(p)$. Then, the Chernoff inequality for the sample mean is given by
\begin{align*}
P(|\bar{X}-E(X)|> \epsilon) \leq 2 \exp\left(-\frac{n\epsilon^2}{2\sigma^2}\right).
\end{align*}
\end{example}

\begin{theorem}[Chernoff inequality and confidence interval]
Consider a random sample of size n i.e. $(X_1,X_2,\ldots, X_n) \to$ n I.I.D. copies of a sub-Gaussian random variable $X$ (i.e. $M_{X-E(X)}(s)\leq \exp\left(\frac{s^2\sigma^2}{2} \right)$) with unknown distribution and $\mu=E(X), \sigma^2 =Var(X)$. Then, for any small positive real number $\epsilon$, we have
\begin{align*}
& P\left(|\bar{X}-\mu|>\epsilon \right) \leq 2 e^{-\frac{n\epsilon^2}{2\sigma^2}} \leq \delta \equiv P\left(|\bar{X}-\mu|\leq \epsilon \right) \geq 1-2 e^{-\frac{n\epsilon^2}{2\sigma^2}} \implies \\
& \bar{X} \in \left[\mu -\sigma \sqrt{\frac{2}{n}\log\left(\frac{2}{\delta} \right)}, \mu + \sigma \sqrt{\frac{2}{n}\log\left(\frac{2}{\delta} \right)}\right] \ \mbox{with high prob.} \ 1-\delta.
\end{align*}
\end{theorem}

\begin{theorem}[Chernoff-Hoeffding inequality]
Consider a random variable X that observes values only in $[a,b]$, then we have
\begin{align*}
P(|X-E(X)|> \epsilon) \leq 2 \exp\left(-2\frac{\epsilon^2}{(b-a)^2}\right).
\end{align*}
Hence, the probability that a random variable deviates from its mean by more than $\epsilon$ (in either direction) is less than or equal to $\frac{Var(X)}{\epsilon^2}$. This agrees with the intuition that the probability of an outcome departing from the mean must become smaller as the width of the distribution decreases or equivalently as the variance decreases.
\end{theorem}

\begin{theorem}[Chernoff-Hoeffding inequality for sample mean]
Consider a random sample of size n i.e. $(X_1,X_2,\ldots, X_n) \to$ n I.I.D. copies of a random variable $X$ having a unknown distribution supported on $[a,b]$. Then, for any small positive real number $\epsilon >0$, we have
\begin{align*}
P\left(|\bar{X}-\mu|>\epsilon \right) \leq 2 \exp\left(-\frac{2n\epsilon^2}{(b-a)^2}\right)\leq \delta.
\end{align*}
In particular, when X is supported on $[a,b]$ we have
\begin{align*}
P\left(|\bar{X}-\mu|>\epsilon \right) \leq 2 \exp\left(-\frac{2n\epsilon^2}{(b-a)^2}\right)\leq \delta.
\end{align*}
\end{theorem}

\begin{theorem}[Chernoff-Hoeffding inequality and confidence interval]
Consider a random sample of size n i.e. $(X_1,X_2,\ldots, X_n) \to$ n I.I.D. copies of a random variable $X$ having a unknown distribution supported on $[a,b]$. Then, for any small positive real number $\epsilon$, we have
\begin{align*}
& P\left(|\bar{X}-\mu|>\epsilon \right) \leq 2 e^{-\frac{2n\epsilon^2}{(b-a)^2}} \leq \delta \equiv P\left(|\bar{X}-\mu|\leq \epsilon \right) \geq 1-2 e^{-\frac{2n\epsilon^2}{(b-a)^2}} \implies \\
& \bar{X} \in \left[\mu -(b-a) \sqrt{\frac{1}{2n}\log\left(\frac{2}{\delta} \right)}, \mu + (b-a) \sqrt{\frac{1}{2n}\log\left(\frac{2}{\delta} \right)}\right] \ \mbox{w. h. p.} \ 1-\delta.
\end{align*}
\end{theorem}

\subsection{Stochastic convergence and asymptotic large sample theory}
\begin{definition}[Convergence in probability]
Consider a sequence of random variables $X_1,X_2,\ldots,X_n,\ldots$ (i.e. $\{X_n\}$). Then,  $\{X_n\}$ converges in probability to a random variable $X$ for a small positive real number $\epsilon$, if 
\begin{align*}
\lim_{n \to \infty} P(|X_n -X| >\epsilon) =0 \equiv \lim_{n \to \infty} P(|X_n -X| \leq \epsilon) =1.
\end{align*}
This is abbreviated as 
\begin{align*}
X_n \to X \equiv plim_{n \to \infty} X_n= X.
\end{align*} 
\end{definition}

\begin{example}[Sample variance]
Consider a random sample of size n i.e. n copies of independent and identically distributed random variables $X_1,X_2,\ldots, X_n$ of a random X with unknown distribution having mean $\mu=E(X)$ and variance $\sigma^2=Var(X)$. Then
\begin{align*}
\lim_{n \to \infty} P(|\bar{X} -E(X)| >\epsilon) =0 \equiv \lim_{n \to \infty} P(|\bar{X} -E(X)| \leq \epsilon) =1
\end{align*}

\noindent We have
\begin{align*}
P(|\bar{X} -E(X)| >\epsilon) \leq \frac{Var(\bar{X})}{\epsilon^2} =\frac{\sigma^2}{n\epsilon^2} \to 0 \ \mbox{as} \ n \to \infty.
\end{align*}
\end{example}

\begin{example}
Let X be a random variable such that $X_n =X +Y_n$ where $E(Y_n)=\frac{1}{n}$ and $Var(Y_n)=\frac{\sigma^2}{n}$. Then prove that $X_n$ converges to X in probability as $n \to \infty$.

\noindent Solution: for a random variable X, it is given that
\begin{align*}
X_n=X+Y_n, \ \mbox{where} \ E(Y_n)=\frac{1}{n}, \ Var(Y_n)=\frac{\sigma^2}{n}.
\end{align*}
Now, we have 
\begin{align}
|Y_n|=|Y_n-E(Y_n)+E(Y_n)|\leq |Y_n-E(Y_n)|+\frac{1}{n}.
\end{align}\label{eq:cpp-01}
Then, using the above inequality \eqref{eq:cpp-01} and Chebyshev inequality, we have 
\begin{align*}
P(|X_n-X|> \epsilon) &= P(|Y_n|> \epsilon) \leq P(|Y_n-E(Y_n)|>\epsilon)\leq \frac{Var(Y_n)}{\epsilon} \\
                     &=\frac{\sigma^2}{n\epsilon^2} \to 0 \ \mbox{as} \ n \to \infty.
\end{align*}
\end{example}

\begin{definition}[Convergence in mean square]
Consider a sequence of random variables $X_1,X_2,\ldots,X_n,\ldots$ (i.e. $\{X_n\}$). Then,  $\{X_n\}$ converges in mean square error to a random variable $X$ for a small positive real number $\epsilon$, if 
\begin{align*}
\lim_{n \to \infty} E((X_n -X)^2) =0.
\end{align*}
This is abbreviated as 
\begin{align*}
X_n \to X \equiv mselim_{n \to \infty} X_n = X.
\end{align*} 
\end{definition}

\begin{remark}[Consistent an estimator in mean square error]
The $2 -$ mean convergence plays a remarkable role in defining a consistent estimator. Consider a sample statistic $\hat{\Theta}=T(X_1,X_2,\ldots , X_n)$ of a parameter $\theta$. We say that the estimator is consistent if its Mean Squared Error (MSE)
\begin{align*}
MSE(\hat{\Theta})& =E\left((\hat{\Theta}-\theta)^2 \right) = E\left(\left(\hat{\Theta}- E(\hat{\Theta})+ E(\hat{\Theta})-\theta\right)^2 \right) \\
& = E\left((\hat{\Theta}- E(\hat{\Theta}))^2 \right) + \left(E(\hat{\Theta})-\theta\right)^2 \\
& -2E\left(\hat{\Theta}-E(\hat{\Theta})\right)\left(E(\hat{\Theta})-\theta\right) \\
& =Var(\hat{\Theta})+Bias^2(\hat{\Theta}) -0 \to 0 \ \mbox{as} \ n\to \infty.
\end{align*}
\end{remark}

\begin{definition}[Convergence in distribution]
Consider a sequence of random variables $X_1,X_2,\ldots,X_n,\ldots$ (i.e. $\{X_n\}$). Then,  $\{X_n\}$ converges in distribution to a random variable $X$ for a small positive real number $\epsilon$, if 
\begin{align*}
\lim_{n \to \infty} P(X_n \leq x) =P(X\leq x) \equiv lim_{n \to \infty }F_{X_n}(x)=F(x)..
\end{align*}
This is abbreviated as 
\begin{align*}
X_n \to X \equiv dlim_{n \to \infty} X_n= X.
\end{align*} 
\end{definition}

\begin{example}[Sample mean]
Consider a random sample of size n i.e. n copies of independent and identically distributed random variables $X_1,X_2,\ldots, X_n$ of a random X with unknown distribution having mean $\mu=E(X)$ and variance $\sigma^2=Var(X)$. Then
\begin{align*}
\lim_{n \to \infty} P\left(\frac{\bar{X}-\mu}{\sigma} \leq z\right) =\phi(z).
\end{align*}
\end{example}

\begin{definition}[Convergence almost surely]  
Consider a sequence of random variables $X_1,X_2,\ldots,X_n,\ldots$ (i.e. $\{X_n\}$). Then,  $\{X_n\}$ converges almost surely to a random variable $X$ if
\begin{align*}
P\left(\{\omega : \lim_{n\to \infty}X_n(\omega)=X(\omega)\} \right).
\end{align*}
A probability of outcomes with point-wise convergence.
\end{definition}

\begin{remark}
Let $X_n$ be a sequence of random variables and X a random variable. Then the following implication diagram is satisfied:
\begin{center}
\begin{figure}[h]
    \centering
    \includegraphics[width=0.75\textwidth]{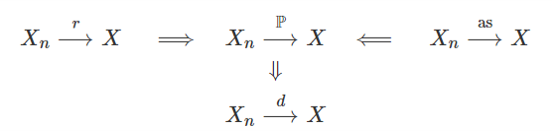}
    \caption{Relation and implication of various stochastic convergences of a sequence of random variables.}
    \label{fig:sci-01}
\end{figure}
\end{center}
\end{remark}

\section{Limit theorems}
Often, our goal is to understand how collections of a large number of independent random variables behave. This is what the laws of large numbers reveal. In general, the idea is that the average of a large number of I.I.D. random variables will approach the expectation. Sometimes that is enough, but usually, we also need to understand how fast does the average converge to the expectation.
\subsection{Law of large numbers}
\begin{theorem}[Strong law of large numbers]
Consider a random sample of size n i.e. n copies of independent and identically distributed random variables $X_1,X_2,\ldots, X_n$ of a random X with unknown distribution having mean $\mu=E(X)$ and variance $\sigma^2=Var(X)$. Then
\begin{align*}
\lim_{n \to \infty} P(|\bar{X} -E(X)| >\epsilon) =0 \equiv \lim_{n \to \infty} P(|\bar{X} -E(X)| \leq \epsilon) =1.
\end{align*}
\end{theorem}

\begin{proof}
Consider the probability of the sample mean random variable deviating from the expected value by more than $\epsilon$, where $\epsilon$ is a small positive number. This probability is given by
\begin{align*}
P\left(|\bar{X}-E(\bar{X})|>\epsilon \right) & =P\left(|\bar{X}-E(X)|>\epsilon \right) \leq \frac{Var(\bar{X})}{\epsilon^2} \\
                                            & =\frac{\sigma^2}{n\epsilon^2} \to 0 \ \mbox{as} \ n \to \infty. 
\end{align*}
It says that for large enough n, the probability of the error in the approximation of $\bar{X}$ by $E(X)$ exceeding $\epsilon$ (which can be chosen as small as desired) will be exceedingly small.
\end{proof}

\begin{example}[Bernoulli law of large numbers]
Prove that if a fair coin is tossed n times in succession, then the relative frequency of heads, i.e., the number of heads observed divided by the number of coin tosses, should be close to $1/2$. This was why we intuitively accepted the assignment of a probability of 1/2 to the event that the outcome of a fair coin toss would be a head.
\end{example}

\begin{proof}
Consider coin toss experiment as a sequence of n Bernoulli sub-experiments i.e. a random sample of size n for the experiment i.e. $(X_1,X_2,\ldots, X_n), X_i\sim Ber(1/2)$.

\noindent, Now, all possible relative frequency is given by the sample mean of the random sample of size n
\begin{align*}
\bar{X}=\frac{X_1+X_2+\ldots +X_n}{n}.
\end{align*}
In order to see asymptotic behaviour of the sample mean $\bar{X}$, we compute the following:
\begin{align*}
& E(\bar{X}) =E\left(\frac{\sum_{i=1}^n}{n}\right)=\frac{\sum_{i=1}^nE(X_i)}{n}=1/2. \\
& Var(\bar{X}) =Var\left(\frac{\sum_{i=1}^n X_i}{n}\right)=\frac{\sum_{i=1}^n Var(X_i)}{n^2}=\frac{1}{4n}.
\end{align*}
Now, as per definition of variance, we have
\begin{align*}
&Var(\bar{X})=E\left((\bar{X}-E(\bar{X}))^2 \right)=\frac{1}{4n} \to 0 \ \mbox{as} \ n\to \infty \\
& \implies \bar{X} \to E(\bar{X})=1/2 \ \mbox{as} \ n \to \infty.
\end{align*}
\begin{center}
\begin{figure}[h]
    \centering
    \includegraphics[width=0.75\textwidth]{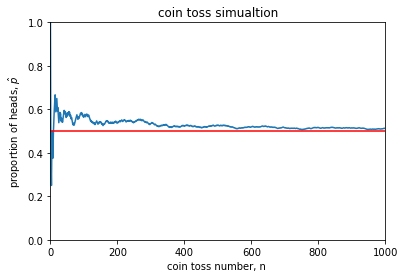}
    \caption{A realization of sample mean random variable of n I.I.D. Bernoulli random variables with $p=1/2$ as n increases.}
    \label{fig:blln-02}
\end{figure}
\end{center}
More generally for a Bernoulli sub-experiment with probability p, we have that
\begin{align*}
&Var(\bar{X})=E\left((\bar{X}-E(\bar{X}))^2 \right)=\frac{p(1-p)}{n} \to 0 \ \mbox{as} \ n\to \infty \\
& \implies \bar{X} \to E(\bar{X})=p \ \mbox{as} \ n \to \infty.
\end{align*}
\end{proof}

\begin{example}[Sample mean for I.I.D. Gaussian random variables]
Consider a sequence of n Gaussian sub-experiments i.e. a random sample of size n for the experiment i.e. $(X_1,X_2,\ldots, X_n), X_i\sim N(\mu, \sigma^2)$. Then $\bar{X} \to \mu$ in probability.
\end{example}

\subsection{Limit theorems}
At an intuitive level, it says that the appropriately scaled sum of a bunch of independent random variables behaves like a Gaussian (AKA Normal) random variable.
\begin{theorem}[Central limit theorem]
Consider a random sample of size n i.e. n copies of independent and identically distributed random variables $X_1,X_2,\ldots, X_n$ of a random X with unknown distribution having mean $\mu=E(X)$ and variance $\sigma^2=Var(X)$. And, define a standardized sample mean is given by
\begin{align*}
\bar{Z}_n =\frac{\bar{X}_n-E(\bar{X}_n)}{\sqrt{Var(\bar{X})}} = \frac{\sum_{i=1}^n X_i-n\mu}{\sqrt{n\sigma^2}}.
\end{align*}
Then, for any $z$, we have
\begin{align*}
\lim_{n\to \infty} P(\bar{Z}_n\leq z) = P(Z\leq z)=\phi(z)=\int_{-\infty}^z\frac{1}{\sqrt{2\pi}}\exp\left(-\frac{x^2}{2} \right)dx.
\end{align*}
In summary, $\bar{Z}_n \to Z\sim N(0,1)$ in distribution. Furthermore, the Berry-Esséen inequality gives a more precise quantification of the speed of this convergence
\begin{align*}
\left\|P(\bar{Z}_n \leq z)-\int_{-\infty}^z \frac{1}{\sqrt{2\pi}} \exp \left(-\frac{x^2}{2}\right)dx \right\| \leq \frac{0.77E\left(|X_1-E(X_1)|^3 \right)}{(Var(X_1))^{3/2}\sqrt{n}}.
\end{align*}
\end{theorem}

\begin{proof}
Consider a random sample of size n i.e. n I.I.D. copies of X as $(X_1,X_2,\ldots,X_n)$ and the corresponding standardized sample mean is given by 
\begin{align*}
\bar{Z}_n=\frac{\bar{X}_n -E(X)}{\sqrt{Var(\bar{X}_n)}}=\frac{n\bar{X}_n-nE(X)}{\sqrt{nVar(X)}}=\frac{S_n-nE(X)}{\sqrt{nVar(X)}}.
\end{align*}
Now, we have
\begin{align*}
M_{\bar{Z}_n}(s) &=E_{\bar{Z}_n}(\exp(s\bar{Z}_n))=E\left(\exp\left(\frac{s}{\sigma\sqrt{n}}(\sum_{i=1}^n X_i -n\mu)\right) \right) \\
                 &=E\left(\exp\left(\left(\frac{s}{\sqrt{n}}\right)\left(\frac{X_1-\mu}{\sigma}\right) +\ldots +\left(\frac{s}{\sqrt{n}}\right)\left(\frac{X_n-\mu}{\sigma}\right)\right)\right) \\
                 &= E\left(\exp\left(\frac{s}{\sqrt{n}}\frac{X_1-\mu}{\sigma}\right)\right) \cdots E\left(\exp\left(\frac{s}{\sqrt{n}}\frac{X_1-\mu}{\sigma}\right)\right) \\
                 & =\left(E\left(\exp\left(\frac{s}{\sqrt{n}}\frac{X-\mu}{\sigma}\right)\right) \right)^n = \left(M_{\frac{X-\mu}{\sigma}}\left(\frac{s}{\sqrt{n}}\right) \right)^n  \\
                 & \approx \left(M_{\frac{X-\mu}{\sigma}}(0) +M_{\frac{X-\mu}{\sigma}}^{'}(0)\frac{s}{\sqrt{n}} +M_{\frac{X-\mu}{\sigma}}^{''}(s_1/\sqrt{n})\frac{s^2}{2n} \right)^n \\
                 & \approx \left(1 + \frac{s^2}{2n} \right)^n \to \exp \left(\frac{s^2}{2}\right) \ \mbox{as} \ n\to \infty \implies \bar{Z}_n \xrightarrow{D} Z\sim \mathcal{N}(0,1).
\end{align*}
\end{proof}

\begin{example}[\citep{BT08}]
We have 100 bags, with weight distributed uniformly in $[5,50]$, so that the mean weight is $27.5$ lbs and the variance is $168.75$lbs. Can we upper bound the probability that the total weight is larger than $3000$lbs?

\noindent For comparison, let's see what happens if we use Chebyshev’s inequality
\begin{align*}
P(S_{100} > 3000) & =\frac{1}{2}\times P\left(\left|\frac{1}{100}\sum_{i=1}^{100}X_i -\frac{2750}{100} \right|> 2.5 \right) \\
                  & \leq \frac{1}{2}\times \frac{Var(X)}{100\cdot (2.5)^2}\approx 0.13 =13\%.
\end{align*}
That is a strict inequality, but intuitively it is overestimating the probability by a factor of two since it is also including the case of the average being significantly smaller than the mean as well.

\noindent If we use the CLT instead,
\begin{align*}
P\left(\sum_{i=1}^{100}X_i > 3000 \right) & =1-P\left(\sum_{i=1}^{100}X_i \leq 3000 \right) \\
& =1-P\left(\frac{\sum_{i=1}^{100}X_i -3000}{\sqrt{100\cdot 168.75}} \leq \frac{275}{\sqrt{100\cdot 168.75}} \right) \\
& \approx 1-P\left(Z\leq \frac{275}{\sqrt{100\cdot 168.75}} \right) =0.027.
\end{align*}
where Z is a random variable distributed as $N(0,1)$. So the CLT predicts that the probability of having a total weight of over 3000lbs is about $3\%$, which is much lower than the $27\%$ bound (or $13\%$ if we divide by two) that we got from Chebyshev’s inequality! It turns out that the Normal approximation here is very good even though the CLT is not a bound and $3\%$ is not an appropriately low number according to the Berry-Esséen inequality. This is because the CLT is very good for the sums of bounded random variables like the uniform.
\end{example}

\begin{remark}[An inequality for computing tail bounds for Gaussian distribution]
In order to use the CLT to get easily calculated bounds, the following approximations will often prove useful: for any $z>0$,
\begin{align*}
\left(1-\frac{1}{z^2} \right)\frac{\exp\left(-\frac{z^2}{2} \right)}{z\sqrt{2\pi}} \leq \int_z^{\infty} \frac{1}{\sqrt{2\pi}}\exp\left(-\frac{x^2}{2} \right)dx \leq \frac{\exp\left(-\frac{z^2}{2} \right)}{z\sqrt{2\pi}}.
\end{align*}
\end{remark}

\begin{example}
Backdrop: Assume a realization/people $\bold{x}=(x_1,\ldots,x_n)$ (of an I.I.D random sample of size n from $X\sim Ber(p)$)vote party A with probability p. We poll n people. Let $M_n$ be the fraction/average of the polled persons for Party A and defined as
\begin{align*}
M_n=\frac{1}{n}\sum_{i=1}^n X_i.
\end{align*}

\noindent Find n such that estimator $M_n$ to be within $\pm 1\%$ of the true value p with probability at least $95\%$, statistically 
\begin{align*}
P(|M_n-p|>0.01)\leq 0.05.
\end{align*}
And geometrically, can be visualized as in the figure \eqref{fig:clttb-03}:
\begin{center}
\begin{figure}[h]
    \centering
    \includegraphics[width=0.75\textwidth]{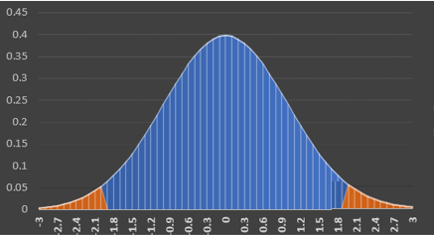}
    \caption{The distribution $M_n -p$ is symmetric, and we are looking for tail bounds.}
    \label{fig:clttb-03}
\end{figure}
\end{center}

\noindent Answer: Since the distribution of $M_n$ is symmetric around its mean, we can use this symmetry to write $P(|M_n-p|>0.01)\approx 2P(M_n-p>0.01)$. Because the $X_i\sim Ber(p)$, we can bound their variance by $1/4$, so $Var(M_n)\leq 1/4n$. Dividing both sides by the square root of the variance to put the probability into the CLT form, we get
\begin{align*}
& P(|M_n-p|>0.01) \approx 2P(M_n-p>0.01) \approx 2P\left(\frac{M_n -p}{\sqrt{1/4n}} > \frac{0.01}{\sqrt{1/4n}} \right) \to  \\
& 2P\left(Z > \frac{0.01}{\sqrt{1/4n}} \right) \implies  P\left(Z > \frac{0.01}{\sqrt{1/4n}} \right)=1-\phi \left(\frac{0.01}{\sqrt{1/4n}} \right) \leq .025 \\
& \implies \phi\left(\frac{0.01}{\sqrt{1/4n}}\right) \geq 0.9750 \implies \frac{0.01}{\sqrt{1/4n}} \geq 1.96 \implies n \geq 9604.
\end{align*}
In this calculation we used the approximation given by the CLT, but how good is it? If we work out the exact computation, using the fact that $M_n$ is a scaled binomial and the worse-case parameter $p=1/2$, then we get that $n=9604$ will give us a $95.1\%$ probability of being within $0.01$ of the true value p: in this case, the CLT’s approximation is very good. So if the CLT is so much better, why use Chebyshev’s inequality at all? The reason is that Chebyshev is a bound that works for all $\epsilon$ and probabilities, is always a valid upper bound but is very conservative. The CLT is often much sharper, however it only works for $\epsilon$ of appropriate scale and does not give a rock-solid bound as it is an approximation and not a bound. Is it possible to strengthen Chebyshev’s inequality to get a bound nearly as sharp as the CLT’s, under appropriate assumptions? Yes, the Chernoff inequality provides the optimal approximation.
\end{example}

\begin{remark}[Chernoff inequality and CLT]
To get some insight into the CLT by exploring how the value of the KL divergence $D(a\|p)$ of a scales when a is in the vicinity of p.
\begin{align*}
& D(a\|p)=E_a\left(\ln\left(\frac{a}{p} \right) \right) = \\
& \left\lbrace\begin{array}{cc}
\sum_{x\in \Omega_X} \ln\left(\frac{a(x)}{p(x)} \right)a(x) =\sum_{x\in \Omega_X}a(x)\ln(a(x)) - \sum_{x\in \Omega_X}a(x)\ln(p(x)) \\
\int_{x\in \Omega_X} \ln\left(\frac{a(x)}{p(x)} \right)a(x)dx =\int_{x\in \Omega_X}a(x)\ln(a(x))dx - \int_{x\in \Omega_X}a(x)\ln(p(x))dx.
\end{array} \right.
\end{align*}
KL divergence $D(a\|p))$ of a in the vicinity of p is the difference of entropy of a and expected value of $\ln(p)$ with respect to a:
\begin{align*}
D(a\|p)=a\ln(a)+(1-a)\ln(1-a)-a\ln(p)-(1-a)\ln(1-p).
\end{align*}
\end{remark}


%% file: statm.tex
In the previous chapter we discuss large sample theory that plays very important role in the statistical inference of statistical methods to be discussed here in details.

\section{Point estimation} \label{sec:pe}
There are various approaches of point estimations such as
\begin{enumerate}
\item[(i)] maximum likelihood estimation i.e. MLE
\item[(ii)] method of moments
\item[(iii)] Max a posterior estimation i.e. MAP.
\end{enumerate}
\subsection{Sample statistic}
Consider a random variable (or a measurement) $X$ with unknown distribution, then it is very difficult to compute statistic or parameter of the observed values under the unknown distribution. Now, we need to define sample statistic instead with respect to a random sample of size n (i.e. $(X_1,X_2,\ldots,X_n) \to$ n independent and identically distributed (i.i.d.) copies of X). Then, recalling concepts of sample statistic from subsection \eqref{ssec:ss-01} of chapter \eqref{lsample}:
\begin{enumerate}
\item[a.] Sample mean:
\begin{align*}
\bar{X}_n =\frac{X_1+X_2+\ldots +X_n}{n} \ \mbox{with a realization} \ \bar{x}=\frac{x_1+x_2+\ldots +x_n}{n}.
\end{align*}
\item[b.] Sample variance:
\begin{align*}
& S_n^2 =\frac{1}{n-1}\sum_{i=1}^n \left(X_i-\bar{X}_n \right)^2=\frac{1}{n-1}\left(\sum_{i=1}^n X_i^2 -n\bar{X}^2 \right) \\
& \ \mbox{with a realization} \ S^2=\frac{1}{n-1}\sum_{i=1}^n \left(x_i-\bar{x} \right)^2=\frac{1}{n-1}\left(\sum_{i=1}^n x_i^2 -n\bar{x}^2 \right).
\end{align*}
\item[c.] Sample standard deviation:
\begin{align*}
S&=\sqrt{S^2}=\sqrt{\frac{1}{n-1}\sum_{i=1}^n \left(X_i-\bar{X}_n \right)^2}, \ \mbox{with a realization} \\
s&=\sqrt{s^2}=\sqrt{\frac{1}{n-1}\sum_{i=1}^n \left(x_i-\bar{x} \right)^2}.
\end{align*}
\end{enumerate}

\subsection{Maximum likelihood estimation}
\subsubsection{Likelihood function}
\begin{definition}
Consider a realization/data $\bold{x}=(x_1,\ldots, x_n)$ an i.i.d. random sample of size n from $X\sim p_X(x|\theta)$ or $f_X(x|\theta)$, where $\theta$ is a parameter (or vector of parameters). Now, we define a likelihood of $\bold{x}$ given $\theta$ to be the probability of observing the data $\bold{x}$ as
\begin{align*}
L(\theta|\bold{x}) =P(\bold{x}|\theta) =\left\lbrace\begin{array}{cc}
\Pi_{i=1}^n p_X(x_i|\theta) & \ \mbox{if X is discrete}\\
\Pi_{i=1}^n f_X(x_i|\theta) & \ \mbox{if X is continuous}.
\end{array} \right.
\end{align*}
\begin{center}
\begin{figure}[h]
    \centering
    \includegraphics[width=1.0\textwidth]{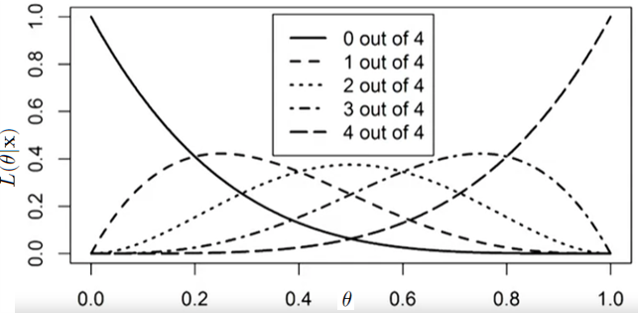}
    \caption{A likelihood function $L(\theta|\bold{x})$ wrt a probabilistic model $p_X(x|\theta)$ or $f_X(x|\theta)$ with a unknown parameter $\theta$.}
    \label{fig:lf-01}
\end{figure}
\end{center}
The likelihood function $L(\theta|\bold{x})$ has the following silent features:
\begin{enumerate}
\item[a.] The likelihood function is not a probability density function.
\item[b.] It is an important component of both frequentist and Bayesian analyses.
\item[c.] It measures the support provided by the data for each possible value of the parameter.
\end{enumerate}
If we compare the likelihood function at two parameter points and find that $L(\theta_1|\bold{x}) > L(\theta_2|\bold{x})$ then the realized sample data $\bold{x}$, we actually observed, is more likely to have occurred if $\theta=\theta_1$ than if $\theta=\theta_2$. This can be interpreted as $\theta_1$ is a more plausible value for $\theta$ than $\theta_2$.
\end{definition}

\begin{example}[A Bernoulli model]
Question: Consider a realization $\bold{x}=(x_1=1,x_2=0, x_3=1)$ of an i.i.d. sample of size 3 from $X\sim Ber(\theta)$, where $\theta$ is the unknown probability of a success. Compute the likelihood function $L(\theta|\bold{x})$.

\noindent Answer: it is given a realization $\bold{x}=(x_1=1,x_2=0, x_3=1)$ of an i.i.d. samples from $X\sim Ber(\theta)$, then we have
\begin{align*}
L(\theta|\bold{x})=P(\bold{x}|\theta)=\Pi_{i=1}^3p_X(x_i|\theta)=\Pi_{i=1}^3\theta^{x_i}(1-\theta)^{1-x_i}=\theta^2(1-\theta).
\end{align*}
\begin{center}
\begin{figure}[h]
    \includegraphics[width=0.80\textwidth]{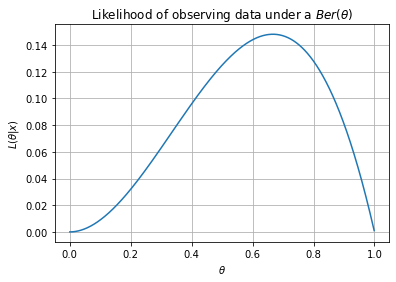}
    \caption{A likelihood function $L(\theta|\bold{x})=\theta^2(1-\theta)$ wrt $X\sim Ber(\theta)$ with a unknown parameter $\theta$.}
    \label{fig:lfbm-01}
\end{figure}
\end{center}
\end{example}

\begin{example}
Consider a realization $\bold{x}=(x_1=3,x_2=0, x_3=2, x_4=7)$ of an I.I.D. sample of size 4 from $X\sim Pois(\theta)$, where $\theta$ is the historical average number of events in a unit of time).

\noindent Answer: it is given a realization $\bold{x}=(x_1=1,x_2=0, x_3=1)$ of an I.I.D. samples from $X\sim Ber(\theta)$, then we have
\begin{align*}
L(\theta|\bold{x})=P(\bold{x}|\theta)=\Pi_{i=1}^4p_X(x_i|\theta)=\Pi_{i=1}^4e^{\theta}\frac{\theta^{x_i}}{(x_i)!}=e^{-4\theta}\frac{\theta^{12}}{2!3!7!}.
\end{align*}
\end{example}

\subsubsection{Maximum likelihood estimation}
\begin{definition}[Maximum likelihood estimate]\label{def:mle-01}
Consider a realization $\bold{x}=(x_1, \ldots, x_n)$ of an I.I.D. random sample of size n from $X\sim p_X(x|\theta)$ or $f_X(x|\theta)$, where $\theta$ is a parameter (or vector of parameters). Then, the maximum likelihood estimate $\hat{\theta}_{MLE}$ of $\theta$ to be the parameter which maximizes the likelihood (or equivalently, the log-likelihood) of the data as:
\begin{align*}
\hat{\theta}=\arg\max_{\theta} L(\theta|\bold{x}) =\arg\max_{\theta)} ln\left(L(\theta|\bold{x}) \right).
\end{align*}
\end{definition}

\begin{example}[Bernoulli model]
Question: consider a realization $\bold{x}=(x_1=1, x_2=1, x_3=1, x_4=1, x_5=0)$ of an I.I.D. sample of size 4 from $X\sim Ber(\theta)$ distribution with unknown parameter $\theta$. Find the maximum likelihood estimate $\hat{\theta}$ of $\theta$.

\noindent We have the likelihood function with respect to the given realization of the sample as
\begin{align*}
L(\theta|\bold{x}) & =\Pi_{i=1}^5 p_X(x_i|\theta)=\Pi_{i=1}^5\theta^{x_i}(1-\theta)^{1-x_i}=\theta^4(1-\theta).
\end{align*}
The plot of the likelihood with $\theta$ on the x-axis and $L(HHHHT|\theta)$ on the y-axis in \eqref{fig:lfbm-02} as:
\begin{center}
\begin{figure}[h]
    \centering
    \includegraphics[width=1.0\textwidth]{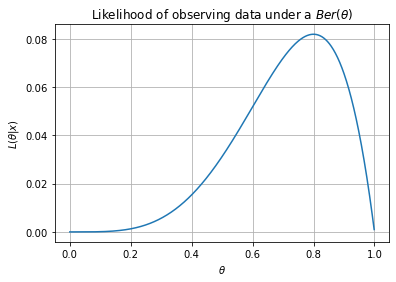}
    \caption{A likelihood function with the Bernoulli model for unknown parameter $\theta$ (the probability of heads).}
    \label{fig:lfbm-02}
\end{figure}
\end{center}
and we can actually see the $\theta$ which maximizes the likelihood is $\hat{\theta}=4/5$. But sometimes it is not possible to plot the likelihood, so we solve it analytically now.

\noindent To find the $\theta$ which maximizes this likelihood, so we take the derivative with respect to $\theta$ and set it to 0:
\begin{align*}
\frac{d}{d\theta}(L(\theta|\bold{x})) =\frac{d}{d\theta}(\theta^4 (1-\theta)) =\theta^3 (4-5\theta)=0 \implies \theta_{opt}=4/5.
\end{align*}
So, in MLE, we're trying to find the $\theta$ that maximizes the likelihood, and we don't care what the maximum value of the likelihood is. We didn't even compute it. We just care that the $\arg\max_{\theta} L(\theta|x)$ is $4/5$.
\end{example}

\begin{remark}[Likelihood vs log-likelihood]
Consider a realization $\bold{x}=(x_1,\ldots, x_n)$ of an I.I.D. random sample of size $n$ from $X\sim p_X(x|\theta)$ or $f_X(x|\theta)$, where $\theta$ is a parameter (or vector of parameters). The log-likelihood of x given $\theta$ are defined as:
\begin{align*}
ln\left(L(\theta|x \right) = \left\lbrace\begin{array}{cc}
\sum_{i=1}^n ln\left(p_X(x_i|\theta \right) & \ \mbox{X is discrete} \\
\sum_{i=1}^n ln\left(f_X(x_i|\theta \right) & \ \mbox{X is continuous}.
\end{array} \right..
\end{align*}
Since, $\ln(\cdot)$ is a monotone increasing function, so it preserves order of the argument, so whatever was the maximizer ($\arg\max$) in the original function, will also be maximizer in the log function. For example, we can see in the following picture 
\begin{center}
\begin{figure}[h]
    \centering
    \includegraphics[width=1.0\textwidth]{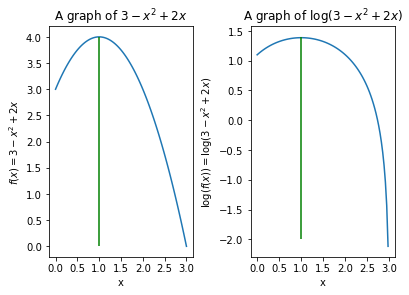}
    \caption{The values are different in y-axis, but if you look at the x-axis, it happens that both functions are maximized at $1$ i.e. the $\arg\max$'s are the same.}
    \label{fig:lf-llf-02}
\end{figure}
\end{center}
\end{remark}

\begin{example}
Consider a realization $\bold{x}=(x_1=1,x_2=0, x_3=1)$ of an I.I.D. sample of size $n=3$ from $X\sim Ber(\theta)$.Compute the MLE estimate of $\theta$.

\noindent Answer: Since $X\sim Ber(\theta)$, so we have $p_X(x|\lambda)=\lambda^{x}(1-\lambda)^{1-x}$). Now, likelihood function is defined as
\begin{align*}
L(\theta|\bold{x}) =P(\bold{x}|\theta)=\Pi_{i=1}^n p_X(x_i|\theta)=\Pi_{i=1}^n \theta^{x_i}(1-\theta)^{1-x_i}.
\end{align*}
Taking log both sides, we get
\begin{align*}
\ln\left(L(\theta|x) \right) & =\sum_{i=1}^n \left(x_i \ln(\theta) +(1-x_i)\ln(1-\theta) \right) \\
                             & =\ln(\theta)\sum_{i=1}^n x_i + \ln(1-\theta)\sum_{i=1}^n (1-x_i).
\end{align*}
Take the derivative(s) with respect to $\theta$ and set to 0, we have
\begin{align*}
\frac{d}{d\theta}\left(\ln\left(L(\theta|x)\right)\right) &=\frac{1}{\theta}\sum_{i=1}^n x_i -\frac{1}{1-\theta}\sum_{i=1}^n(1-x_i) \\
                                                          &=\frac{1}{\theta (1-\theta)} \left((1-\theta)\sum_{i=1}^n x_i -\theta \sum_{i=1}^n (1-x_i) \right) \\
                                                          &= \frac{1}{\theta (1-\theta)} \left(\sum_{i=1}^n x_i -n\theta \right) \\
\frac{d^2}{d\theta^2}\left(\ln\left(L(\theta|x)\right)\right) &=\frac{-1}{\theta^2}\sum_{i=1}^n x_i -\frac{1}{(1-\theta)^2}\sum_{i=1}^n(1-x_i)
\end{align*}
Equating the derivative of $\ln(L(\theta|x))$ to 0, we get
\begin{align*}
\hat{\theta}_{MLE}=\frac{1}{n}\sum_{i=1}^n x_i =\bar{x} \ \mbox{with} \ \frac{d^2}{d\theta^2}\left(\ln\left(L(\theta|x)\right)\right)\leq 0. \ \mbox{So,} \ \hat{\Theta}=\bar{X}.
\end{align*}
\end{example}

\begin{example}[Poisson model]
Consider a realization $\bold{x}=(x_1,\ldots, x_n)$ of an I.I.D. sample of size n from $X\sim Pois(\theta)$. Compute the MLE estimate of $\theta$.

\noindent Answer: Since $X\sim Ber(\theta)$, so we have $p_X(x|\theta)=e^{-\theta}\frac{\theta^x}{x!}$). Now, likelihood function is defined as:
\begin{align*}
L(\theta|\bold{x}) =P(\bold{x}|\theta) =\Pi_{i=1}^n p_X(x_i|\theta)=\Pi_{i=1}^n e^{-\theta}\frac{\theta^{x_i}}{(x_i)!}.            
\end{align*}
Now, taking log sides, we get
\begin{align*}
\ln\left(L(x|\theta)\right) & =\ln\left(\Pi_{i=1}^n p_X(x_i|\theta)\right)=\sum_{i=1}^n\ln\left(p_X(x_i|\theta)\right)  \\
            & = \sum_{i=1}^n -\theta + x_i \ln(\theta) -\ln((x_i)!).
\end{align*}
Take the derivative(s) with respect to $\theta$ and set to 0, we have
\begin{align*}
& \frac{d}{d\theta}\ln\left(L(\theta|x) \right)=\sum_{i=1}^n \left(-1 +\frac{x_i}{\theta} \right)=0 \implies -n+\frac{1}{\theta}\sum_{i=1}^n =0 \implies \\
& \hat{\theta}=\frac{1}{n}\sum_{i=1}^n x_i =\bar{x} \ \mbox{with} \ \frac{d^2}{d\theta^2}\ln\left(L(\theta|x) \right)=\sum_{i=1}^n\left( -\frac{x_i}{\theta^2}\right) \leq 0. \ \mbox{So,} \ \hat{\Theta}=\bar{X}.
\end{align*}
\end{example}

\begin{example}
Consider a realization $\bold{x}=(x_1,\ldots, x_n)$ of an I.I.D. sample of size $n$ from $X\sim Exp(\theta)$. Compute the MLE estimate of $\theta$.

\noindent Answer: Since $X\sim Ber(\theta)$, so we have $p_X(x|\theta)=\theta e^{-x\theta}$). Now, likelihood function is defined as:
\begin{align*}
L(\theta|\bold{x}) =P(\bold{x}|\theta) =\Pi_{i=1}^n p_X(x_i|\theta)=\Pi_{i=1}^n \theta e^{-x_i\theta}.            
\end{align*}
Now, taking log sides, we get
\begin{align*}
\ln\left(L(x|\theta)\right) & =\ln\left(\Pi_{i=1}^n p_X(x_i|\theta)\right)=\sum_{i=1}^n\ln\left(p_X(x_i|\theta)\right)  \\
            & = \sum_{i=1}^n -x_i\theta +  \ln(\theta).
\end{align*}
Take the derivative(s) with respect to $\theta$ and set to 0, we have
\begin{align*}
& \frac{d}{d\theta}\ln\left(L(\theta|x) \right)=\sum_{i=1}^n \left(-x_i +\frac{1}{\theta} \right)=0 \implies \frac{n}{\theta}-\sum_{i=1}^n x_i =0 \implies \\
& \hat{\theta}=\frac{1}{n}\sum_{i=1}^n x_i =\bar{x} \ \mbox{with} \ \frac{d^2}{d\theta^2}\ln\left(L(\theta|x) \right)=\sum_{i=1}^n\left( -\frac{1}{\theta^2}\right) \leq 0. \ \mbox{So,} \ \hat{\Theta}=\bar{X}.
\end{align*}
\end{example}

\begin{example}
Consider a realization $\bold{x}=(x_1,\ldots, x_n)$ of an I.I.D. sample of size $n$ from $X\sim Unif(0,\theta)$. Compute the MLE estimate of $\theta$.

\noindent Answer: Since $X\sim Ber(\theta)$, so we have $p_X(x|\theta)=\frac{1}{\theta}I_{0\leq x \leq \theta}$). Now, likelihood function is defined as:
\begin{align*}
L(x|\theta) & =\Pi_{i=1}^n p_X(x_i|\theta)=p_X(x_1|\theta)\cdot p_X(x_2|\theta)\cdots p_X(x_n|\theta) \\
            & = \frac{1}{\theta}I_{\{0\leq x_1\leq \theta\}} \cdot \frac{1}{\theta}I_{\{0\leq x_2\leq \theta\}} \cdots \frac{1}{\theta}I_{\{0\leq x_n\leq \theta\}} \\
            & =\frac{1}{\theta^n}I_{\{0\leq x_1,x_2,\ldots,x_n\leq \theta\}}.
\end{align*}
\begin{center}
\begin{figure}[h]
    \centering
    \includegraphics[width=0.80\textwidth]{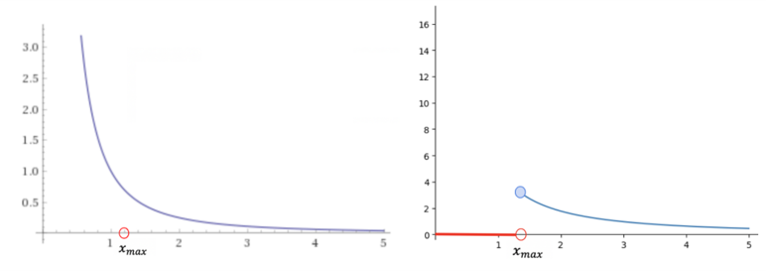}
    \caption{The first (left) plot is graph of $\frac{1}{\theta^n}$, and the second (right) plot of the actual likelihood function $\frac{1}{\theta^n}I_{\{0\leq x_1,\ldots,x_n\leq \theta\}} \equiv \frac{1}{\theta^n}I_{\{x_{max}\leq \theta\}}$.}
    \label{fig:udlf-03}
\end{figure}
\end{center}
Take the derivative(s) with respect to $\theta$ and set to 0, we have
\begin{align*}
& \frac{d}{d\theta}L(\theta|x)=\frac{-n}{\theta^{n+1}}I_{\{0\leq x_1,x_2,\ldots,x_n\leq \theta\}}=0 \implies  \hat{\theta}=\max_{\{i=1,2,\ldots,n\}} x_i \\
&=x_{max} \ \mbox{with} \ \frac{d^2}{d\theta^2} L(\theta|x) =\frac{n(n+1)}{\theta^{n+2}}I_{\{0\leq x_1,x_2,\ldots,x_n\leq \theta\}} \geq 0. \\
& \ \mbox{So,} \ \hat{\Theta}=max_{\{i=1,2,\ldots,n\}}X_i=X_{max}.
\end{align*}
%
\end{example}

\subsection{Method of moments}
In the last sub-section, we have seen maximum likelihood estimation (MLE) had a nice intuition but mathematically is a bit tedious to solve. So, we'll learn an another technique for estimating parameters called the Method of Moments (MoM) with several examples which hopefully makes things clearer.

\begin{definition}[kth moment]\label{def-km-01}
Let X be a random variable and $c\in \mathbb{R}$ be a scalar. Then: (a.) the kth moment of X and (b.) the kth moment of X (about c) are defined as:
\begin{align*}
(a.) \ E(X^k) \ \mbox{and} \ (b.) \ E((X-c)^k) \ \mbox{respectively}.
\end{align*}
Usually, we are interested in the first moment of X: $\mu =E(X)$, and the second moment of X about $\mu$: $Var(X)=E((X-\mu)^2)$.
\end{definition}

\begin{definition}[kth sample moment]\label{def-ksm-02}
Consider a realization $\bold{x}=(x_1,\ldots,x_n)$ of an I.I.D. random sample of size $n$ from $X\sim p_X(x|\theta)$ or $f_X(x|\theta)$ with unknown parameter $\theta$ and $c\in \mathbb{R}$ be a scalar. Then: (a.) the kth sample  moment of X and (b.) the kth sample moment of X (about c) are defined as:
\begin{align*}
(a.) \ \frac{1}{n}\sum_{i=1}^n x_i^k \ \mbox{and} \ (b.) \ \frac{1}{n}\sum_{i=1}^n(X-c)^k \ \mbox{respectively}.
\end{align*}
Usually, we are interested in the first sample moment of X: $\bar{x} =\frac{1}{n}\sum_{i=1}^n x_i$, and the second sample moment of X about $\mu$: $S^2=\frac{1}{n}\sum_{i=1}^n(x_i-\bar{x})^2$.
\end{definition}

\begin{example}
Consider a realization $\bold{x}=(x_1,\ldots , x_n)$ of an I.I.D. random sample size $n$ from $X\sim Unif(0,\theta)$, where $\theta$ is a unknown parameter. Compute MoM estimator of $\theta$.

\noindent Answer: since $X\sim Ber(\theta)$, so we have $p_X(x|\theta)=\frac{1}{\theta}I_{0\leq x \leq \theta}$), so $E(X)=\frac{\theta}{2}$. Then, from first moment approximation, we have
\begin{align*}
E(X)=\frac{\theta}{2}=\frac{1}{n}\sum_{i=1}^n \implies \hat{\theta}_{MoM} =\frac{2}{n}\sum_{i=1}^n x_i.
\end{align*}
It can be observed that
\begin{align*}
\hat{\theta}_{MoM}=\frac{2}{n}\sum_{i=1}^n x_i \neq \hat{\theta}_{MLE}=x_{max}.
\end{align*}
And the MoM estimator of $\theta$ is given by
\begin{align*}
\hat{\Theta}_{MoM}=\frac{2}{n}\sum_{i=1}^n X_i=2\bar{X}.
\end{align*}
\end{example}

\begin{example}
Consider a realization $\bold{x}=(x_1,\ldots , x_n)$ of an I.I.D. sample of size $n$ from $X\sim Exp(\theta)$. Compute the MoM estimator of $\theta$.

\noindent Answer: since $X\sim Exp(\theta)$, so we have $f_X(x|\theta)=\theta e^{-x\theta}$), so $E(X)=\frac{1}{\theta}$. Then, from first moment approximation, we have
\noindent We set the first true moment to the first sample moment as follows (recall that $X\sim Exp(\lambda), E(X)=\frac{1}{\lambda}$):
\begin{align*}
E(X)=\frac{1}{\theta}=\frac{1}{n}\sum_{i=1}^n x_i \implies \hat{\theta}_{MoM}=\frac{1}{\frac{1}{n}\sum_{i=1}^n x_i}=\hat{\theta}_{MLE}=\frac{1}{\bar{x}}.
\end{align*}
So, the MoM estimator of the $\theta$ is given by
\begin{align*}
\hat{\Theta}_{MoM}=\frac{1}{\frac{1}{n}\sum_{i=1}^n X_i}=\frac{1}{\bar{X}}.
\end{align*}
\end{example}

\begin{example}
Consider a realization $\bold{x}=(x_1,\ldots , x_n)$ of an I.I.D. random sample size $n$ from $X\sim Pois(\theta)$. Compute MoM estimator of $\theta$.

\noindent Answer: since $X\sim Ber(\theta)$, so we have $p_X(x|\theta)=\frac{1}{\theta}I_{0\leq x \leq \theta}$), so $E(X)=\theta$. Then, from first moment approximation, we have
\begin{align*}
E(X)=\theta=\frac{1}{n}\sum_{i=1}^n \implies \hat{\theta}_{MoM} =\frac{1}{n}\sum_{i=1}^n x_i=\bar{x}.
\end{align*}
Notice that in this case, the MoM estimate disagrees with the MLE  as follows:
\begin{align*}
\frac{1}{n}\sum_{i=1}^n x_i=\hat{\theta}_{MoM} = \hat{\theta}_{MLE}=\bar{x}.
\end{align*}
So, the MoM estimator of $\theta$ is given by
\begin{align*}
\hat{\Theta}_{MoM}=\frac{1}{n}\sum_{i=1}^n X_i=\bar{X}.
\end{align*}
\end{example}

%

\subsection{Max a posterior estimation}
In order to begin with Bayesian estimation, it requires us to learn at least few special distributions (e.g. Beta distribution) especially meant for prior distribution assumption.

\subsubsection{The Beta Random Variable}
In order to proceed in Bayesian framework, it is very much essential to define few useful distributions to be utilized for prior distribution of the unknown parameter $\theta$.

\begin{definition}[Beta distribution \citep{T20}]\label{def:betadis-01}
A random variable X is having beta distribution i.e. $X\sim Beta(\alpha,\beta)$, if and only if X has the following density function (and range $\Omega_X =[0,1]$):
\begin{align*}
f_X(x)=\left\lbrace\begin{array}{cc}
\frac{1}{B(\alpha,\beta)}x^{\alpha-1}(1-x)^{\beta-1} & \ x\in [0,1],\\
0 & \ \mbox{otherwise}.
\end{array} \right.
\end{align*}
X is typically the belief distribution about some unknown probability of success, where we pretend we've seen $\alpha -1$ successes and $\beta -1$ failures. Hence the mode (most likely value of the probability/point with highest density) $\arg\max_{x\in [0,1]}f_X(x)$, is
\begin{align*}
mode(X)=\frac{\alpha -1}{\alpha -1 +\beta -1}.
\end{align*}
Also note the following:
\begin{enumerate}
\item[a.] The first term in the pdf, $\frac{1}{B(\alpha,\beta)}$ is just a normalizing constant (ensures the pdf to integrates to
1). It is called the Beta function, and so our random variable is named as a Beta random variable.
\item[b.] There is an annoying off-by-1 issue: ($\alpha -1$ heads and $\beta -1$ tails), so when choosing these parameters, be careful.
\item[c.] x is the probability of success, and $(1 - x)$ is the probability of failure.
\end{enumerate}
\end{definition}

\begin{example}[\citep{T20}]
If you flip a coin with unknown probability of heads X, identify the parameters of the most appropriate Beta distribution to model your belief:
\begin{enumerate}
\item[a.] You didn't observe anything (i.e. $\alpha = \# heads +1=0+1=1, \beta=\# tails +1=0+1=1$)?

\noindent It's $B(0+1,0+1) \equiv B(1,1) \equiv Unif(0,1) \to$ NO mode (because it follows the Uniform distribution, every point has same density).
\item[b.] You observed 8 heads and 2 tails?

\noindent It's $B(8+1,2+1)\equiv B(9,3) \to$ mode(X) = $\frac{9-1}{9-1 +3-1}=\frac{8}{10}=0.8$.
\item[c.] You observed 80 heads and 20 tails?

\noindent It's $B(80+1,20+1)\equiv B(81,21) \to$ mode(X) = $\frac{81-1}{81-1 +21-1}=\frac{80}{100}=0.8$.
\item[d.] You observed 2 heads and 3 tails?

\noindent It's $B(2+1,3+1)\equiv B(3,4) \to$ mode(X) = $\frac{3-1}{3-1 +4-1}=\frac{2}{5}=0.4$.
\end{enumerate}
\end{example}

\noindent The Dirichlet random vector generalizes the Beta random variable to having a belief distribution over $(p_1, p_2,\ldots , p_r)$ (like in the multinomial distribution so $\sum_{i=1}^rp_i=1$), and has r parameters $(\alpha_1, \alpha_2,\ldots,\alpha_r)$. It has the similar interpretation of pretending you've seen $\alpha_i -1$ outcomes of type i.

\begin{definition}[Dirichlet distribution]
A random variable X is having Dirichlet distribution i.e. $X\sim Dir(\alpha_1,\alpha_2,\ldots,\alpha_r)$, if and only if X has the following density function (and range $\Omega_X =[0,1]$):
\begin{align*}
f_X(x)=\left\lbrace\begin{array}{cc}
\frac{1}{B(\alpha_1,\alpha_2,\ldots,\alpha_r)}\Pi_{i=1}^r x^{\alpha_i-1} & \ x_i\in [0,1], \sum_{i=1}^r x_i=1\\
0 & \ \mbox{otherwise}.
\end{array} \right.
\end{align*}
This is a generalization of the Beta random variable from 2 outcomes to r. X is typically the belief distribution about some unknown probability of different outcomes, where we pretend we've seen $\alpha_1 -1$ outcomes of type 1, $\alpha_2 -1$ outcomes of type 2 $\ldots$, $\alpha_r -1$ outcomes of type r. Hence the mode (most likely value of the probability/point with highest density) $\arg\max_{x_i\in [0,1], \sum_{i=1}^r x_i=1}f_X(x)$, is
\begin{align*}
mode(X)=\left( \frac{\alpha_1 -1}{\sum_{i=1}^r(\alpha_i -1)}, \frac{\alpha_2 -1}{\sum_{i=1}^r(\alpha_i -1)}, \ldots, \frac{\alpha_r -1}{\sum_{i=1}^r(\alpha_i -1)} \right).
\end{align*}
Also note the following:
\begin{enumerate}
\item[a.] Similar to the Beta RV, the first term in the pdf, $\frac{1}{B(\alpha)}$ is just a normalizing constant (ensures the pdf to integrates to 1), where $\alpha=(\alpha_1,\alpha_2,\ldots,\alpha_r)$. 
\item[b.] Notice that this is the probability distribution over the random vector $x_i$'s, which is the vector of probabilities, so they must sum to 1 i.e. $\sum_{i=1}^r x_i =1$.
\end{enumerate}
\end{definition}

\subsubsection{MAP}

\begin{definition}[MAP]
Consider a realization/data $x=(x_1,\ldots, x_n)$ of an I.I.D. sample of size $n$ from probability mass function $p_X(x:\Theta =\theta)$ (if X discrete), or from density $f_X(x;\Theta=\theta)$ (if X continuous), where $\Theta$ is the random variable representing the
parameter (or vector of parameters). We define the Maximum A Posteriori (MAP) estimate $\hat{\theta}_{MAP}$ of $\theta$ as 
\begin{align*}
\hat{\theta}_{MAP}=\arg\max_{\theta}\pi_{\Theta}(\theta|x) =\arg\max_{\theta} \pi_{\Theta}(\theta)L_{\Theta}(\theta|x).
\end{align*}
That is, it's exactly the same as maximum likelihood, except instead of just maximizing the likelihood, we are maximizing the likelihood multiplied by the prior.
\end{definition}

\begin{example}[uniform prior distribution and $\hat{\theta}_{MLE}=\hat{\theta}_{MAP}$]
Question: consider a realization $\bold{x}=(0,0,1,1,0)$ of an I.I.D. random sample of size 5 from $X\sim Ber(\Theta)$, where $\Theta \in \{0.2,0.5,0.7\}$. Compute the (a.) $\hat{\theta}_{MLE}$ and $\hat{\theta}_{MAP}$ with respect to prior distribution $\pi_{\Theta}(\theta)=1/3$ for $\theta \in \{0.2,0.5,0.7\}$.

\noindent Solution: it is given a realization $\bold{x}=(0,0,1,1,0)$ of an i.i.d. sample of size 5 from $X\sim Ber(\theta)$, then we have
\begin{align*}
L(\theta|\bold{x})=P(\bold{x}|\theta)=\Pi_{i=1}^5p_X(x_i|\theta)=\Pi_{i=1}^5\theta^{x_i}(1-\theta)^{1-x_i}=\theta^2(1-\theta)^3.
\end{align*}
Then,
\begin{align*}
L(0.2|x)& =P(x|0.2)=(0.2)^2\cdot (0.8)^3=0.02048, \\
L(0.5|x)& =P(x|0.5)=(0.5)^2\cdot (0.5)^3=0.03125, \\
L(0.7|x)& =P(x|0.7)=(0.7)^2\cdot (0.3)^3=0.01323.
\end{align*}
So, $\hat{\theta}_{MLE}=\arg\max_{\theta \in \{0.2, 0.5, 0.7\}}L(\theta|x)=0.5$. Furthermore,
\begin{align*}
\pi_{\Theta}(0.2)L(0.2|x) &=\frac{1}{3}\cdot (0.2)^2 \cdot(0.8)^3 \approx 0.00683, \\
\pi_{\Theta}(0.5)L(0.5|x) &=\frac{1}{3}\cdot (0.5)^2\cdot(0.5)^3 \approx 0.01042, \\ 
\pi_{\Theta}(0.7)L(0.7|x) &=\frac{1}{3}\cdot (0.7)\cdot (0.3)^3 =0.00441.
\end{align*}
So, $\hat{\theta}_{MAP}=\arg\max_{\theta \in \{0.2, 0.5, 0.7\}}\pi_{\Theta}(\theta)L(\theta|x)=0.5$. i.e. with a uniform prior distribution of $\Theta$, we have
\begin{align*}
\hat{\theta}_{MLE}=0.5=\hat{\theta}_{MAP}.
\end{align*}
\end{example}

\begin{example}[Non-uniform prior distribution and $\hat{\theta}_{MLE} \neq \hat{\theta}_{MAP}$]
Question: consider a realization $\bold{x}=(x_1=0,x_2=0,x_3=1,x_4=0)$ of an I.I.D. random sample of size 4 from $X\sim Ber(\Theta)$, where $\Theta \in \{0.2,0.5,0.7\}$ with a prior distribution $p_{\Theta}(0.2)=0.10, p_{\Theta}(0.5)=0.01$ and $p_{\Theta}(0.7)=0.89$. Compute the (a.) $\hat{\theta}_{MLE}$ and $\hat{\theta}_{MAP}$.

\noindent Solution: it is given a realization $\bold{x}=(0,0,1,0)$ of an i.i.d. sample of size 4 from $X\sim Ber(\theta)$, then we have
\begin{align*}
L(\theta|\bold{x})=P(\bold{x}|\theta)=\Pi_{i=1}^4p_X(x_i|\theta)=\Pi_{i=1}^4\theta^{x_i}(1-\theta)^{1-x_i}=\theta(1-\theta)^3.
\end{align*}
Then,
\begin{align*}
L(0.2|x)& =P(x|0.2)=(0.2)\cdot (0.8)^3=0.01024, \\
L(0.5|x)& =P(x|0.5)=(0.5)\cdot (0.5)^3=0.06250, \\
L(0.7|x)& =P(x|0.7)=(0.7)\cdot (0.3)^3=0.01890.
\end{align*}
So, $\hat{\theta}_{MLE}=\arg\max_{\theta \in \{0.2, 0.5, 0.7\}}L(\theta|x)=0.5$. Furthermore,
\begin{align*}
\pi_{\Theta}(0.2)L(0.2|x) &=0.10\cdot (0.2)\cdot(0.8)^3 = 0.001024, \\
\pi_{\Theta}(0.5)L(0.5|x) &=0.01\cdot (0.5)\cdot(0.5)^3 \approx 0.000625, \\ 
\pi_{\Theta}(0.7)L(0.7|x) &=0.89\cdot (0.7)\cdot (0.3)^3 =0.16821.
\end{align*}
So, $\hat{\theta}_{MAP}=\arg\max_{\theta \in \{0.2, 0.5, 0.7\}}\pi_{\Theta}(\theta)L(\theta|x)=0.7$. i.e. with a non-uniform prior distribution of $\Theta$, we have
\begin{align*}
\hat{\theta}_{MLE}=0.5 \neq 0.7=\hat{\theta}_{MAP}.
\end{align*}
\end{example}

\subsection{Goodness of fit}
In this section, we have discussed various estimation techniques to find an estimate of a unknown parameter that leads to an estimator or sample statistic of the corresponding parameter and runs over all possible realizations/observations of an I.I.D. random sample of size n. Now, we need to find out which is relatively best i.e. how can we determine which estimator is better (rather than the technique)? There are even more different ways to estimate besides MLE/MoM/MAP, and in different scenarios, different techniques may work better. Here, we will consider some properties of estimators that allow us to compare their goodness.
\subsubsection{Bias}
Once, we come up with an estimator of a parameter using some estimation technique, we need to measure the estimator whether or not in expectation equals to the true parameter and the measure is known as bias of the estimator.
\begin{definition}[Bias]
Consider a random sample of size n i.e. $(X_1,X_2,\ldots,X_n) \to$ n I.I.D. copies of a random variable X with unknown but parametrized distribution $p_X(x|\theta)$ or $f_X(x|\theta)$. Suppose $\hat{\Theta}$ be an estimator of the unknown (but fixed) parameter $\theta$ under a suitable estimation technique. Then, bias of the estimator $\hat{\Theta}$ for $\theta$ is defined as
\begin{align*}
Bias(\hat{\Theta})=E\left(\hat{\Theta} \right)-\theta.
\end{align*}
Now, we have
\begin{enumerate}
\item[a.] If $Bias(\hat{\Theta}) = 0 \equiv E\left(\hat{\Theta} \right)=\theta$, then we say $\hat{\Theta}$ is an unbiased estimator of $\theta$.
\item[b.] If $Bias(\hat{\Theta}) >0$, then $\hat{\Theta}$ typically overestimates $\theta$.
\item[c.] If $Bias(\hat{\Theta}) <0$, then $\hat{\Theta}$ typically underestimates $\theta$.
\end{enumerate}
\end{definition}

\begin{example}[Sample mean]
Backdrop: Consider a realization $\bold{x}=(x_1, ..., x_n)$ of an i.i.d. random sample of size n i.e. $(X_1,X_2,\ldots,X_n) \to$ n I.I.D. copies of a random variable X from $Pois(\theta)$, then the MLE and MoM both estimate the same given by
\begin{align*}
\hat{\theta}_{MLE}=\frac{1}{n}\sum_{i=1}^nx_i =\bar{x}= \hat{\theta}_{MoM} \implies \hat{\Theta}=\frac{1}{n}\sum_{i=1}^n X_i =\bar{X}.
\end{align*}
Question: Prove that the estimator (or the sample mean), $\hat{\Theta}=\bar{X}$, is a unbiased estimator of the $\theta$ (or the unknown mean).

\noindent Solution: We have
\begin{align*}
E\left(\hat{\Theta} \right) = E\left(\frac{1}{n}\sum_{i=1}^n X_i \right) =\frac{1}{n}\sum_{i=1}^n E(X_i)=\theta.
\end{align*}
This makes sense: the average of your samples should be on-target for the true average.
\end{example}

\begin{example}[Sample statistic for $X\sim Unif(0,\theta)$]
Backdrop: Consider a realization $\bold{x}=(x_1, ..., x_n)$ of an i.i.d. random sample of size n i.e. $(X_1,X_2,\ldots,X_n) \to$ n I.I.D. copies of a random variable X from $X\sim Unif(0,\theta)$, then the MLE and MoM both estimate the same given by
\begin{align*}
\hat{\theta}_{MLE}=x_{max}, \hat{\theta}_{MoM}=\frac{2}{n}\sum_{i=1}^nx_i=2\bar{x} \implies \hat{\Theta}_{MLE}=X_{max}, \hat{\Theta}_{MoM}=2\bar{X}.
\end{align*}
Question: Sure, $\hat{\Theta}_{MLE}$ maximizes the likelihood, so in a way $\hat{\Theta}_{MLE}$ is better than $\hat{\Theta}_{MoM}$. But, what are the biases of these estimators? Before doing any computation: do you think $\hat{\Theta}_{MLE}$ and $\hat{\Theta}_{MoM}$ are overestimates, underestimates, or unbiased?

\noindent Solution: Actually, it can be considered that $\hat{\Theta}_{MoM}$ is spot-on since the average of the samples should be close to $\frac{\theta}{2}$ and multiplying by 2 would seem to give the true $\theta$. On the other hand, $\hat{\Theta}_{MLE}$ might be a bit of an underestimate, since we probably wouldn't have $\theta$ be exactly the largest (maybe a little larger).
\begin{enumerate}
\item[a.] Computation of $Bias(\hat{\Theta}_{MLE})$ wrt $\theta$: we have
\begin{align*}
f_{\hat{\Theta}_{MLE}}(\theta)& =f_{X_{max}}(x)=\frac{d}{dx}\left(F_{X_{max}}(x)\right)=\frac{d}{dx}\left(P(X_{max}\leq x)\right)\\
                              & =\frac{d}{dx}\left(P(X_1\leq x, X_2\leq x, \ldots X_n \leq x)\right)=\frac{d}{dx}\left((F_X(x))^n\right) \\
                              &=n(F_X(x))^{n-1}f_X(x) =n\left(\frac{x}{\theta}\right)^{n-1}\left(\frac{1}{\theta}\right).
\end{align*}
Then
\begin{align*}
E\left(\hat{\Theta}_{MLE} \right)&=E(X_{max})=\int_0^{\theta}xn\left(\frac{x}{\theta}\right)^{n-1}\left(\frac{1}{\theta}\right)dx \\
                                 &=\frac{n}{\theta^n}\int_0^{\theta}x^ndx=\frac{n}{n+1}\theta.
\end{align*}
So, we have
\begin{align*}
Bias\left(E\left(\hat{\Theta}_{MLE} \right)\right)=E\left(\hat{\Theta}_{MLE} \right)-\theta = \frac{n}{n+1}\theta -\theta=\frac{-\theta}{n+1} <0.
\end{align*}
This makes sense because if we had 3 sample points from $Unif(0,1)$ for example, we would expect them at $1/4, 2/4, 3/4$, and so it would be $n/(n+1)$ as our expected max. Similarly, if we had 4 sample points, then we would expect them at $1/5, 2/5, 3/5, 4/5$, and so it would again be $n/(n+1)$ as our expected max.
\item[b.] Computation of $Bias\left(\hat{\Theta}_{MoM} \right)$ wrt $\theta$: we have
\begin{align*}
E\left(\hat{\Theta}_{MoM} \right)&=E\left(\frac{2}{n}\sum_{i=1}^nX_i \right)=\frac{2}{n}\sum_{i=1}^n E(X_i)=\frac{2}{n}\cdot \frac{\theta}{2}=\theta.
\end{align*}
\end{enumerate}
Consider a realization $\bold{x}=(x_1=1, x_2=9,x_3=2)$ of an I.I.D. sample of size 3 from $X\sim Unif(0,\theta)$, then we have the following estimates
\begin{align*}
\hat{\theta}_{MLE}=x_{max}=9, \hat{\theta}_{MoM}=\frac{2}{3}\sum_{i=1}^3x_i=\frac{2}{3}\cdot (1+9+2)=8.
\end{align*}
However, based on our realization of the sample, the MoM estimator is impossible. If the actual parameter were 8, then that means that the distribution we pulled the sample from is $Unif(0,8)$, in which case the likelihood that we get a $9$ is 0. But we did see a 9 in our sample. So, even though $\hat{\Theta}_{MoM}$ is unbiased, it still yields an impossible estimate. This just goes to show that finding the right estimator is actually quite tricky.

\noindent A good solution would be to de-bias the MLE by scaling it appropriately. If we decided to have a new estimator based on the MLE:
\begin{align*}
\hat{\Theta}=\frac{n+1}{n}\hat{\Theta}_{MLE}=\frac{n+1}{n}X_{max}.
\end{align*}
Now, we have an unbiased estimator, but now it does not maximize the likelihood any more. Actually, the MLE is what we say to be asymptotically unbiased as follows:
\begin{align*}
Bias\left(\hat{\Theta}_{MLE}\right)=E\left(\hat{\Theta}_{MLE} \right)-\theta = \frac{n}{n+1}\theta -\theta=\frac{-\theta}{n+1} \to 0 \ \mbox{as} \ n \to \infty.
\end{align*}
\end{example}

\begin{example}[A reciprocal of sample mean]
Backdrop: Consider a realization $\bold{x}=(x_1, ..., x_n)$ of an i.i.d. random sample of size n i.e. $(X_1,X_2,\ldots,X_n) \to$ n I.I.D. copies of a random variable X from $Exp(\theta)$, then the MLE and MoM both estimate the same given by
\begin{align*}
\hat{\theta}_{MLE}=\frac{1}{\frac{1}{n}\sum_{i=1}^nx_i} =\frac{1}{\bar{x}}= \hat{\theta}_{MoM} \implies \hat{\Theta}=\frac{1}{\frac{1}{n}\sum_{i=1}^n X_i} =\frac{1}{\bar{X}}.
\end{align*}
Question: Prove that the estimator, $\hat{\Theta}=\frac{1}{\bar{X}}$, is a biased estimator of the $\theta$ (or the unknown parameter).

\noindent Solution: We have
\begin{align*}
E\left(\hat{\Theta} \right) = E\left(\frac{1}{\frac{1}{n}\sum_{i=1}^n X_i} \right) \geq \frac{1}{E\left(\frac{1}{n}\sum_{i=1}^n X_i \right)}=\frac{1}{\theta}.
\end{align*}
The function $g(x_1,\ldots,x_n) =\frac{1}{\sum_{i=1}^n x_i}$ is convex (at least in the positive octant when all $x_i\geq 0$) since $\frac{1}{x}$ is a convex function. Then, Jensen's inequality says that $E(g(x_1,\ldots, x_n))\geq g(E(x_1),E(x_2),\ldots,E(x_n))$ because of convexity. So $E\left(\hat{\Theta} \right)\geq \theta$ systematically, and we typically have an overestimate.
\end{example}

\subsubsection{Variance and mean squared error}
We are often also interested in accuracy of an estimator i.e. how much an estimator varies (we would like it to be unbiased and have small variance to that it is more accurate). There is a metric namely mean squared error that captures this property of estimators and defined as below.
\begin{definition}[Mean squared error of an estimator]
Consider a random sample of size n i.e. $(X_1,X_2,\ldots,X_n) \to$ n I.I.D. copies of a random variable X with unknown but parametrized distribution $p_X(x|\theta)$ or $f_X(x|\theta)$. Suppose $\hat{\Theta}$ be an estimator of the unknown (but fixed) parameter $\theta$ under a suitable estimation technique. Then, mean squared error of the estimator $\hat{\Theta}$ for $\theta$ is defined as
\begin{align*}
MSE(\hat{\Theta})=E\left(\left(\hat{\Theta}-\theta\right)^2 \right)=E\left(\left(\hat{\Theta}-E(\hat{\Theta})+E(\hat{\Theta})-\theta\right)^2\right)
\end{align*}
We have
\begin{align*}
& MSE(\hat{\Theta}) =E\left(\left(\hat{\Theta}-\theta\right)^2 \right) = E\left(\left(\hat{\Theta}- E(\hat{\Theta})+ E(\hat{\Theta})-\theta\right)^2 \right) \\
& = E\left(\left(\hat{\Theta}- E(\hat{\Theta})\right)^2 \right) + \left(E(\hat{\Theta})-\theta\right)^2 +2E\left(\hat{\Theta}-E(\hat{\Theta})\right)\left(E(\hat{\Theta})-\theta\right) \\
& =Var\left(\hat{\Theta}\right)+Bias^2\left(\hat{\Theta}\right) +0=Var\left(\hat{\Theta}\right)+Bias^2\left(\hat{\Theta}\right).
\end{align*}
Variance-Bias trade-off: Usually, we want to minimize MSE, and these two quantities are often inversely related. That is, decreasing one leads to an increase in the other, and finding the balance will minimize the MSE i.e.a small MSE often involves a trade-off between
variance and bias. For unbiased estimators, $MSE(\hat{\Theta})=Var(\hat{\Theta})$, so no trade-off can be made.. 
\end{definition}

\begin{example}[MSE of sample mean]
Backdrop: Consider a realization $\bold{x}=(x_1, ..., x_n)$ of an i.i.d. random sample of size n i.e. $(X_1,X_2,\ldots,X_n) \to$ n I.I.D. copies of a random variable X from $Pois(\theta)$, then the MLE and MoM both estimate the same given by
\begin{align*}
\hat{\theta}_{MLE}=\frac{1}{n}\sum_{i=1}^nx_i =\bar{x}= \hat{\theta}_{MoM} \implies \hat{\Theta}=\frac{1}{n}\sum_{i=1}^n X_i =\bar{X}.
\end{align*}
Question: Compute the mean squared error of the estimator (or the sample mean), $\hat{\Theta}=\bar{X}$, with respect to the unknown parameter $\theta$ (or the unknown mean).

\noindent Solution: We have
\begin{align*}
E\left(\hat{\Theta} \right) = E\left(\frac{1}{n}\sum_{i=1}^n X_i \right) =\frac{1}{n}\sum_{i=1}^n E(X_i)=\theta \implies Bias(\hat{\Theta})=0.
\end{align*}
And
\begin{align*}
Var(\hat{\Theta})=Var(\bar{X})=\frac{Var(X)}{n}=\frac{\theta}{n}.
\end{align*}
So, mean squared error of $\hat{\Theta}$ is given by
\begin{align*}
MSE(\hat{\Theta})=E\left(\left( \right)^2 \right)=Var(\hat{\Theta}) +Bias^2(\hat{\Theta})=\frac{\theta}{n}.
\end{align*}
\end{example}

\subsubsection{Consistency of an estimator}
Consistency of an estimator means that as the sample size gets large the estimate gets closer and closer to the true value of the parameter. 
\begin{definition}[Consistency of an estimator]
Consider a random sample of size n i.e. $(X_1,X_2,\ldots,X_n) \to$ n I.I.D. copies of a random variable X with unknown but parametrized distribution $p_X(x|\theta)$ or $f_X(x|\theta)$. Suppose $\hat{\Theta}$ be an estimator of the unknown (but fixed) parameter $\theta$ under a suitable estimation technique. Then, the estimator $\hat{\Theta}$ is consistent if $\hat{\Theta}$ converges in probability to $\theta$. That is, for any $\epsilon>0$, we have
\begin{align*}
\lim_{n\to \infty}P\left(|\hat{\Theta}-\theta|>\epsilon \right)=0 \equiv \lim_{n\to \infty}P\left(|\hat{\Theta}-\theta|\leq \epsilon \right)=1.
\end{align*}
Also, the estimator $\hat{\Theta}$ is consistent in mean squared error i.e. $\hat{\Theta}$ converges in mean squared error to $\theta$ if 
\begin{align*}
MSE(\hat{\Theta})=Var(\hat{\Theta}) +Bias^2(\hat{\Theta} \to 0 \ \mbox{as} \ n\to \infty.
\end{align*}
\end{definition}

\begin{example}[Sample statistic for $X\sim Unif(0,\theta)$]
Backdrop: Consider a realization $\bold{x}=(x_1, ..., x_n)$ of an i.i.d. random sample of size n i.e. $(X_1,X_2,\ldots,X_n) \to$ n I.I.D. copies of a random variable X from $X\sim Unif(0,\theta)$, then the MLE and MoM both estimate the same given by
\begin{align*}
\hat{\theta}_{MLE}=x_{max}, \hat{\theta}_{MoM}=\frac{2}{n}\sum_{i=1}^nx_i=2\bar{x} \implies \hat{\Theta}_{MLE}=X_{max}, \hat{\Theta}_{MoM}=2\bar{X}.
\end{align*}
Question: Prove that $\hat{\Theta}$ is a consistent estimator of $\theta$.

\noindent Answer: 
\begin{enumerate}
\item[a.] For unbiased estimator $\hat{\Theta}_{MoM}$ and $\epsilon>0$ we have
\begin{align*}
P\left(|\hat{\Theta}_{MoM}-\theta|>\epsilon \right)&=P\left(|\hat{\Theta}_{MoM}-E(\hat{\Theta}_{MoM})|>\epsilon \right) \\
                                                   &\leq \frac{Var(\hat{\Theta}_{MoM})}{\epsilon^2} =\frac{Var(2\bar{X})}{\epsilon^2} \\
                                                   & =\frac{4Var(X)}{n\epsilon^2} \to 0 \ \mbox{as} \ n\to \infty.
\end{align*}
\item[b.] For biased estimator $\hat{\Theta}_{MLE}$, we cannot use Chebyshev's inequality unfortunately. For computation of $Bias(\hat{\Theta}_{MLE})$ wrt $\theta$: we have
\begin{align*}
F_{\hat{\Theta}_{MLE}}(\theta)& =F_{X_{max}}(x)=P(X_{max}\leq x)=P(X_1\leq x, X_2\leq x, \ldots X_n \leq x) \\
                              & =(F_X(x))^n =\left(\frac{x}{\theta}\right)^n.
\end{align*}
Now, we have
\begin{align*}
& P(|\hat{\Theta}_{MoM}-\theta|>\epsilon) =P(\hat{\Theta}_{MoM} <\theta -\epsilon) + P(\hat{\Theta}_{MoM} >\theta +\epsilon)= \\                   &P(\hat{\Theta}_{MoM} <\theta -\epsilon)=\left\lbrace\begin{array}{cc}
\left(\frac{\theta -\epsilon}{\theta}\right)^n & \epsilon < \theta \ \mbox{i.e.} \ \theta -\epsilon >0 \\
0 & \epsilon \geq \theta \ \mbox{i.e.} \ \theta -\epsilon \leq 0.
\end{array} \right. \to 0 \ \mbox{as} \ n\to \infty.
\end{align*}
\end{enumerate}
Now we've seen that, even though the MLE and MoM estimators of $\theta$ for an i.i.d. sample of size n from $Unif(0,\theta)$ are different, they are both consistent. Furthermore for $Unif(0,\theta)$
\begin{enumerate}
\item[a.] $\hat{\Theta}_{MoM}$ is an unbiased and consistent estimator of $\theta$ for an I.I.D. sample of size n from $Unif(0,\theta)$ meaning it is correct in expectation, also converges to the true parameter (consistent) since the variance goes to 0.
\item[b.] However, if we ignore all the sample points and just take the first one and multiply it by 2, $\hat{\Theta}=2X_1$, it is
unbiased (as $E(\hat{\Theta})=E(2X_1)=2E(X_1)=2\cdot \frac{\theta}{2}=\theta$), but it's not consistent i.e. the estimator doesn't get better and better with more n because we're not using all n sample points. Consistency requires that as we get more samples, we approach
the true parameter.
\item[c.] Biased but consistent, on the other hand, we established that the MLE estimator $\hat{\Theta}_{MLE}$ (with expectation $E(\hat{\Theta}_{MLE})=\frac{n}{n+1}\theta \to \theta$ as $n\to \infty$) is biased but consistent.
\item[d.] $\hat{\Theta}=\frac{1}{\bar{X}^2}$ is neither unbiased nor consistent would just be some random expression.
\end{enumerate}
\end{example}

\subsubsection{Efficiency}
Efficiency says that the estimator has as low variance as possible. This property combined with consistency and unbiasedness make our estimator is on target (unbiased).
\begin{definition}
Consider a realization $\bold{x}=(x_1,x_2,\ldots,x_n)$ of an I.I.D. random sample of size n i.e. $(X_1,X_2,\ldots,X_n) \to$ n I.I.D. copies of a random variable X with unknown but parametrized distribution $p_X(x|\theta)$ or $f_X(x|\theta)$. Then, define the following
\begin{enumerate}
\item[a.] Score function:
\begin{align*}
s(\bold{x})=\frac{\partial}{\partial \theta}\left(\ln(f_X(x|\theta))\right)=\frac{f_X^{'}(x|\theta)}{f_X(x|\theta)}.
\end{align*}
Then, expectation of score statistic is given by
\begin{align*}
E(S(\theta))&=E\left(\frac{f_X^{'}(x|\theta)}{f_X(x|\theta)} \right)=\int \frac{f_X^{'}(x|\theta)}{f_X(x|\theta)}f_X(x|\theta^*)dx \\
            &\approx \frac{\partial}{\partial \theta}\int f_X(x|\theta^*)dx=\frac{\partial}{\partial \theta}(1)=0.
\end{align*}
\item[b.] Fisher information matrix:
\begin{align*}
& I(\theta)=E\left((S(\theta)-E(S(\theta)))(S(\theta)-E(S(\theta))^{\top} \right)=E\left(S(\theta)S^{\top}(\theta) \right)= \\
&nE\left(\left(\frac{\partial}{\partial}\left(\ln(L(\theta|\bold{x}) \right)\right)^2 \right)=-E\left(\frac{\partial^2}{\partial \theta^2}\left(\ln(L(\theta|\bold{x})) \right) \right) =-E\left(H_{\ln(L(\theta|\bold{x})} \right).
\end{align*}
\end{enumerate}
\end{definition}

\begin{definition}[CRLB]
Consider a realization $\bold{x}=(x_1,\ldots, x_n)$ of an i.i.d. random sample of size n i.e. $(X_1,X_2,\ldots,X_n)$ from $X\sim p_X(x|\theta)$ or $f_X(x|\theta)$, where $\theta$ is a parameter (or vector of parameters). If $\hat{\Theta}$ is an unbiased estimator for $\theta$, then
\begin{align*}
MSE(\hat{\Theta})=Var(\hat{\Theta})\geq \frac{1}{I(\theta)}.
\end{align*}
where $I(\theta)$ is the Fisher information. What this is saying is, for any unbiased estimator $\hat{\Theta}$ for $\theta$, the variance $Var(\hat{\Theta})=MSE(\hat{\Theta})$ is at least $\frac{1}{I(\theta)}$. If we achieve this lower bound, meaning our variance
is exactly equal to $\frac{1}{I(\theta)}$, then we have the best variance possible for our estimate. That is, we have the minimum variance unbiased estimator (MVUE) for $\theta$.

\noindent If If $\hat{\Theta}$ is an unbiased estimator for $\theta$, then
\begin{align*}
MSE(\hat{\Theta})=Var(\hat{\Theta}) +Bias^2(\Theta)\geq \frac{1}{I(\theta)}.
\end{align*}
\end{definition}
If $\hat{\Theta}$ is any unbiased estimator for $\theta$, there is a minimum possible variance (variance = MSE for unbiased estimators) and the estimator achieves this lowest possible variance, it is said to be efficient.

\begin{definition}[Efficiency]
Consider a realization $\bold{x}=(x_1,\ldots, x_n)$ of an i.i.d. random sample of size n i.e. $(X_1,X_2,\ldots,X_n)$ from $X\sim p_X(x|\theta)$ or $f_X(x|\theta)$, where $\theta$ is a parameter (or vector of parameters). If $\hat{\Theta}$ is an unbiased estimator for $\theta$, then its efficiency is defined as
\begin{align*}
e(\hat{\Theta})=\frac{I^{-1}(\theta)}{Var(\hat{\Theta})}\leq 1.
\end{align*}
This will always be between $0$ and $1$ because if our variance is equal to the CRLB, then it equals $1$, and anything greater will result in a smaller value. A larger variance will result in a smaller efficiency, and we want our efficiency to be as high as possible.
An unbiased estimator is said to be efficient if it achieves the CRLB - meaning $e(\hat{\Theta})=1$. That is, it could not possibly have a lower variance. Again, the CRLB is not guaranteed for biased estimators.

\noindent Let $\hat{\Theta}_1$ and $\hat{\Theta}_2$ be two unbiased estimators of a unknown parameter $\theta$ with equal sample sizes. Then,  $\hat{\Theta}_1$ is a more efficient estimator than  $\hat{\Theta}_2$ if $Var(\hat{\Theta}_1) < Var(\hat{\Theta}_2)$ i.e. $\epsilon(\hat{\Theta}_1)>\epsilon(\hat{\Theta}_2)$.
\end{definition}

\begin{example}[Efficiency of sample mean]
Backdrop: Consider a realization $\bold{x}=(x_1, ..., x_n)$ of an i.i.d. random sample of size n i.e. $(X_1,X_2,\ldots,X_n) \to$ n I.I.D. copies of a random variable X from $Pois(\theta)$, then the MLE and MoM both estimate the same given by
\begin{align*}
\hat{\theta}_{MLE}=\frac{1}{n}\sum_{i=1}^nx_i =\bar{x}= \hat{\theta}_{MoM} \implies \hat{\Theta}=\frac{1}{n}\sum_{i=1}^n X_i =\bar{X}.
\end{align*}
Question: Prove that the estimator (or the sample mean), $\hat{\Theta}=\bar{X}$, with respect to the unknown parameter $\theta$ (or the unknown mean) is efficient or not.

\noindent Solution: First, we need to check that it's unbiased, as the CRLB only holds for unbiased estimators.
\begin{align*}
E\left(\hat{\Theta} \right) = E\left(\frac{1}{n}\sum_{i=1}^n X_i \right) =\frac{1}{n}\sum_{i=1}^n E(X_i)=\theta \implies Bias(\hat{\Theta})=0.
\end{align*}
Now, we compute the variance of $\hat{\Theta}$:
\begin{align*}
Var(\hat{\Theta})=Var(\bar{X})=\frac{Var(X)}{n}=\frac{\theta}{n}.
\end{align*}
Then, we need to define likelihood function $L(\theta|x)$ to compute Fisher Information that gives us the CRLB, and see if our
variance matches.
\begin{align*}
& L(\theta|\bold{x})=P(\bold{x}|\theta)=\Pi_{i=1}^np_X(x_i|\theta)=\Pi_{i=1}^n e^{-\theta}\frac{\theta^{x_i}}{(x_i)!} \implies \\
& \ln(L(\theta|\bold{x}))=\sum_{i=1}^n\left(-\theta +x_i\ln(\theta)-\ln((x_i)!) \right) \implies \frac{\partial }{\partial \theta} \left(\ln(L(\theta|\bold{x})) \right) \\
& =\sum_{i=1}^n\left(-1 +\frac{x_i}{\theta} \right) \implies \frac{\partial^2 }{\partial \theta^2} \left(\ln(L(\theta|\bold{x})) \right)
=\sum_{i=1}^n\left(-\frac{x_i}{\theta^2}\right)=\frac{-n\bar{x}}{\theta^2}.
\end{align*}
Then, in order to compute information matrix $I(\theta)$, we compute the expectation of second derivative of log-likelihood function as
\begin{align*}
E\left(\frac{\partial^2 }{\partial \theta^2} \left(\ln(L(\theta|\bold{x})) \right) \right)=E\left(\sum_{i=1}^n -\frac{x_i}{\theta^2} \right) =\frac{-\sum_{i=1}^nE(X_i)}{\theta^2}=\frac{-n\theta}{\theta^2}=\frac{-n}{\theta}.
\end{align*}
So, Fisher Information is the negative expected value of the second derivative of the log-likelihood, so, we have
\begin{align*}
I(\theta)=-E\left( \frac{\partial^2 }{\partial \theta^2} \left(\ln(L(\theta|\bold{x})) \right) \right)=\frac{n}{\theta}.
\end{align*}
Finally, efficiency  of the estimator $\hat{\Theta}$ as the inverse of the Fisher Information over the variance computed as
\begin{align*}
e(\hat{\Theta})=\frac{I^{-1}(\theta)}{Var(\hat{\Theta}}=1.
\end{align*}
Thus, we've shown that, since the efficiency is 1, the estimator is efficient. That is, it has the best possible variance among all unbiased estimators of $\theta$. 
\end{example}

\subsubsection{Sufficiency}
Consider a realization $\bold{x}=(x_1,x_2,\ldots,x_n)$ of an i.i.d. sample of size n from a known distribution model with unknown parameter $\theta$. Imagine we have two people:
\begin{enumerate}
\item[a.] Statistician A: Knows the entire sample, gets n quantities: $\bold{x}=(x_1,\ldots, x_n)$.
\item[b.] Statistician B: Knows $T(x_1,\ldots,x_n)=t$, a single number which is a function of the realizations of an I.I.D. sample of size n. For example, the sum or the maximum of the sample points. 
\end{enumerate}
Heuristically, $T(x_1,\ldots, x_n)$ is a sufficient statistic if Statistician B can do just as good a job as Statistician A, given information. For example, if the samples are from the Bernoulli distribution, knowing $T(x_1,\ldots, x_n) =\sum_{i=1}^nx_i$ (the number of heads) is just as good as knowing all the individual outcomes, since a good estimate would be the number of heads over the number of total trials. Hence, we don't actually care the ORDER of the outcomes, just how many heads occurred. The word sufficient in English roughly means enough, and so this terminology is well-chosen.

\begin{definition}[Statistic]
Consider a realization $\bold{x}=(x_1,x_2,\ldots,x_n)$ of an i.i.d. sample of size n i.e. $(X_1,X_2,\ldots,X_n)\to$ n I.I.D. copies of X having a known distribution model with unknown parameter $\theta$. Then, we define a statistic as a function $T:\mathbb{R}^n\to\mathbb{R}$ of realizations $\bold{x}=(x_1,x_2,\ldots,x_n)$ of the I.I.D. sample of size n. For examples:
\begin{enumerate}
\item[a.] $T(x_1,\ldots,x_n)=\sum_{i=1}^nx_i \to$ the sum.
\item[b.] $T(x_1,\ldots,x_n)=\frac{1}{n}\sum_{i=1}^nx_i \to$ the sample mean.
\item[c.] $T(x_1,\ldots,x_n)=max_{i=1,2,\ldots,n}x_i=x_{max} \to$ the max/largest value.
\item[d.] $T(x_1,\ldots,x_n)=x_1 \to$ just take the first sample point.
\item[e.] $T(x_1,\ldots,x_n)=7 \to$ ignore all the sample points.
\end{enumerate}
All estimators are statistics because they take in our n data points and produce a single number.
\end{definition}

\begin{definition}[Sufficient statistic]
A statistic $T=T(X_1,\ldots,X_n)$ is a sufficient statistic if the conditional distribution of $(X_1,\ldots,X_n)$ given $T=t$ and $\theta$ does not depend on $\theta$ i.e.
\begin{align*}
P(X_1=x_1,\ldots,X_n=x_n|T=t,\theta)=P(X_1=x_1,\ldots,X_n=x_n|T=t).
\end{align*}
\end{definition}

\begin{theorem}
Consider a realization $\bold{x}=(x_1,\ldots,x_n)$ of an i.i.d. random sample of size n with likelihood $L(\theta|\bold{x})$. A statistic $T= T(x_1,\ldots, x_n)$ is sufficient if and only if there exist non-negative functions g and h such that:
\begin{align*}
L(\theta|\bold{x})=g(\bold{x})\cdot h(T(\bold{x}),\theta)
\end{align*}
That is, the likelihood of the data can be split into a product of two terms: the first term g can depend on the entire data, but not $\theta$, and the second term h can depend on $\theta$, but only on the data through the sufficient statistic T. (In other words, T is the only thing that allows the data $x_1,\ldots,x_n$ and $\theta$ to interact.) That is, we don't have access to the n individual quantities $x_1,\ldots,x_n$ just the single number (T, the sufficient statistic).

\noindent But basically, we want to split the likelihood into a product of two terms/functions:
\begin{enumerate}
\item[a.] For the first term g, you are allowed to know each individual sample point if you want, but NOT $\theta$.
\item[b.] For the second term h, you can only know the sufficient statistic (single number) $T(x_1,\ldots,x_n)$ and $\theta$. We may not know each individual $x_i$.
\end{enumerate}
\end{theorem}

\begin{example}[Sample statistic for $X\sim Unif(0,\theta)$]
Backdrop: Consider a realization $\bold{x}=(x_1, ..., x_n)$ of an i.i.d. random sample of size n i.e. $(X_1,X_2,\ldots,X_n) \to$ n I.I.D. copies of a random variable X from $X\sim Unif(0,\theta)$, then the MLE and MoM both estimate the same given by
\begin{align*}
\hat{\theta}_{MLE}=x_{max}, \hat{\theta}_{MoM}=\frac{2}{n}\sum_{i=1}^nx_i=2\bar{x} \implies \hat{\Theta}_{MLE}=X_{max}, \hat{\Theta}_{MoM}=2\bar{X}.
\end{align*}
Question: Prove that $\hat{\Theta}_{LME}=X_{max}$ is a sufficient estimator of $\theta$.

\noindent Answer: We have the likelihood function $L(\theta|\bold{x})$ as:
\begin{align*}
L(\theta|\bold{x})&=\Pi_{i=1}^n\frac{1}{\theta}I_{\{0\leq x_i\leq \theta\}}=\frac{1}{\theta^n}I_{\{0\leq x_1,\ldots,x_n\leq \theta\}} =\frac{1}{\theta^n}I_{\{0\leq x_{max}\leq \theta\}} \\ &=\frac{1}{\theta^n}I_{\{T(x_1,\ldots,x_n)\leq \theta\}}.
\end{align*}
Now, we choose
\begin{align*}
g(x_1,\ldots,x_n)=1, \ h(T(x_1,\ldots,x_n),\theta)=\frac{1}{\theta^n}I_{\{T(x_1,\ldots,x_n)\leq \theta\}}
\end{align*}
By the Neyman-Fisher Factorization Criterion, $\hat{\Theta}_{MLE}=X_{max}$ is a sufficient estimator. For the h term, notice that we just need to know the max of the samples $T(x_1,\ldots,x_n)$ to compute h: we don't actually need to know each individual $x_i$. It is noted that here the only interaction between the data and parameter $\theta$ happens through the sufficient statistic (the max of all the values).
\end{example}

\begin{example}[Sample statistic for $X\sim Pois(\theta)$]
Backdrop: Consider a realization $\bold{x}=(x_1, ..., x_n)$ of an i.i.d. random sample of size n i.e. $(X_1,X_2,\ldots,X_n) \to$ n I.I.D. copies of a random variable X from $X\sim Pois(\theta)$, then the we have
\begin{align*}
& T(x_1,\ldots,x_n)=\sum_{i=1}^nx_i, \ \hat{\theta}_{MLE}=\frac{1}{n}\sum_{i=1}^nx_i=\bar{x} \implies \\
& T(X_1,\ldots,X_n)=S_n, \hat{\Theta}_{MLE}=\bar{X}.
\end{align*}
Question: Prove that $T(X_1,\ldots,X_n)=S_n$ is a sufficient statistic and hence $\hat{\Theta}_{LME}=\bar{X}$ is a sufficient estimator. (The reason this is true is because we don't need to know each individual sample to have a good estimate for $\theta$, we just need to know how many events happened total.)

\noindent Answer: We have the likelihood function $L(\theta|\bold{x})$ as:
\begin{align*}
L(\theta|\bold{x})&=\Pi_{i=1}^np_X(x_i|\theta)=\Pi_{i=1}^ne^{-\theta}\frac{\theta^{x_i}}{(x_i)!}=\frac{e^{-n\theta}\theta^{\sum_{i=1}^nx_i}}{\Pi_{i=1}^n(x_i)!} \\
                  &=\frac{1}{\Pi_{i=1}^n(x_i)!}\cdot e^{-n\theta}\theta^{T(x_1,\ldots,x_n)}.
\end{align*}
Now, we choose
\begin{align*}
g(x_1,\ldots,x_n)=\frac{1}{\Pi_{i=1}^n(x_i)!}, \ h(T(x_1,\ldots,x_n),\theta)=e^{-n\theta}\theta^{T(x_1,\ldots,x_n)}.
\end{align*}
By the Neyman-Fisher Factorization Criterion, $T(X_1,\ldots,X_n)=\sum_{i=1}^nX_i$ is a sufficient statistic. This implies $\hat{\Theta}_{MLE}=\frac{T(X_1,\ldots,X_n)}{n}$ is also a sufficient statistic, since knowing the total number of events and the average number of
events is equivalent (since we know n).

\noindent It is noted that here we had the g term handle some function of only $x_1,\ldots,x_n$ but not $\theta$. For the h term though, we do have $\theta$ but don't need the individual samples $x_1,\ldots,x_n$ to compute h. Imagine being just given $T(x_1,\ldots,x_n)$, now you have enough information to compute h.
\end{example}

\begin{example}[Sample statistic for $X\sim Ber(\theta)$]
Backdrop: Consider a realization $\bold{x}=(x_1, ..., x_n)$ of an i.i.d. random sample of size n i.e. $(X_1,X_2,\ldots,X_n) \to$ n I.I.D. copies of a random variable X from $X\sim Ber(\theta)$, then the we have
\begin{align*}
& T(x_1,\ldots,x_n)=\sum_{i=1}^nx_i, \ \hat{\theta}_{MLE}=\frac{1}{n}\sum_{i=1}^nx_i=\bar{x} \implies \\
& T(X_1,\ldots,X_n)=S_n, \hat{\Theta}_{MLE}=\bar{X}.
\end{align*}
Question: Prove that $T(X_1,\ldots,X_n)=S_n$ is a sufficient statistic and hence $\hat{\Theta}_{LME}=\bar{X}$ is a sufficient estimator. (The reason this is true is because we don't need to know each individual sample to have a good estimate for $\theta$, we just need to know how many events happened total.)

\noindent Answer: We have the likelihood function $L(\theta|\bold{x})$ as:
\begin{align*}
L(\theta|\bold{x})&=\Pi_{i=1}^np_X(x_i|\theta)=\Pi_{i=1}^n\theta^{x_i}(1-\theta)^{1-x_i}=\theta^{x_1+\ldots+x_n}(1-\theta)^{n-x_1-\ldots-x_n} \\
                  &=\theta^{\sum_{i=1}^nx_i}(1-\theta)^{n-\sum_{i=1}^nx_i}.
\end{align*}
Now, we choose
\begin{align*}
g(x_1,\ldots,x_n)=1, \ h(T(x_1,\ldots,x_n),\theta)=\theta^{T(x_1,\ldots,x_n)}(1-\theta)^{n-T(x_1,\ldots,x_n)}.
\end{align*}
By the Neyman-Fisher Factorization Criterion, $T(X_1,\ldots,X_n)=\sum_{i=1}^nX_i$ is a sufficient statistic. This implies $\hat{\Theta}_{MLE}=\frac{T(X_1,\ldots,X_n)}{n}$ is also a sufficient statistic, since knowing the total number of events and the average number of
events is equivalent (since we know n).

\noindent It is noted that here we had the g term handle some function of only $x_1,\ldots,x_n$ but not $\theta$. For the h term though, we do have $\theta$ but don't need the individual samples $x_1,\ldots,x_n$ to compute h. Imagine being just given $T(x_1,\ldots,x_n)$, now you have enough information to compute h. 

\noindent We can notice that here the only interaction between the data and parameter $\theta$ happens through the sufficient statistic (the sum/mean of all the values). We don't actually need to know each individual $x_i$.
\end{example}

\section{Interval estimation}
Till now, we have discussed several techniques to estimate unknown parameters and the goodness criterion of the estimator. Even if the estimator is having all the good properties, the probability that the estimator for $\theta$ is exactly correct is 0 i.e. $P(\hat{\Theta}=\theta)=0$, since $\theta$ is continuous (a decimal number). So, we need to construct confidence intervals around the estimator, containing the true value of the parameter with high probability i.e. $\hat{\theta}$ is very close to $\theta$ with high probability. Confidence intervals are used in the Frequentist setting, which means the population parameters are assumed to be unknown but will always be fixed, not random variables. Credible intervals, on the other hand, are a Bayesian version of a Frequentist's confidence interval.

\subsection{Confidence interval}

\begin{theorem}[A $100(1-\alpha)\%$ confidence interval]\label{thm:ci-01}
Consider a realization $\bold{x}=(x_1,\ldots,x_n)$ of an I.I.D. random sample of size n from $X\sim p_X(x|\theta)$ or $f_X(x|\theta)$. Using CLT, an $100(1-\alpha)\%$ confidence interval of an unbiased estimate $\hat{\theta}$ for $\theta$ is given by
\begin{align*}
\hat{\theta} \in \left[\theta-z_{1-\frac{\alpha}{2}}\sqrt{Var(\hat{\theta})},\theta+z_{1-\frac{\alpha}{2}}\sqrt{Var(\hat{\theta})}\right].
\end{align*}
Where $z_{1-\frac{\alpha}{2}}$ is z-score given by
\begin{align*}
z_{1-\frac{\alpha}{2}}=\phi^{-1}(a-\frac{\alpha}{2})=\left\lbrace\begin{array}{cc}
1.96  & \ \mbox{for} \ \alpha=5\%=0.05 \\
2.576 & \ \mbox{for} \ \alpha=1\%=0.01.
\end{array} \right.
\end{align*}
\end{theorem}

\begin{proof}
We have the following information:
\begin{enumerate}
\item[i.] $X\sim p_X(x|\theta)$ or $f_X(x|\theta)$, so
\begin{align*}
E(\Theta)=g_1(\theta) \ \mbox{and} \ Var(\Theta)=g_2(\theta).
\end{align*}
\item[ii.] a realization $\bold{x}=(x_1,\ldots,x_n)$ with given n, computes a value of $\hat{\theta}$. Also, $Var(\Theta)=\frac{g_1(\theta)}{n}$.
\item[iii.]In $100(1-\alpha)\%$ the $\alpha$ is given (e.g. $\alpha=5\%$ or $1\%$ etc).
\item[iv.] A technique-central limit theorem.
\end{enumerate}
To compute a $100(1-\alpha)\%$ confidence interval, we have
\begin{align*}
& P(|\hat{\theta}-\theta|>\epsilon) \leq \alpha \implies P\left(|\frac{\hat{\theta}-\theta}{\sqrt{Var(\hat{\theta})}}|>\frac{\epsilon}{\sqrt{Var(\hat{\theta})}}\right) \leq \alpha \approx \\
& P\left(|Z|> \frac{\epsilon}{\sqrt{Var(\hat{\theta})}} \right) \leq \alpha \implies 1- P\left(|Z|\leq  \frac{\epsilon}{\sqrt{Var(\hat{\theta})}} \right) \geq 1-\alpha \\
&P\left(|Z|\leq \frac{\epsilon}{\sqrt{Var(\hat{\theta})}} \right) \geq 1-\alpha \implies 2\phi\left(\frac{\epsilon}{\sqrt{Var(\hat{\theta})}} \right) \geq 2-\alpha \implies \\
& \frac{\epsilon}{\sqrt{Var(\hat{\theta})}} \geq \phi^{-1}\left(1-\frac{\alpha}{2}\right) \implies \frac{\epsilon}{\sqrt{Var(\hat{\theta})}} \geq z_{1-\frac{\alpha}{2}} \implies \\
& \epsilon \geq z_{1-\frac{\alpha}{2}}\sqrt{Var(\hat{\theta})} \implies \hat{\theta} \in \left[\theta-z_{1-\frac{\alpha}{2}}\sqrt{Var(\hat{\theta})},\theta+z_{1-\frac{\alpha}{2}}\sqrt{Var(\hat{\theta})}\right].
\end{align*}
\end{proof}

\begin{example}[Confidence interval wrt $Pois(\theta)$]
Backdrop: Consider a realization  $\bold{x}=(x_1,\ldots,x_n)$ of an i.i.d. sample of size from $X\sim Pois(\theta)$ where $\theta$ is unknown. Then, both MLE and MoM estimates agreed at the sample mean as:
\begin{align*}
\hat{\theta}=\hat{\theta}_{MLE}=\hat{\theta}_{MoM}=\frac{1}{n}\sum_{i=1}^nx_i=\bar{x}.
\end{align*} 
Question: Construct an  $95\%$ confidence interval of $\hat{\theta}$ (i.e. it contains $\theta$ with probability $95\%$).

\noindent Answer: We have the following information:
\begin{enumerate}
\item[i.] $X\sim Pois(\theta)$, so 
\begin{align*}
E(X)=\theta \ \mbox{and} \ Var(X)=\theta.
\end{align*}
\item[ii.] a realization $\bold{x}=(x_1,\ldots,x_n)$ with given n, computes a value of $\hat{\theta}=\bar{x}$. Also, $Var(\hat{\Theta})=\frac{\theta}{n}$.
\item[iii.]In $100(1-\alpha)\%$ the $\alpha=5\%=0.05$. So, z-score is given by
\begin{align*}
z_{1-\frac{\alpha}{2}}=\phi^{-1}\left(a-\frac{\alpha}{2}\right)=\phi^{-1}(0.975)=1.96.
\end{align*}
\item[iv.] A technique-central limit theorem.
\end{enumerate}
Now, from theorem \eqref{thm:ci-01}, we have
\begin{align*}
& \hat{\theta} \in \left[\theta-1.96\sqrt{\frac{\theta}{n}},\theta+1.96\sqrt{\frac{\theta}{n}}\right] \equiv  \theta \in \left[\hat{\theta}-1.96\sqrt{\frac{\hat{\theta}}{n}},\hat{\theta}+1.96\sqrt{\frac{\hat{\theta}}{n}}\right] \equiv \\
&\theta \in \left[\bar{x}-1.96\sqrt{\frac{\bar{x}}{n}},\bar{x}+1.96\sqrt{\frac{\bar{x}}{n}}\right].
\end{align*}

\end{example}

%

\begin{example}
Backdrop: consider a realization $\bold{x}=(x_1,\ldots,x_n)$ of an I.I.D. sample of size n from $X\sim Ber(\theta)$ with the unknown $
\theta$ as probability of success.

\noindent Question: given $n=400$ and $\sum_{i=1}^nx_i=136$ (observed 136 successes out of 400), construct a $99\%$ confidence interval for $\theta$ (the unknown probability of success).

\noindent Answer: We have the following information:
\begin{enumerate}
\item[i.] $X\sim Ber(\theta)$, so 
\begin{align*}
E(X)=\theta \ \mbox{and} \ Var(X)=\theta(1-\theta).
\end{align*}
\item[ii.] a realization $\bold{x}=(x_1,\ldots,x_n)$ such that $\sum_{i=1}^nx_i=136$ with given $n=400$, computes a value of $\hat{\theta}=\bar{x}=\frac{136}{400}=0.34$. Also, $Var(\hat{\Theta})=\frac{\theta(1-\theta)}{n}\leq \frac{1}{4n}$.
\item[iii.]In $100(1-\alpha)\%$ the $\alpha=1\%=0.01$. So, z-score is given by
\begin{align*}
z_{1-\frac{\alpha}{2}}=\phi^{-1}(a-\frac{\alpha}{2})=\phi^{-1}(0.995)=2.576.
\end{align*}
\item[iv.] A technique-central limit theorem.
\end{enumerate}
Now, from theorem \eqref{thm:ci-01}, we have
\begin{align*}
& \hat{\theta} \in \left[\theta-2.576\sqrt{\frac{\theta}{n}},\theta+2.576\sqrt{\frac{\theta}{n}}\right] \equiv  \theta \in \left[\hat{\theta}-2.576\sqrt{\frac{\hat{\theta}(1-\hat{\theta})}{n}}, \hat{\theta} + \right. \\
& \left. 2.576\sqrt{\frac{\hat{\theta}(1-\hat{\theta})}{n}}\right] = \left[\bar{x}-2.576\sqrt{\frac{\bar{x}(1-\bar{x})}{n}},\bar{x}+2.576\sqrt{\frac{\bar{x}(1-\bar{x})}{n}}\right] =\\
& \left[0.34-2.576\cdot\sqrt{\frac{0.34\cdot 0.66}{400}}, 0.34 +2.576\cdot\sqrt{\frac{0.34\cdot 0.66}{400}} \right]=[0.28,0.40].
\end{align*}

\noindent Approach-02: Central limit theorem and an optimal bound on the variance 
\begin{align*}
Var(\hat{\Theta})=\frac{\theta(1-\theta)}{n}\leq \frac{1}{4n}.
\end{align*}
We have
\begin{align*}
0.005 &\geq P(\hat{\Theta}-\theta>\epsilon)=\frac{1}{2}P(|\hat{\Theta}-\theta|>\epsilon) \\
      &=\frac{1}{2}P\left(\left|\frac{\hat{\Theta}-\theta}{\sqrt{\theta(1-\theta)/400}} \right| >\frac{\epsilon}{\sqrt{\theta(1-\theta)/400}}\right) \\
      & \approx \frac{1}{2}P\left(|Z|>\frac{\epsilon}{\sqrt{\theta(1-\theta)/400}}\right)=P\left(Z>\frac{\epsilon}{\sqrt{\theta(1-\theta)/400}}\right)\\
      &=1-P\left(Z\leq\frac{\epsilon}{\sqrt{\theta(1-\theta)/400}}\right)=1-\phi\left(\frac{\epsilon}{\sqrt{\theta(1-\theta)/400}} \right).    
\end{align*}
So, we get
\begin{align*}
& \phi\left(\frac{\epsilon}{\sqrt{\theta(1-\theta)/400}} \right) \geq 1-0.005=0.995=\phi(2.576)=\phi(z_{0.995}) \implies \\
& \frac{\epsilon}{\sqrt{\theta(1-\theta)/400}} \geq 2.576 \ \mbox{for any} \ \theta \in (0,1) \implies 20\epsilon\cdot 4 \geq 2.576 \implies  \\
& \epsilon \geq \frac{2.576}{20}\cdot \frac{1}{4}=\frac{2.576}{20}\cdot\frac{1}{4}=0.032 \implies  \mbox{for a realization} \ \hat{\theta} =0.34, \\
& \hat{\Theta}  \in [\theta-0.032,\theta+0.032] \equiv \theta \in [0.307,0.372]\approx [0.31,0.37]. 
\end{align*}
If we repeat this process several times (getting $n=400$ sample points each time and constructing different confidence intervals), about $99\%$ of the confidence intervals we construct will contain $\theta$.
\end{example}

\begin{example}[Finding $n$ with a given confidence interval]
Backdrop: Assume a realization/people $\bold{x}=(x_1,\ldots,x_n)$ (of an I.I.D random sample of size $n$ from $X\sim Ber(p)$) vote party $A$ with probability $p$. We poll $n$ people. Let $M_n$ be the fraction/average of the polled persons for Party $A$ and defined as
\begin{align*}
M_n=\frac{1}{n}\sum_{i=1}^n X_i.
\end{align*}

\noindent Find n such that estimator $M_n$ to be within $\pm 1\%$ of the true value p with probability at least $95\%$, statistically 
\begin{align*}
P(|M_n-p|>0.01)\leq 0.05.
\end{align*}
And geometrically, can be visualized as in the figure \eqref{fig:bstpm-01}:
\begin{center}
\begin{figure}[h]
    \centering
    \includegraphics[width=0.75\textwidth]{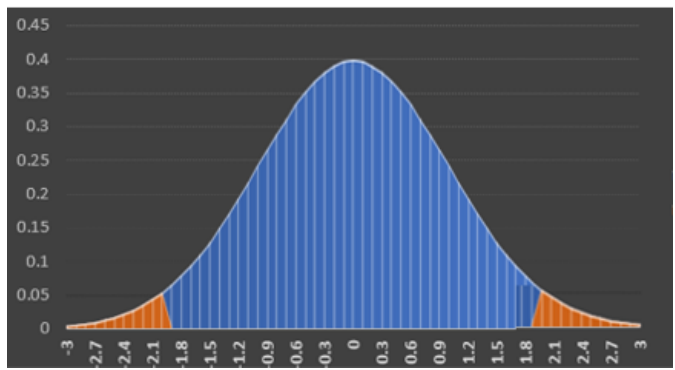}
    \caption{The distribution $M_n -p$ is symmetric, and we are looking for tail bounds.}
    \label{fig:bstpb-01}
\end{figure}
\end{center}

\noindent Answer: Since opinion of ith polled person for party A is $X_i\sim Ber(p)$ and the distribution of $M_n=\frac{1}{n}\sum_{i=1}^nX_i$ is symmetric around its mean $E(M_n)=p$ with $Var(M_n)=\frac{p(1-p)}{n}$, then, we have
\begin{align*}
0.025 &\geq P(M_n-p>0.01)=\frac{1}{2}P(|M_n-p|>0.01) \\
      &=\frac{1}{2}P\left(\left|\frac{M_n-p}{\sqrt{p(1-p)/n}} \right| >\frac{0.01}{\sqrt{p(1-p)/n}}\right) \\
      & \approx \frac{1}{2}P\left(|Z|>\frac{0.01}{\sqrt{p(1-p)/n}}\right)=P\left(Z>\frac{0.01}{\sqrt{p(1-p)/n}}\right)\\
      &=1-P\left(Z\leq\frac{0.01}{\sqrt{p(1-p)/n}}\right)=1-\phi\left(\frac{0.01}{\sqrt{p(1-p)/n}} \right)  \\
\implies & \phi\left(\frac{0.01}{\sqrt{p(1-p)/n}} \right) \geq 1-0.025=0.975=\phi(1.96)=\phi(z_{1-0.025}) \\
\implies & \frac{0.01}{\sqrt{p(1-p)/n}} \geq 1.96 \ \mbox{for any} \ p \in (0,1) \implies 0.01\sqrt{n}\cdot 4 \geq 1.96 \\
\implies & n\geq \left(\frac{1.96}{0.01}\cdot \frac{1}{4}\right)^2=9604.    
\end{align*}
In this calculation we used the approximation given by the CLT, but how good is it? If we work out the exact computation, using the fact that $M_n$ is a scaled binomial and the worse-case parameter $p=1/2$, then we get that $n=9604$ will give us a $95.1\%$ probability of being within $0.01$ of the true value p: in this case, the CLT’s approximation is very good. So if the CLT is so much better, why use Chebyshev’s inequality at all? The reason is that Chebyshev is a bound that works for all $\epsilon$ and probabilities, is always a valid upper bound but is very conservative. The CLT is often much sharper, however it only works for $\epsilon$ of appropriate scale and does not give a rock-solid bound as it is an approximation and not a bound. Is it possible to strengthen Chebyshev’s inequality to get a bound nearly as sharp as the CLT’s, under appropriate assumptions? Yes, the Chernoff inequality provides the optimal approximation.
\end{example}

\section{Hypothesis testing}
\begin{example}[A victim is not guilty]
Consider a victim is in trial for a crime in the court. Then the Judge will have following hypotheses to determine whether or not the victim is guilty (a given statement happens to be source of an alternative hypothesis) as:
\begin{align*}
H_0: & \ \mbox{the victim is not guilty} \\
H_a: & \ \mbox{the victim is guilty}.
\end{align*}
Based on evidence/observation/realization/data and its statistical analysis, the Judge come up with a conclusion that either reject $H_0$ or fail to reject $H_0$.
\end{example}

\begin{example}[Magician's fair coin]
An statistical method of hypothesis testing to prove statistically a claim by selecting a suitable hypothesis based on data. Suppose, we have a Magician Sarkar, who says
\begin{align*}
\mbox{Magician Sarkar: I have here a fair coin.}
\end{align*}
And then an audience, an Agnostic (described in \eqref{fig:agnostic-01}) statistician named Agyeya, engages him in a conversation:
\begin{enumerate}
\item[---] Agnostic statistician Agyeya: I don't believe you. Can we examine it?
\item[---] Magician Sarkar: Be my guest.
\item[---] Agnostic Statistician Agyeya: I am going to flip your coin 100 times and see how many heads I get. (Agyeya flips the coin 100 times and sees 99 heads.) You cannot be telling the truth, there's no way this coin is fair.
\item[---]Magician Sarkar: Wait I was just unlucky, I swear I am not lying.
\end{enumerate}
So let's give Sarkar the benefit of the doubt. We compute the probability that we observed an outcome at least as extreme as this, given that Sarkar isn't lying i.e. the coin is fair, so the number of heads observed should be $X\sim Bin(100,0.5)$, because there are 100 independent trials and a $50\%$ of heads since it's fair. So, the probability that we observe at least $99$ heads (because we're looking for something as least as extreme), is the sum of the probability of 99 heads and the probability of 100 heads as:
\begin{align*}
P(X\geq 99)&=\binom{100}{99}(0.5)^{99}(0.5)^1 + \binom{100}{100}(0.5)^{100}(0.5)^0 \\
           &=\frac{101}{2^{100}}\approx 7.96\times 10^{-29} \approx 0.
\end{align*}
Basically, if the coin were fair, the probability of what we just observed (99 heads or more) is basically 0. This is strong statistical evidence that the coin is NOT fair. Our assumption was that the coin is fair, but if this were the case, observing such an extreme outcome would be extremely unlikely. Hence, our assumption is probably wrong. So, this is like a Probabilistic Proof by Contradiction.
\begin{center}
\begin{figure}[h]
    \centering
    \includegraphics[width=1.0\textwidth]{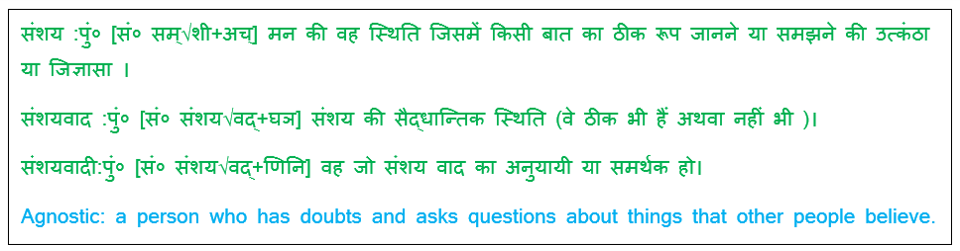}
    \caption{A self-explanatory meaning of Agnostic in Hindi with corresponding translation in English.}
    \label{fig:agnostic-01}
\end{figure}
\end{center}
\end{example}

\begin{example}[to determine a fair coin]
Consider a coin toss. Then, to determine whether the coin is fair or not (happens to be source of an alternative hypothesis), we come up with following hypotheses as:
\begin{align*}
H_0: & \ \mbox{the victim is fair i.e.} \ p=0.5 \\
H_a: & \ \mbox{the victim is not fairy i.e.} \ p\neq 0.5.
\end{align*}
Based on evidence/observation/realization/data and its statistical analysis, the we can conclude that either reject $H_0$ or fail to reject $H_0$.
\end{example}

\subsubsection{A summarized hypothesis testing procedure}
\begin{enumerate}[{Step}~1:]
\item Make a claim (e.g. like Indian food is spicy, Dirty boots belong on the table etc.)
\item Set up a null hypothesis $H_0$ and alternative hypothesis $H_a$ (based on an underlined statistical statement).
\begin{itemize}
\item Alternative hypothesis can be one-sided or two-sided.
\item The null hypothesis is usually a baseline, no effect, or benefit of the doubt.
\item The alternative is what you want to prove, and is opposite the null.
\end{itemize}
\item Choose a significance level $\alpha$ (usually $\alpha=0.05$ or $0.01$).
\item Collect data (e.g. a realization $\bold{x}=(x_1,x_2,\ldots,x_n)$ of an I.I.D. random sample of size n from $X\sim p_X(x|\theta)$ or $f_X(x|\theta)$). And come up with an estimate $\hat{\theta}$ for the unknown parameter $\theta$ (e.g. $\hat{\theta}=\bar{x}$). 
\item Compute a p-value as: 
\begin{align*}
p=P(\mbox{observing data at least as extreme as our} |H_0 \ \mbox{is true}).
\end{align*}
\item State the conclusion as:
\begin{itemize}
\item if $p<\alpha$, reject the null hypothesis $H_0$ in favour of the alternative $H_a$ (i.e. our result is statistically significant in this case).
\item if $p\geq \alpha$, fail to reject the null hypothesis $H_0$.
\end{itemize}
\end{enumerate}

\begin{example}
Backdrop: consider a realization $\bold{x}=(x_1,\ldots,x_n)$ of an I.I.D. random sample of size n from $X\sim Ber(\theta)$, where $\theta$ is the unknown probability of success.

\noindent Question: We want to determine whether or not more than $3/4$ of Indian would vote for A in 2024. In a random poll sampling $n= 137$ Indian, we collected responses $\bold{x}=(x_1,\ldots, x_n)$ (each is 1 or 0, if they would vote for him or not). We observe 131 yes responses: $\sum_{i=1}^nx_i=131$. Perform a hypothesis test and state your conclusion.

\noindent Answer:
\begin{enumerate}
\item[a.] We have 
\begin{align*}
X\sim Ber(\theta) \implies E(X)=\theta, \ Var(X)=\theta(1-\theta).
\end{align*}
\item[b.] Let p denote the true proportion of Americans that would vote for A. Then, the null and alternative hypotheses are:
\begin{align*}
& H_0:\ \mbox{Exactly $3/4$ Indian voted for party A i.e.} \ p=0.75 \\
& H_a: \ \mbox{More than $3/4$ Indian voted for party A i.e.} \ p\geq 0.75.
\end{align*}
\item[c.] Let $X_i\to$ the response of ith polled person in the poll sampling. So, we have
\begin{align*}
& \hat{\Theta}=\bar{X} \implies E(\Theta)=E(\bar{X})=\theta, \ Var(\hat{\Theta})=Var(\bar{X})=\frac{Var(X)}{n}\\
& =\frac{\theta(1-\theta)}{n}.
\end{align*}
\noindent As, we have a realization $\bold{x}=(x_1,x_2,\ldots,x_{137})$ (such that $\sum_{i=1}^{137}=131$) of an I.I.D. random sample of size $n=137$ from $X\sim Ber(x|\theta)$. So, we get
\begin{align*}
\hat{\theta}=\bar{x}=\frac{1}{n}\sum_{i=1}^n x_i=\frac{131}{137}=0.96.
\end{align*}
\item[d.] Now, test the hypothesis at the $\alpha=0.01$ significance level. Then, under the null hypothesis, we have $p=0.75$, so, the p-value (observing data at least as extreme), is given by
\begin{align*}
P(\bar{X}\geq \bar{x}) &=P\left(\bar{X}\geq \frac{131}{137}\right)=P\left(\frac{\bar{X}-0.75}{\sqrt{(0.37)^2}} \geq \frac{131/137 -0.75}{\sqrt{(0.37)^2}} \right) \\
                       & \approx P(Z\geq 5.42643) \approx 0.
\end{align*}
\item[e.] With a p-value so close to 0 (and certainly $<\alpha=0.01$, we reject the null hypothesis that (only) $75\%$ of Indian would vote for A. There is strong evidence that this proportion is actually larger.
\end{enumerate}
\noindent Again, what we did was: assume $p=0.75$ (null hypothesis), then note that the probability of observing data so extreme (in fact very close to $100\%$ of people), was nearly 0. Hence, we reject this null hypothesis because what we observed would've been so unlikely if it were true.
\end{example}

%% file: statmod.tex
\section{Statistical models}
It is fortunate, however, that a large number of estimation problems can be represented by a data model in linear framework as a linear model that allows us to easily determine this estimator. Not only is the minimum variance valued (MVU) estimator immediately evident once the linear model has been identified, but in addition, the statistical performance follows naturally. The key, then, to finding the optimal estimator is in structuring the problem in the linear model form to take advantage of its unique properties.

\subsubsection{Derivation of least mean square estimator}
\begin{theorem}[Condition expectation as an LMS estimator]
The covariance between two random variables has an important bearing on the predictability of Y based on knowledge of the outcome of $X$. So, consider $Y$ is statistically related to another random variable $X$ i.e. $Y=g(X)$ and $g$ is unknown. Then, least-mean-square estimation of $Y$ is given by the conditional expectation as
\begin{align*}
\hat{Y}=E\left(Y|X \right)=\hat{g}(X).
\end{align*}
Moreover, estimation error $Y-E(Y|X)$ and the attribute $X$ are uncorrelated as:
\begin{align*}
& E(X(Y-E(Y|X)))=E(XY)-E(XY)=0=E(X)E(Y-E(Y|X))  \\
& \implies Cov(X, Y-E(Y|X))=0.
\end{align*}
Since, we have
\begin{align*}
E(XE(Y|X))&=E\left(\int_{x\in \Omega_X}xE(Y|X=x)f_X(x)dx \right)  \\
          &=\int_{y\in \Omega_Y}\int_{x\in \Omega_X}xyf_X(x)f_{Y|X}(y|x)dxdy \\
          &=\int_{(x,y)\in \Omega_X \times \Omega_Y}xyf_{X,Y}(x,y)dxdy=E(XY). 
\end{align*}
\end{theorem}

\begin{proof}
For one observation of $X=x$, we have
\begin{align*}
MSE=E((Y-c)^2)& =E\left(\left(Y-E(Y|X=x)+E(Y|X=x)-c\right)^2\right) \\
              & =E((Y-E(Y|X=x))^2)+(E(Y|X=x)-c)^2.
\end{align*}
To find least mean square estimate, we have
\begin{align*}
\frac{d}{dc}(MSE)&=\frac{d}{dc}\left(\left(E(Y|X=x)-c\right)^2\right)=-2\left(E(Y|X=x)-c \right)=0 \\
\implies & c=E(Y|X=x)=\hat{g}(x).
\end{align*}
That is, we have a least mean square estimate of $Y$ as
\begin{align*}
\hat{Y}=\hat{y}=E(Y|X=x)=\hat{g}(x).
\end{align*}
Then, varies for every possible observation of $X$, we have the least mean square estimator of $Y$ as:
\begin{align*}
\hat{Y}=E(Y|X)=\hat{g}(X) \ \mbox{such that} \ E\left((Y-g(X))^2 \right)\geq E\left(Y-E(Y|X) \right) \ \mbox{for any} \ g.
\end{align*}
Alternatively, consider an arbitrary function g of $X$ and define mean squared error with respect to Y as
\begin{align*}
MSE(g(X))&=E((Y-g(X))^2)=E((Y-E(Y|X) + E(Y|X)-g(X))^2) \\
         &=E\left(\left(Y-E\left(Y|X\right)\right)^2\right) +E\left(\left(E\left(Y|X\right)-g(X)\right)^2\right) \\
         &=E\left(E\left(\left(Y-E\left(Y|X\right)\right)^2|X\right)\right) +E\left(\left(E\left(Y|X\right)-g(X)\right)^2\right) \\
         &=E\left(Var\left(Y|X\right)\right) + E\left(\left(E\left(Y|X\right)-g(X)\right)^2\right) \\
         &\geq E\left(\left(E\left(Y|X\right)-g(X)\right)^2\right) \ \mbox{for any other function g of} \ X.
\end{align*}
\end{proof}
\section{Linear models}
Consider $Y$ is statistically dependent on another random variable (or vector), as a variate $X$,  in a linear framework as 
\begin{align*}
E(Y|X)=\theta_0+\theta_1X.
\end{align*}

\begin{theorem}[Computation of parameters for a linear model]\label{thm:comp-par-01}
Consider $Y$ is statistically dependent on another random variable (or vector), as a variate $X$,  in a linear framework as 
\begin{align*}
E(Y|X)=\theta_0+\theta_1X.
\end{align*}
Then, the parameters can be computed as
\begin{align*}
\theta_1=\left(Var(X) \right)^{-1}Cov(X,Y), \ \theta_0=E(Y)- \left(Var(X) \right)^{-1}Cov(X,Y)E(X).
\end{align*}
So,
\begin{align*}
& \hat{Y}=E(Y|X)=E(Y)+\left(Var(X) \right)^{-1}Cov(X,Y)(X-E(X)) \equiv \\
& Y=E(Y|X)+\epsilon=E(Y)+\left(Var(X) \right)^{-1}Cov(X,Y)(X-E(X))+\epsilon.
\end{align*}
\end{theorem}

\begin{proof}
We have
\begin{align*}
& \hat{Y} =E(Y|X)=aX+b \implies Y=aX+b+\epsilon \implies E(Y)=aE(X)+b \implies \\
& E\left(E(X)E(Y|X) \right) =a\left(E(X) \right)^2 +bE(X), \ E(XY)=aE(X^2)+bE(X) \implies \\
& a=\left(E(X^2)-(E(X))^2\right)^{-1}(E(XY)-E(X)E(Y)), \ b=E(Y)-aE(X) \\
& \implies \hat{Y}=E(Y|X)=E(Y)+\left(Var(X) \right)^{-1}Cov(X,Y)(X-E(X)).
\end{align*}
\end{proof}

\subsection{Simple univariate linear model (SULM)}
A linear model for regression assumes that the LMS estimate of $Y=f(X)$ as a conditional expectation/regression function $E(Y|X)$ is
linear in the inputs $X_1,\ldots,X_d$. Linear models were largely developed in the pre-computer age of statistics, but even in today's computer era there are still good reasons to study and use them. They are simple and often provide an adequate and interpretable description of how the inputs affect the output. For prediction purposes they can sometimes outperform fancier non-linear models, especially in situations with small numbers of training cases, low signal-to-noise ratio or sparse data. Finally, linear methods can be
applied to transformations of the inputs to higher dimensional embedded spaces using suitable basis-function approaches.

\subsubsection{Computation of parameters for SULM}
\begin{theorem}[Univariate linear model]
Consider $Y$ is statistically dependent on another random variable $X$ in linear framework as 
\begin{align*}
E(Y|X)=\theta_0+\theta_1X.
\end{align*}
Then, the parameters can be computed as
\begin{align*}
\theta_1=\frac{Cov(X,Y)}{Var(X)}, \ \theta_0=E(Y)- \frac{Cov(X,Y)}{Var(X)}E(X).
\end{align*}
So,
\begin{align*}
& \hat{Y}=E(Y|X)=E(Y)+\frac{Cov(X,Y)}{Var(X)}(X-E(X)) \equiv \\
& Y=E(Y|X)+\epsilon=E(Y)+\frac{Cov(X,Y)}{Var(X)}(X-E(X))+\epsilon.
\end{align*}
\end{theorem}

\begin{proof}
We have
\begin{align*}
& \hat{Y} =E(Y|X)=aX+b \implies Y=aX+b+\epsilon \implies E(Y)=aE(X)+b \implies \\
& E\left(E(X)E(Y|X) \right) =a\left(E(X) \right)^2 +bE(X), \ E(XY)=aE(X^2)+bE(X) \implies \\
& a=\frac{E(XY)-E(X)E(Y)}{E(X^2)-(E(X))^2}=\frac{Cov(X,Y)}{Var(X)}, \ b=E(Y)-\frac{Cov(X,Y)}{Var(X)}E(X) \\
& \implies \hat{Y}=E(Y|X)=E(Y)+\frac{Cov(X,Y)}{Var(X)}(X-E(X)).
\end{align*}
\end{proof}

\begin{example}[Signal detection]
Question: Consider a noisy observation $Y=S+E$ of a signal $S$, where the signal and noisy components are uncorrelated i.e. $Cov(S,E)=0$. Let $Var(S)=\sigma_S^2$ and $Var(E)=\sigma_E^2$. Then, find an LMS estimate of the signal.

\noindent Answer: Using the above derivation for linear model, we get the conditional expectation of the signal $S$ conditional upon the observation $Y$ as
\begin{align*}
E(S|Y)& =E(S)+\frac{Cov(S,Y)}{Var(Y)}(Y-E(Y)) \\
      & =E(S)+\frac{\sigma_S^2}{\sigma_S^2+\sigma_E^2}(Y-E(Y)).
\end{align*}
\end{example}

\begin{example}[\citep{BT08}]
Question: Let $\Theta \sim \mathcal{U}(4,10)$ and we observe $\Theta$ with some noise $W\sim \mathcal{U}(-1,1)$ independent of $\Theta$. That is we observe the experimental value of the random variable $Y=\Theta +W$. Find the least square estimate of $\Theta$.

\noindent Answer: As, $\Theta \sim \mathcal{U}(4,10)$, noise $W\sim \mathcal{U}(-1,1)$  are both are independent. So, conditioned on $\Theta =\theta$, X is the same as $\theta +W$, and is uniformly distributed over the interval $[\theta-1,\theta+1]$. Thus, joint PDF is given by
\begin{align*}
f_{\Theta,X}(\theta,x)=f_{\Theta}(\theta)f_{X|\Theta}(x|\theta)=\left\lbrace\begin{array}{cc}
\frac{1}{6}\cdot \frac{1}{2} & \ \theta \in [4,6], x\in [\theta-1,\theta+1] \\ 
0, & \ \mbox{otherwise} \ .
\end{array} \right.
\end{align*}
The slanted rectangle in the right-hand side of figure \eqref{fig:lmsud-10} is the set of ordered-pairs $(\theta,x)$ for which $f_{\Theta,X}(\theta,x)$ is non-zero. The optimal estimate $E(\Theta|X=x)$ is the midpoint of that section. It depends on the observed value x and hence (a piecewise linear) function of x.
\begin{center}
\begin{figure}[h]
    \centering
    \includegraphics[width=1.0\textwidth]{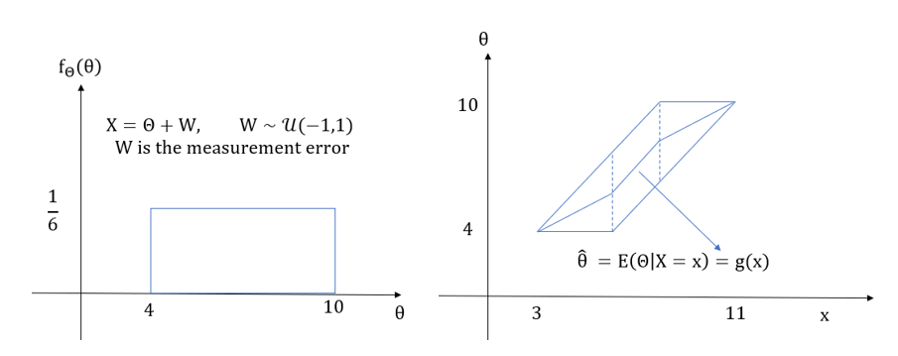}
    \caption{Uniform distribution in left plot and least mean squared estimate $E(\Theta|X=x)$ of $\Theta$ in the right plot.}
    \label{fig:lmsud-10}
\end{figure}
\end{center}
\end{example}

\subsubsection{Estimation of parameters for SULM}
\begin{theorem}[Least squared estimate of parameter of a SULM]
Backdrop: Consider a realization $\bold{z}=(z_1,\ldots, z_n)$ of an i.i.d. random sample of size n i.e. $(Z_1,\ldots,Z_n) \to$ n I.I.D. copies of a joint random variable $Z=(X,Y)$ with unknown linear model $\hat{Y}=E(Y|X)= \theta_0 +\theta_1 X \equiv Y=E(Y|X)+\epsilon=\theta_0 +\theta_1X$. Now, running the linear model over n training examples namely in the data-set $\mathcal{S}=\left\lbrace z_i=( x_i,y_i)\right\rbrace_{i=1}^n$ as:
\begin{align*}
& y_i=\theta_0+\theta_1x_i +\epsilon_i \ \mbox{for} \ i=1,2,\ldots,n \implies 
\left(\begin{array}{cccc}
y_1 \\
y_2 \\
\vdots \\
y_n
\end{array} \right)
=
\left(\begin{array}{cccc}
\theta_0 + \theta_1x_1 \\
\theta_0 + \theta_1x_2 \\
\vdots +  \vdots \\
\theta_0 + \theta_1x_n
\end{array} \right)
\\
& + \left(\begin{array}{ccc}
\epsilon_1 \\
\epsilon_2 \\
\vdots \\
\epsilon_n
\end{array} \right)
\implies 
\left(\begin{array}{cccc}
y_1 \\
y_2 \\
\vdots \\
y_n
\end{array} \right)
=
\left(\begin{array}{cccc}
1 & x_1 \\
1 & x_2 \\
\vdots &  \vdots \\
1 & x_n
\end{array} \right)
\left(\begin{array}{cc}
\theta_0 \\
\theta_1
\end{array} \right)
+ 
\left(\begin{array}{ccc}
\epsilon_1 \\
\epsilon_2 \\
\vdots \\
\epsilon_n
\end{array} \right)
\implies \\
& \bold{y}=\mathrm{X}\theta+\epsilon, \ \mbox{where} \ \bold{x}, \bold{y}, \epsilon, \bold{1} \in \mathbb{R}^n, \mathrm{X}=\left(\bold{1},\bold{x}\right) \in \mathbb{R}^{n\times 2}, \theta =\left(\begin{array}{cc}
\theta_0 \\
\theta_1
\end{array} \right) \in \mathbb{R}^2. 
\end{align*}

\noindent So, running the linear model over n training examples $\mathcal{S}=\left\lbrace z_i=( x_i,y_i)\right\rbrace_{i=1}^n$ transform to a linear inverse problem (of the linear model) as: $\bold{y}=\mathrm{X}\theta +\epsilon$ (a general form of the linear model), where
\begin{enumerate}
\item[a.] $\bold{y}$ is called the response such that $E(\bold{y})=\mathrm{X}\theta$.
\item[b.] $\bold{x}$ is called the explanatory variable, the regressor, or the covariates.
\item[c.] $\mathrm{X}$ is called the design matrix.
\item[d.] A least square estimate of $\theta$: 
\begin{align*}
\hat{\theta}=\mathrm{X}^{\dagger}\bold{y} \equiv \arg\min_{\theta}\left(\|\bold{y}-\mathrm{X}\theta\|^2 \right) = \left(\mathrm{X}^{\top}\mathrm{X}\right)^{-1}\mathrm{X}^{\top}\bold{y}= \left(\begin{array}{cc}
\hat{\theta}_0 \\
\hat{\theta}_1
\end{array} \right).
\end{align*}
\end{enumerate}
And, the MoM estimate of the parameters are given by
\begin{align*}
\hat{\theta}_1=\frac{\sum_{i=1}^n (x_i-\bar{x})(y_i-\bar{y})}{\sum_{i=1}^n(x_i-\bar{x})^2}, \ \hat{\theta}_0=\bar{y}-\hat{\theta}_1\bar{x}.
\end{align*}
Finally, we get the best linear predictor of Y as:
\begin{align*}
\hat{Y}=\hat{\theta}_0+\hat{\theta}_1X.
\end{align*}
\end{theorem}

\begin{example}
Backdrop: Consider a realization $\bold{z}=(z_1,\ldots, z_n)$ of an i.i.d. random sample of size n i.e. $(Z_1,\ldots,Z_n) \to$ n I.I.D. copies of a joint random variable $Z=(X,Y)$ with $Y=E(Y|X)+\epsilon = \theta_0 +\theta_1 X +\epsilon$.

\noindent Question: Students in a probability course claimed that doing the assignments had not helped prepare them for the end-term exam. The end-term exam score y and assignment score x (averaged up to the time of the end-term) for the $n=18$ students in the class are in the table below in table \eqref{tab:pesulm-01}:


\begin{table}[htb]
\centering 
\caption{The exam score y and the homework score x averaged upto the time of midterm for 18 students.}
\scalebox{0.80}{
\begin{tabular}{| c | c | c | c | c | c | c | c | c | c | c | c | c | c | c | c | c | c | c | }
\hline
x & 96 & 77 & 0 & 0 & 78 & 64 & 89 & 47 & 90 & 93 & 18 & 86 & 0 & 30 & 59 & 77 & 74 & 67 \\
\hline 
y & 95 & 80 & 0 & 0 & 79 & 77 & 72 & 66 & 98 & 90 & 0 & 95 & 35 & 50 & 72 & 55 & 75 & 66 \\
\hline
\end{tabular}\label{tab:pesulm-01}}
\end{table}

\noindent Answer: consider a linear model $\hat{Y}=E(Y|X)= \theta_0 +\theta_1 X \equiv Y=E(Y|X)+\epsilon=\theta_0 +\theta_1X$, and running over $n=18$ training examples namely in the data-set $\mathcal{S}=\left\lbrace x_i,y_i)\right\rbrace_{i=1}^{18}$ as
\begin{align*}
& y_i=\theta_0+\theta_1x_i +\epsilon_i \ \mbox{for} \ i=1,2,\ldots,18 \implies 
\left(\begin{array}{cccc}
y_1 \\
y_2 \\
\vdots \\
y_n
\end{array} \right)
=
\left(\begin{array}{cccc}
\theta_0 + \theta_1x_1 \\
\theta_0 + \theta_1x_2 \\
\vdots +  \vdots \\
\theta_0 + \theta_1x_n
\end{array} \right)
\\
& + \left(\begin{array}{ccc}
\epsilon_1 \\
\epsilon_2 \\
\vdots \\
\epsilon_n
\end{array} \right)
\implies 
\left(\begin{array}{cccc}
95 \\
80 \\
\vdots \\
66
\end{array} \right)
=
\left(\begin{array}{cccc}
1 & 96 \\
1 & 77 \\
\vdots &  \vdots \\
1 & 67
\end{array} \right)
\left(\begin{array}{cc}
\theta_0 \\
\theta_1
\end{array} \right)
+ 
\left(\begin{array}{ccc}
\epsilon_1 \\
\epsilon_2 \\
\vdots \\
\epsilon_n
\end{array} \right)
\implies \\
& \bold{y}=\mathrm{X}\theta+\epsilon, \ \mbox{where} \ \bold{x}, \bold{y}, \epsilon, \bold{1} \in \mathbb{R}^{18}, \mathrm{X}=\left(\bold{1},\bold{x}\right) \in \mathbb{R}^{18\times 2}, \theta =\left(\begin{array}{cc}
\theta_0 \\
\theta_1
\end{array} \right) \in \mathbb{R}^2. 
\end{align*}
Furthermore, we have
\begin{align*}
& \mathrm{X}^{\top}\mathrm{X}=\left(\begin{array}{cc}
18 & 1045 \\
1045 & 80199
\end{array} \right)
\implies
\left(\mathrm{X}^{\top}\mathrm{X} \right)^{-1}=
\left(\begin{array}{cc}
0.228 & -0.003 \\
-0.003 & 0.00005
\end{array} \right)
, \\
& \mathrm{X}^{\top}\bold{y}=\left(\begin{array}{cc}
1105 \\
8195
\end{array} \right)
\implies
\hat{\theta}=
\left(\begin{array}{cc}
\hat{\theta}_0 \\
\hat{\theta}_1
\end{array} \right)=
\left(\mathrm{X}^{\top}\mathrm{X}\right)^{-1}\mathrm{X}^{\top}\bold{y}=
\left(\begin{array}{cc}
10.75 \\
0.8726
\end{array} \right)
.
\end{align*}

\noindent So, $\bold{y}=\mathrm{X}\theta +\epsilon$ is the general form of a linear model, where
\begin{enumerate}
\item[a.] $\bold{y}$ is called the response such that $E(\bold{y})=\mathrm{X}\theta$.
\item[b.] $\bold{x}$ is called the explanatory variable, the regressor, or the covariates.
\item[c.] $\mathrm{X}$ is called the design matrix.
\item[d.] A least square estimate of $\theta$: 
\begin{align*}
\hat{\theta}=\left(\mathrm{X}^{\top}\mathrm{X}\right)^{-1}\mathrm{X}^{\top}\bold{y}=(\hat{\theta}_0,\hat{\theta}_1)^{\top}=(10.73, 0.8726)^{\top}.
\end{align*}
\end{enumerate}
Alternatively, using method of moments estimate of the parameters, we get
\begin{align*}
\hat{\theta}_1=\frac{\sum_{i=1}^n (x_i-\bar{x})(y_i-\bar{y})}{\sum_{i=1}^n(x_i-\bar{x})^2}\approx 0.8726, \ \hat{\theta}_0=\bar{y}-\hat{\theta}_1\bar{x}\approx 10.73.
\end{align*}
\begin{center}
\begin{figure}[h]
    \centering
    \includegraphics[width=1.0\textwidth]{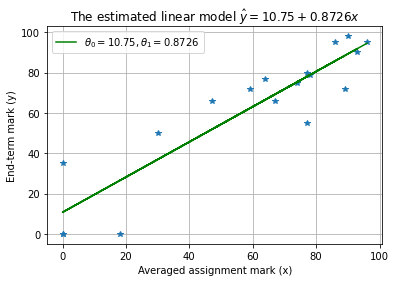}
    \caption{The estimated linear model $\hat{y}=10.73+0.8726x$ that can predict the end-term exam mark for a student with a given averaged assignment mark.}
    \label{fig:peulm-12}
\end{figure}
\end{center}
\end{example}

\begin{theorem}
Least squared estimate of parameters of a simple univariate linear models, under the Gaussian noise, is the maximum likelihood estimate of the parameters of the SULM.
\end{theorem}

\begin{theorem}
Method of moments estimate of parameters of a simple univariate linear model is given by
\begin{align*}
\hat{\theta}_1=\frac{\sum_{i=1}^n (x_i-\bar{x})(y_i-\bar{y})}{\sum_{i=1}^n(x_i-\bar{x})^2}, \ \hat{\theta}_0=\bar{y}-\hat{\theta}_1\bar{x}.
\end{align*}
\end{theorem}

\begin{theorem}
Max a posteriori estimate of parameters of a simple univariate linear model.
\end{theorem}

\subsubsection{Goodness of fit for estimated parameters of SULM}
\begin{definition}[Goodness of fit]
Consider a realization $\bold{z}=(z_1,\ldots, z_n)$ of an i.i.d. random sample of size n i.e. $(Z_1,\ldots,Z_n) \to$ n I.I.D. copies of a joint random variable $Z=(X,Y)$ with $Y=E(Y|X)+\epsilon = \theta_0 +\theta_1 X +\epsilon$. Then, goodness of fit, $R^2$ is defined as
\begin{align*}
R^2=1-\frac{\mbox{Current deviance}}{\mbox{Total deviance}}=1-\frac{\sum_{i=1}^n(y_i-\hat{y}_i)^2}{\sum_{i=1}^n(y_i-\bar{y})^2}=\frac{\sum_{i=1}^n(\hat{y}_i-\bar{y})^2}{\sum_{i=1}^n(y_i-\bar{y})^2}.
\end{align*}
\end{definition}

\begin{theorem}[Gauss-Markov Theorem]
Consider $\bold{y}$ be a random vector with
\begin{align*}
\bold{y}=\mathrm{X}\theta +\epsilon, E(\bold{y})=\mathrm{X}\theta, Var(\bold{y})=\sigma^2I, \mathrm{X} \in \mathbb{R}^{n\times (d+1)} \ \mbox{rank} \ (d+1).
\end{align*}
Then, $\bold{a}^{\top}\hat{\Theta}$ is the unique linear unbiased estimator of $\bold{a}^{\top}\theta$ with minimum variance i.e. BLUE estimator i.e. $\hat{\Theta}\sim \mathcal{N}\left(\theta,\sigma^2\left(\mathrm{X}^{\top}\mathrm{X}\right)^{-1}\right)$ and $\sigma^2$ is having an unbiased estimate as:
\begin{align*}
s^2=\frac{1}{n-2}(\bold{y}-\mathrm{X}\hat{\theta})^{\top}(\bold{y}-\mathrm{X}\hat{\theta})=\frac{1}{n-2}\sum_{i=1}^n(y_i-\hat{y})^2=\frac{SSE}{n-2}.
\end{align*}
Also,
\begin{enumerate}
\item[a.] $\hat{\Theta}_1 \sim \mathbb{N}\left(\theta_1,\frac{\sigma^2}{\sum_{i=1}^n(x_i-\bar{x})^2} \right)$.
\item[b.] $\frac{(n-2)S^2}{\sigma^2}\sim \Xi^2(n-2)$.
\item[c.] $\hat{\Theta_1}$ and $S^2$ are independent, and hence
\begin{align*}
T=\frac{\hat{\Theta}_1}{\frac{S}{\sqrt{\sum_{i=1}^n(X_i-\bar{X})^2}}}
\end{align*}
is a $t$ distribution $T(n-2,\delta)$ with a non-centrality parameter $\delta=\frac{\theta}{\frac{\sigma}{\sqrt{\sum_{i=1}^n(x_i-\bar{x})^2}}}$. A  $100(1-\alpha)\%$ confidence interval for $\theta_1$ is given by
\begin{align*}
\hat{\theta}_1 \pm t_{(\alpha/2, n-2)}\cdot \frac{s}{\sqrt{\sum_{i=1}^n(x_i-\bar{x})^2}}. 
\end{align*}
Here, with a realization $\hat{\theta}_1=0.8726$, and $s=13.8547$, we have
\begin{align*}
t=\frac{\hat{\theta}_1}{\frac{s}{\sqrt{\sum_{i=1}^n(x_i-\bar{x})^2}}}=8.8025> t_{(0.025,16)}=2.210.
\end{align*}
So, we reject $H_0: \theta_1=0$ at the a $\alpha=0.05$ level of significance. And hence $\theta_1 \in [0.8726-2.120\cdot 0.09914, 0.8726+2.210\cdot 0.09914]=[0.6624,1.0828]$.
\end{enumerate}
\end{theorem}

\subsubsection{Sufficient statistics for the SULM}

\subsubsection{A linear model system: a strategy of data-driven linear modelling and application}
\begin{enumerate}[{Step}~1:]
\item Hypothesis class and model structure: propose a model for the data (i.e. a parametric formula linking the response variable with the input variables, recognising the stochastic nature of the response).
\item Estimation of model: fit the model (e.g. find the best set of parameters).
\item Goodness of fit: ask is the model adequate? (i.e. consistent with the data). Does it allow the main questions of the analysis to be answered?
\item Model evaluation and performance: fit other plausible models, compare them, and choose the best.
\item Deployment of the best linear model and decision making.
\end{enumerate}

\subsection{Simple multivariate linear model}
Motivation for simple multivariate linear model:
\subsubsection{Computation of parameters for SUMLM}
\begin{definition}[Multivariate linear model]
Consider Y is statistically dependence on a number of covariate random variables $X_1,X_2,\ldots,X_d$ (i.e. a random vector $\bold{X}=(X_1,\ldots, X_d)^{\top} \in \mathbb{R}^{d\times 1}$) in linear framework as: 
\begin{align*}
E(Y|X)=\theta_0+\theta_1X_1+\theta_2X_2+\ldots+\theta_dX_d=\theta_0+\theta^{\top}\bold{X}.
\end{align*}
Then, the parameters can be computed as
\begin{align*}
\theta &=(\theta_1,\ldots,\theta_d)^{\top}=Cov(\bold{X},Y)(Cov(\bold{X}))^{-1}, \\
\theta_0&=E(Y)- Cov(\bold{X},Y)(Var(\bold{X}))^{-1}E(\bold{X}).
\end{align*}
So,
\begin{align*}
& \hat{Y}=E(Y|\bold{X})=E(Y)+Cov(\bold{X},Y)(Var(\bold{X})^{-1}(\bold{X}-E(\bold{X})) \equiv \\
& Y=E(Y|\bold{X})+\epsilon=E(Y)+Cov(\bold{X},Y)(Var(\bold{X})^{-1}(\bold{X}-E(\bold{X}))+\epsilon.
\end{align*}
\end{definition}

\begin{example}
xxxxxxxxxxxxxxxxx
\end{example}

\subsubsection{Estimation of parameters for SMLM}
\begin{theorem}[Parameter estimation of a simple multivariate linear model]
Backdrop: Consider a realization $\bold{z}=(z_1,\ldots, z_n)$ of an i.i.d. random sample of size n i.e. $(Z_1,\ldots,Z_n) \to$ n I.I.D. copies of a joint random variable $Z=(\bold{X},Y)$ with unknown linear model $\hat{Y}=E(Y|\bold{X})= \theta_0 +\theta_1X_1+\ldots+\theta_dX_d \equiv Y=E(Y|\bold{X})+\epsilon=\theta_0 +\theta_1X_1+\ldots+\theta_dX_d$. Now, running the linear model over n training examples namely in the data-set $\mathcal{S}=\left\lbrace \bold{x}_i,y_i)\right\rbrace_{i=1}^n$ as
\begin{align*}
& y_i=\theta_0+\theta_1x_{i1} + \ldots + \theta_dx_{id}  +\epsilon_i \ \mbox{for} \ i=1,2,\ldots,n 
\implies \\
&
\left(\begin{array}{cccc}
y_1 \\
y_2 \\
\vdots \\
y_n
\end{array} \right)
=
\left(\begin{array}{cccc}
\theta_0 + \theta_1x_{11} + \ldots +\theta_dx_{1d} \\
\theta_0 + \theta_1x_{21} + \ldots + \theta_dx_{2d} \\
\vdots +  \vdots + \ldots +\vdots \\
\theta_0 + \theta_1x_{n1} + \ldots + \theta_dx_{nd}
\end{array} \right)
+ 
\left(\begin{array}{ccc}
\epsilon_1 \\
\epsilon_2 \\
\vdots \\
\epsilon_n
\end{array} \right)
\implies 
\\
&
\left(\begin{array}{cccc}
y_1 \\
y_2 \\
\vdots \\
y_n
\end{array} \right)
=
\left(\begin{array}{cccc}
1 & x_1 \\
1 & x_2 \\
\vdots &  \vdots \\
1 & x_n
\end{array} \right)
\left(\begin{array}{cccc}
\theta_0 \\
\theta_1 \\
\vdots \\
\theta_d
\end{array} \right)
+ 
\left(\begin{array}{ccc}
\epsilon_1 \\
\epsilon_2 \\
\vdots \\
\epsilon_n
\end{array} \right)
\implies 
\bold{y}=\mathrm{X}_{n\times (d+1)}\bold\theta+\epsilon, \\
& \ \mbox{where} \ \bold{x}_j=(x_{1j},\ldots x_{nj})^{\top}, j=1,2,\ldots,d, \bold{y}, \epsilon, \bold{1} \in \mathbb{R}^{n \times 1}, \mathrm{X}_{n\times (d+1)}= \\
&\left(\bold{1},\bold{x}_1,\ldots,\bold{x}_d\right) \in \mathbb{R}^{n\times (d+1)}, \theta =(\theta_0,\theta_1,\ldots,\theta_d)^{\top} \in \mathbb{R}^{(d+1)\times 1}. 
\end{align*}
So, $\bold{y}=\mathrm{X}\theta +\epsilon$ is the general form of a linear model, where
\begin{enumerate}
\item[a.] $\bold{y}$ is called the response such that $E(\bold{y})=\mathrm{X}\theta$.
\item[b.] $\bold{x}$ is called the explanatory variable, the regressor, or the covariates.
\item[c.] $\mathrm{X}$ is called the design matrix.
\item[d.] A least square estimate of $\omega$: 
\begin{align*}
\hat{\theta}=\left(\mathrm{X}^{\top}\mathrm{X}\right)^{-1}\mathrm{X}^{\top}\bold{y}=(\hat{\theta}_0,\hat{\theta}_1,\ldots,\hat{\theta}_d)^{\top}.
\end{align*}
\end{enumerate}
\end{theorem}

\begin{example}
Backdrop: Consider a realization $\bold{z}=(z_1,\ldots, z_n)$ of an i.i.d. random sample of size n i.e. $(Z_1,\ldots,Z_n) \to$ n I.I.D. copies of a joint random variable $Z=(\bold{X},Y)$ with $Y=E(Y|\bold{X})+\epsilon = \theta_0 +\theta_1 X_1 +\theta_2 X_2 +\epsilon$.

\noindent Question: Consider a response y with bivariate input $(x_1,x_2)$ and $n=12$ from \citep{FM79} in the table below:

\begin{center}
\begin{tabular}{| c | c | c | c | c | c | c | c | c | c | c | c | c | }
\hline
$x_1$ & 0 &2 & 2 & 2 & 4 & 4 & 4 & 6 & 6 & 6 & 8 & 8 \\
\hline 
$x_2$ & 2 & 6 & 7 & 5 & 9 & 8 & 7 & 10 & 11 & 9 & 15 & 13  \\
\hline
y & 2 & 3 & 2 & 7 & 6 & 8 & 10 & 7 & 8 & 12 & 11 & 14  \\
\hline
\end{tabular}
\end{center}

\noindent Answer: consider a linear model $\hat{Y}=E(Y|\bold{X})= \theta_0 +\theta_1 X_1 +\theta_2 X_2 \equiv Y=E(Y|\bold{X})+\epsilon=\theta_0 +\theta_1X_1 +\theta_2 X_2$, and running over $n$ training examples namely in the data-set $\mathcal{S}=\left\lbrace (x_{i1},x_{i2}),y_i)\right\rbrace_{i=1}^n$ as
\begin{align*}
& y_i=\theta_0+\theta_1x_{i1}+\theta_2x_{i2} +\epsilon_i \ \mbox{for} \ i=1,2,\ldots,n \implies \bold{y}=\mathrm{X}\theta+\epsilon, \\
& \ \mbox{where} \ \bold{x}, \bold{y}, \epsilon, \bold{1} \in \mathbb{R}^{n \times 1}, \mathrm{X}=\left(\bold{1},\bold{x}_1, \bold{x}_2\right) \in \mathbb{R}^{n\times 3}, \theta =(\theta_0,\theta_1,\theta_2)^{\top} \in \mathbb{R}^{3\times 1}. 
\end{align*}
So, $\bold{y}=\mathrm{X}\theta +\epsilon$ is the general form of a linear model, where
\begin{enumerate}
\item[a.] $\bold{y}$ is called the response such that $E(\bold{y})=\mathrm{X}\theta$.
\item[b.] $\bold{x}$ is called the explanatory variable, the regressor, or the covariates.
\item[c.] $\mathrm{X}$ is called the design matrix.
\item[d.] A least square estimate of $\theta$: 
\begin{align*}
\hat{\theta}=\left(\mathrm{X}^{\top}\mathrm{X}\right)^{-1}\mathrm{X}^{\top}\bold{y}=(\hat{\theta}_0,\hat{\theta}_1)^{\top}=(10.73, 0.8726)^{\top}.
\end{align*}
\end{enumerate}

\noindent Now, with respect to above data, we have the following response data vector, design matrix and the associated computations:
\begin{align*}
y=\left(\begin{array}{cccccccccccc}
2 \\
3 \\ 
2 \\
7 \\
6 \\
8 \\
10 \\
7 \\
8 \\
12 \\
11 \\
14 
\end{array} \right)
,
\begin{array}{cc}
& X=\left(\begin{array}{cccccccccccc}
1 & 1 & 1 & 1 & 1 & 1 & 1 & 1 & 1 & 1 & 1 & 1 \\
0 &2 & 2 & 2 & 4 & 4 & 4 & 6 & 6 & 6 & 8 & 8 \\
2 & 6 & 7 & 5 & 9 & 8 & 7 & 10 & 11 & 9 & 15 & 13
\end{array} \right)^{\top}
\\ 
& X^{\top}X=\left(\begin{array}{ccc}
12 & 52 & 102 \\
52 & 395 & 536  \\
102 & 536 & 1004 
\end{array} \right)
\\
& \left(X^{\top}X\right)^{-1}=\left(\begin{array}{ccc}
0.97476 & 0.2429 & -0.22871 \\
0.2429 & 0.16207 & -0.11120  \\
-0.22871 & -0.11120 & 0.08360 
\end{array} \right)
\\
& \hat{\theta}=\left(\mathrm{X}^{\top}\mathrm{X}\right)^{-1}\mathrm{X}^{\top}\bold{y}=\left(\begin{array}{ccc}
5.3754 \\
3.0118 \\
-1.2855
\end{array} \right).
\end{array}
\end{align*} 
So, the covariance matrix of $\hat{\theta}$ is given by
\begin{align*}
Cov(\hat{\theta})=\sigma^2\left(X^{\top}X\right)^{-1}
=\sigma^2\left(\begin{array}{ccc}
0.975 & 0.243 & -0.229 \\
0.243 & 0.162 & -0.111  \\
-0.229 & -0.111 & 0.084 
\end{array} \right)
\end{align*}
The negative value of $Cov(\hat{\theta}_1,\hat{\theta}_2)=-0.111$ indicates that in repeated sampling (using the same $12$ values of $x_1$ and $x_2$),$(\hat{\theta}_1$ and $\hat{\theta}_2)$ would tend to move in opposite directions; that is, an increase in one would be accompanied by a decrease in the other.
\end{example}

\subsubsection{Goodness of fit for estimated parameters of SMLM}
\begin{theorem}[Gauss-Markov Theorem]
Consider $\bold{y}$ be a random vector with
\begin{align*}
\bold{y}=\mathrm{X}\theta +\epsilon, E(\bold{y})=\mathrm{X}\theta, Var(\bold{y})=\sigma^2I, \mathrm{X} \in \mathbb{R}^{n\times (d+1)} \ \mbox{rank} \ (d+1).
\end{align*}
Then, $\bold{a}^{\top}\hat{\Theta}$ is the unique linear unbiased estimator of $\bold{a}^{\top}\theta$ with minimum variance i.e. BLUE estimator i.e. $\hat{\Theta}\sim \mathcal{N}\left(\theta,\sigma^2\left(\mathrm{X}^{\top}\mathrm{X}\right)^{-1}\right)$ and $\sigma^2$ is having an unbiased estimate as:
\begin{align*}
s^2=\frac{1}{n-d-1}(\bold{y}-\mathrm{X}\hat{\theta})^{\top}(\bold{y}-\mathrm{X}\hat{\theta})=\frac{1}{n-d-1}\sum_{i=1}^n(y_i-\hat{y})^2=\frac{SSE}{n-2}.
\end{align*}
Also,
\begin{enumerate}
\item[a.] $\hat{\Theta}\sim \mathcal{N}\left(\theta,\sigma^2\left(\mathrm{X}^{\top}\mathrm{X}\right)^{-1}\right)$.
\item[b.] $\frac{(n-d-1)S^2}{\sigma^2}\sim \Xi^2(n-d-1)$.
\item[c.] $\hat{\Theta_1}$ and $S^2$ are independent.
Here, with a realization $\hat{\theta}=(5.3754,3.0118,-1.2855)^{\top}$, we get $s^2=2.829$.
\end{enumerate}
\end{theorem}

\subsection{Embedded linear models (ELM)}
An embedded simple multivariate linear model is deep-down inside a simple multivariate linear model in the transformed higher dimensional embedded space, so computation and estimation of parameters for the ESMLM are similar to that of the SMLM.
\subsubsection{Computation of parameters for ELM}
\begin{definition}[Embedded simple multivariate linear model]
Consider $Y$ is statistically dependent on a number of covariate random variables $X_1,X_2,\ldots,X_d$ (i.e. a random vector $\bold{X}=(X_1,\ldots, X_d)^{\top} \in \mathbb{R}^{d\times 1}$) (that got transformed to $\bold{\phi(X))}=(\phi_1(\bold{X}),\ldots, \phi_p(\bold{X}))^{\top}$ in the p-dimensional embedded space using the non-linear basis functions $\phi_i$) in linear framework as: 
\begin{align*}
E(Y|\bold{X})&=\theta_0+\theta_1\phi_1(\bold{X})+\theta_2\phi_2(\bold{X})+\ldots+\theta_p\phi_p(\bold{X})\\
             &=\theta_0+\theta^{\top}\bold{\phi(X)}.
\end{align*}
Then, the parameters can be computed as
\begin{align*}
\theta &=(\theta_1,\ldots,\theta_p)^{\top} =Cov(\bold{\phi(X)},Y)(Cov(\bold{\phi(X)}))^{-1}, \\
\theta_0&=E(Y)- Cov(\bold{\phi(X)},Y)(Var(\bold{\phi(X))})^{-1}E(\bold{\phi(X)}).
\end{align*}
So,
\begin{align*}
& \hat{Y}=E(Y|\bold{X})=E(Y)+Cov(\bold{X},Y)(Var(\bold{X})^{-1}(\bold{X}-E(\bold{X})) \equiv \\
& Y=E(Y|\bold{X})+\epsilon=E(Y)+Cov(\bold{X},Y)(Var(\bold{X})^{-1}(\bold{X}-E(\bold{X}))+\epsilon.
\end{align*}
\end{definition}

\begin{theorem}[A Kernel embedded and ridge regularized linear model (KRLM)]
Let us now consider a reproducing Kernel Hilbert space (RKHS), $\mathcal{H}$, associated to a positive definite Kernel K on input space $\mathcal{X}$. Then, an estimate of KRLM is obtained by regularizing the mean squared error (MSE) criterion by the RKHS (a Hilbert space produced by a Kernel inspired feature map $\phi(x)=k(\cdot,x)$) norm:
\begin{align*}
\hat{f} =\arg\min_{f\in \mathcal{H}} \left(\frac{1}{n}\sum_{i=1}^n \left(y_i-f(x_i)\right)^2 + \lambda \|f\|_{\mathcal{H}}^2 \right).
\end{align*}
\end{theorem}

\begin{proof}
Now, the modified estimation problem can be posed as:
\begin{align*}
\hat{f} & =\arg\min_{f\in \mathcal{H}} \left(\frac{1}{n}\sum_{i=1}^n\ell(y_i,f(x_i)) + \lambda \|f\|_{\mathcal{H}}^2\right) \\
       & = \arg\min_{f\in \mathcal{H}} \left(\frac{1}{n}\sum_{i=1}^n \left(y_i-f(x_i)\right)^2  + \lambda \|f\|_{\mathcal{H}}^2\right).
\end{align*}
By the representation theorem, the solution of the unconstrained problem can be expanded as:
\begin{align*}
\hat{f}(x)=\sum_{i=1}^n \alpha_i k(x_i,x) =\hat{\omega}^{\top}\phi(x), \phi: \mathcal{X} \to \mathbb{R} \ \mbox{such that} \ \phi(x)=k(\cdot,x).
\end{align*}
To estimate $\omega$, we need to consider the following:
\begin{align*}
& y = (y_1,y_2,\ldots, y_n)^{\top}, K=\left(k(x_i,x_j)\right)_{n \times n}, k(x_i,x_j)=\phi(x_i)^{\top}\phi(x_j), \\
& \alpha = (\alpha_1,\alpha_2,\ldots, \alpha_n)^{\top}, \left(\hat{f}(x_1),\ldots, \hat{f}(x_n)\right)^{\top} \approx K\alpha.
\end{align*}
Then, the estimation problem can be reposed as 
\begin{align*}
\hat{\alpha}=\arg\min_{\alpha in \mathbb{R}^n} \left(\frac{1}{n}(K\alpha-y)^{\top}(K\alpha -y) +\lambda \alpha^{\top}K\alpha
\right)
\end{align*}
So, the estimate of $\alpha$ is computed as:
\begin{align*}
& \nabla J(\alpha)=\frac{2}{n}K(K\alpha -y)+2\lambda K\alpha = \frac{2}{n}K\left((K\alpha +\lambda nI)-y \right)=0 \implies \\
& \alpha - \left(K+\lambda nI \right)^{-1}y \in Ker(K) \implies \alpha = \left(K+\lambda nI \right)^{-1}y + \epsilon, \epsilon \in Ker(K) \implies \\
& \hat{\alpha} = \left(K+\lambda nI \right)^{-1}y \ \mbox{with} \ K\epsilon = 0.
\end{align*}
\end{proof}

\subsubsection{Estimation of parameters for ELM}
\begin{theorem}[Parameter estimation of an embedded simple multivariate linear model]
Backdrop: Consider a realization $\bold{z}=(z_1,\ldots, z_n)$ of an I.I.D. random sample of size $n$ i.e. $(Z_1,\ldots,Z_n) \to n$ I.I.D. copies of a joint random variable $Z=(\bold{X},Y) \to (\bold{\phi(X)},Y)$ with unknown linear model $\hat{Y}=E(Y|\bold{\phi(X)})= \theta_0 +\theta_1\phi_1(\bold{X})+\ldots+\theta_p\phi_p(\bold{X}) \equiv Y=E(Y|\bold{\phi(X)})+\epsilon=\theta_0 +\theta_1\phi_1(\bold{X})+\ldots+\theta_p\phi_p(\bold{X}) +\epsilon$. Now, running the linear model over n training examples namely in the data-set $\mathcal{S}=\left\lbrace \bold{x}_i,y_i)\right\rbrace_{i=1}^n$ as
\begin{align*}
& y_i=\theta_0+\theta^{\top}\bold{\phi{x}}_i +\epsilon_i \ \mbox{for} \ i=1,2,\ldots,n \implies \bold{y}=\mathrm{X}_{n\times (p+1)}\bold\theta+\epsilon, \\
& \ \mbox{where} \ \bold{\phi(x)}_j=(\phi_{1j}\bold{x},\ldots \phi_{nj}(\bold{x}))^{\top}, j=1,2,\ldots,p, \bold{y}, \epsilon, \bold{1} \in \mathbb{R}^{n \times 1}, \mathrm{X}_{n\times (p+1)}= \\
&\left(\bold{1},\bold{\phi(x)}_1,\ldots,\bold{\phi(x)}_p\right) \in \mathbb{R}^{n\times (p+1)}, \theta =(\theta_0,\theta_1,\ldots,\theta_p)^{\top} \in \mathbb{R}^{(p+1)\times 1}. 
\end{align*}
So, $\bold{y}=\mathrm{X}\theta +\epsilon$ is the general form of a linear model, where
\begin{enumerate}
\item[a.] $\bold{y}$ is called the response such that $E(\bold{y})=\mathrm{X}\theta$.
\item[b.] $\bold{x}$ is called the explanatory variable, the regressor, or the covariates.
\item[c.] $\mathrm{X}$ is called the design matrix.
\item[d.] A least square estimate of $\omega$: 
\begin{align*}
\hat{\theta}=\left(\mathrm{X}^{\top}\mathrm{X}\right)^{-1}\mathrm{X}^{\top}\bold{y}=(\hat{\theta}_0,\hat{\theta}_1,\ldots,\hat{\theta}_p)^{\top}.
\end{align*}
\end{enumerate}
\end{theorem}

\begin{example}
Consider a synthetic data generated from the function:
\begin{align*}
f(x)=1 + cos(\pi x) + sin(\pi x) + cos(2\pi x) + sin(2*\pi x) + 0.5\epsilon, \epsilon \sim \mathcal{N}(0,1).
\end{align*}
Then, the estimate of the embedded linear model is given by
\begin{align*}
\hat{y}=1.07 + 1.05\cos(\pi x) + 1.16\sin(\pi x) + 1.15\cos(2\pi x) + 0.98\sin(2\pi x).
\end{align*}
\begin{center}
\begin{figure}[h]
    \centering
    \includegraphics[width=1.0\textwidth]{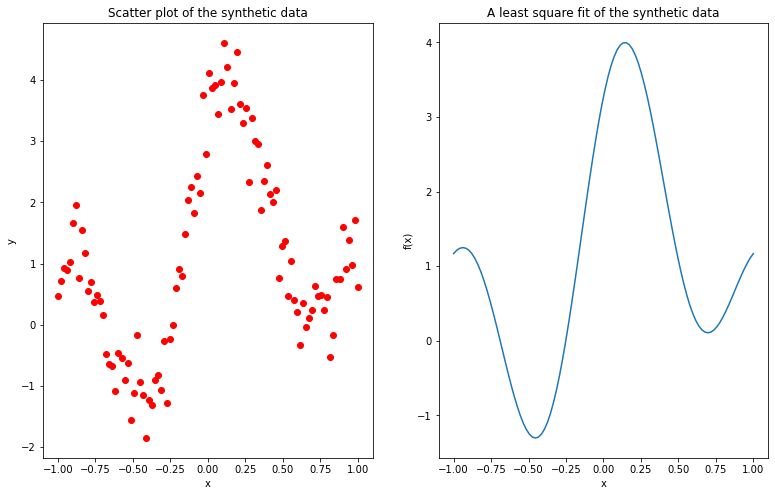}
    \caption{In left, there is a scatter plot of the synthetic data and in right, an embedded linear model with sinusoidal basis functions from the synthetic data.}
    \label{fig:peelm-01}
\end{figure}
\end{center}
\end{example}

\begin{example}
Backdrop: Consider a realization $\bold{z}=(z_1,\ldots, z_n)$ of an I.I.D. random sample of size $n$ i.e. $(Z_1,\ldots,Z_n) \to n$ I.I.D. copies of a joint random variable $Z=(X,Y)$ with $Y=E(Y|X)+\epsilon =\theta_0 +\theta_1 X + \theta_2 X^2 + \epsilon$ or $\theta_0 +\theta_1 X + \theta_2 X^2 + \theta_3X^3 + \theta_4 X^4 + \epsilon$.

\noindent Question: Employees in a company claimed that getting a higher position had not helped prepare them to get higher salary. The salary y and the position for the $n=10$ employees are given in table \eqref{tab:peeblm-01}:
\begin{table}[htb]
\centering 
\caption{The salary y and the position level x for 10 employees.}
\scalebox{0.70}{
\begin{tabular}{| c | c | c | c | c | c | c | c | c | c | c | c | c | c | c | c | c | c | c | }
\hline
x & 1 & 2 & 3 & 4 & 5 & 6 & 7 & 8 & 9 & 10 \\
\hline 
y & 45000 & 50000 & 60000 & 80000 & 110000 & 150000 & 200000 & 300000 & 500000 & 1000000 \\
\hline
\end{tabular}\label{tab:peeblm-01}}
\end{table}

\noindent Answer: consider a linear model $\hat{Y}=E(Y|\bold{X})= \theta_0 +\theta_1 X_1 +\theta_2 X_2 \equiv Y=E(Y|\bold{X})+\epsilon=\theta_0 +\theta_1X_1 +\theta_2 X_2$, and running over $n$ training examples namely in the data-set $\mathcal{S}=\left\lbrace (x_{i1},x_{i2}),y_i)\right\rbrace_{i=1}^n$ as
\begin{align*}
& y_i=\theta_0+\theta_1x_{i1}+\theta_2x_{i2} +\epsilon_i \ \mbox{for} \ i=1,2,\ldots,n \implies \bold{y}=\mathrm{X}\theta+\epsilon, \\
& \ \mbox{where} \ \bold{x}, \bold{y}, \epsilon, \bold{1} \in \mathbb{R}^{n \times 1}, \mathrm{X}=\left(\bold{1},\bold{x}_1, \bold{x}_2\right) \in \mathbb{R}^{n\times 3}, \theta =(\theta_0,\theta_1,\theta_2)^{\top} \in \mathbb{R}^{3\times 1}. 
\end{align*}
So, $\bold{y}=\mathrm{X}\theta +\epsilon$ is the general form of a linear model, where
\begin{enumerate}
\item[a.] $\bold{y}$ is called the response such that $E(\bold{y})=\mathrm{X}\theta$.
\item[b.] $\bold{x}$ is called the explanatory variable, the regressor, or the covariates.
\item[c.] $\mathrm{X}$ is called the design matrix.
\item[d.] A least square estimate of $\theta$: 
\begin{align*}
\hat{\theta}=\left(\mathrm{X}^{\top}\mathrm{X}\right)^{-1}\mathrm{X}^{\top}\bold{y}=(\hat{\theta}_0,\hat{\theta}_1)^{\top}=(10.73, 0.8726)^{\top}.
\end{align*}
\end{enumerate}
\begin{center}
\begin{figure}[h]
    \centering
    \includegraphics[width=1.0\textwidth]{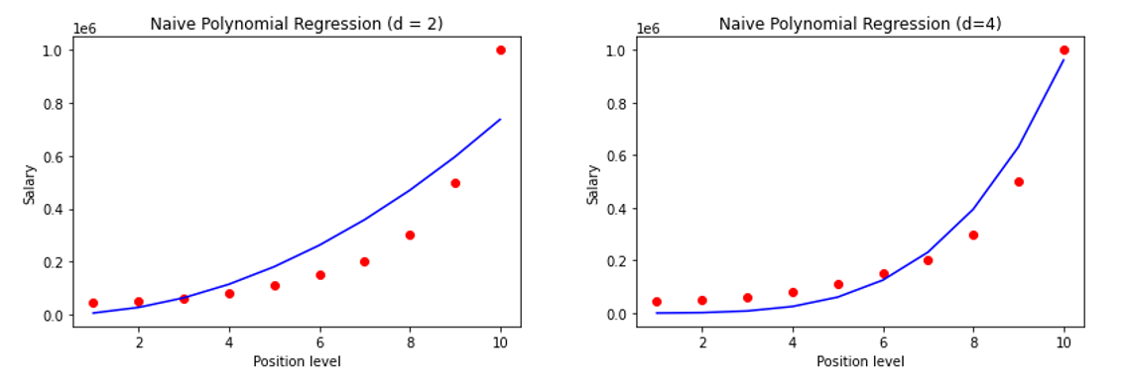}
    \caption{Estimated embedded linear models $\hat{y}=2.25\times 10^5-1.40\times 10^5 x+ 2.09 \times 10^4 x^2$ and $\hat{y}=1.35\times 10^5-1.43\times 10^5x+6.82\times 10^4 x^2 -1.16\times 10^4 x^3 + 7.09 \times 10^2 x^4$ that can predict the salary for an employee with given position number.}
    \label{fig:peeblm-02}
\end{figure}
\end{center}
\end{example}

\subsection{Multiple multivariate linear model (MMLM)}
Motivation and intuition of multiple multivariate linear model:
\begin{definition}[Kronecker product]
Consider two $2\times 2$ matrices as:
\begin{align*}
A=\left(\begin{array}{cc}
2 & 0\\
1 & 3
\end{array} \right)
, \ 
B=\left(\begin{array}{cc}
5 & -1\\
-1 & 4
\end{array} \right)
\end{align*}
Then, Kronecker product is defined as:
\begin{align*}
A\otimes B=\left(\begin{array}{cc}
2B & 0B \\
B & 3B
\end{array} \right)
=
\left(\begin{array}{cccc}
10 & -2 & 0 & 0\\
-2 & 8 & 0 & 0 \\
5 & -1 & 15 & -3 \\
-1 & 4 & -3 & 12
\end{array} \right).
\end{align*}
\end{definition}

\subsubsection{Computation of parameters for MMLM}
\begin{theorem}[Simplification of MMLM in column and row pictures]
Backdrop: Consider a realization $\bold{z}=(z_1,\ldots, z_n)$ of an I.I.D. random sample of size $n$ i.e. $(Z_1,\ldots,Z_n) \to n$ I.I.D. copies of a joint random variable $Z=(\bold{X},\bold{Y})$ with unknown linear model $\hat{Y}_i=E(Y_i|\bold{X})= \theta_{i0} +\theta_{i1}X_1+\ldots+\theta_{id}X_d, i=1,\ldots,p \equiv Y_i=E(Y_i|\bold{X})+\epsilon=\theta_{i0} +\theta_{i1}X_1+\ldots+\theta_{id}X_d, i=1,\ldots,p$. Now, running the linear model over n training examples namely in the data-set $\mathcal{S}=\left\lbrace \bold{x}_i,\bold{y}_i)\right\rbrace_{i=1}^n$ as a column picture:
\begin{align*}
\mathrm{Y} & =\mathrm{X}\mathrm{\Theta}+\mathrm{\epsilon} \implies vec(\mathrm{Y})=vec(\mathrm{X}\mathrm{\Theta})+vec(\epsilon) \implies \\
& = \left(\begin{array}{cc}
y_1\\
y_2 \\
\vdots\\
y_p
\end{array} \right)
=\left(\begin{array}{cc}
\mathrm{X}\theta_1\\
\mathrm{X}\theta_2
\vdots\\
\mathrm{X}\theta_p
\end{array} \right)
+
\left(\begin{array}{cc}
\epsilon_1\\
\epsilon_2 \\
\vdots\\
\epsilon_p
\end{array} \right)
=(I_p\otimes \mathrm{X})vec(\Theta)
+
\left(\begin{array}{cc}
\epsilon_1\\
\vdots\\
\epsilon_p
\end{array} \right) \\
& =  
\left(\begin{array}{cccccc}
\mathrm{X} & 0 & \ldots & 0 & 0\\
0 & \mathrm{X} & \ldots & 0 & 0\\
\vdots & 0 & \ldots & \vdots & \vdots \\
0 & 0 & \ldots & 0 & \mathrm{X}
\end{array} \right)
\left(\begin{array}{cc}
\theta_1\\
\theta_2 \\
\vdots\\
\theta_p
\end{array} \right)
+
\left(\begin{array}{cc}
\epsilon_1\\
\epsilon_2 \\
\vdots\\
\epsilon_p
\end{array} \right)  \implies \\
y_j & =\mathrm{X}\theta_j +\epsilon_j, j=1,\ldots,p.
\end{align*}
As a row picture:
\begin{align*}
\mathrm{Y} & =\mathrm{X}\mathrm{\Theta}+\mathrm{\epsilon} \implies vec(\mathrm{Y}^{\top})=vec((\mathrm{X}\mathrm{\Theta})^{\top})+vec(\epsilon^{\top}) \implies \\
& = \left(\begin{array}{cc}
Y_1^{\top}\\
Y_2^{\top} \\
\vdots\\
Y_n^{\top}
\end{array} \right)
=\left(\begin{array}{cc}
(\mathrm{X}_1\Theta)^{\top}\\
(\mathrm{X}_2\Theta)^{\top}
\vdots\\
(\mathrm{X}_n\Theta)^{\top}
\end{array} \right)
+
\left(\begin{array}{cc}
\epsilon_1^{\top}\\
\epsilon_2^{\top} \\
\vdots\\
\epsilon_n^{\top}
\end{array} \right)
=(\mathrm{X}\otimes I_p)vec(\Theta^{\top})
+
\left(\begin{array}{cc}
\epsilon_1^{\top}\\
\epsilon_2^{\top} \\
\vdots\\
\epsilon_n^{\top}
\end{array} \right) \\
& =  
\left(\begin{array}{ccccc}
x_{11}I_p & x_{12}I_p & \ldots & x_{1d}I_p\\
x_{21}I_p & x_{22}I_p & \ldots & x_{2d}I_p\\
\vdots & \vdots & \ldots & \vdots  \\
x_{n1}I_p & x_{n2}I_p & \ldots  & x_{nd}I_p
\end{array} \right)
\left(\begin{array}{cc}
\Theta_1\\
\Theta_2 \\
\vdots\\
\Theta_p
\end{array} \right)
+
\left(\begin{array}{cc}
\epsilon_1^{\top}\\
\epsilon_2^{\top} \\
\vdots\\
\epsilon_n^{\top}
\end{array} \right)  \implies \\
Y_i & =(\mathrm{X}_i \otimes \Theta) +\epsilon_i, 1=1,\ldots,n.
\end{align*}
\end{theorem}

\begin{remark}
Computation and estimation of parameters for the multiple multivariate linear models can be carried with the help of simple multivariate linear models approach using a suitable decomposition technique either in row or column picture form. 
\end{remark}

\section{Generalized linear models}
First, we consider a model conditional probability of success for $Y\in \{0,1\}$ given $X, P(Y=1|X),$ as:
\begin{align*}
& logit(P(Y=1|X))=\log\left(\frac{P(Y=1|X)}{1-P(Y=1|X)} \right)=\theta^{\top} X \implies \\
& P(Y=1|X)=\frac{1}{1+\exp(-\theta^{\top}X)}=\frac{\exp(\theta^{\top}X)}{1+\exp(\theta^{\top}X)}=\sigma\left(\theta^{\top}X \right).
\end{align*}
So, in general a statistical modelling (i.e. a categorical modelling), we wish to estimate/predict a discrete class label (or a posteriori probability), with a linear argument and a known activation as:
\begin{align*}
logit\left(\hat{Y} \right)=\ln\left(\frac{\hat{Y}}{1-\hat{Y}} \right)=\theta^{\top} X \implies \hat{Y}=\frac{1}{1+\exp(-\theta^{\top}X)}=\sigma\left(\theta^{\top}X \right).
\end{align*} 

\subsection{Logistic regression: Bernoulli logit model}
Algorithmic approach of logistic regression:
\begin{enumerate}[{Step}~1:]
\item First transform the space of class probability to the space of logit of the probability using the transformation 
\begin{align*}
logit(p)=\log\left(\frac{p}{1-p}\right) \in (-\infty,\infty), \ p\in [0,1].
\end{align*}
\item Now, we perform linear regression as the linear model
\begin{align*}
logit(\hat{Y})=logit(p)=\theta^{\top}X
\end{align*} 
in order to maximize the likelihood of observing the data.
\item Finally, we remap the space of $logit(p)$ to class probability using the inverse of the logit as the logistic map as:
\begin{align*}
\hat{Y}=\hat{p}=logit^{-1}(\theta^{\top}X)=\frac{\exp(\theta^{\top}X)}{1+\exp(\theta^{\top}X)}=\sigma(\theta^{\top}X).
\end{align*}
\item Further, as $p \in [0,1]$, we have a decision boundary at:
\begin{align*}
p=0.5 \implies \theta^{\top}X=logit(0.5)=\log\left(\frac{0.5}{1-0.5}\right)=\log(1)=0.
\end{align*}
\end{enumerate}

\begin{example}
\begin{center}
\begin{figure}[htb]
    \centering
    \includegraphics[width=1.0\textwidth]{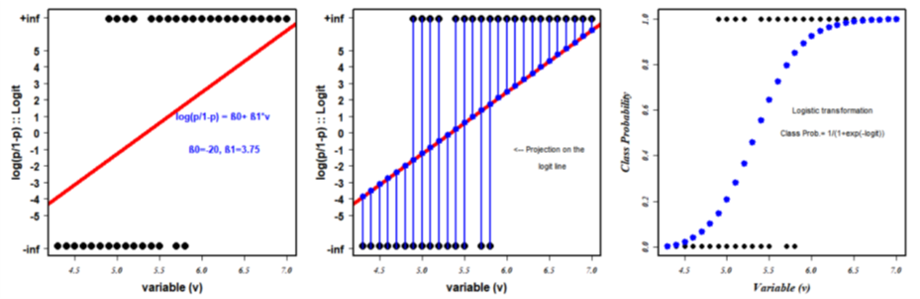}
    \caption{A visualization of logistic regression via linear model of logit(p) followed by inversion map of the logit function.}
    \label{fig:logitplm-01}
\end{figure}
\end{center}
\end{example}
Now, running over the given dataset $\mathcal{S}=\{(x_i,y_i)\}_{i=1}^n$, we have

\begin{theorem}
Consider $\mathcal{X}$ be a set of inputs. $Y\in \{-1,1\}$ binary outputs, $\mathcal{S}=(x_i,y_i )_{i=1}^n \in \left(\mathcal{X} \times \mathcal{Y} \right)^n$ a training set of $n$ pairs. Estimate a function $f:\mathcal{X} \to \mathbb{R}$ to predict y by $sign(f(x))$.
\end{theorem}

\begin{proof}
Consider a $0/1$ loss measures if a prediction is correct or not as:
\begin{align*}
\ell_{0/1}(y,f(x))=I\left(yf(x)<0 \right) = \left\lbrace\begin{array}{cc}
0 & \ \mbox{if} \ y=sign(f(x)) \\
1 & \ \mbox{otherwise}.
\end{array} \right.
\end{align*}
And, the estimate of $f$, we get by solving
\begin{align*}
\hat{f} =\arg\min_{f\in \mathcal{H}} \left(\frac{1}{n}\sum_{i=1}^n\ell_{0/1}(y_i,f(x_i)) + \lambda \|f\|_{\mathcal{H}}^2\right).
\end{align*}
The minimization problem of minimizing misclassification error with regularization is having various limitation as follows:
\begin{enumerate}
\item[a.] the problem is non-smooth, and typically NP-hard to solve, 
\item[b.] the regularization has no effect since the $0/1$ loss is invariant by scaling of $f$. In fact, no function achieves the minimum when $\lambda >0$.
\end{enumerate}
So, an alternative is to define a probabilistic model of $y \in \{-1,1\}$ parametrized by $f(x)$, as follows:
\begin{align*}
& y\omega^{\top}x=yf(x)=logit\left(E(Y=y|X=x) \right)=logit\left(P(Y=y|X=x)\right) \implies \\
& p(y|x)=logit^{-1}(yf(x))=\sigma(yf(x))=\frac{1}{1+\exp(-yf(x))} \implies  \\
& \ell_{logistic}(y,f(x))=-ln\left(p(y|x) \right)=ln(1+\exp(-yf(x))).
\end{align*}
A few properties of sigmoid-logistic function and logistic loss can be characterized as:
\begin{align*}
& \sigma(u)= \frac{1}{1+exp(-u)}, \sigma(-u)=1-\sigma(u), \sigma'(u)=\sigma(u)\sigma(-u), \\
& \ell_{logistic}(u)=\ln(1+\exp(-u)), \ell'_{logistic}(u)=-\sigma(-u), \ell''_{logistic}(u)=\sigma'(u).
\end{align*}
\begin{center}
\begin{figure}[h]
    \centering
    \includegraphics[width=1.0\textwidth]{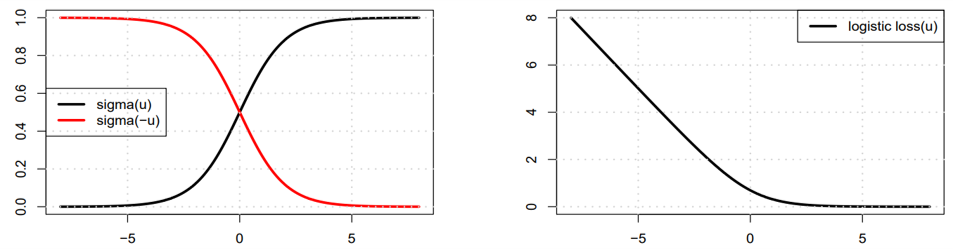}
    \caption{Plots of sigmoid-logistic function and logistic loss.}
    \label{fig:sigm-logit-loss-01}
\end{figure}
\end{center}
Now, the modified estimation problem can be posed as:
\begin{align*}
\hat{f} & =\arg\min_{f\in \mathcal{H}} \left(\frac{1}{n}\sum_{i=1}^n\ell_{logistic}(y_i,f(x_i)) + \lambda \|f\|_{\mathcal{H}}^2\right) \\
       & = \arg\min_{f\in \mathcal{H}} \left(\frac{1}{n}\sum_{i=1}^n \ln\left(1+\exp(-y_if(x_i))\right) + \lambda \|f\|_{\mathcal{H}}^2\right).
\end{align*}
By the representation theorem, the solution of the unconstrained problem can be expanded as:
\begin{align*}
\hat{f}(x)=\sum_{i=1}^n \alpha_i k(x_i,x) \implies \left(\hat{f}(x_1), \ldots ,\hat{f}(x_n) \right)^{\top} = K \alpha , \|f\|_{\mathcal{H}} =\alpha^{\top} K\alpha.
\end{align*}
To estimate $\alpha$, we need to solve the following:
\begin{align*}
\hat{\alpha} & = \arg\min_{\alpha \in \mathbb{R}} \left(\frac{1}{n}\sum_{i=1}^n \ln\left(1+\exp(-y_i(K\alpha)_i)\right) + \lambda \alpha^{\top} K\alpha \right) \\
             & \approx \arg\min_{\alpha \in \mathbb{R}} \left(\frac{1}{n} (K\alpha -z)^\top W (K\alpha -z) + \lambda \alpha^{\top} K\alpha \right).
\end{align*}
The above convex quadratic approximation (with respect to an affine transformation $z=K\alpha_0 -W^{-1}Py$) is as follows:
\begin{align*}
& \nabla J(\alpha) =\frac{1}{n}KP(\alpha)y + \lambda K\alpha, P(\alpha)=diag(P_1(\alpha), \ldots, P_n(\alpha)), \\
& P_i(\alpha)=\ell'_{logistic}(y_f(x_))=-\sigma(-y_if(x_i)) \\
& \nabla^2J(\alpha) = \frac{1}{n} KW(\alpha)K +\lambda K, W(\alpha) =diag(W_1(\alpha),\ldots,W_n(\alpha)), \\
& W_i(\alpha)=\ell"_{logistic}(y_if(x_i))=\sigma(y_if(x_i))\sigma(-y_if(x_i)).
\end{align*}
So, the estimate of $\alpha$ is computed as:
\begin{align*}
& \nabla J_{quadratic}(\alpha)=\frac{2}{n}KW(K\alpha -x)+2\lambda K\alpha = \frac{2}{n}K\left((WK+\lambda nI)-Wz \right)=0 \\
& \implies \alpha - \left(WK+\lambda nI \right)^{-1}Wz \in Ker(K) \implies \alpha = \left(WK+\lambda nI \right)^{-1}Wz + \epsilon, \\ 
& \epsilon \in Ker(K) \implies  \hat{\alpha} = \left(WK+\lambda nI \right)^{-1}Wz \ \mbox{with} \ K\epsilon = 0.
\end{align*}
\end{proof}

\subsection{SoftMax regression: a multinomial logit model}

\section{Future note: control, discovery models, and generation}
In order to capture discrete pattern in the response variable, we need to introduce some kind of non-linear function on it (e.g. logit, multinomial logit etc.), that leads to generalized linear models (i.e. a model having known $\&$ non-linear activation function with  linear argument e.g. logistic regression, multinomial logistic regression etc.). Furthermore, discovery of patterns within input space with a few or no labels (i.e. prediction of labels or response variable via representation learning) lead to discovery models. At the end, the linear models, the generalized linear models and the discovery models can be unified by universal approximation of neural networks.

\subsubsection{A machine learning system}
\begin{enumerate}[{Step}~1:]
\item Methods:-mathematical modelling with scientific meaning, estimation/optimization, optimization algorithm/ numerical methods, generalization/concentration inequalities,  inference/prediction, and numerical methods etc.
\item Techniques:-machine learning algorithms for decision making, control, discovery and generation (e.g. regression, classification, neural networks etc.).
\item Tools: Implementation in computing devices using some platform e.g. Scikit Learn (for machine learning algorithms other than neural networks), TenFlow, PyTorch (neural networks), Keras (deep learning).
\item Deployment in the real world system e.g. a self deriving car etc.
\end{enumerate}

\subsection{A non-linear inverse problem}
xxxxxxxxxxxxxxxxxxxxxxxxxxxxxx.